\pdfoutput=1
\documentclass[12pt]{article}
\usepackage[margin=20mm]{geometry}
\usepackage[T1]{fontenc}
\usepackage[UKenglish]{babel}
\usepackage{indentfirst,graphicx}
\usepackage[usenames,dvipsnames,svgnames,table]{xcolor}
\usepackage{amsfonts,amsmath,amssymb,amsbsy,amsthm,makeidx}
\allowdisplaybreaks
\usepackage{hyperref}
\hypersetup{ 
unicode=true,
colorlinks=true,
linkcolor={black},
citecolor={black},
urlcolor={blue!60!black},
pdftitle={Clustered Graph Coloring and Layered Treewidth},
pdfauthor={Chun-Hung Liu and David~R.~Wood}}
\usepackage{enumitem}
\setlist[itemize]{itemsep=0ex}
\setlist[enumerate]{itemsep=0ex}
\setlist[description]{itemsep=0ex}
\usepackage[longnamesfirst,numbers,sort&compress]{natbib}
\makeatletter
\def\NAT@spacechar{~}
\makeatother
\usepackage[noabbrev,capitalise]{cleveref}
\crefname{lem}{Lemma}{Lemmas}
\crefname{thm}{Theorem}{Theorems}
\crefname{prop}{Proposition}{Propositions}
\crefname{obs}{Observation}{Observations}
\crefname{conj}{Conjecture}{Conjectures}
\crefname{claim}{Claim}{Claims}
\crefformat{equation}{(#2#1#3)}
\Crefformat{equation}{Equation #2(#1)#3}
\newtheorem{theorem}{Theorem}
\newtheorem{lemma}[theorem]{Lemma}
\newtheorem{proposition}[theorem]{Proposition}

\newtheorem{corollary}[theorem]{Corollary}
\newtheorem{claim}{Claim}[theorem]
\makeatletter
\renewcommand\section{\@startsection {section}{1}{\z@}{-3ex \@plus -1ex \@minus -.2ex}{2ex \@plus.2ex}{\normalfont\large\bfseries}}
\renewcommand\subsection{\@startsection{subsection}{2}{\z@}{-2.5ex\@plus -1ex \@minus -.2ex}{1.5ex \@plus .2ex}{\normalfont\normalsize\bfseries}}
\renewcommand\subsubsection{\@startsection{subsubsection}{3}{\z@}{-2ex\@plus -1ex \@minus -.2ex}{1ex \@plus .2ex}{\normalfont\normalsize\bfseries}}
 \renewcommand\paragraph{\@startsection{paragraph}{4}{\z@}{1.5ex \@plus.5ex \@minus.2ex}{-1em}{\normalfont\normalsize\bfseries}}
\renewcommand\subparagraph{\@startsection{subparagraph}{5}{\parindent}  {1.5ex \@plus.5ex \@minus .2ex}  {-1em} {\normalfont\normalsize\bfseries}}
\makeatother
\def\X {{\mathcal X}}
\def\Se {{\mathcal S}}
\def\C {{\mathcal C}}
\def\V {{\mathcal V}}

\def\Q {{\mathcal{Q}}}
\def\L {{\mathcal L}}

\newcommand{\ngs}[2][]{N_{#1}^{\geq s}(#2)}
\newcommand{\nls}[2][]{N_{#1}^{< s}(#2)}

\newcommand{\ceil}[1]{\lceil{#1}\rceil}

\renewcommand{\geq}{\geqslant}
\renewcommand{\leq}{\leqslant}

\renewcommand{\thefootnote}{\fnsymbol{footnote}}
\newcommand{\defn}[1]{\textcolor{Maroon}{\emph{#1}}\index{#1}}
\newcommand{\mathdef}[2]{\textcolor{Maroon}{\emph{#1-#2}}\index{#2@#1-#2}}

\makeindex

\begin{document}

\title{\bf\Large\boldmath Clustered Graph Coloring\\ and Layered Treewidth\footnote{This material is based upon work supported by the National Science Foundation under Grant No.\ DMS-1664593, DMS-1929851, DMS-1954054 and DMS-2144042.}}
\author{%
Chun-Hung Liu\footnote{Department of Mathematics, Texas A\&M University, Texas, USA, \texttt{chliu@math.tamu.edu}. Partially supported by NSF grants under award No.\ DMS-1664593, DMS-1929851, DMS-1954054 and CAREER award DMS-2144042.}
\qquad David R. Wood\footnote{School of Mathematics, Monash University, Melbourne, Australia, \texttt{david.wood@monash.edu}. Research supported by the Australian Research Council.}
}
\date{}
\footnotetext[0]{First released on 22nd May 2019 as part of arXiv:1905.08969. Revised \today.}
\maketitle

\begin{abstract}
A graph coloring has bounded clustering if each monochromatic component has bounded size. This paper studies such a coloring, where the number of colors depends on an excluded complete bipartite subgraph. This is a much weaker assumption than previous works, where typically the number of colors depends on an excluded minor. This paper focuses on graph classes with bounded layered treewidth, which include planar graphs, graphs of bounded Euler genus, graphs embeddable on a fixed surface with a bounded number of crossings per edge, amongst other examples. Our main theorem says that for fixed integers $s,t,k$, every graph with layered treewidth at most $k$ and with no $K_{s,t}$ subgraph is $(s+2)$-colorable with bounded clustering. The $s=3$ case implies that every graph with a drawing on a fixed surface with a bounded number of crossings per edge is 5-colorable with bounded clustering. Our main theorem is also a critical component in two companion papers that study clustered coloring of graphs with no $K_{s,t}$ subgraph and excluding a fixed minor, odd minor or topological minor. 
\end{abstract}

\renewcommand{\thefootnote}{\arabic{footnote}}

\newpage\tableofcontents

\newpage
\section{Introduction}
\label{Introduction}

This paper considers graph colorings where the condition that adjacent vertices are assigned distinct colors is relaxed. Instead, we require that every monochromatic component has bounded size (for a given graph class). 
More formally, a \defn{coloring} of a graph $G$ is a function that assigns one color to each vertex of $G$. A \defn{monochromatic component} with respect to a coloring of $G$ is a connected component of a subgraph of $G$ induced by all the vertices assigned the same color. 
A coloring has \defn{clustering} $\eta$ if every monochromatic component has at most $\eta$ vertices. Our focus is on minimizing the number of colors, with small monochromatic components as a secondary goal. 
The \defn{clustered chromatic number} of a graph class $\mathcal{G}$ is the minimum integer $k$ such that for some integer $\eta$, every graph in $\mathcal{G}$ is $k$-colorable with clustering $\eta$. There have been several recent papers on this topic~\citep{NSSW19,vdHW18,HST03,ADOV03,CE19,Kawa08,KM07,LMST08,KO19,EJ14,EO16,DN17,LO17,HW19,MRW17}; see \citep{WoodSurvey} for a survey.

This paper is one of four companion papers \citep{LW2,LW3,LW4}. The unifying theme that distinguishes this work from previous contributions is that the number of colors is determined by an excluded subgraph. Excluding a subgraph is a much weaker assumption than excluding a minor, which is a typical assumption in previous results. In particular, we consider graphs with no $K_{s,t}$ subgraph plus various other structural properties, and prove that every such graph is colorable with bounded clustering, where the number of colors depends only on $s$. All the dependence on $t$ and the structural property in question is hidden in the clustering function. 

Note that the case $s=1$ is already interesting, since a graph has no $K_{1,t}$ subgraph if and only if it has maximum degree less than $t$. 
Many of our theorems generalize known results \citep{LO17,EJ14,ADOV03} for graphs of bounded maximum degree to the setting of an excluded $K_{s,t}$ subgraph. 

Moreover, for $s\geq 2$, no result with bounded clustering is possible for graphs with no $K_{s,t}$ subgraph, without making some extra assumption. 
In particular, for every graph $H$ that contains a cycle, and for all $k,\eta\in\mathbb{N}$, if $G$ is a graph with chromatic number greater than $k\eta$ and girth greater than $|V(H)|$ (which exists \citep{Erdos59}), then $G$ contains no $H$ subgraph and $G$ is not $k$-colorable with clustering $\eta$, for otherwise $G$ would be $k\eta$-colorable. 

\subsection{Main Results}

In this paper, the ``other structural property'' mentioned above is ``bounded layered treewidth''. First we explain what this means. A \defn{tree-decomposition} of a graph $G$ is a pair $(T,\mathcal{X} = (X_x:x\in V(T))$, where $T$ is a tree,  and for each node $x\in V(T)$,  $X_x$ is a non-empty subset of $V(G)$ called a \defn{bag}, such that for each vertex $v\in V(G)$, the set $\{x\in V(T):v\in X_x\}$ induces a non-empty (connected) subtree of $T$, and for each edge $vw\in E(G)$ there is a node $x\in V(T)$ such that $\{v,w\}\subseteq X_x$. 
The \defn{width} of a tree-decomposition $(T,\mathcal{X})$ is $\max\{|X_x|-1: x\in V(T)\}$. The \defn{treewidth} of a graph $G$ is the minimum width of a tree-decomposition of $G$. Treewidth is a key parameter in algorithmic and structural graph theory; see \citep{Reed03,Reed97,Bodlaender-TCS98,HW17} for surveys.

A \defn{layering} of a graph $G$ is an ordered partition $(V_1,\dots,V_n)$ of $V(G)$ into (possibly empty) sets such that for each edge $vw\in E(G)$ there exists $i\in[1,n-1]$ such that $\{v,w\}\subseteq V_i\cup V_{i+1}$. 
The \defn{layered treewidth} of a graph $G$ is the minimum nonnegative integer $\ell$ such that $G$ has a tree-decomposition $(T, \mathcal{X} = (X_x:x\in V(T)))$ and a layering $(V_1,\dots,V_n)$, such that $|X_x\cap V_i|\leq\ell$ for each bag $X_x$ and layer $V_i$. This says that the subgraph induced by each layer has bounded treewidth, and moreover, a single tree-decomposition of $G$ has bounded treewidth when restricted to each layer. In fact, these properties hold when considering a bounded sequence of consecutive layers.

Layered treewidth was independently introduced by \citet{DMW17} and \citet{Shahrokhi13}. \citet{DMW17} proved that every planar graph has layered treewidth at most 3;  more generally, that every graph with Euler genus\footnote{The \defn{Euler genus} of an orientable surface with $h$ handles is $2h$. The \defn{Euler genus} of  a non-orientable surface with $c$ cross-caps is $c$. The \defn{Euler genus} of a graph $G$ is the minimum Euler genus of a surface in which $G$ embeds (with no crossings).} at most $g$ has layered treewidth at most $2g+3$; and most generally, that a minor-closed class has bounded layered treewidth if and only if it excludes some apex graph as a minor. Layered treewidth is of interest beyond minor-closed classes, since as described in \cref{Crossings}, there are several natural graph classes that have bounded layered treewidth but contain arbitrarily large complete graph minors. 

Consider coloring a graph $G$ with layered treewidth $w$. Say $(V_1,\dots,V_n)$ is the corresponding layering. 
Then $G[ \bigcup_{i \text{ odd}} V_i]$ has treewidth at most $w-1$, and is thus properly $w$-colorable. Similarly, $G[ \bigcup_{i \text{ even}} V_i]$ is properly $w$-colorable. Thus $G$ is properly $2w$-colorable. This bound is best possible since $K_{2w}$ has layered treewidth $w$. 
In fact, for every integer $d$ and integer $w$, there is a graph $G$ with treewidth $2w-1$, such that every $(2w-1)$-coloring of $G$ has a vertex of monochromatic degree at least $d$, implying there is a monochromatic component with more than $d$ vertices \citep{WoodSurvey}. \citet{BDDEW19} observed that every graph with treewidth $k$ has layered treewidth at most $\ceil{\frac{k+1}{2}}$ (using two layers), so the aforementioned graph $G$ has layered treewidth at most $w$. This says that the clustered chromatic number of the class of graphs with layered treewidth $w$ equals $2w$, and indeed, every such graph is properly $2w$-colorable. 

On the other hand, we prove that for clustered coloring of graphs excluding a $K_{s,t}$ subgraph in addition to having bounded layered treewidth, only $s+2$ colors are needed (no matter how large the upper bound on layered treewidth). This is the main result of the paper. 

\begin{theorem} \label{ltwbasic}
For all $s,t,w\in\mathbb{N}$ there exists $\eta\in\mathbb{N}$ such that every graph with layered treewidth at most $w$ and with no $K_{s,t}$ subgraph is $(s+2)$-colorable with clustering $\eta$. 
\end{theorem}

The number of colors in \cref{ltwbasic} is nearly optimal. At least $s+1$ colors are required since for all $s,\eta \in \mathbb{N}$, there exists a graph with treewidth at most $s$ (and hence with bounded layered treewidth) and with no $K_{s,s+2}$ subgraph such that no $s$-coloring has clustering $\eta$ \citep{WoodSurvey}. When $s=1$, $s+2=3$ colors are required, since the Hex Lemma~\citep{Gale79} says that every 2-coloring of the $n\times n$ planar triangular grid (which has layered treewidth 2, maximum degree 6, and thus contains no $K_{1,7}$ subgraph) contains a monochromatic path of length at least $n$.
Note that the $s=1$ case of \cref{ltwbasic} can be proved by a much simpler argument with much improved bounds on the clustering~\citep{LW4, DEMWW22}.

Further motivation for \cref{ltwbasic} is that it is a critical ingredient in the proofs of results in our companion papers~\citep{LW2,LW3} about clustered colorings of graphs excluding a minor, odd minor, or topological minor. These results can be viewed as clustered analogues of weakenings of Hadwiger's Conjecture, the Gerard--Seymour Conjecture, and Haj\'os' Conjecture. 
For example, in \cite{LW3} we prove that every graph with no $K_{t}$-minor is $t$-colorable with bounded clustering and every graph with no $K_t$-topological minor is $(4t-9)$-colorable with bounded clustering.
The former is the currently best result for $K_t$-minor free graphs, and it is known that the number of colors cannot be reduced to $t-2$; the latter is the currently best result for $K_t$-topological minor free graphs and the only known linear bound for the number of colors.
We refer readers to \cite{LW2,LW3} for other results.

In fact, for our applications in \cite{LW2,LW3}, we actually need the following stronger result  proved in \cref{AllowingApexVertices}  (where the $\xi=0$ case is \cref{ltwbasic}). 

\begin{theorem} 
\label{ltwmain}
For all $s,t,w\in\mathbb{N}$ and $\xi \in \mathbb{N}_0$, there exists $\eta\in\mathbb{N}$ such that if $G$ is a graph with no $K_{s,t}$ subgraph and $G-Z$ has layered treewidth at most $w$ for some $Z \subseteq V(G)$ with $\lvert Z \rvert \leq \xi$, then $G$ is $(s+2)$-colorable with clustering $\eta$.
\end{theorem}

Note that \cref{ltwmain} is stronger than \cref{ltwbasic}, so at least $s+1$ colors are required and the number of colors in \cref{ltwmain} is nearly optimal.
In addition, it is optimal when $s\leq 2$: Suppose that the theorem holds with $s=2$, $t=7$, $w=2$ and $\xi=1$, but with only 3 colors. Let $n\gg \eta^2$. Let $G$ be obtained from the $n\times n$ triangular grid by adding one dominant vertex $v$. Then $G$ contains no $K_{2,7}$ subgraph, and $G-v$ has layered treewidth 2. By assumption, $G$ is 3-colorable with clustering $\eta$. Say $v$ is blue. Since $v$ is dominant, at most $\eta$ rows or columns contain a blue vertex. The non-blue induced subgraph contains an $\eta\times\eta$ triangular grid (since $n\gg\eta^2$). This contradicts the Hex Lemma mentioned above. 

We remark that the class of graphs mentioned in \cref{ltwmain} is more general than the class of graphs with bounded layered treewidth. For example, the class of graphs that can be made planar (and hence bounded layered treewidth) by deleting one vertex contains graphs of arbitrarily large layered treewidth.

\subsection{Applications}
\label{Crossings}

We now give several examples of graph classes with bounded layered treewidth, for which \cref{ltwbasic} gives interesting results. 

\subsubsection*{\boldmath $(g,k)$-Planar Graphs}

A graph is \mathdef{$(g,k)$}{planar} if it can be drawn in a surface of Euler genus at most $g$ with at most $k$ crossings on each edge (assuming no three edges cross at a single point). 
Note that $(g,k)$-planar graphs can contain arbitrarily large complete graph minors, even in the $g=0$ and $k=1$ case \cite{DEW17}. 

\citet{DEW17} proved that every $(g,k)$-planar graph has layered treewidth at most $(4g + 6)(k + 1)$. Furthermore, \citet{OOW19} proved that every $(g,k)$-planar graph contains no $K_{3,t}$ subgraph for some $t=O(kg^2)$. Thus \cref{ltwbasic} with $s=3$ implies:

\begin{corollary}
\label{gkPlanar}
For all $g,k\in\mathbb{N}_0$, there exists $\eta\in\mathbb{N}$, such that every $(g,k)$-planar graph is 5-colorable with clustering $\eta$. 
\end{corollary}

\cref{gkPlanar} highlights the utility of excluding a $K_{s,t}$ subgraph. 
It improves an earlier result of \citet{WoodSurvey} who proved that $(g,k)$-planar graphs are 12-colorable with clustering bounded by a function of $g$ and $k$.
It also generalizes a theorem of Esperet and Ochem~\cite{EO16} who proved the $k=0$ case, which says that every graph with bounded Euler genus is 5-colorable with bounded clustering. Note that \citet{DN17} recently proved that every graph with bounded Euler genus is 4-colorable (in fact, 4-choosable) with bounded clustering. It is open whether every $(g,k)$-planar graph is 4-colorable with bounded clustering.

\subsubsection*{Map Graphs}
\label{MapGraphs}

Map graphs are defined as follows. Start with a graph $G_0$ embedded in a surface of Euler genus $g$, with each face labelled a ``nation'' or a ``lake'', where each vertex of $G_0$ is incident with at most $d$ nations. Let $G$ be the graph whose vertices are the nations of $G_0$, where two vertices are adjacent in $G$ if the corresponding faces in $G_0$ share a vertex. Then $G$ is called a \mathdef{$(g,d)$}{map graph}. A $(0,d)$-map graph is called a \emph{(plane)} \mathdef{$d$}{map graph}; such graphs have been extensively studied \citep{FLS-SODA12,Chen07,DFHT05,CGP02,Chen01}. The $(g,3)$-map graphs are precisely the graphs of Euler genus at most $g$ (see \citep{CGP02,DEW17}). 
So $(g,d)$-map graphs provide a natural generalization of graphs embedded in a surface that allows for arbitrarily large cliques even in the $g=0$ case (since if a vertex of $G_0$ is incident with $d$ nations then $G$ contains $K_d$). 

\citet{DEW17} proved that every $(g,d)$-map graph has layered treewidth at most $(2g+3)(2d+1)$. They also proved that every $(g,d)$-map graph is $(g,O(d^2))$-planar.  Thus \cref{gkPlanar} implies:

\begin{corollary}
\label{MapGraph5Color}
For all $g,d\in\mathbb{N}$, there exists $\eta\in\mathbb{N}$, such that every $(g,d)$-map graph is 5-colorable with clustering $\eta$. 
\end{corollary}

It is straightforward to prove that, in fact, if a $(g,d)$-map graph $G$ contains a $K_{3,t}$ subgraph, then $t\leq 6d(g+1)$. 
Thus \cref{MapGraph5Color} is also implied by \cref{ltwbasic} directly with $s=3$, and the obtained clustering bounds is slightly improved than the bound obtained by using \cref{gkPlanar}. 

It is open whether every $(g,d)$-map graph is 4-colorable with clustering bounded by a function of $g$ and $d$.

\subsection{Notation}

Let $G$ be a graph with vertex-set $V(G)$ and edge-set $E(G)$. 
For $v\in V(G)$, let $N_G(v):=\{w\in V(G): vw\in E(G)\}$ be the neighborhood of $v$, and let  $N_G[v] :=N_G(v) \cup\{v\}$. 
For $X\subseteq V(G)$, let $N_G(X):= \bigcup_{v\in X} (N_G(v) \setminus X)$ and $N_G[X]:= N_G(X)\cup X$.
Denote the subgraph of $G$ induced by $X$ by $G[X]$.
Let $\mathbb{N}:=\{1,2,\dots\}$ and $\mathbb{N}_0:=\{0,1,2,\dots\}$. 
For $m,n\in\mathbb{N}_0$, let $[m,n]:=\{m,m+1,\dots,n\}$ and $[n]:=[1,n]$.  
Denote the set of all positive real numbers by ${\mathbb R}^+$.
For a coloring $c$ of $G$, a \mathdef{$c$}{monochromatic component} is a monochromatic component with respect to $c$.

\subsection{Organization of the Paper}

We prove (a stronger form of) \cref{ltwbasic} in \cref{MainSection}, assuming the key lemma, \cref{no apex 0}. This lemma is the most complicated part of the paper and is proved in \cref{sec:mainlemma}. An overview of the structure of \cref{sec:mainlemma} is given at the start of that section.

\cref{sec:bddnbrhood,sec:fences} are preparations for proving \cref{no apex 0}. 
 \cref{sec:bddnbrhood} shows that for every subset $X$ of vertices of a graph with no $K_{s,t}$ subgraph, the number of vertices outside $X$ but adjacent to at least $s$ vertices in $X$ is bounded. This helps control the growth of monochromatic components each time we extend a precoloring.  \cref{sec:fences} introduces fences, parades and fans, which are notions defined on tree-decompositions that will be used in the proof of \cref{no apex 0}. These tools are inspired by Robertson and Seymour's separator theory for graphs of bounded treewidth. 

\cref{AllowingApexVertices} proves \cref{ltwmain} using the stronger form of \cref{ltwbasic} proved in \cref{MainSection}. We conclude by mentioning some open problems in \cref{sec:openproblems}.

\section{Bounded Neighborhoods}
\label{sec:bddnbrhood}

We now set out to prove \cref{ltwbasic}. Consider a graph $G$ with no $K_{s,t}$ subgraph for fixed $s$ and $t$. In the $s=1$ case, $G$ has bounded maximum degree, so if $X$ is a bounded-size set of vertices in $G$, then $N_G(X)$ also has bounded size. This fails for $s\geq 2$, since vertices in $X$ might have unbounded degree. To circumvent this issue we focus on those vertices with at least $s$ neighbors in $X$, and then show that there are a bounded number of such vertices. The following notation formalizes this simple but important idea. For a graph $G$, a set $X\subseteq V(G)$, and $s\in\mathbb{N}$, define
\begin{align*}
\ngs[G]{X} & := \{v \in V(G)-X:  \lvert N_G(v) \cap X \rvert \geq s\} \text{ and}\\
\nls[G]{X} & := \{v \in V(G)-X:  1 \leq \lvert N_G(v) \cap X \rvert < s\} .
\end{align*}
When the graph $G$ is clear from the context, we write $\ngs{X}$ instead of $\ngs[G]{X}$ and $\nls{X}$ instead of $\nls[G]{X}$. 

\begin{lemma} 
\label{BoundedGrowth}
For all $s,t\in\mathbb{N}$, there exists a function $f_{s,t}:\mathbb{N}_0\rightarrow\mathbb{N}_0$ such that for every graph $G$ with no $K_{s,t}$ subgraph, if $X\subseteq V(G)$ then $| \ngs{X} |  \leq f_{s,t}(|X|)$.
\end{lemma}

\begin{proof}
Define $f_{s,t}(x)={x \choose s}(t-1)$ for every $x\in\mathbb{N}_0$.
For every $y \in N^{\geq s}(X)$, let $Z_y$ be a subset of $N_G(y) \cap X$ with size $s$.
Since there are exactly ${\lvert X \rvert} \choose s$ subsets of $X$ with size $s$, if $\lvert N^{\geq s}(X) \rvert > f_{s,t}(\lvert X \rvert)$, then there exists a subset $Y$ of $N^{\geq s}(X)$ with size $t$ such that $Z_y$ is identical for all $y \in Y$, implying $G[Y \cup \bigcup_{y \in Y}Z_y]$ contains a $K_{s,t}$ subgraph, which is a contradiction. Thus $| N^{\geq s}(X) | \leq f_{s,t}(|X|)$. 
\end{proof}

The function $f$ in \cref{BoundedGrowth} can be improved if we know more about the graph $G$, which helps to get better upper bounds on the clustering. 
A \mathdef{$1$}{subdivision} of a graph $H$ is the graph obtained from $H$ by subdividing every edge exactly once. For a graph $G$, let $\nabla(G)$ be the maximum average degree of a graph $H$ for which the 1-subdivision of $H$ is a subgraph of $G$. The following result is implicit in \cite{OOW19}. We include the proof for
completeness.

\begin{lemma} \label{nabla}
For all $s,t\in\mathbb{N}$ and positive $\nabla \in \mathbb{R}^+$ there is a number $c:=\max\{t-1,\frac{\nabla}{2} + (t-1) \binom{\nabla}{s-1}\}$, such that for every graph $G$ with no $K_{s,t}$-subgraph and with $\nabla(G)\leq\nabla$, if $X\subseteq V(G)$ then $| N^{\geq s}(X) |  \leq c \lvert X \rvert$.
\end{lemma}

\begin{proof}
In the case $s=1$, the proof of \cref{BoundedGrowth} implies this lemma since $c \geq t-1$. 
Now assume that $s\geq 2$. 
Let $H$ be the bipartite graph with bipartition $\{N^{\geq s}_G(X),\binom{X}{2}\}$, where $v\in N^{\geq s}_G(X)$ is adjacent in $H$ to $\{x,y\}\in\binom{X}{2}$ whenever $x,y\in N_G(v)\cap X$. 
Let $M$ be a maximal matching in $H$. 
Let $Q$ be the graph with vertex-set $X$, where $xy\in E(Q)$ whenever $\{v,\{x,y\}\}\in M$ for some vertex $v\in N^{\geq s}_G(X)$. 
Thus, the 1-subdivision of every subgraph of $Q$ is a subgraph of $G$. 
Hence $|M| =  |E(Q)| \leq\frac{\nabla}{2}|V(Q)| = \frac{\nabla}{2} |X|$. 
Moreover, $Q$ is $\nabla$-degenerate, implying $Q$ contains at most $\binom{\nabla}{s-1}|X|$ cliques of size exactly $s$. 
Exactly $|M|$ vertices in $N^{\geq s}_G(X)$ are incident with an edge in $M$. 
For each vertex $v\in N^{\geq s}_G(X)$ not incident with an edge in $M$, by maximality, $N_G(v)\cap X$ is a clique in $Q$ of size at least $s$.  
Define a mapping from each vertex $v \in N^{\geq s}_G(X)$ to a clique of size exactly $s$ in $Q[N_G(v)\cap X]$.  Since $G$ has no $K_{s,t}$ subgraph, at most $t-1$ vertices $v\in N^{\geq s}_G(X)$ are mapped to each fixed $s$-clique in $Q$. Hence $|N^{\geq s}_G(X)| \leq \frac{\nabla}{2}|X| + (t-1) \binom{\nabla}{s-1}|X| \leq c|X|$. 
\end{proof}

\cref{nabla} is applicable in many instances. 
In particular, every graph $G$ with treewidth $k$ has $\nabla(G)\leq 2k$ (since if a 1-subdivision of some graph $G'$ is a subgraph of $G$, then $G'$ has treewidth at most $k$, and every graph with treewidth at most $k$ has average degree less than $2k$). 
By an analogous argument, using bounds on the average degree independently due to \citet{Thomason84} and \citet{Kostochka82}, every $H$-minor-free graph $G$ has $\nabla(G)\leq O(|V(H)| \sqrt{\log |V(H)|})$. Similarly, using bounds on the average degree independently due to \citet{KS96} and \citet{BT98}, every $H$-topological-minor-free graph $G$ has $\nabla(G)\leq O(|V(H)|^2)$. Finally, \citet[Lemmas 8,9]{DMW17} proved that $\nabla(G) \leq 5w$ for every graph with layered treewidth $w$. In all these cases, \cref{nabla} implies that the function $f$ in \cref{BoundedGrowth} can be made linear.

\begin{corollary} \label{BoundedGrowthLtw}
For all $s,t,w\in\mathbb{N}$, let $f_{s,t,w}:\mathbb{N}_0\rightarrow \mathbb{Q}$ be the function defined by $f_{s,t,w}(x) :=(\frac{5w}{2} + (t-1) \binom{5w}{s-1})x$ for every $x \in \mathbb{N}_0$. 
Then for every graph $G$ with no $K_{s,t}$-subgraph and with layered treewidth at most $w$, if $X$ is a subset of $V(G)$, then $\lvert N^{\geq s}(X) \rvert \leq f_{s,t,w}(\lvert X \rvert)$
\end{corollary}

\begin{proof}
\citet[Lemmas 8,9]{DMW17} showed that $\nabla(G) \leq 5w$ for every graph with layered treewidth $w$. 
So this result immediately follows from \cref{nabla}.
\end{proof}

\section{Fences in a Tree-decomposition}
\label{sec:fences}

This section introduces the notion of a fence, which is used in the proof of \cref{no apex 0}. The main result is \cref{fence lemma}. We start with a variant of a well-known result about separators in graphs of bounded treewidth.

\begin{lemma} \label{one node sep}
Let $w \in \mathbb{N}_0$. 
Let $G$ be a graph and $(T,\X)$ a tree-decomposition of $G$ of width at most $w$, 
where $\X=(X_t: t \in V(T))$.
If $Q$ is a subset of $V(G)$ with $\lvert Q \rvert \geq 12w+13$, then there exists $t^* \in V(T)$ such that for every component $T'$ of $T-t^*$, $\lvert (Q \cap (\bigcup_{t \in V(T')}X_t)) \cup X_{t^*} \rvert < \frac{2}{3}\lvert Q \rvert$.
\end{lemma}

\begin{proof}
Suppose that there exists an edge $xy$ of $T$ such that $\lvert (Q \cap (\bigcup_{t \in V(T_i)}X_t)) \cup X_x \cup X_y \rvert \geq \frac{2}{3}\lvert Q \rvert$ for each $i \in [2]$, where $T_1,T_2$ are the components of $T-xy$.
Then 
$$\lvert Q \rvert + 2\lvert X_x \cup X_y \rvert \geq \sum_{i=1}^2 \lvert (Q \cap (\bigcup_{t \in V(T_i)}X_t)) \cup X_x \cup X_y \rvert \geq \frac{4}{3} \lvert Q \rvert.$$
Since the width of $(T,\X)$ is at most $w$, $4(w+1) \geq 2\lvert X_x \cup X_y \rvert \geq \frac{1}{3}\lvert Q \rvert \geq 4w+\frac{13}{3}$, a contradiction. 

First assume that there exists an edge $xy$ of $T$ such that $\lvert (Q \cap (\bigcup_{t \in V(T_i)}X_t)) \cup X_x \cup X_y \rvert < \frac{2}{3}\lvert Q \rvert$ for each $i \in [2]$, where $T_1,T_2$ are the components of $T-xy$.
Let $t^*:=x$.
Then for every component $T'$ of $T-t^*$, $(Q \cap (\bigcup_{t \in V(T')}X_t)) \cup X_{t^*} \subseteq (Q \cap (\bigcup_{t \in V(T_i)}X_t)) \cup X_x \cup X_y$ for some $i \in [2]$, and hence $\lvert (Q \cap (\bigcup_{t \in V(T')}X_t)) \cup X_{t^*} \rvert < \frac{2}{3}\lvert Q \rvert$.
So the lemma holds.

Now assume that for every edge $xy$ of $T$, there exists a unique $r \in \{x,y\}$ such that $\lvert (Q \cap (\bigcup_{t \in V(T_r)}X_t)) \cup X_x \cup X_y \rvert \geq \frac{2}{3}\lvert Q \rvert$, where $T_x,T_y$ are the components of $T-xy$ containing $x,y$, respectively. Orient the edge $xy$ so that $r$ is the head of this edge.
We obtain an orientation of $T$.
Note that the sum of the out-degree of the nodes of $T$ equals $\lvert E(T) \rvert=\lvert V(T) \rvert-1$.
So some node $t^*$ has out-degree 0.

For each component $T'$ of $T-t^*$, let $t_{T'}$ be the node in $T'$ adjacent in $T$ to $t^*$. By the definition of the direction of $t_{T'}t^*$, 
\begin{align*}
\lvert (Q \cap (\bigcup_{t \in V(T_{t_{T'} } )} X_t)) \cup X_{t^*} \rvert \leq \lvert (Q \cap (\bigcup_{t \in V(T_{t_{T'}}) } X_t)) \cup X_{t_{T'}} \cup X_{t^*} \rvert < \frac{2}{3}\lvert Q \rvert.
\end{align*}
Note that $T_{t_{T'}}=T'$ for every component $T'$ of $T-t^*$, which completes the proof.
\end{proof}

Let $T$ be a tree and $F$ a subset of $V(T)$.
An \mathdef{$F$}{part} of $T$ is an induced subtree of $T$ obtained from a component of $T-F$ by adding the nodes in $F$ adjacent in $T$ to this component.
For an $F$-part $T'$ of $T$, define $\partial T'$ to be $F \cap V(T')$.

\begin{lemma} \label{good tree part}
Let $(T,\X)$ be a tree-decomposition of a graph $G$ of width at most $w\in \mathbb{N}_0$, where $\X=(X_t: t \in V(T))$.
Then for every $\epsilon \in \mathbb{R}$ with $1 \geq \epsilon \geq \frac{1}{w+1}$ and every $Q \subseteq V(G)$, there exists a subset $F$ of $V(T)$ with $\lvert F \rvert \leq \max\{\epsilon(\lvert Q \rvert-3w-3),0\}$ such that for every $F$-part $T'$ of $T$, $\lvert (\bigcup_{t \in V(T')}X_t \cap Q) \cup \bigcup_{t \in \partial T'}X_t \rvert \leq \frac{1}{\epsilon}\cdot (12w+13)$.
\end{lemma}

\begin{proof}
We shall prove this lemma by induction on $\lvert Q \rvert$.
If $\lvert Q \rvert \leq \frac{1}{\epsilon}(12w+13)$, then let $F:=\emptyset$; then for every $F$-part $T'$ of $T$, $\lvert (\bigcup_{t \in V(T')}X_t \cap Q) \cup \bigcup_{t \in \partial T'}X_t \rvert = \lvert Q \rvert \leq \frac{1}{\epsilon}(12w+13)$.
So we may assume that $\lvert Q \rvert \geq \frac{1}{\epsilon}(12w+13) \geq 12w+13$ and the lemma holds for all $Q$ with smaller size.

By \cref{one node sep}, there exists $t^* \in V(T)$ such that for every component $T'$ of $T-t^*$, $\lvert (Q \cap (\bigcup_{t \in V(T')}X_t)) \cup X_{t^*} \rvert < \frac{2}{3}\lvert Q \rvert$.
That is, for every $\{t^*\}$-part $T'$ of $T$, $\lvert (\bigcup_{t \in V(T')}X_t \cap Q) \cup \bigcup_{t \in \partial T'}X_t \rvert <\frac{2}{3}\lvert Q \rvert$.

For each $\{t^*\}$-part $T'$ of $T$, 
let $Q_{T'}:=(\bigcup_{t \in V(T')}X_t \cap Q) \cup X_{t^*}$, so $\lvert Q_{T'} \rvert <\frac{2}{3}\lvert Q \rvert$.
Let $T_1,T_2,\dots,T_k$ be the $\{t^*\}$-parts $T'$ of $T$ with $\lvert Q_{T'} \rvert \geq \frac{1}{\epsilon}(12w+13)$.

Let $f$ be the function defined by $f(x):=\max\{\epsilon(x-3w-3),0\}$ for every $x\in\mathbb{R}$.
For each $i \in [k]$, since $\lvert Q_{T_i} \rvert < \frac{2}{3}\lvert Q \rvert < \lvert Q \rvert$, the induction hypothesis implies that there exists $F_i \subseteq V(T_i)$ with $\lvert F_i \rvert \leq f(\lvert Q_{T_i} \rvert)$ such that for every $F_i$-part $T'$ of $T_i$, $\lvert (\bigcup_{t \in V(T')}X_t \cap Q_{T_i}) \cup \bigcup_{t \in \partial T'}X_t \rvert \leq \frac{1}{\epsilon}(12w+13)$.

Define $F= \{t^*\} \cup \bigcup_{i=1}^k F_i$.
Note that for every $F$-part $T'$ of $T$, either $T'$ is a $\{t^*\}$-part of $T$ with $\lvert Q_{T'} \rvert \leq \frac{1}{\epsilon}(12w+13)$, or there exists $i \in [k]$ such that $T'$ is an $F_i$-part of $T_i$.
In the former case, $\lvert (\bigcup_{t \in V(T')}X_t \cap Q) \cup \bigcup_{t \in \partial T'}X_t \rvert = \lvert Q_{T'} \rvert \leq \frac{1}{\epsilon}(12w+13)$. In the latter case, $\lvert (\bigcup_{t \in V(T')}X_t \cap Q) \cup \bigcup_{t \in \partial T'}X_t \rvert \leq \lvert (\bigcup_{t \in V(T')}X_t \cap Q_{T_i}) \cup \bigcup_{t \in \partial T'}X_t \rvert \leq \frac{1}{\epsilon}(12w+13)$ since $X_{t^*} \subseteq Q_{T_i}$.
Hence $\lvert (\bigcup_{t \in V(T')}X_t \cap Q) \cup \bigcup_{t \in \partial T'}X_t \rvert \leq \frac{1}{\epsilon}(12w+13)$ for every $F$-part $T'$ of $T$.

To prove this lemma, it suffices to prove that $\lvert F \rvert \leq f(\lvert Q \rvert)$.
Note that $\lvert F \rvert  \leq 1+\sum_{i=1}^k \lvert F_i \rvert \leq 1+\sum_{i=1}^kf(\lvert Q_{T_i} \rvert)$.
Since $\lvert Q_{T_i} \rvert \geq \frac{1}{\epsilon}(12w+13) \geq 12w+13$ for every $i \in [k]$, $\lvert F \rvert \leq 1+\sum_{i=1}^k \epsilon(\lvert Q_{T_i} \rvert - 3w-3)$.

If $k=0$, then $f(\lvert Q \rvert)=\epsilon(\lvert Q \rvert-3w-3) = \epsilon\lvert Q \rvert - \epsilon(3w+3) \geq 12w+13-(3w+3)>1=\lvert F \rvert$ since $\epsilon\lvert Q \rvert \geq 12w+13$ and $\epsilon \leq 1$.
If $k=1$, then $\lvert F \rvert \leq 1+\epsilon(\lvert Q_{T_1} \rvert-3w-3) \leq 1 +\epsilon(\frac{2}{3}\lvert Q \rvert-3w-3) =\epsilon(\lvert Q \rvert-3w-3) +1-\frac{\epsilon}{3}\lvert Q \rvert \leq f(\lvert Q \rvert)$ since $\epsilon\lvert Q \rvert \geq 12w+13$.
Hence we may assume that $k \geq 2$.
Then 
\begin{align*}
\lvert F \rvert 
& \leq 1+\sum_{i=1}^k \epsilon(\lvert Q \cap \bigcup_{t \in V(T_i)-\{t^*\}}X_t \rvert + \lvert X_{t^*} \rvert - 3w-3)\\
& \leq 1+\epsilon\lvert Q \rvert + k\epsilon(\lvert X_{t^*} \rvert - 3w-3)\\ 
& = \epsilon(\lvert Q \rvert - 3w-3) + 1+k\epsilon\lvert X_{t^*} \rvert-(k-1)\epsilon(3w+3) \\
& \leq f(\lvert Q \rvert) + 1+\epsilon(k(w+1)-3(k-1)(w+1))\\
& \leq f(\lvert Q \rvert)
\end{align*} since $k \geq 2$.
This completes the proof.
\end{proof}

\begin{lemma} \label{fence lemma}
Let $(T,\X)$ be a tree-decomposition of a graph $G$ of width at most $w \in \mathbb{N}_0$, where $\X=(X_t: t \in V(T))$.
Then for every $Q \subseteq V(G)$, there exists a subset $F$ of $V(T)$ with $\lvert F \rvert \leq \max\{\lvert Q \rvert-3w-3,0\}$ such that:
	\begin{enumerate}
		\item for every $F$-part $T'$ of $T$, $\lvert \bigcup_{t \in V(T')}X_t \cap Q \rvert \leq 12w+13$, and 
		\item if $Q \neq \emptyset$, then for every $t^* \in F$, there exist at least two $F$-parts $T'$ of $T$ satisfying $t^* \in \partial T'$ and $Q \cap \bigcup_{t \in V(T')}X_t - X_{t^*} \neq \emptyset$. 
	\end{enumerate}
\end{lemma}

\begin{proof}
Let $Q \subseteq V(G)$.
By \cref{good tree part} with $\epsilon=1$, there exists $F \subseteq V(T)$ with $\lvert F \rvert \leq \max\{\lvert Q \rvert-3w-3,0\}$ such that for every $F$-part $T'$ of $T$, $\lvert \bigcup_{t \in V(T')}X_t \cap Q \rvert \leq 12w+13$. Assume further that $F$ is minimal.

Suppose that $Q \neq \emptyset$ and there exists $t^* \in F$ such that there is at most one $F$-part $T'$ of $T$ satisfying $t^* \in \partial T'$ and $Q \cap \bigcup_{t \in V(T')}X_t-X_{t^*} \neq \emptyset$.
Let $F^* := F-\{t^*\}$.
Note that for every $F$-part $W$ of $T$ with $t^* \not \in \partial W$, $W$ is an $F^*$-part of $T$ and $\partial W \subseteq F-\{t^*\}$, so $\lvert \bigcup_{t \in V(W)}X_t \cap Q \rvert \leq 12w+13$.
In addition, there is exactly one $F^*$-part $T^*$ of $T$ with $t^* \in V(T^*)$, and every $F^*$-part of $T$ other than $T^*$ is an $F$-part of $T$.
Since there is at most one $F$-part $T'$ of $T$ satisfying $t^* \in \partial T'$ and $Q \cap \bigcup_{t \in V(T')}X_t-X_{t^*} \neq \emptyset$, if $T'$ exists, then $\lvert \bigcup_{t \in V(T^*)}X_t \cap Q \rvert = \lvert \bigcup_{t \in V(T')}X_t \cap Q \rvert \leq 12w+13$; otherwise $\lvert \bigcup_{t \in V(T^*)}X_t \cap Q \rvert \leq \lvert X_{t^*} \rvert \leq 12w+13$.
This contradicts the minimality of $F$ and proves the lemma.
\end{proof}

We call the set $F$ mentioned in \cref{fence lemma} a \mathdef{$(T,\X,Q)$}{fence}.

\subsection{Parades and Fans} \label{subsec:paradefans}

This subsection introduces parades and fans, and proves \cref{paradefans} which is used in the proof of \cref{no apex 0} below. 

Let $(T,\X)$ be a rooted tree-decomposition of a graph $G$ of width $w$, where $\X=(X_t: t \in V(T))$.
For every note $t''$ of $T$, let $T_{t''}$ be the maximal subtree of $T$ rooted at $t''$.
Let $t,t'$ be distinct nodes of $T$ with $t' \in V(T_t)$.
For $k \in [0,w+1]$ and $m \in \mathbb{N}_0$, a \mathdef{$(t,t',k)$}{fan} of size $m$ is a sequence $(t_1,t_2,\dots,t_m)$ of nodes of $T_t$ such that:
	\begin{description}
		\item{(FAN1)} for every $j \in [m-1]$, $t_{j+1} \in V(T_{t_j})-\{t_j\}$ and $t' \in V(T_{t_j})$, 
		\item{(FAN2)} for every $j \in [m]$, $\lvert X_{t_j} \cap X_t \rvert = k$, and 
		\item{(FAN3)} $X_{t_j}-X_t$ are pairwise disjoint for all $j \in [m]$.
	\end{description}
Note that we do not require $t' \in V(T_{t_m})$ in (FAN1); in fact, having this requirement or not does not affect our future arguments much because we can drop the last term in a $(t,t',k)$-fan to obtain a $(t,t',k-1)$-fan of size one less satisfying this requirement.
Moreover, by the definition of a tree-decomposition, (FAN2) implies that the sets $X_{t_j} \cap X_t$ are identical for all $j \in [m]$.

Let $T$ be a rooted tree. A \defn{parade} in $T$ is a sequence $(t_1,t_2,\dots,t_k)$ of nodes of $T$ (for some $k \in {\mathbb N})$ such that $t_{\alpha+1} \in V(T_{t_\alpha})-\{t_\alpha\}$ for every $\alpha \in [k-1]$. 
Note that every $(t,t',k)$-fan is a parade by (FAN1).

A sequence $(a_1,a_2,\dots,a_\alpha)$ is a \defn{subsequence} of a sequence $(b_1,b_2,\dots,b_\beta)$ (for some $\alpha,\beta\in\mathbb{N}$) if there exists an injection $\iota: [\alpha] \to [\beta]$ such that $\iota(1)<\iota(2)<\dots<\iota(\alpha)$ and $a_i = b_{\iota(i)}$ for every $i \in [\alpha]$.

\begin{lemma} \label{paradefans}
For every $w \in {\mathbb N}_0$ and $k \in {\mathbb N}$ with $k \geq 2$, there exists $N:=N(w,k) \in {\mathbb N}$ such that for every rooted tree-decomposition $(T,\X)$ of a graph $G$, where $\X=(X_t: t \in V(T))$, and for every parade $(t_1,t_2,\dots,t_N)$ in $T$ with $\lvert X_{t_i} \rvert \leq w+1$ for every $i \in [N]$, if $X_{t_i} \not \subseteq X_{t_j}$ for every $i > j$, then there exist $\ell \in [0,w]$ and a $(t_1',t_k',\ell)$-fan $(t_1',t_2',\dots,t_k')$ of size $k$ that is a subsequence of $(t_1,t_2,\dots,t_N)$.
\end{lemma}

\begin{proof}
Define $N(0,k) := \max\{k,2\}$, and for every $x \in [w]$, define $N(x,k): = N(x-1,k) \cdot (k-1)(w+1)$.

We proceed by induction on $w$.
When $w=0$, since $X_{t_i} \not \subseteq X_{t_j}$ for every $i > j$, $(t_1,t_2,\dots,t_N)$ is a $(t_1,t_N,0)$-fan of size $N \geq k$.
So we may assume that $w \geq 1$ and this lemma holds for every smaller $w$.

Suppose to the contrary that for every $\ell \in [0,w]$, there exists no $(t_1',t_k',\ell)$-fan $(t_1',t_2',\dots,t'_k)$ of size $k$ that is a subsequence of $(t_1,t_2,\dots,t_N)$.
In particular, there exists no $k$ elements in $\{t_1,t_2,\dots,t_N\}$ with pairwise disjoint bags.
By greedily collecting nodes in $\{t_1,t_2,\dots, t_N\}$ with pairwise disjoint bags starting from $t_1$, we obtain a parade $(q_1,q_2,\dots,q_{k'})$ that is a subsequence of $(t_1,t_2,\dots,t_N)$ for some $k' \in [k-1]$ such that for every $i \in [N]$, there exists $i' \in [k']$ such that $X_{t_i} \cap X_{q_{i'}} \neq \emptyset$ and $t_i \in V(T_{q_{i'}})$.
Since $N \geq k-1$, we may assume $k'=k-1$ by adding redundant nodes into $(q_1,q_2,\dots,q_{k'})$.
For each $i \in [k-1]$, let $S_i := \{t_j \in V(T_{q_i}): j \in [N], X_{t_j} \cap X_{q_i} \neq \emptyset\}$.
So there exists $i_1 \in [k-1]$ such that $\lvert S_{i_1} \rvert \geq N/(k-1)$.
For every $\alpha \in [w+1]$, let $S'_\alpha := \{t_j \in S_{i_1}: \lvert X_{t_j} \cap X_{q_{i_1}} \rvert = \alpha\}$.
So there exists $\alpha^* \in [w+1]$ such that $\lvert S'_{\alpha^*} \rvert \geq \lvert S_{i_1} \rvert/(w+1) \geq \frac{N}{(k-1)(w+1)} \geq 2$.
Since $\lvert X_{t_i} \rvert \leq w+1$ for every $i \in [N]$, and $X_{t_i} \not \subseteq X_{t_j}$ for every $i > j$, we know $\alpha^* \in [w]$.

Let $(z_1,z_2,\dots,z_{\lvert S'_{\alpha^*} \rvert})$ be the parade formed by the elements of $S'_{\alpha^*}$.
Since $(T,\X)$ is a tree-decomposition, there exists $Z \subseteq X_{q_{i_1}}$ with $\lvert Z \rvert=\alpha^*$ such that for every $i \in [\lvert S'_{\alpha^*} \rvert]$, $X_{z_i} \cap X_{q_{i_1}} = Z$.

For every $t \in V(T)$, let $X'_t := X_t-Z$.
Let $\X' := (X'_t: t \in V(T))$.
Then $\lvert X'_{z_i} \rvert = \lvert X_{z_i} \rvert - \lvert Z \rvert \leq (w-\lvert Z \rvert)+1$ for every $i \in [\lvert S'_{\alpha^*} \rvert]$.
If there exist $i,j \in [\lvert S'_{\alpha^*} \rvert]$ with $i>j$ such that $X'_{z_i} \subseteq X'_{z_j}$, then $X_{z_i} = X'_{z_i} \cup Z \subseteq X'_{z_j} \cup Z = X_{z_j}$, a contradiction.
So for any $i,j \in [\lvert S'_{\alpha^*} \rvert]$ with $i>j$, $X'_{z_i} \not \subseteq X'_{z_j}$.
Since $\lvert S'_{\alpha^*} \rvert \geq \frac{N}{(k-1)(w+1)} \geq N(w-1,k) \geq N(w-\lvert Z \rvert,k)$, by the induction hypothesis, there exist $\ell \in [0,w-\lvert Z \rvert]$ and a $(t_1',t_k',\ell)$-fan $(t_1',t_2',\dots,t_k')$ in $(T,\X')$ of size $k$ that is a subsequence of $(z_1,z_2,\dots,z_{\lvert S'_{\alpha^*} \rvert})$ and hence is a subsequence of $(t_1,t_2,\dots,t_N)$.
Since $Z \subseteq X_{z_i}$ for every $i \in [\lvert S'_{\alpha^*} \rvert]$, $(t_1',t_2',\dots,t_k')$ is a $(t_1',t_k',\ell+\lvert Z \rvert)$-fan in $(T,\X)$.
Note that $\ell+\lvert Z \rvert \in [0,w]$.
This proves the lemma.
\end{proof}

\section{Proof of Main Theorem}
\label{MainSection}

This section presents several definitions that lead to the statement of \cref{no apex}, which implies and strengthens \cref{ltwbasic}. This lemma is proved assuming \cref{no apex 0}, which is then proved in \cref{sec:mainlemma} below. 

\subsection{List-Coloring Setup}

The proof of \cref{ltwbasic} uses a list-coloring argument, where if $(V_1,V_2,\dots)$ is a layering of our graph, then we assume that color $i$ does not appear in the lists of vertices in layers $V_j$ with $j\equiv i \pmod{s+2}$. This ensures that each monochromatic component is contained within at most $s+1$ consecutive layers. 

For our purposes a \defn{color} is an element of $\mathbb{Z}$. A \defn{list-assignment} of a graph $G$ is a function $L$ with domain containing $V(G)$, such that $L(v)$ is a non-empty set of colors for each vertex $v\in V(G)$. 
For a list-assignment $L$ of $G$, an \mathdef{$L$}{coloring} of $G$ is a function $c$ with domain $V(G)$ such that $c(v) \in L(v)$ for every $v \in V(G)$. So an $L$-coloring has clustering $\eta$ if every monochromatic component has at most $\eta$ vertices. A list-assignment $L$ of a graph $G$ is an \mathdef{$\ell$}{list-assignment} if $|L(v)|\geq\ell$ for every vertex $v\in V(G)$. 

Let $G$ be a graph and $Z\subseteq V(G)$. A \mathdef{$Z$}{layering} $\V$ of $G$ is an ordered partition $(V_1,V_2,\dots)$ of $V(G)-Z$ into (possibly empty) sets such that for every edge $e$ of $G-Z$, there exists $i\in\mathbb{N}$ such that both ends of $e$ are contained in $V_i \cup V_{i+1}$. Note that a layering is equivalent to an $\emptyset$-layering. For a tree-decomposition $(T,\X)$ of $G$, with $\X=(X_t: t \in V(T))$, the \mathdef{$\V$}{width} of $(T,\X)$ is  
$$\max_{i\in\mathbb{N}} \max_{t \in V(T)}\lvert X_t \cap V_i \rvert.$$

Let $G$ be a graph and let $Z\subseteq V(G)$. 
For every $s\in\mathbb{N}$ and $\ell \in [s+2]$, an \mathdef{$s$}{segment} of a $Z$-layering $(V_1,V_2,\dots)$ of \defn{level $\ell$} is $\bigcup_{j=a}^{a+s}V_j$ for some (possibly non-positive) integer $a$ with $a \equiv \ell+1$ (mod $s+2$), where $V_a=\emptyset$ if $a \leq 0$. When the integer $s$ is clear from the context, we write \defn{segment} instead of $s$-segment. 

Let $G$ be a graph and let $s\in\mathbb{N}$. 
Let $Z \subseteq V(G)$ and $\V=(V_1,V_2,\dots)$ be a $Z$-layering of $G$.
A list-assignment $L$ of $G$ is \mathdef{$(s,\V)$}{compatible} if the following conditions hold:
	\begin{itemize}
		\item $L(v) \subseteq [s+2]$ for every $v \in V(G)$.
		\item $i \not \in L(v)$ for every $i \in [s+2]$ and $v \in \bigcup (V_j: j \equiv i \pmod{s+2})$. 
	\end{itemize}
Note that there is no condition on $L(v)$ for $v \in Z$ other than $L(v) \subseteq [s+2]$. 

We remark that for every $i \in [s+2]$, if $v \in V(G)$ with $i \in L(v)$, then either $v \in Z$, or $v$ belongs to a segment of $\V$ with level $i$.
This leads to the following easy observation that we frequently use.

\begin{proposition} \label{inside segment suffices}
Let $G$ be a graph and let $s\in\mathbb{N}$. 
Let $\V=(V_1,V_2,\dots)$ be a layering of $G$, and let $L$ be an $(s,\V)$-compatible list-assignment. If $k\in\mathbb{N}$ and $c$ is an $L$-coloring such that for every $i \in [s+2]$ and every $s$-segment $S$ with level $i$, every $c$-monochromatic component contained in $G[S]$ with color $i$ contains at most $k$ vertices, then $c$ has clustering at most $k$.
\end{proposition}

\begin{proof}
For every $i \in [s+2]$, if $v \in V(G)$ with $i \in L(v)$, then $v$ belongs to a segment of $\V$ with level $i$. So every $c$-monochromatic component with color $i$ is contained in some segment of $\V$ with level $i$.
\end{proof}

Let $G$ be a graph, $Z \subseteq V(G)$, $\V$ a $Z$-layering of $G$, $s\in\mathbb{N}$, and $L$ an $(s,\V)$-compatible list-assignment of $G$. For $Y_1\subseteq V(G)$, we say that $(Y_1,L)$ is an \mathdef{$(s,\V)$}{standard pair} if the following conditions hold:
	\begin{itemize}
		\item[(L1)] $Y_1 = \{v \in V(G): \lvert L(v) \rvert=1\}$. 
		\item[(L2)] For every $y \in N^{< s}(Y_1)$, 
		$$\lvert L(y) \rvert = s+1-\lvert N_G(y) \cap Y_1 \rvert ,$$ 
		and $L(y) \cap L(u)=\emptyset$ for every $u \in N_G(y) \cap Y_1$. (Note that $| L(y) | \geq 2$.)\ 
		\item[(L3)] For every $v \in V(G)-N_G[Y_1]$, we have $\lvert L(v) \rvert =s+1$.
	\end{itemize}

This definition says nothing about $L(v)$ for $v\in \ngs{Y_1}$. Indeed, it is possible that for such a vertex $v$, each color in $L(v)$ appears on some neighbor of $v$ in $Y_1$. However, this is not a problem since we use standard pairs only when $G$ has no $K_{s,t}$ subgraph, in which case $\ngs{Y_1}$ has bounded size by \cref{BoundedGrowth}. 

\subsection{Proof of Main Theorem} \label{subsec:ProofMainThm}

The next lemma and \cref{no apex 0} in the next section	motivate the above definitions, since they provide conditions under which a graph with no $K_{s,t}$ subgraph is $L$-colorable for some $(s,\V)$-compatible list-assignment $L$ and $(s,\V)$-standard pair $(Y_1,L)$. Most of the work in proving \cref{no apex} is done by \cref{no apex 0}, where we make an extra assumption about the precolored set $Y_1$. So we prove \cref{no apex} first, assuming \cref{no apex 0}. \cref{no apex} is also used to prove \cref{ltwmain} in \cref{AllowingApexVertices}.

\begin{lemma} 
\label{no apex}
For every $s,t,w \in \mathbb{N}$ and $\eta \in \mathbb{N}_0$, there exists $\eta^*=\eta^*(s,t,w,\eta) \in \mathbb{N}$ such that if $G$ is a graph with no $K_{s,t}$-subgraph, $\V$ is a layering of $G$, $L$ is an $(s,\V)$-compatible list-assignment, $Y_1$ is a subset of $V(G)$, $(Y_1,L)$ is an $(s,\V)$-standard pair, and $(T,\X)$ is a tree-decomposition of $G$ of $\V$-width at most $w$ such that $\lvert Y_1 \cap S \rvert \leq \eta$ for every $s$-segment\footnote{The condition $\lvert Y_1 \cap S \rvert \leq \eta$ for every $s$-segment $S$ is equivalent to the condition $\lvert Y_1 \cap V_i \rvert \leq \eta'$ for every layer $V_i$ of $\V$, for another constant $\eta'$.} $S$ of $\V$, then there exists an $L$-coloring of $G$ with clustering $\eta^*$.
\end{lemma}

\begin{proof}[Proof of \cref{no apex} assuming \cref{no apex 0}] 
Let $s,t,w \in \mathbb{N}$ and $\eta \in \mathbb{N}_0$. Let  $\eta^*$ be the number $\eta^*(s,t,w+\eta)$ from \cref{no apex 0}. Let $G$ be a graph with no $K_{s,t}$-subgraph, $\V$ a layering of $G$, $L$ an $(s,\V)$-compatible list-assignment, $Y_1$ a subset of $V(G)$, $(Y_1,L)$ an $(s,\V)$-standard pair, and $(T,\X)$ a tree-decomposition of $G$ of $\V$-width at most $w$ such that $\lvert Y_1 \cap S \rvert \leq \eta$ for every $s$-segment $S$ of $\V$. Say $\X=(X_t: t \in V(T))$.

For each $t \in V(T)$, let $X^*_{t} := X_t \cup Y_1$.
Let $\X^* = (X^*_t: t \in V(T))$. Let $t^*$ be a node of $T$.
Then $(T,\X^*)$ is a tree-decomposition of $G$ such that $X^*_{t^*}$ contains $Y_1$.
Since for every $s$-segment $S$ of $\V$, $\lvert Y_1 \cap S \rvert \leq \eta$, the $\V$-width of $(T,\X^*)$ is at most $w+\eta$.
Therefore, by \cref{no apex 0}, there exists an $L$-coloring of $G$ with clustering $\eta^*$.
\end{proof}

We now show that \cref{no apex} implies \cref{ltwbasic}. Indeed, 
\cref{no apex} is stronger than \cref{ltwbasic} in that it allows a non-empty precolored set $Y_1$ of vertices (satisfying certain technical assumptions relative to the given layering $\V$). 

\begin{proof}[Proof of \cref{ltwbasic}] 
Given $s,t,w \in \mathbb{N}$, let $\eta$ be the number $\eta^*(s,t,w,0)$ from \cref{no apex}. 
Let $G$ be a graph with no $K_{s,t}$ subgraph and with layered treewidth at most $w$. 
By definition, $G$ has a tree-decomposition with $\V$-width at most $w$ for some layering $\V=(V_1,V_2,\dots)$ of $G$. 
Define $L(v):=[s+2]\setminus\{i\}$ for each $v\in \bigcup( V_j : j \equiv i \pmod{s+2})$ and for each $i\in[s+2]$. 
Then $L$ is an $(s,\V)$-compatible list-assignment and $(\emptyset,L)$ is an $(s,\V)$-standard pair. 
By \cref{no apex} with $Y_1=\emptyset$, there exists an $L$-coloring of $G$ with clustering $\eta$. 
The number of colors is at most $s+2$.
\end{proof}

\subsection{Progress and Gates}

Before jumping to the proof of \cref{no apex 0} we define two notions that are used in its proof.

Let $G$ be a graph, $Z \subseteq V(G)$, $\V$ a $Z$-layering, and $(Y_1,L)$ an $(s,\V)$-standard pair. For a subset $W\subseteq V(G)$ and color $i$ (not necessarily belonging to $\bigcup_{v \in V(G)}L(v)$), a \mathdef{$(W,i)$}{progress} of $(Y_1,L)$ is a pair $(Y_1',L')$ defined as follows:
\begin{itemize}
	\item Let $Y_1':=Y_1 \cup W$. 
	\item For every $y \in Y_1$, let $L'(y):=L(y)$.
	\item For every $y \in Y_1'-Y_1$, let $L'(y)$ be a 1-element subset of $L(y)-\{i\}$, which exists by (L1).
	\item For each $v \in N^{< s}(Y_1')$, let $L'(v)$ be a subset of 
	$$L(v)- \bigcup \{L'(w): w \in N_G(v) \cap (W-Y_1)\}$$ of size $\lvert L(v) \rvert - \lvert N_G(v) \cap (W-Y_1) \rvert$, which exists since $|L'(w)|=1$ for $w \in N_G(v) \cap (W-Y_1)$.
	
	\item For every $v \in V(G)-(Y_1' \cup N^{< s}(Y_1'))$, define $L'(v):=L(v)$.
\end{itemize}

The intuition here is that in a $(W,i)$-progress, vertices in $W-Y_1$ are assigned a color different from $i$. As explained in the definition, a $(W,i)$-progress $(Y_1',L')$ exists. Since $L'(v)\subseteq L(v)$ for each $v\in V(G)$, every $L'$-coloring of $G$ is an $L$-coloring of $G$. Moreover, since $L$ is an $(s,\V)$-compatible list-assignment, so too is $L'$. By construction, $(Y_1',L')$ is an $(s,\V)$-standard pair. See \citep{LW2} for a more general notion of ``progress''.

Let $(Y_1,L)$ be an $(s,\V)$-standard pair in a graph $G$.
For $y \in Y_1$, a \defn{gate} for $y$ (with respect to $(Y_1,L)$) is a vertex $v \in N_G(y)-Y_1$ such that $L(v) \cap L(y) \neq \emptyset$.
For $W\subseteq Y_1$, let $$A_{L}(W):=\{v \in V(G)-Y_1: v\text{ is a gate  for some }y \in W\text{ with respect to }(Y_1,L)\}.$$
Note that $Y_1$ is determined by $L$ by (L1).
Gates are important since this is where a monochromatic component of $G[Y_1]$ can grow into the remainder of $G$.

\section{Main Lemma}
\label{sec:mainlemma}

The next lemma is the heart of the proof of  \cref{ltwbasic}.

\begin{lemma} \label{no apex 0}
For every $s,t,w \in \mathbb{N}$, there exists $\eta^*:=\eta^*(s,t,w) \in \mathbb{N}$ such that if $G$ is a graph with no $K_{s,t}$-subgraph, $\V$ is a layering of $G$, $L$ is an $(s,\V)$-compatible list-assignment, $Y_1$ is a subset of $V(G)$, $(Y_1,L)$ is an $(s,\V)$-standard pair, and $(T,\X)$ is a tree-decomposition of $G$ of $\V$-width at most $w$ such that some bag contains $Y_1$, then there exists an $L$-coloring of $G$ with clustering $\eta^*$.
\end{lemma}

All of \cref{sec:mainlemma} is dedicated to the proof of \cref{no apex 0}, which is completed at the end of the section. We now describe the structure of \cref{sec:mainlemma}. In \cref{subsec:const_main_lemma}, we define constants, including the number $\eta^*$ stated in \cref{no apex 0}, various functions and other notions that are used throughout the proof of \cref{no apex 0}. We also give a superficial sketch of the proof. \cref{no apex 0} is proved by constructing the desired coloring using an algorithm. We give an intuition of the algorithm in \cref{subsec:intuition_algo}, where notation is defined informally. The reader should consult \cref{subsec:alog} for formal definitions and the complete statement of the algorithm. 

The rest of \cref{sec:mainlemma} is dedicated to showing that the algorithm indeed constructs the desired coloring. The proof is split into a series of claims that can be verified independently. The claims are grouped into subsections according to the required knowledge about the parts of the algorithm or the partial goal of the proof. In each subsection that requires the knowledge of a new portion of the algorithm, 
for the convenience of the reader, we first restate the relevant part of the algorithm before proving the related claims. Later subsections might require the knowledge of other parts of the algorithm stated in earlier subsections.

These claims can be read in two parts. 
Sections~\ref{subsec:claims_fences}--\ref{subsec:PrecoloredSets} show that one type of monochromatic component has bounded size and that the number of precolored vertices in the ``interesting region'' is always bounded.
Then Sections~\ref{subsec:KKstar}--\ref{subsec:SecondType} show that the second type of monochromatic component has bounded size.
These two parts use different techniques and can be read separately, though reading the first part might be helpful for being acquainted with our algorithm.

\subsection{Constants, functions and notions} \label{subsec:const_main_lemma}

Let $s,t,w \in \mathbb{N}$.
We start by defining several values used throughout the proof:
	\begin{itemize}
		\item Let $f$ be the function $f_{s,t,w}$ in \cref{BoundedGrowthLtw}.
		\item Let $f_0$ be the identity function $f$ on $\mathbb{N}_0$; for every $i \geq 1$, let $f_i$ be the function from $\mathbb{N}_0$ to $\mathbb{N}_0$ such that $f_i(x)=f_{i-1}(x)+f(f_{i-1}(x))$ for every $x \in \mathbb{N}_0$.
		\item Let $s^*:=24(s+2)$.
		\item Let $w_0:=12w\,s^*+13$.
		\item Let $g_0: \mathbb{N}_0 \rightarrow \mathbb{N}_0$ be the function defined by $g_0(0) :=w_0$ and $g_0(x):=f_1(g_0(x-1))+2w_0$ for every $x \in \mathbb{N}$.
		\item Let $g_1: \mathbb{N}_0 \rightarrow \mathbb{N}_0$ be the function defined by $g_1(0):=g_0(w_0)$ and $g_1(x):=f_1(g_1(x-1)+3w_0)$ for every $x \in \mathbb{N}$.
		\item Let $g_2: \mathbb{N}_0 \times \mathbb{N}_0 \rightarrow \mathbb{N}_0$ be the function defined by  $g_2(0,y):=y+w_0$ for every $y \in \mathbb{N}_0$, and $g_2(x,y):=f_2(g_2(x-1,y))+3w_0$ for every $x \in \mathbb{N}$.
		\item Let $g_3: \mathbb{N}_0 \rightarrow \mathbb{N}_0$ be the function defined by $g_3(0):=g_1(s+2)$, and  $g_3(x):=g_2(s+2,g_3(x-1))$ for every $x \in \mathbb{N}$.
		\item Let $\eta_1:=w_0g_1(s+2)+3w_0g_3(4w_0)+w_0$.
		\item Let $g_4: \mathbb{N}_0 \times \mathbb{N}_0 \rightarrow \mathbb{N}_0$ be the function defined by
		 $g_4(0,y):=y$ and $g_4(x,y):=(s+2) \cdot f_1(g_4(x-1,y)+2\eta_1)$ 
		   for every $x \in \mathbb{N}$ and $y \in \mathbb{N}_0$,
		\item Let $g_5: \mathbb{N}_0 \rightarrow \mathbb{N}_0$ be the function defined by $g_5(x):=g_4(s+2,g_4(w_0,x))+2g_3(4w_0)$ for every $x \in \mathbb{N}_0$. 
		
		\item Let $g_6: \mathbb{N}_0 \rightarrow \mathbb{N}_0$ be the function defined by $g_6(0):=w_0$, and  $g_6(x):=g_5(g_6(x-1))+w_0$ for every $x \in \mathbb{N}$.
		\item Let $\eta_2:=g_6(5w_0(2w_0+1)+9w_0)$.
		\item Let $\eta_3:=(\eta_1+2\eta_2+1)w_0$.
		\item Let $g_7: \mathbb{N}_0 \times \mathbb{N}_0 \rightarrow \mathbb{N}_0$ be the function defined by
		$g_7(0,y):=y$ and $g_7(x,y):=f_1(g_7(x-1,y))+w_0$ for every $x \in \mathbb{N}$ and $y \in \mathbb{N}_0$.
		\item Let $g_8: \mathbb{N}_0 \rightarrow \mathbb{N}_0$ 
		be the function defined by $g_8(0):=\eta_1+2\eta_2$ and $g_8(x):=g_7(s+2,g_8(x-1))$ for every $x \in \mathbb{N}$.
		\item Let $g_9: \mathbb{N}_0 \rightarrow \mathbb{N}_0$ be the function defined by $g_9(0):=2g_8(w_0)+\eta_2$ and  $g_9(x):=f_1(g_9(x-1))+w_0$ for every $x \in \mathbb{N}$.
		\item Let $\eta_4:=g_9(s+2)$.
		\item Let $g_{10}: \mathbb{N}_0 \rightarrow \mathbb{N}_0$  be the function defined by  
		$g_{10}(0) := 2\eta_4$ and $g_{10}(x) := f_{s+3}(g_{10}(x-1)+\eta_3+2\eta_4)$ for every $x \in \mathbb{N}$.
		\item Let $\eta_5 := 3g_{10}(\eta_3)$.
		\item Let $\psi_1: [0,w_0]^2 \rightarrow {\mathbb N}$ be the function such that for every $(x_1,x_2) \in [0,w_0]^2$, $\psi_1(x_1,x_2) := (10(f_{w_0}(\eta_5))^{3w_0+3}w_0^3 + w_0) \cdot (\min\{x_1,x_2\} \cdot (w_0+1) \cdot (f(\eta_5)+1)^{w_0+1}+1)$.
		
		\item Let $\psi_2: [0,w_0]^2 \rightarrow {\mathbb N}$ be the function such that for every $(x_1,x_2) \in [0,w_0]^2$, $\psi_2(x_1,x_2) := \psi_1(100x_1,100x_2)$.
		
		\item Let $\psi_3: [0,w_0]^2 \rightarrow {\mathbb N}$ be the function such that for every $(x_1,x_2) \in [0,w_0]^2$, $\psi_3(x_1,x_2) := \psi_1(x_1,x_2) \cdot \frac{(f(\eta_5))^{w_0}}{\psi_1(0,0)}$.
		\item Let $\kappa_0 := \psi_2(w_0,w_0)+\psi_3(w_0,w_0)$. 
		\item Let $\kappa_1 := w_0^2 \cdot N_{\ref{paradefans}}(w_0,\kappa_0+w_0+3)$, where $N_{\ref{paradefans}}$ is the number $N$ mentioned in \cref{paradefans}.
		\item Let $\phi_1: {\mathbb N} \rightarrow \mathbb{N}$ be the function defined by $\phi_1(1):=\kappa_1$, and $\phi_1(x+1) := \kappa_1 \cdot (1+\sum_{i=1}^{x}\phi_1(i))+1$ for every $x \in {\mathbb N}$.
		\item Let $\phi_2: {\mathbb N}_0 \rightarrow \mathbb{N}_0$ be the function defined by $\phi_2(0):=0$, and $\phi_2(x) := ((\phi_1(x)+1)w_0f(\eta_5)+1) \cdot (1+\sum_{i=1}^{x-1}\phi_2(i))$ for every $x \in {\mathbb N}$.
		\item Let $\phi_3: {\mathbb N} \rightarrow \mathbb{N}$ be the function defined by $\phi_3(x):= (1+\sum_{i=1}^x\phi_2(i)) \cdot(\phi_1(x)+1) \cdot x$ for every $x \in {\mathbb N}$.
		\item Let $h: {\mathbb N}_0 \rightarrow {\mathbb N}$ be the function defined by $h(0) := 2^{2w_0+6} (\phi_3(w_0) \cdot f(\eta_5) \cdot w_0^2)^2(\eta_5+3)$, and for every $x \in {\mathbb N_0}$, $h(x+1) := (h(x)+2) \cdot 2^{2w_0+7} (\phi_3(w_0) \cdot f(\eta_5) \cdot w_0^2)^2(\eta_5+3)$.
		
		\item Let $\eta_6:=h(w_0-1)$. 
		\item Let $\eta_7:=\eta_5 \cdot (f(\eta_5))^{\eta_6}$.
		\item Define $\eta^* := \eta_4+\eta_7$.
	\end{itemize}

Let $G$ be a graph with no $K_{s,t}$-subgraph, $\V=(V_1,V_2,\dots)$ a layering of $G$, $L$ an $(s,\V)$-compatible list-assignment, $Y_1$ a subset of $V(G)$, $(Y_1,L)$ an $(s,\V)$-standard pair, and $(T,\X)$ a tree-decomposition of $G$ of $\V$-width at most $w$ such that some bag contains $Y_1$, where  $\X=(X_t: t \in V(T))$.
For every integer $i$ with $i \not \in [\lvert \V \rvert]$, define $V_i=\emptyset$.

We also use the following notation:
    \begin{itemize}
        \item For every $R \subseteq V(T)$, define $X_R := \bigcup_{t \in R}X_t$.
        \item For every $U \subseteq V(G)$, define $\X|_U=(X_t \cap U: t \in V(T))$.
        \item Let $r^*$ be the node of $T$ such that $X_{r^*}$ contains $Y_1$. 
    Consider $T$ to be rooted at $r^*$. 
Orient the edges of $T$ away from $r^*$. 
        \item For each node $q$ of $T$, let $T_q$ be the subtree of $T$ induced by $q$ and all the descendants of $q$.
        \item For each vertex $v$ of $G$, let $r_v$ be the node of $T$ closest to $r^*$ with $v \in X_{r_v}$. 
    \end{itemize}

\medskip
We construct the desired $L$-coloring of $G$ by an algorithm. We now give an informal description of several notions that are used in the algorithm below. The layers are partitioned into pairwise disjoint belts, where each belt consists of a very large (but still bounded) set of consecutive layers. An interface consists of roughly the last one third of the layers within one belt, along with the first roughly one third of the layers in the next belt. Then the interior of an interface is the last few layers of the first belt, along with the first few layers of the next belt. The algorithm also uses a linear ordering of $V(G)$ that never changes. We associate with each subgraph of $G$ the first vertex in the ordering that is in the subgraph. We use this vertex to order subgraphs, whereby a  subgraph that has a vertex early in the ordering is considered to be ``old''. We now formalize these ideas. 

\medskip 
Define $\sigma_T$ to be a depth-first-search order of $T$ rooted at $r^*$.
That is, $\sigma_T(r^*)=0$, and if a node $t$ of $T$ is visited earlier than a node $t'$ of $T$ in the depth-first-search starting at $r^*$ then $\sigma_T(t)<\sigma_T(t')$.
For every node $t$ of $T$, we define $i_t$ to be the integer such that $\sigma_T(t)=i_t$. 

Define $\sigma$ to be a linear order of $V(G)$ such that for any distinct vertices $u,v$, if $\sigma_T(r_u)<\sigma_T(r_v)$, then $\sigma(u)<\sigma(v)$.
For every subgraph $H$ of $G$, define $\sigma(H):=\min\{\sigma(v): v \in V(H)\}$.
Note that for any $k \in \mathbb{N}_0$, any set $Z$ consisting of $k$ consecutive layers, any $t \in V(T)$, any $(s,\V)$-standard pair $(Y',L')$, any monochromatic component $M$ with respect to any $L'$-coloring in $G[Y']$ intersecting $Z \cap X_t$, there exist at most $\lvert Z \cap X_t \rvert \leq kw$ monochromatic components $M'$ with respect to any $L'$-coloring in $G[Y']$ intersecting $Z \cap X_t$ such that $\sigma(M')<\sigma(M)$.

A \defn{belt} of $(T,\X)$ is a subset of $V(G)$ of the form $\bigcup_{i=a+1}^{a+s^*}V_i$ for some nonnegative integer $a$ with $a \equiv 0$ (mod $s^*$).
(Recall that $s^*$ has been defined earlier in this section.)
Note that each belt consists of $s^*$ layers.
So for every belt $B$ of $(T,\X)$, $(T,\X|_B)$ is a tree-decomposition of $G[B]$ of width at most $s^*w$.

For every $j \in [\lvert \V \rvert-1]$, let 
\begin{equation*}
I_j:=\bigcup_{i=(j-\frac{1}{3})s^*}^{(j+\frac{1}{3})s^*}V_i \quad\quad\text{and}\quad\quad
\overline{I_j} := \bigcup_{i=(j-\frac{1}{3})s^*-(2s+5)}^{(j+\frac{1}{3})s^*+(2s+5)}V_i.
\end{equation*}
$I_j$ is called the \defn{interface} at $j$. 
(Note that $s^*$ is a multiple of 3, so the indices in the definition of $I_j$ and $\overline{I_j}$ are integers.)\ 
For every $j \in [0,\lvert \V \rvert-1]$, define 
\begin{equation*}
I_{j,0}:=\bigcup_{i=js^*-(s+2)+1}^{js^*}V_i\quad\quad\text{and}\quad\quad
\overline{I_{j,0}}:=\bigcup_{i=js^*-2(s+2)+1}^{js^*}V_i,
\end{equation*}
and define 
\begin{equation*}
I_{j,1}:=\bigcup_{i=js^*+1}^{js^*+s+2}V_i\quad\quad\text{and}\quad\quad
\overline{I_{j,1}}:=\bigcup_{i=js^*+1}^{js^*+2(s+2)}V_i.
\end{equation*}
Also define $I_j^\circ := I_{j,0} \cup I_{j,1}$ and $\overline{I_j^\circ} := \overline{I_{j,0}} \cup \overline{I_{j,1}}$.
For every $j \in [\lvert \V \rvert-1]$, define $\Se_j^\circ$ to be the set of all $s$-segments $S$ intersecting $I_j^\circ$.
Note that $I_j^\circ \subseteq \bigcup_{S \in \Se_j^\circ}S \subseteq \overline{I_j^\circ} \subseteq I_j$ for every $j \in [\lvert \V \rvert-1]$.
In addition, for every $j \in [\lvert \V \rvert-1]$, $I_j$ is contained in a union of two consecutive belts, so for every $t \in V(T)$, $\lvert I_j \cap X_t \rvert \leq 2s^*w \leq w_0$.
(Recall that $w_0$ has been defined earlier in this section.)

For a collection $E$ of 2-element subsets of $V(G)$ and a coloring $c$ of $G$, a \defn{monochromatic $E$-pseudocomponent} (with respect to $c$) is a component of the subgraph of $G+E$ induced by a color class of $c$.

\subsection{Intuition for the algorithm} \label{subsec:intuition_algo}

We now give some intuition about the algorithm which is formally stated in \cref{subsec:alog}. The input to the algorithm is a tree-decomposition of $G$ with bounded layered width, a set $Y_1$ of precolored vertices, and a list-assignment $L$ of $G$. Throughout the algorithm, $Y^{(i,\star,\star)}$ refers to the current set of colored vertices, where the superscript $(i,\star,\star)$ indicates the stage of the algorithm. This vector is incremented in lexicographic order as the algorithm proceeds. The algorithm starts with stage $(0,-1,0)$ which initializes several variables. 
Throughout the algorithm, a monochromatic component refers to a component of the subgraph of $G$ induced by the current precolored set of vertices (or to be more precise, vertices with one color in their list). 

The algorithm does a depth-first-search on the tree $T$ indexing the tree-decomposition, considering the nodes $t$ of $T$ with $\sigma_T(t)=i$ in turn (for $i=0,1,2,\dots$).
The algorithm first builds fences around subgraphs of vertices in bags of nodes of some subtrees rooted at node $t$, where the subtrees vary for different belts.
At stage $(i,-1,\star)$, the algorithm tries to isolate the $k$-th oldest component intersecting a segment that intersects $I_j^\circ$. 
Then at stage $(i,0,\star)$, the algorithm isolates the other monochromatic components intersecting $X_t$. In both these stages, $\star$ refers to the color given to vertices around the component that we are trying to isolate. Then stage $(i,\star,\star)$ isolates the fences associated with subtree $T_{j,t}$. Then we add ``fake edges'' to the graph between precolored vertices to merge some monochromatic components according to certain rules.
Finally, the algorithm moves to the next node in the tree, and builds new fences with respect to the next node in the tree. 


Now we give some more detailed intuition of the algorithm. We simplify some notation in this explanation for ease of presentation. The formal description of the algorithm is described later. During the algorithm, we construct the following.

\begin{itemize}
	\item Subsets $Y^{(i,j,k)}$ of $V(G)$ for $i \in [0,\lvert V(T) \rvert]$, $j \in \{-1\} \cup [0,\lvert V(T) \rvert+1]$ and $k \in [0,s+2]$:
		The set $Y^{(i,j,k)}$ is the set of precolored vertices at stage $(i,j,k)$, where a vertex is said to be \defn{precolored} if its list contains only one color.
		At each stage, we precolor more vertices, so these sets have the property that $Y^{(i,j,k)} \subseteq Y^{(i',j',k')}$ if $(i,j,k)$ is lexicographically smaller than $(i',j',k')$.

	\item Subsets $F_{j,t}$ of $V(T)$ for $j \in [\lvert \V \rvert-1]$ and $t \in V(T)$:
		Each of these sets $F_{j,t}$ is a union of fences.
		For any fixed $j \in [\lvert \V \rvert-1]$ and node $t$ of $T$, assuming a subset $F_{j,p}'$ of $V(T)$ is given, where $p$ is the parent of $t$, we do the following:
		\begin{itemize}
			\item We first construct $F_{j,p}$ as follows:
				We consider the $F_{j,p}' \cap V(T_t)$-parts of $T_t$ containing $t$.
				Note that it is possible to have more than one such part since $t$ could be in $F_{j,p}'$.
				Then for each such part $T'$, we restrict ourselves to the subgraph of $G$ induced by $X_{V(T')} \cap \overline{I_j}$.
				That is, the subgraph induced by the vertices in the bags of the nodes in the part $T'$ and in the closure of the interface at $j$.
				Note that this subgraph has a tree-decomposition naturally given by $(T,\X)$.
				We construct a fence $F_{j,T'}$ in this tree-decomposition based on the current set of precolored vertices $Y^{(i,-1,0)}$ in this subgraph, where $i=\sigma(t)$. 
				Then $F_{j,p}$ is the union of $F'_{j,p}$ and $F_{j,T'}$ among all such parts $T'$.
			\item When $F_{j,p}$ is defined, we extend our precolored set $Y^{(i,-1,0)}$ to $Y^{(i+1,-1,0)}$ by some procedure that is explained later. After $Y^{(i+1,-1,0)}$ is given, we construct $F_{j,t}'$ in a similar way as we construct $F_{j,p}$ from $F_{j,p}'$, except we replace $F_{j,p}'$ by $F_{j,p}$ and replace $Y^{(i,-1,0)}$ by $Y^{(i+1,-1,0)}$.
		\end{itemize}
		The key idea is to ensure that for every $F_{j,p} \cap V(T_t)$-part containing $t$, the subgraph induced by the intersection of $\overline{I_j}$ and bags of nodes in this part only contains few precolored vertices.
	\item Subtrees $T_{j,t}$ of $T_t$ and the ``boundary'' $\partial T_{j,t}$ of $T_{j,t}$, for $j \in [\lvert \V \rvert-1]$ and $t \in V(T)$:
		$T_{j,t}$ is the union of all $F_{j,p} \cap V(T_t)$-parts containing $t$, and $\partial T_{j,t}$ is the union of the boundary of those parts but we do not include $t$.
  (Note that in \cref{sec:fences}, we define $\partial T'$ for a part $T'$ when a fence is given. Here $T_{j,t}$ is a union of parts instead of a part, so the notion $\partial T_{j,t}$ is a new notion and was not defined in \cref{sec:fences}.)
		Roughly speaking, $T_{j,t}$ is a small portion of $T_t$ containing $t$, and $X_{V(T_{j,t})}$ contains only few precolored vertices by the construction of the fences.
	\item Subsets $Z_t$ of $X_{V(T_t)}$ for every $t \in V(T)$:
		$Z_t$ consists of the vertices in $X_{V(T_t)}$ that are either not in the interior of any interface, or in the interior of the interface at $j$ for some $j$ but also in $X_{V(T_{j,t})}$.
	
	\item Subsets $D^{(*,*,*)}$ of $V(G)$: 
		These sets help to extend the current set of precolored vertices.
		We explain the usage of these sets later.
	\item Subsets $S_{j,t}^{(*,*)}$ of $V(G)$ for $j \in [\lvert \V \rvert-1]$ and $t \in V(T)$: These sets play a minor role in the algorithm. 
		Some vertices trigger the extension of the coloring during the algorithm and become ``useless'' afterwards.
		These sets collect such vertices and are used in the proof in order to bound the size of the monochromatic components.
	\item Sets $E_{j,t}^{(i)}$ and $E_{j,t}^{(i,*)}$ of ``fake edges'', for $i \in [0,\lvert V(T) \rvert]$, $j \in [\lvert \V \rvert-1]$ and $t \in V(T)$:
		Each fake edge is a pair of two distinct vertices that are contained in the same $s$-segment in $\Se_j^\circ$.
		At any moment of the algorithm, the current set of precolored vertices form some monochromatic components, and it is possible that more than one of them will be merged into a bigger monochromatic component after we further color more vertices.
		These fake edges join current monochromatic components that will possibly be merged into a bigger monochromatic component in the future to form ``pseudocomponents''. Whether a pair of two vertices is a fake edge depends on $j$ and $t$.
\end{itemize}

Now we sketch the algorithm and explain the intuition.
See \cref{subsec:alog} for a detailed and formal description of the algorithm.

Recall that during the algorithm, $Y^{(i,*,*)}$ always contains all vertices of $G$ contained in $X_t$ with $\sigma_T(t) \leq i$. 
So at the very beginning of the algorithm, we color every uncolored vertex in $X_{r^*}$ to obtain $Y^{(0,-1,0)}$.
Then we execute the algorithm for $i=0,1,2,\dots$.
Fix $i\in\mathbb{N}_0$ and the node $t$ of $T$ with $\sigma_T(t)=i$. 
Here we explain what the algorithm does at this stage. 
\begin{itemize}
	\item First, for each $j \in [\lvert \V \rvert-1]$, we build fences $F'_{j,p}$ and $F_{j,p}$, where $p$ is the parent of $t$, and subtrees $T_{j,t}$ of $T_t$.
		Note that it defines $Z_t$.
		We  further color vertices in $Z_t$.
		The idea for constructing those fences and $Z_t$ is to ensure that few precolored vertices in $Z_t$ are close to the boundary of two different belts.
	\item Then we move to stage $(i,-1,*)$.
		This stage  ``isolates'' the monochromatic pseudocomponents in $Z_t$ intersecting $X_t$ and some segment in $\Se_j^\circ$ for some $j \in [\lvert \V \rvert-1]$.
		That is, we color some vertices in $Z_t$ to stop the growth of those monochromatic components by coloring their gates using a color different from the color of the monochromatic pseudocomponent.
		We cannot isolate all such components at once, since some vertex can be a gate of two monochromatic pseudocomponents with different colors. 
		We break ties according to the ordering $\sigma$ of those monochromatic components.
		For each $j \in [\lvert \V \rvert-1]$, we isolate the monochromatic pseudocomponent with minimum $\sigma$-order, and then isolate the monochromatic pseudocomponent with second minimum $\sigma$-order and so on.
		Say we are isolating the $k$-th monochromatic pseudocomponent among all those monochromatic pseudocomponents intersecting $X_t$ and some segment in $\Se_j^\circ$ for each $j$.
		Then we denote by $W_0^{(i,-1,k)}$ the set consisting of the vertices in the $k$-th monochromatic pseudocomponent for all $j$.
		Then for each $j$, there uniquely exists a color $q+1$ such that this monochromatic pseudocomponent uses color $q+1$.
		We write $W_1^{(i,-1,k,q+1)}$ to denote the vertices in $W_0^{(i,-1,k)}$ with color $q+1$.
		Then we color the gates of $W_1^{(i,-1,k,q+1)}$ contained in $Z_t$ by using a color different from $q+1$.
		Then we move to color the $(k+1)$-th monochromatic pseudocomponent, and so on.
		Note that the $(k+1)$-th monochromatic pseudocomponent can become larger due to the above process of coloring gates.
	\item Now we move to stage $(i,0,*)$.
		In this stage we isolated all monochromatic pseudocomponents in $Z_t$ intersecting $X_t$.
		Note that all monochromatic pseudocomponents intersecting some $s$-segment in $\Se_j^\circ$ for some $j \in [\lvert \V \rvert-1]$ were isolated in the previous stage.
		So all the remaining monochromatic pseudocomponents are disjoint from the $s$-segments in $\Se_j^\circ$ for every $j \in [\lvert \V \rvert-1]$, and hence they are disjoint from all fake edges.
		Hence all monochromatic pseudocomponents that are dealt with at this stage are indeed monochromatic components.
		We isolate all such monochromatic components with color 1, and then the ones with color 2 and so on.
	\item Now we move to stage $(i, \geq 1, *)$.
		In this stage we restrict to $X_{V(T_{j,t})} \cap I_j$, for each $j \in [\lvert \V \rvert-1]$.
		We denote the union of them by $W_4^{(i)}$.
		Now we fix $j$ to be an integer in $[\lvert \V \rvert-1]$.
		The objective is that at the end of this stage, the monochromatic components contained in $X_{V(T_{j,t})} \cap I_j$ intersecting $\partial T_{j,t}$ will not have more new vertices in $X_{V(T_{j,t})} \cap I_j$ in the future.
		The idea for achieving this is to color $X_{t'} \cap I_j$ for some nodes $t' \in \partial T_{j,t}$ in the following way.
		Consider sets $W_3^{(i,*)}$ of ``base nodes''.
		At the beginning (that is, for $W_3^{(i,-1)}$), the only base node is $t$.
		If there exists a monochromatic path $P$ from the bag of a base node to the bag of a node $t'$ in $\partial T_{j,t}$, then we add $t'$ into the set of base nodes (i.e. we add $t'$ into $W_3^{(i,*)}$ to obtain $W_3^{(i,*+1)}$) and color $X_{t'} \cap I_j$, unless the vertex in $V(P) \cap X_{t'}$ belongs to $D^{(i,*,*)}-X_t$ or is in the intersection of $D^{(i,*,*)} \cap X_t$ and a bag of a some base nodes at stage $(i',*,*)$ for some $i'<i$, where $D^{(i,*,*)}$ is a special set of vertices roughly consisting of all vertices contained in the previous monochromatic paths for generating new base nodes.
		Then we isolate those newly colored vertices by coloring their gates.
		It might create new monochromatic paths from base nodes to other node in $\partial T_{j,t}$.
		We repeat the above process until no new base nodes can be added.
		The intuition of considering $D^{(i,*,*)}$ is to make sure that we only include ``necessary'' nodes in $\partial T_{j,t}$ in the set of base nodes in order to limit the size of the set of precolored vertices.
	\item Then we move to the stage for adding new fake edges.
 We first explain some intuition.
		Recall that when coloring more vertices, we always try to isolate some monochromatic pseudocomponents.
		But such isolation is ``incomplete'' because of the definition of $Z_t$.
		In addition, the isolation process depends on the order of the pseudocomponents.
		So we require some extra information to indicate whether two monochromatic pseudocomponents will be merged in the future so that they should be treated the same when we order the pseudocomponents.
		This is the purpose of the fake edges.
		Whether a pair of vertices forms a fake edge depends on where we ``observe'' it. 
		That is, when ``standing at'' a node $t'$ of $T$, we consider whether we should add a fake edge to join two vertices if in the future there will be monochromatic paths in $X_{V(T_{t})}$ connecting these two vertices. 
		So it is only relevant when $t' \in V(T_t)$.
		In addition, all nodes $t' \in V(T_t)-\{t\}$ will get the same new fake edges at this stage; and the sets $E^{(i_t)}_{j,t'}$ and $E^{(i_t)}_{j,t''}$ of fake edges for two different nodes $t'$ and $t''$, respectively, are different only when $t$ is a descendant of the common ancestor of $t'$ and $t''$ with the largest $\sigma_T$-value.
		
		Consider a node $t' \in V(T_t)$.
		Each fake edge is a pair of distinct vertices $u,v$ in $X_{t''} \cap S$ for some $j \in [\lvert \V \rvert-1]$, $t'' \in \partial T_{j,t}$ and $S \in \Se_j^\circ$.
		Also, $u$ and $v$ must be already colored by the same color.
		So there exist monochromatic pseudocomponents $M_u$ and $M_v$ containing $u$ and $v$, respectively.	
		And $M_u \neq M_v$, otherwise $u$ and $v$ are already contained in the same monochromatic pseudocomponent and there is no need to add a fake edge $uv$.
		And $S$ is chosen so that the level of $S$ is the color of $u$ and $v$.
  So no matter how $M_u$ or $M_v$ will grow, they are contained in $S$.
		By symmetry, we may assume $\sigma(M_u)<\sigma(M_v)$.
		Note that we want to add a fake edge to link $u$ and $v$ if it is possible that $M_u$ and $M_v$ will be joined into one monochromatic component in the future.
		So each $M_u$ and $M_v$ must have gates contained in $X_{V(T_{t''})}-X_{t''}$.
		Recall that during the isolation process, the monochromatic pseudocomponent with minimum $\sigma$-order will be isolated first, so it will not grow in the future.
		Hence $M_u$ and $M_v$ cannot be the monochromatic pseudocomponent with the minimum $\sigma$-order.	

		In addition, the need for adding a fake edge is ``triggered'' by another monochromatic pseudocomponent intersecting $X_{t''}$ with smaller $\sigma$-value.
		That is, there exists a monochromatic pseudocomponent intersecting $X_{t''}$ with smaller $\sigma$-order than $\sigma(M_u)$ and $\sigma(M_v)$ so that isolating this monochromatic pseudocomponent might lead to possible needs for coloring more vertices to eventually merge $M_u$ and $M_v$  into one monochromatic pseudocomponent.
		Hence, there is no danger for merging $M_u$ and $M_v$ into one monochromatic pseudocomponent in the future if the gates of those monochromatic pseudocomponents with smaller $\sigma$-value than $\sigma(M_u)$ and $\sigma(M_v)$ are separated by a ``thick buffer'' from the gates of $M_u$ and $M_v$.
		Here a ``thick buffer'' means a sequence of nodes that is similar to a fan defined in the previous section, so it is a sequence of many nodes of $T$ such that they more or less have disjoint bags, and there exist a directed path in $T$ passing through all of them. 

		Moreover, due to the way we isolate monochromatic pseudocomponents in stages $(\leq i, *, *)$, there is no need to add the fake edge $\{u,v\}$ if $T_{j,t} \subseteq T_{j,q}$ for some node $q$ with $i_q<i_t=i$, and there exists some ``evidence'' that if we decide not to add a fake edge to join $M_u$ and $M_v$ at the stage $(i_q,*,*)$, then there is no need to add the fake edge now.
	\item Finally, color all vertices in the bag of the node with $\sigma_T$-value $i+1$, build fences $F_{j,t}'$ and move to the next $i$.
\end{itemize}


\subsection{Formal description of the algorithm} \label{subsec:alog}

This subsection gives the formal description of the algorithm.
Before reading this section, we recommend reading the intuitive description of the algorithm in  \cref{subsec:intuition_algo}.
Some other comments for certain parts of the algorithm that might be helpful for understanding the algorithm will be included in the footnotes when they are described.
Note that in future subsections, we will prove why this algorithm constructs a desired coloring and recall parts of the algorithm.

\medskip
\noindent\textbf{\boldmath Stage $(0,-1,0)$: Initialization}
\begin{itemize}
	\item Let $(Y^{(0,-1,0)},L^{(0,-1,0)})$ be an $(X_{r^*},0)$-progress of $(Y_1,L)$.
	\item Let $D^{(0,0,0)} := \emptyset$. 
	\item Let $F_{j,0} := \emptyset$ for every $j \in [\lvert \V \rvert-1]$.
	\item Let $S^{(0,0)}_{j,t} := \emptyset$, $S^{(0,1)}_{j,t}:=\emptyset$ and $S^{(0,2)}_{j,t} := \emptyset$ for every $j \in [\lvert V \rvert-1]$ and $t \in V(T)$.
	\item Let $E^{(0)}_{j,t} := \emptyset$ for every $j \in [\lvert \V \rvert-1]$ and $t \in V(T)$.
\end{itemize}
For $i=0,1,2,\dots$, let $t$ be the node of $T$ with $\sigma_T(t)=i$ and perform the following steps:\\[1ex]
\hspace*{4mm} 
\textbf{\boldmath Stage: Building Fences}
		\begin{itemize}
 			\item Define the following:
				\begin{itemize}
					\item For every $j \in [\lvert \V \rvert-1]$, 
						\begin{itemize}
\item Let $F'_{j,p}:=F_{j,0}$ if $t=r^*$, and let $p$ be the parent of $t$ if $t \neq r^*$, \footnote{Note that when $t \neq r^*$, $p$ has been seen earlier so $F'_{j,p}$ has been defined.}
\item For every $F'_{j,p} \cap V(T_t)$-part $T'$ of $T_t$ containing $t$, define $F_{j,T'}$ to be a\\
   $(T', \X|_{X_{V(T')} \cap \overline{I_j}},  (Y^{(i,-1,0)} \cup X_{\partial T'}) \cap X_{V(T')} \cap \overline{I_j})$-fence. 
							\item Let $F_{j,p} := F'_{j,p} \cup \bigcup_{T'}F_{j,T'}$, where the union is over all $F'_{j,p} \cap V(T_t)$-parts $T'$ of $T_t$ containing $t$. 
							\item Let $T_{j,t} := \bigcup_{T'}T'$, where the union is over all $F_{j,p} \cap V(T_t)$-parts $T'$ of $T_t$ containing $t$.
							\item Let $\partial T_{j,t} := \bigcup_{T'} (\partial T'-\{t\})$, where the union is over all $F_{j,p} \cap V(T_t)$-parts $T'$ of $T_t$ containing $t$.
						\end{itemize}
					\item Let $Z_t := (X_{V(T_t)}-\bigcup_{j=1}^{\lvert \V \rvert-1}I_j^\circ) \cup (\bigcup_{j=1}^{\lvert \V \rvert-1} (I_j^\circ \cap X_{V(T_{j,t})}))$. 
				\end{itemize}
		
\hspace*{-6mm} 
\noindent\textbf{\boldmath Stage $(i,-1,\star)$: Isolate the $k$-th Oldest Component Intersecting a Segment in $\Se_j^\circ$:}
 				\item
				For each $k \in [0,w_0-1]$, define the following:
				\begin{itemize}
					\item For every $j \in [\lvert \V \rvert-1]$, define $M^{(t)}_{j,k}$ to be the monochromatic $E_{j,t}^{(i)}$-pseudocomponent in $G[Y^{(i,-1,k)}]$ such that $\sigma(M^{(t)}_{j,k})$ is the smallest among all monochromatic \linebreak $E^{(t)}_{j,t}$-pseudocomponents $M$ in $G[Y^{(i,-1,k)}]$ intersecting $X_t$ with $A_{L^{(i,-1,k)}}(V(M)) \cap X_{V(T_t)}-X_t \neq \emptyset$ \footnote{Recall that $A_{L^{(i,-1,k)}}(V(M))$ denotes the set of gates for vertices in $V(M)$ as defined in \cref{subsec:ProofMainThm}.} and contained in some $s$-segment $S \in \Se_j^\circ$ whose level equals the color of $M$ such that $V(M) \cap \bigcup_{\ell=0}^{k-1}V(M_{j,\ell}^{(t)})=\emptyset$.					
					\item Let $W_0^{(i,-1,k)} := X_{V(T_t)} \cap \bigcup_{j \in [\lvert \V \rvert-1]}V(M_{j,k}^{(t)})$.
					\item Let $(Y^{(i,-1,k,0)},L^{(i,-1,k,0)}) := (Y^{(i,-1,k)},L^{(i,-1,k)})$.
					\item For every $q \in [0,s+1]$, define the following:
						\begin{itemize}
							\item Let $W_1^{(i,-1,k,q)} := \{v \in W_0^{(i,-1,k)}: L^{(i,-1,k)}(v)=q+1\}$.
							\item Let $W_2^{(i,-1,k,q)} := A_{L^{(i,-1,k,q)}}(W_1^{(i,-1,k,q)}) \cap Z_t$. 
							\item Let $(Y^{(i,-1,k,q+1)},L^{(i,-1,k,q+1)})$ be a $(W_2^{(i,-1,k,q)},q+1)$-progress of $(Y^{(i,-1,k,q)},L^{(i,-1,k,q)})$.
						\end{itemize}
					\item Let $(Y^{(i,-1,k+1)},L^{(i,-1,k+1)}):= (Y^{(i,-1,k,s+2)},L^{(i,-1,k,s+2)})$.
				\end{itemize}
				\hspace*{-4mm}\textbf{\boldmath  Stage $(i,0,\star)$: Isolate the other components intersecting $X_t$}
				\item Let $(Y^{(i,0,0)},L^{(i,0,0)}) := (Y^{(i,-1,w_0)},L^{(i,-1,w_0)})$.
				
				\item 
				For each $k \in [0,s+1]$, define the following:
				\begin{itemize}
					\item 
					\begin{minipage}[t]{100mm}
\vspace*{-5ex}
\begin{align*}
\text{Let }& W_0^{(i,0,k)} :=  \{v \in Y^{(i,0,k)} \cap X_{V(T_t)} :   L^{(i,0,k)}(v)=\{k+1\} \text{ and}\\
& \text{there exists a path } P \text{ in } G[Y^{(i,0,k)} \cap X_{V(T_t)}]  \text{ from }v \text{ to } X_t \\
&  \text{such that }  L^{(i,0,k)}(q)=\{k+1\} \text{ for all }q \in V(P)\}.
						\end{align*}
					\end{minipage}
				\item Let $W_2^{(i,0,k)} := A_{L^{(i,0,k)}}(W_0^{(i,0,k)}) \cap Z_t$. 
				\item Let $(Y^{(i,0,k+1)},L^{(i,0,k+1)})$ be a $(W_2^{(i,0,k)},k+1)$-progress of $(Y^{(i,0,k)},L^{(i,0,k)})$.
				\end{itemize}

\hspace*{-4mm}\textbf{\boldmath Stage $(i,\geq 1,\star)$: Isolating the Fences of $T_{j,t}$}
				\item Let $W_3^{(i,-1)} := \bigcup_{j=1}^{\lvert \V \rvert-1}(X_{t} \cap I_j)$. 
				\item Let $W_4^{(i)} := \bigcup_{j=1}^{\lvert \V \rvert-1}(X_{V(T_{j,t})} \cap I_j)$. 
				\item For each $\ell \in [0,\lvert V(T) \rvert]$, define the following:
					\begin{itemize}
						\item Let $W_3^{(i,\ell)} := W_3^{(i,\ell-1)} \cup \bigcup_{j=1}^{\lvert \V \rvert-1}\bigcup_{q} (X_q \cap I_j)$, where the last union is over all nodes $q$ in $\partial T_{j,t}$ for which there exists a monochromatic path in $G[Y^{(i,\ell,s+2)} \cap I_j \cap W_4^{(i)}]$ from $W_3^{(i,\ell-1)}$ to 
						\begin{align*}
						X_q \cap & I_j-((D^{(i,\ell,0)}-X_t) \cup \{v \in D^{(i,\ell,0)} \cap X_t \cap X_{q'} \cap I_j: q' \in V(T_t)-\{t\}, \\
						& q'\text{ is a witness for }X_{q'} \cap I_j \subseteq W_3^{(i',\ell')}\text{ for some }i' \in [0,i-1]\text{ and }\ell' \in [0,\lvert V(T) \rvert]\})
						\end{align*}
						internally disjoint from $X_t \cup X_{\partial T_{j,t}}$.
						\item Let $(Y^{(i,\ell+1,0)}, L^{(i,\ell+1,0)})$ be a $(W_{3}^{(i,\ell)},0)$-progress of $(Y^{(i,\ell,s+2)},L^{(i,\ell,s+2)})$.
						\item For each $k \in [0,s+1]$, define the following:
										\begin{itemize}
											\item 
											\begin{minipage}[t]{100mm}
												\vspace*{-5ex}
\begin{align*}
\text{Let }&W_{0}^{(i,\ell+1,k)}  := \{v \in Y^{(i,\ell+1,k)} \cap W_{4}^{(i)}: L^{(i,\ell+1,k)}(v)=\{k+1\} \\
& \text{ and there exists a path } P \text{ in } G[Y^{(i,\ell+1,k)} \cap W_{4}^{(i)}] \text{ from }v \text{ to } W_{3}^{(i,\ell)} \\
& \text{ such that } L^{(i,\ell+1,k)}(q)=\{k+1\} \text{ for all }q \in V(P)\}.
												\end{align*}
											\end{minipage}
											\item Let $W_2^{(i,\ell+1,k)} := A_{L^{(i,\ell+1,k)}}(W_0^{(i,\ell+1,k)}) \cap W_4^{(i)}$. 
											\item Let $(Y^{(i,\ell+1,k+1)},L^{(i,\ell+1,k+1)})$ be the $(W_2^{(i,\ell+1,k)},k+1)$-progress of $(Y^{(i,\ell+1,k)},L^{(i,\ell+1,k)})$.
											\item Let $D^{(i,\ell,k+1)} := D^{(i,\ell,k)} \cup W_0^{(i,\ell+1,k)}$.
										\end{itemize}
						\item Let $D^{(i,\ell+1,0)}:=D^{(i,\ell,s+2)}$.
					\end{itemize}

\hspace*{-10mm}\textbf{Stage: Adding Fake Edges}
				\item For each $j \in [\lvert \V \rvert-1]$ and $t' \in V(T)-(V(T_t)-\{t\})$, let $E^{(i+1)}_{j,t'}:=E^{(i)}_{j,t'}$.
				\item For each $j \in [\lvert \V \rvert-1]$, let $E^{(i,0)}_{j,t}:=E^{(i)}_{j,t}$.
				\item For each $j \in [\lvert \V \rvert-1]$ and $t' \in V(T_t)-\{t\}$, let $E^{(i+1)}_{j,t'}:=E^{(i,\lvert V(G) \rvert)}_{j,t}$, where for every $\ell \in [0,\lvert V(G) \rvert-1]$, let $E^{(i,\ell+1)}_{j,t}$ be the union of $E^{(i,\ell)}_{j,t}$ and the set consisting of the 2-element sets $\{u,v\}$ satisfying the following.
				\begin{itemize}
					\item $\{u,v\} \subseteq Y^{(i,\lvert V(T) \rvert+1,s+2)}$. 
					\item $L^{(i,\lvert V(T) \rvert+1,s+2)}(u) = L^{(i,\lvert V(T) \rvert+1,s+2)}(v)$. 
					\item There exists an $s$-segment $S$ in $\Se_{j}^\circ$ whose level equals the color of $u$ and $v$ such that $\{u,v\} \subseteq S$. \footnote{The above three conditions imply that both $u$ and $v$ are precolored by using the same color, and it is possible that $u$ and $v$ would be contained in the same monochromatic component in the future.}
					\item There exists $t'' \in \partial T_{j,t}$ such that: \footnote{Roughly speaking, the following conditions for $t''$ come from the intuition that we need the fake edge $uv$ to indicate that $u,v$ could possibly lie in the same monochromatic component in the future because we have to color the gates for some monochromatic pseudocomponents with smaller $\sigma$-values than the pseudocomponents containing $u$ or $v$, and those gates are contained in $X_{V(T_{t''})}-X_{t''}$.}
					\begin{itemize}
						\item $\{u,v\} \subseteq X_{t''}$,
						\item $V(M) \cap X_{t''} \neq \emptyset$, where $M$ is the monochromatic $E^{(i,\ell)}_{j,t}$-pseudocomponent in \linebreak $G[Y^{(i,\lvert V(T) \rvert+1,s+2)}]$ such that $\sigma(M)$ is the $(\ell+1)$-th smallest among all monochromatic $E^{(i,\ell)}_{j,t}$-pseudocomponents in  $G[Y^{(i,\lvert V(T) \rvert+1,s+2)}]$,
						\item $M$ is contained in some $s$-segment in $\Se_{j}^\circ$ whose level equals the color of $M$,
						\item $t''$ is a witness for $X_{t''} \cap I_j \subseteq W_3^{(i,\ell')}$ for some $\ell' \in [0,\lvert V(T) \rvert]$,
						\item $A_{L^{(i,\lvert V(T) \rvert+1,s+2)}}(V(M)) \cap X_{V(T_{t''})}-X_{t''} \neq \emptyset$,
						\item $A_{L^{(i,\lvert V(T) \rvert+1,s+2)}}(V(M_u)) \cap X_{V(T_{t''})}-X_{t''} \neq \emptyset$,
						\item $A_{L^{(i,\lvert V(T) \rvert+1,s+2)}}(V(M_v)) \cap X_{V(T_{t''})}-X_{t''} \neq \emptyset$, and
						\item $\sigma(M) < \min\{\sigma(M_u),\sigma(M_v)\}$, 
						
					\end{itemize}
					where $M_u$ and $M_v$ are the monochromatic $E^{(i,\ell)}_{j,t}$-pseudocomponents in $G[Y^{(i,\lvert V(T) \rvert+1,s+2)}]$ containing $u$ and $v$, respectively.
				\item Let $M$, $M_u$ and $M_v$ be the monochromatic $E_{j,t}^{(i,\ell)}$-pseudocomponents as in the previous bullet.
						If $(V(M_u) \cup V(M_v)) \cap X_t \neq \emptyset$, then at least one of the following does not hold\footnote{The previous bullet describes conditions showing that $u$ and $v$ can possibly be in the same monochromatic component in the future because we have to color gates for other monochromatic pseudocomponents. This bullet say that if all of the following conditions hold, then actually coloring those gates would not make $u$ and $v$ be in the same monochromatic component in the future, so we should not add a fake edge. We will prove why it is true in later subsections.},
						\begin{itemize}
							\item for every monochromatic $E_{j,t}^{(i,\ell)}$-pseudocomponent $M'$ in $G[Y^{(i,\lvert V(T) \rvert+1,s+2)}]$ with $\sigma(M') \leq \sigma(M)$ such that $M'$ is contained in some $s$-segment in $\Se_{j}^\circ$ whose level equals the color of $M'$, $M'$ satisfies that $A_{L^{(i,\lvert V(T) \rvert+1,s+2)}}(V(M')) \cap X_{V(T_{t})}-X_{t} \subseteq X_{V(T_{t''})}-(X_{t''} \cup Z_t)$,
							\item there exists a parade $(t_{-\beta_{u,v,\ell}},t_{\beta_{u,v,\ell}+1},\dots,t_{-1},t_1,t_2,\dots,t_{\alpha_{u,v,\ell}})$ for some $\alpha_{u,v,\ell} \in {\mathbb N}$ and $\beta_{u,v,\ell} \in {\mathbb N}$ such that 
								\begin{itemize}
									\item $t_{-\beta_{u,v,\ell}} \in V(T_{t''})-\{t''\}$, 
									\item $\alpha_{u,v,\ell} = \psi_2(\alpha_{u,\ell},\alpha_{v,\ell})$ and $\beta_{u,v,\ell} = \psi_3(\alpha_{u,\ell},\alpha_{v,\ell})$, \footnote{Recall that $\psi_2$ and $\psi_3$ were defined in \cref{subsec:const_main_lemma}.}\\
									where $\alpha_{u,\ell}$ and $\alpha_{v,\ell}$ are the integers such that $\sigma(M_{u,\ell})$ and $\sigma(M_{v,\ell})$ are the $\alpha_{u,\ell}$-th and the $\alpha_{v,\ell}$-th smallest among all monochromatic $E_{j,t}^{(i,\ell)}$-pseudocomponents $M'$ in $G[Y^{(i,\lvert V(T) \rvert+1,s+2)}]$ contained in some $s$-segment in $\Se_j^\circ$ whose color equals the color of $M_u$ and $M_v$ with $A_{L^{(i,\lvert V(T) \rvert+1,s+2)}}(V(M')) \cap X_{V(T_{t''})}-X_{t''} \neq \emptyset$, respectively, 
									\item for any distinct $\ell_1,\ell_2 \in [-\beta_{u,v,\ell}, \alpha_{u,v,\ell}]-\{0\}$, $X_{t_{\ell_1}} \cap I_j - X_{t''}$ and $X_{t_{\ell_2}} \cap I_j - X_{t''}$ are disjoint non-empty sets,
									\item $\bigcup_{M''} (A_{L^{(i,\lvert V(T) \rvert+1,s+2)}}(V(M'')) \cap X_{V(T_{t})}-X_{t}) \subseteq (X_{V(T_{t_{\alpha_{u,v,\ell}}})}-X_{t_{\alpha_{u,v,\ell}}}) \cap I_j^\circ$, where the union is over all monochromatic $E_{j,t}^{(i,\ell)}$-pseudocomponents $M''$ in \linebreak $G[Y^{(i,\lvert V(T) \rvert+1,s+2)}]$ such that $V(M'')$ is contained in some $s$-segment in $\Se_j^\circ$ whose level equals the color of $M''$, $\sigma(M'') \leq \sigma(M)$, $V(M'') \cap X_{t''} \neq \emptyset$, and \linebreak $A_{L^{(i,\lvert V(T) \rvert+1,s+2)}}(V(M'')) \cap X_{V(T_{t''})}-X_{t''} \neq \emptyset$,
									\item there exists $x_\ell \in \{u,v\}$ such that $A_{L^{(i,\lvert V(T) \rvert+1,s+2)}}(V(M_{x_\ell})) \cap (X_{V(T_{t})}-X_{t}) \cap X_{V(T_{t_1})}-X_{t_1} = \emptyset$, and $A_{L^{(i,\lvert V(T) \rvert+1,s+2)}}(V(M_{x_\ell})) \cap (X_{V(T_{t})}-X_{t}) \cap X_{V(T_{t_{-\beta_{u,v,\ell}}})}-(X_{t_{-\beta_{u,v,\ell}}} \cup I_j^\circ) = \emptyset$, 
								\end{itemize}
							\item there exists a monochromatic $E_{j,t}^{(i)}$-pseudocomponent $M^*$ in $G[Y^{(i,\lvert V(T) \rvert+1,s+2)}]$ with $\sigma(M^*) \leq \sigma(M)$ such that $M^*$ is contained in some $s$-segment in $\Se_{j}^\circ$ whose level equals the color of $M^*$ satisfying that $A_{L^{(i,\lvert V(T) \rvert+1,s+2)}}(V(M^*)) \cap X_{V(T_{t})}-X_{t} \subseteq X_{V(T_{t''})}-(X_{t''} \cup Z_t)$.
						\end{itemize}
					\item There do not exist $q \in V(T)$ with $T_{j,t} \subseteq T_{j,q}$ and $i_q<i_t$, a witness $q' \in \partial T_{j,q}$ for $X_{q'} \cap I_j \subseteq W_3^{(i_{q},\ell')}$ for some $\ell' \in [0,\lvert V(T) \rvert]$, and a monochromatic $E_{j,q}^{(i_{q})}$-pseudocomponent in $G[Y^{(i_{q},\lvert V(T) \rvert+1,s+2)}]$ intersecting $X_{q'}$ and $\{u,v\}$.
				\end{itemize}

\hspace*{-10mm}\textbf{Stage: Moving to the Next Node in the Tree}
				\item Let $(Y^{(i+1,-1,0)},L^{(i+1,-1,0)})$ be a $(X_{t'},0)$-progress of $(Y^{(i,\lvert V(T) \rvert+1,s+2)},L^{(i,\lvert V(T) \rvert+1,s+2)})$, where $t'$ is the node of $T$ with $\sigma_T(t')=i+1$. 
				\item Let $D^{(i+1,0,0)}:=D^{(i,\lvert V(T) \rvert,s+2)}$.
				\item For every node $t'$ with $t '\in V(T_t)-\{t\}$ and $j \in [\lvert V \rvert-1]$, 
					\begin{itemize}
						\item If $t' \in V(T_{j,t})-\partial T_{j,t}$, then define 
							\begin{itemize}
								\item $S^{(i+1,0)}_{j,t'} := S^{(i,0)}_{j,t'} \cup \bigcup_{\ell=0}^{\lvert V(T) \rvert}\bigcup_{k=0}^{s+1}W_0^{(i,\ell+1,k)}$,
								\item $S^{(i+1,1)}_{j,t'} := S^{(i,1)}_{j,t'} \cup \bigcup_{j=1}^{\lvert \V \rvert-1}\bigcup_q (X_q \cap I_j)$, where the last union is over all nodes $q$ of $T$ such that $q \in V(T_{t'})-\{t'\}$ and $q$ is a witness for $X_q \cap I_j \subseteq W_3^{(i,\ell')}$ for some $\ell' \in [0,\lvert V(T) \rvert]$, and 
								\item $S^{(i+1,2)}_{j,t'} := S^{(i,2)}_{j,t'} \cup \bigcup_{k=0}^{w_0-1}W_0^{(i,-1,k)} \cup \bigcup_{k=0}^{s+1}W_0^{(i,0,k)}$;
							\end{itemize}
						\item otherwise, define $S^{(i+1,0)}_{j,t'} := S^{(i,0)}_{j,t'}$, $S^{(i+1,1)}_{j,t'} := S^{(i,1)}_{j,t'}$ and $S^{(i+1,2)}_{j,t'} := S^{(i,2)}_{j,t'}$.
					\end{itemize}
\hspace*{-10mm}\textbf{Stage: Building a New Fence}
				\item For every $j \in [\lvert \V \rvert-1]$, define the following:
				\begin{itemize}
					\item For every $F_{j,p} \cap V(T_t)$-part $T'$ of $T_t$ containing $t$, let $F'_{j,T'}$ be a \\
$(T', \X|_{X_{V(T')} \cap \overline{I_j}},  (Y^{(i+1,-1,0)} \cup X_{\partial T'}) \cap X_{V(T')} \cap \overline{I_j})$-fence. 
					\item Let $F'_{j,t} := F_{j,p} \cup \bigcup_{T'}F'_{j,T'}$, where the union is over all $F_{j,p} \cap V(T_t)$-parts $T'$ of $T_t$ containing $t$. 
				\end{itemize}
\end{itemize}

\subsection{Remark about the main goal}

It is clear that if a triple $(i',j',k')$ is lexicographically smaller than another triple $(i,j,k)$, then $L^{(i,j,k)}(v) \subseteq L^{(i',j',k')}(v)$ for every $v \in V(G)$, so $Y^{(i,j,k)} \supseteq Y^{(i',j',k')}$. 
In addition, $X_t \subseteq Y^{(i_t,-1,0)}$. 
Recall that $i_t$ is the number such that $\sigma_T(t)=i_t$.
In particular, $\lvert L^{(\lvert V(T) \rvert,-1,0)}(v) \rvert=1$ for every $v \in V(G)$, so there exists a unique $L^{(\lvert V(T) \rvert,-1,0)}$-coloring $c$ of $G$. By construction, $c$ is an $L$-coloring. We prove below that $c$ has clustering $\eta^*$.

\subsection{Claims related to fences} \label{subsec:claims_fences}

We first recall the relevant parts of the algorithm.

\medskip
\noindent\textbf{\boldmath Stage $(0,-1,0)$: Initialization}

	$(Y^{(0,-1,0)},L^{(0,-1,0)})$ is defined, and $F_{j,0} := \emptyset$ for every $j \in [\lvert \V \rvert-1]$. 
	Some other sets are defined in this stage. 
	See \cref{subsec:alog} for details.

\medskip

\noindent For $i=0,1,2,\dots$, let $t$ be the node of $T$ with $\sigma_T(t)=i$ and perform the following steps:\\[1ex]
\hspace*{4mm} 
\textbf{\boldmath Stage: Building Fences}
		\begin{itemize}
 			\item Define the following:
				\begin{itemize}
					\item For every $j \in [\lvert \V \rvert-1]$, 
						\begin{itemize}
\item Let $F'_{j,p}:=F_{j,0}$ if $t=r^*$, and let $p$ be the parent of $t$ if $t \neq r^*$,
\item For every $F'_{j,p} \cap V(T_t)$-part $T'$ of $T_t$ containing $t$, define $F_{j,T'}$ to be a\\
   $(T', \X|_{X_{V(T')} \cap \overline{I_j}},  (Y^{(i,-1,0)} \cup X_{\partial T'}) \cap X_{V(T')} \cap \overline{I_j})$-fence. 
							\item Let $F_{j,p} := F'_{j,p} \cup \bigcup_{T'}F_{j,T'}$, where the union is over all $F'_{j,p} \cap V(T_t)$-parts $T'$ of $T_t$ containing $t$. 
							\item Let $T_{j,t} := \bigcup_{T'}T'$, where the union is over all $F_{j,p} \cap V(T_t)$-parts $T'$ of $T_t$ containing $t$.
							\item Let $\partial T_{j,t} := \bigcup_{T'} (\partial T'-\{t\})$, where the union is over all $F_{j,p} \cap V(T_t)$-parts $T'$ of $T_t$ containing $t$.
						\end{itemize}
					\item Let $Z_t := (X_{V(T_t)}-\bigcup_{j=1}^{\lvert \V \rvert-1}I_j^\circ) \cup (\bigcup_{j=1}^{\lvert \V \rvert-1} (I_j^\circ \cap X_{V(T_{j,t})}))$. 
				\end{itemize}
		\end{itemize}		
\hspace*{4mm} 
\noindent\textbf{\boldmath Stage $(i,-1,\star)$:} See \cref{subsec:alog} for details. \\
\hspace*{4mm}
\noindent\textbf{\boldmath  Stage $(i,0,\star)$:} See \cref{subsec:alog} for details. \\
\hspace*{4mm}
\noindent\textbf{\boldmath Stage $(i,\geq 1,\star)$:} See \cref{subsec:alog} for details. \\
\hspace*{4mm}
\noindent\textbf{Stage: Adding Fake Edges:} See \cref{subsec:alog} for details. \\
\hspace*{4mm}
\noindent\textbf{Stage: Moving to the Next Node in the Tree}

			$(Y^{(i+1,-1,0)},L^{(i+1,-1,0)})$ is defined. 
			Other sets are defined.
			See \cref{subsec:alog} for details.\\
\hspace*{4mm}
\noindent\textbf{Stage: Building a New Fence}
				
				For every $j \in [\lvert \V \rvert-1]$, define the following:
				\begin{itemize}
					\item For every $F_{j,p} \cap V(T_t)$-part $T'$ of $T_t$ containing $t$, let $F'_{j,T'}$ be a \\
$(T', \X|_{X_{V(T')} \cap \overline{I_j}},  (Y^{(i+1,-1,0)} \cup X_{\partial T'}) \cap X_{V(T')} \cap \overline{I_j})$-fence. 
					\item Let $F'_{j,t} := F_{j,p} \cup \bigcup_{T'}F'_{j,T'}$, where the union is over all $F_{j,p} \cap V(T_t)$-parts $T'$ of $T_t$ containing $t$. 
				\end{itemize}

Now we prove claims related to fences.

\begin{claim} \label{claim:boundaryZ}
Let $t \in V(T)$ and $j \in [\lvert V \rvert-1]$.
Let $(Y',L')$ be an $(s,\V)$-standard pair and let $Z \subseteq Y'$.
Then $A_{L'}(Z) \subseteq N^{\geq s}_G[Z]$, $A_{L'}(Z) \cap X_{V(T_{j,t})} \subseteq N_G^{\geq s}(Z \cap X_{V(T_{j,t})}) \cup X_{\partial T_{j,t}} \cup X_t$, and $\lvert A_{L'}(Z) \cap X_{V(T_{j,t})} \rvert \leq f(\lvert Z \cap X_{V(T_{j,t})} \rvert) + \lvert X_{\partial T_{j,t}} \rvert + \lvert X_t \rvert$.
\end{claim}

\begin{proof}
It is obvious that $A_{L'}(Z) \subseteq N^{\geq s}_G[Z]$ since $(Y',L')$ is an $(s,\V)$-standard pair.
Since $(Y',L')$ is an $(s,\V)$-standard pair, $A_{L'}(Z) \cap X_{V(T_{j,t})} \subseteq N_G^{\geq s}(Z) \cap X_{V(T_{j,t})} \subseteq N_G^{\geq s}(Z \cap X_{V(T_{j,t})}) \cup (X_{\partial T_{j,t}} \cup X_t)$.
So $\lvert A_{L'}(Z) \cap X_{V(T_{j,t})} \rvert \leq \lvert  N_G^{\geq s}(Z \cap X_{V(T_{j,t})}) \rvert + \lvert X_{\partial T_{j,t}} \rvert + \lvert X_t \rvert \leq f(\lvert Z \cap X_{V(T_{j,t})} \rvert) + \lvert X_{\partial T_{j,t}} \rvert + \lvert X_t \rvert$ by \cref{BoundedGrowthLtw}.
\end{proof}

\begin{claim} \label{claim:I_jbasic1}
Let $i,i' \in \mathbb{N}_0$ with $i'<i$, and let $t$ and $t'$ be nodes of $T$ with $\sigma_T(t)=i$ and $\sigma_T(t')=i'$. 
Let $j \in [\lvert V \rvert-1]$.
Then the following hold:
	\begin{itemize}
		\item $\lvert Y^{(i',-1,0)} \cap \overline{I_j} \cap X_{V(T_t)} \cap X_{V(T_{j,t'})} \rvert \leq w_0$.
		\item $\lvert X_{\partial T_{j,t'}} \cap X_{V(T_t)} \cap \overline{I_j} \rvert \leq w_0$.
	\end{itemize}
\end{claim}

\begin{proof}
We may assume that $t \in V(T_{t'})$, for otherwise, $X_{V(T_t)} \cap X_{V(T_{j,t'})} = \emptyset$ and we are done.

Let $F'_{j,p} := F_{j,0}$ if $t'=r^*$; let $p$ be the parent of $t'$ if $t' \neq r^*$.
By definition, $F_{j,p} = F'_{j,p} \cup \bigcup_{T'}F_{j,T'}$, where the union is over all $F'_{j,p} \cap V(T_{t'})$-parts $T'$ of $T_{t'}$ containing $t'$, and each $F_{j,T'}$ is a $(T', \X_{X_{V(T')} \cap \overline{I_j}}, (Y^{(i',-1,0)} \cup X_{\partial T'}) \cap X_{V(T')} \cap \overline{I_j})$-fence.
So for each $F_{j,p} \cap V(T_{t'})$-part $T'$ of $T_{t'}$, $\lvert (Y^{(i',-1,0)} \cap \overline{I_j} \cap X_{V(T')}) \cup (X_{\partial T'} \cap \overline{I_j}) \rvert \leq w_0$.
By definition, $T_{j,t'} = \bigcup_{T'}T'$, where the union is over all $F_{j,p} \cap V(T_{t'})$-parts $T'$ of $T_{t'}$ containing $t'$.
Since $t \in V(T_{t'})-\{t'\}$, there exists at most one $F_{j,p} \cap V(T_{t'})$-part of $T_{t'}$ containing both $t'$ and $t$.
If there exists an $F_{j,p} \cap V(T_{t'})$-part of $T_{t'}$ containing both $t'$ and $t$, then denote it by $T^*$; otherwise  let $T^*:=\emptyset$.
So $X_{V(T_t)} \cap X_{V(T_{j,t'})} \subseteq X_{V(T^*)}$.
Hence $\lvert Y^{(i',-1,0)} \cap \overline{I_j} \cap X_{V(T_t)} \cap X_{V(T_{j,t'})} \rvert \leq \lvert Y^{(i',-1,0)} \cap \overline{I_j} \cap X_{V(T^*)} \rvert \leq w_0$.
Similarly, $\lvert X_{\partial T_{j,t'}} \cap X_{V(T_t)} \cap \overline{I_j} \rvert \leq \lvert X_{\partial T_{j,t'}} \cap X_{V(T_{j,t'})} \cap X_{V(T_t)} \cap \overline{I_j} \rvert \leq \lvert X_{\partial T_{j,t'}} \cap X_{V(T^*)} \cap \overline{I_j} \rvert \leq \lvert X_{\partial T^*} \cap \overline{I_j} \rvert \leq w_0$.
\end{proof}

\begin{claim} \label{claim:fencedownside}
Let $i \in \mathbb{N}_0$ and let $t \in V(T)$ with $\sigma_T(t)=i$.
Let $p$ be the parent of $t$ if $t \neq r^*$; otherwise let $F_{j,p}$ be the set mentioned in the algorithm when $t=r^*$.
Then 
	\begin{itemize}
		\item for every $q \in F_{j,p}$, $Y^{(i,-1,0)} \cap \overline{I_j} \cap X_{V(T_q)}-X_q \neq \emptyset$, and 
		\item for every $q \in F'_{j,t}$, $Y^{(i+1,-1,0)} \cap \overline{I_j} \cap X_{V(T_q)}-X_q \neq \emptyset$.
	\end{itemize}
\end{claim}

\begin{proof}
We shall prove this claim by induction on $i$.
We first assume $i=0$.
So $t=r^*$.
Since $F_{j,0}=\emptyset$, $T$ is the only $F_{j,0} \cap V(T_{t})$-part containing $t$, and $\partial T = \emptyset$.
Since $Y^{(0,-1,0)} \subseteq X_{r^*}$, we know $F_{j,T} = \emptyset$ by the definition of a fence (Statement~2 in \cref{fence lemma}).
Hence the set $F_{j,p}$ defined in the algorithm is $\emptyset$.
So Statement~1 in this claim holds.
Hence we may assume that $q \in F'_{j,t}$.
Since $F_{j,p}=\emptyset$, $q \in F'_{j,T'}$ for some $\emptyset$-part $T'$ of $T_t$ containing $t$.
So $T'=T$, and $F'_{j,T'}$ is a $(T,\X|_{X_{V(T)} \cap \overline{I_j}}, Y^{(1,-1,0)} \cap \overline{I_j})$-fence.
By the definition of a fence (Statement~2 in \cref{fence lemma}), there exist at least two $F'_{j,T'}$-parts $T''$ satisfying $q \in \partial T''$ and $X_{V(T'')} \cap Y^{(1,-1,0)} \cap \overline{I_j} - X_{q} \neq \emptyset$.
Note that some such part $T''$ is contained in $T_q$.
So $Y^{(1,-1,0)} \cap \overline{I_j} \cap X_{V(T_q)}-X_q \neq \emptyset$.
Hence the claim holds when $i=0$.

So we may assume $i>0$, and the lemma holds for every smaller $i$.
We first assume $q \in F_{j,p}$ and prove Statement~1.
So either $q \in F'_{j,p}$ or $q \in F_{j,T'}-F'_{j,p}$ for some $F'_{j,p} \cap V(T_t)$-part $T'$ of $T_t$ containing $t$.
If $F'_{j,p} \cap V(T_q) \neq \emptyset$, then there exists $q' \in F'_{j,p} \cap V(T_q)$ and $Y^{(i,-1,0)} \cap \overline{I_j} \cap X_{V(T_q)}-X_q \supseteq Y^{(i_p+1,-1,0)} \cap \overline{I_j} \cap X_{V(T_{q'})}-X_{q'} \neq \emptyset$ by the induction hypothesis, where $i_p = \sigma_T(p)$.
So we may assume $F'_{j,p} \cap V(T_q) = \emptyset$.
In particular, $q \in F_{j,T'}-F'_{j,p}$ for some $F'_{j,p} \cap V(T_t)$-part $T'$ of $T_t$ containing $t$. 
By the definition of a fence (Statement~2 in \cref{fence lemma}), there exist at least two $F_{j,T'}$-parts $T''$ satisfying $q \in \partial T''$ and $(Y^{(i,-1,0)} \cup X_{\partial T'}) \cap X_{V(T'')} \cap \overline{I_j} - X_{q} \neq \emptyset$.
So one such $T''$ is contained in $T_q$.
For every $t' \in \partial T'$, $t' \in F'_{j,p}$, so by the induction hypothesis, $Y^{(i_p+1,-1,0)} \cap \overline{I_j} \cap X_{V(T_{t'})}-X_{t'} \neq \emptyset$.
For every $t' \in \partial T' \cap V(T'')$, $T_{t'} \subseteq T_q$, so $Y^{(i,-1,0)} \cap \overline{I_j} \cap X_{V(T_{q})}-X_{q} \supseteq Y^{(i_p+1,-1,0)} \cap \overline{I_j} \cap X_{V(T_{t'})}-X_{t'} \neq \emptyset$.
Hence if $X_{\partial T'} \cap X_{V(T'')} \neq \emptyset$, then $\partial T' \cap V(T'') \neq \emptyset$, so $Y^{(i,-1,0)} \cap \overline{I_j} \cap X_{V(T_{q})}-X_{q} \neq \emptyset$; otherwise, $Y^{(i,-1,0)} \cap \overline{I_j} \cap X_{V(T_q)}-X_q \supseteq (Y^{(i,-1,0)} \cup X_{\partial T'}) \cap X_{V(T'')} \cap \overline{I_j} \cap - X_{q} \neq \emptyset$.
This proves Statement~1 of this claim.

Now we assume $q \in F'_{j,t}$ and prove Statement 2 of this claim.
So either $q \in F_{j,p}$ or $q \in F'_{j,T'}-F_{j,p}$ for some $F_{j,p} \cap V(T_t)$-part $T'$ of $T_t$ containing $t$.
If $q \in F_{j,p}$, then $Y^{(i,-1,0)} \cap \overline{I_j} \cap X_{V(T_q)}-X_q \neq \emptyset$ by Statement 1 of this claim.
So we may assume $q \in F'_{j,T'}-F_{j,p}$ for some $F_{j,p} \cap V(T_t)$-part $T'$ of $T_t$ containing $t$. 
By the definition of a fence (Statement~2 in \cref{fence lemma}), there exist at least two $F'_{j,T'}$-parts $T''$ satisfying $q \in \partial T''$ and $(Y^{(i+1,-1,0)} \cup X_{\partial T'}) \cap X_{V(T'')} \cap \overline{I_j} - X_{q} \neq \emptyset$.
So one such $T''$ is contained in $T_q$.
For every $t' \in \partial T'$, $t' \in F_{j,p}$, so by Statement 1 of this claim, $Y^{(i,-1,0)} \cap \overline{I_j} \cap X_{V(T_{t'})}-X_{t'} \neq \emptyset$.
For every $t' \in \partial T' \cap V(T'')$, $T_{t'} \subseteq T_q$, so $Y^{(i+1,-1,0)} \cap \overline{I_j} \cap X_{V(T_{q})}-X_{q} \supseteq Y^{(i,-1,0)} \cap \overline{I_j} \cap X_{V(T_{t'})}-X_{t'} \neq \emptyset$.
Hence if $X_{\partial T'} \cap X_{V(T'')} \neq \emptyset$, then $\partial T' \cap V(T'') \neq \emptyset$, so $Y^{(i+1,-1,0)} \cap \overline{I_j} \cap X_{V(T_{q})}-X_{q} \neq \emptyset$; otherwise, $Y^{(i+1,-1,0)} \cap \overline{I_j} \cap X_{V(T_q)}-X_q \supseteq (Y^{(i+1,-1,0)} \cup X_{\partial T'}) \cap X_{V(T'')} \cap \overline{I_j} \cap - X_{q} \neq \emptyset$.
This proves Statement~2 of this claim.
\end{proof}

\subsection{Claims related to Stage $(i,-1,\star)$ and Stage $(i,0,\star)$}

We first recall the relevant part of the algorithm.

\medskip
\noindent\textbf{\boldmath Stage $(0,-1,0)$: Initialization}

	Let $(Y^{(0,-1,0)},L^{(0,-1,0)})$ be an $(X_{r^*},0)$-progress of $(Y_1,L)$.
	Let $E^{(0)}_{j,t} := \emptyset$ for every $j \in [\lvert \V \rvert-1]$ and $t \in V(T)$.
	Other sets are defined.
	See \cref{subsec:alog} for details.

\medskip

\noindent For $i=0,1,2,\dots$, let $t$ be the node of $T$ with $\sigma_T(t)=i$ and perform the following steps:\\[1ex]
\hspace*{4mm} 
\textbf{\boldmath Stage: Building Fences:}
	Let $Z_t := (X_{V(T_t)}-\bigcup_{j=1}^{\lvert \V \rvert-1}I_j^\circ) \cup (\bigcup_{j=1}^{\lvert \V \rvert-1} (I_j^\circ \cap X_{V(T_{j,t})}))$. 
	Other sets are defined.
	See \cref{subsec:alog} for details.\\
\hspace*{4mm} 
\noindent\textbf{\boldmath Stage $(i,-1,\star)$: Isolate the $k$-th Oldest Component Intersecting a Segment in $\Se_j^\circ$}
		\begin{itemize}
			\item For each $k \in [0,w_0-1]$, define the following:
				\begin{itemize}
					\item For every $j \in [\lvert \V \rvert-1]$, define $M^{(t)}_{j,k}$ to be the monochromatic $E_{j,t}^{(i)}$-pseudocomponent in $G[Y^{(i,-1,k)}]$ such that $\sigma(M^{(t)}_{j,k})$ is the smallest among all monochromatic \linebreak $E^{(t)}_{j,t}$-pseudocomponents $M$ in $G[Y^{(i,-1,k)}]$ intersecting $X_t$ with $A_{L^{(i,-1,k)}}(V(M)) \cap X_{V(T_t)}-X_t \neq \emptyset$ and contained in some $s$-segment $S \in \Se_j^\circ$ whose level equals the color of $M$ such that $V(M) \cap \bigcup_{\ell=0}^{k-1}V(M_{j,\ell}^{(t)})=\emptyset$.					
					\item Let $W_0^{(i,-1,k)} := X_{V(T_t)} \cap \bigcup_{j \in [\lvert \V \rvert-1]}V(M_{j,k}^{(t)})$.
					\item Let $(Y^{(i,-1,k,0)},L^{(i,-1,k,0)}) := (Y^{(i,-1,k)},L^{(i,-1,k)})$.
					\item For every $q \in [0,s+1]$, define the following:
						\begin{itemize}
							\item Let $W_1^{(i,-1,k,q)} := \{v \in W_0^{(i,-1,k)}: L^{(i,-1,k)}(v)=q+1\}$.
							\item Let $W_2^{(i,-1,k,q)} := A_{L^{(i,-1,k,q)}}(W_1^{(i,-1,k,q)}) \cap Z_t$. 
							\item Let $(Y^{(i,-1,k,q+1)},L^{(i,-1,k,q+1)})$ be a $(W_2^{(i,-1,k,q)},q+1)$-progress of $(Y^{(i,-1,k,q)},L^{(i,-1,k,q)})$.
						\end{itemize}
					\item Let $(Y^{(i,-1,k+1)},L^{(i,-1,k+1)}):= (Y^{(i,-1,k,s+2)},L^{(i,-1,k,s+2)})$.
				\end{itemize}
		\end{itemize}
\hspace*{4mm}
\noindent\textbf{\boldmath  Stage $(i,0,\star)$: Isolate the other components intersecting $X_t$}
		\begin{itemize}
				\item Let $(Y^{(i,0,0)},L^{(i,0,0)}) := (Y^{(i,-1,w_0)},L^{(i,-1,w_0)})$.
				
				\item 
				For each $k \in [0,s+1]$, define the following:
				\begin{itemize}
					\item 
					\begin{minipage}[t]{100mm}
\vspace*{-5ex}
\begin{align*}
\text{Let }& W_0^{(i,0,k)} :=  \{v \in Y^{(i,0,k)} \cap X_{V(T_t)} :   L^{(i,0,k)}(v)=\{k+1\} \text{ and}\\
& \text{there exists a path } P \text{ in } G[Y^{(i,0,k)} \cap X_{V(T_t)}]  \text{ from }v \text{ to } X_t \\
&  \text{such that }  L^{(i,0,k)}(q)=\{k+1\} \text{ for all }q \in V(P)\}.
						\end{align*}
					\end{minipage}
				\item Let $W_2^{(i,0,k)} := A_{L^{(i,0,k)}}(W_0^{(i,0,k)}) \cap Z_t$. 
				\item Let $(Y^{(i,0,k+1)},L^{(i,0,k+1)})$ be a $(W_2^{(i,0,k)},k+1)$-progress of $(Y^{(i,0,k)},L^{(i,0,k)})$.
				\end{itemize}
		\end{itemize}
\hspace*{4mm}
\noindent\textbf{\boldmath Stage $(i,\geq 1,\star)$:} 
	See \cref{subsec:alog} for details.\\
\hspace*{4mm}
\noindent\textbf{Stage: Adding Fake Edges:}
	For each $j \in [\lvert \V \rvert-1]$ and $t' \in V(T)$, $E^{(i+1)}_{j,t'}$ is defined.
	See \cref{subsec:alog} for details.\\
\hspace*{4mm}
\noindent\textbf{Stage: Moving to the Next Node in the Tree:}
	See \cref{subsec:alog} for details.\\
\hspace*{4mm}
\noindent\textbf{Stage: Building a New Fence:}
	See \cref{subsec:alog} for details.

	\medskip

Now we prove related claims.
Those claims bound the size of precolored vertices in certain regions.

\begin{claim} \label{claim:-1jumpsize}
Let $i,i' \in \mathbb{N}_0$ with $i' < i$, and let $t$ and $t'$ be the nodes of $T$ with $\sigma_T(t)=i$ and $\sigma_T(t')=i'$.
Let $j \in [\lvert V \rvert-1]$.
Then $\lvert Y^{(i',0,0)} \cap \overline{I_j} \cap X_{V(T_t)} \cap X_{V(T_{j,t'})} \rvert \leq g_0(w_0)$.
\end{claim}

\begin{proof}
For every $k \in [0,w_0-1]$, let $q_k \in [0,s+1]$ be the number such that the color of $M_{j,k}^{(t')}$ is $q_k+1$.
So for every $k \in [0,w_0-1]$, 
\begin{align*}
& (Y^{(i',-1,k+1)}-Y^{(i',-1,k)}) \cap \overline{I_j} \cap X_{V(T_t)} \cap X_{V(T_{j,t'})}-X_t \\
\subseteq \;& A_{L^{(i',-1,k,q_k)}}(W_1^{(i',-1,k,q_k)}) \cap Z_{t'} \cap \overline{I_j} \cap X_{V(T_t)} \cap X_{V(T_{j,t'})}-X_t \\
\subseteq \;& (N_G^{\geq s}(W_1^{(i',-1,k,q_k)} \cap X_{V(T_{j,t'})}) \cup X_{\partial T_{j,t'}} \cup X_{t'}) \cap Z_{t'} \cap \overline{I_j} \cap X_{V(T_t)} \cap X_{V(T_{j,t'})}-X_t,
\end{align*} where the last inclusion follows from \cref{claim:boundaryZ}.
Note that for every $k \in [0,w_0-1]$, since $W_1^{(i',-1,k,q_k)} \cap \overline{I_j} \subseteq \bigcup_{S \in \Se_j^\circ}S$, we know $N_G[W_1^{(i',-1,k,q_k)} \cap \overline{I_j}] \subseteq \overline{I_j}$. 
So for every $k \in [0,w_0-1]$, 
\begin{align*}
& (N_G^{\geq s}(W_1^{(i',-1,k,q_k)} \cap X_{V(T_{j,t'})}) \cup X_{\partial T_{j,t'}} \cup X_{t'}) \cap Z_{t'} \cap \overline{I_j} \cap X_{V(T_t)} \cap X_{V(T_{j,t'})}-X_t \\
\subseteq \;& (N_G^{\geq s}(W_1^{(i',-1,k,q_k)} \cap X_{V(T_{j,t'})}) \cap Z_{t'} \cap \overline{I_j} \cap X_{V(T_t)} \cap X_{V(T_{j,t'})}-X_t) \cup (X_{\partial T_{j,t'}} \cap \overline{I_j} \cap X_{V(T_t)}) \\
\subseteq \;& (N_G^{\geq s}(W_1^{(i',-1,k,q_k)} \cap X_{V(T_{j,t'})} \cap \overline{I_j}) \cap \overline{I_j} \cap X_{V(T_t)} \cap X_{V(T_{j,t'})}-X_t) \cup (X_{\partial T_{j,t'}} \cap \overline{I_j} \cap X_{V(T_t)}) \\
\subseteq \;& N_G^{\geq s}(W_1^{(i',-1,k,q_k)} \cap X_{V(T_{j,t'})} \cap \overline{I_j} \cap X_{V(T_t)}) \cup (X_{\partial T_{j,t'}} \cap \overline{I_j} \cap X_{V(T_t)}) \\
\subseteq \;& N_G^{\geq s}(Y^{(i',-1,k)} \cap X_{V(T_{j,t'})} \cap \overline{I_j} \cap X_{V(T_t)}) \cup (X_{\partial T_{j,t'}} \cap \overline{I_j} \cap X_{V(T_t)}).
\end{align*}
Hence for every $k \in [0,w_0-1]$, 
\begin{align*}
& \lvert Y^{(i',-1,k+1)} \cap \overline{I_j} \cap X_{V(T_t)} \cap X_{V(T_{j,t'})} \rvert \\
\leq \;& \lvert Y^{(i',-1,k)} \cap \overline{I_j} \cap X_{V(T_t)} \cap X_{V(T_{j,t'})} \rvert \\
& \quad + \lvert X_t \cap \overline{I_j} \rvert + \lvert (Y^{(i',-1,k+1)}-Y^{(i',-1,k)}) \cap \overline{I_j} \cap X_{V(T_t)} \cap X_{V(T_{j,t'})}-X_t \rvert \\
\leq \;& \lvert Y^{(i',-1,k)} \cap \overline{I_j} \cap X_{V(T_t)} \cap X_{V(T_{j,t'})} \rvert \\
& \quad + \lvert X_t \cap \overline{I_j} \rvert + f(\lvert Y^{(i',-1,k)} \cap X_{V(T_{j,t'})} \cap \overline{I_j} \cap X_{V(T_t)} \rvert) + \lvert X_{\partial T_{j,t'}} \cap \overline{I_j} \cap X_{V(T_t)} \rvert\\
 = \;& f_1(\lvert Y^{(i',-1,k)} \cap \overline{I_j} \cap X_{V(T_t)} \cap X_{V(T_{j,t'})} \rvert) + \lvert X_t \cap \overline{I_j} \rvert + \lvert X_{\partial T_{j,t'}} \cap X_{V(T_t)} \cap \overline{I_j} \rvert \\
\leq\; & f_1(\lvert Y^{(i',-1,k)} \cap \overline{I_j} \cap X_{V(T_t)} \cap X_{V(T_{j,t'})} \rvert) + 2w_0, 
\end{align*}
where the last inequality follows from \cref{claim:I_jbasic1}.

By \cref{claim:I_jbasic1}, $\lvert Y^{(i',-1,0)} \cap \overline{I_j} \cap X_{V(T_t)} \cap X_{V(T_{j,t'})} \rvert \leq w_0=g_0(0)$.
So it is easy to verify that for every $k \in [0,w_0]$, $\lvert Y^{(i',-1,k)} \cap \overline{I_j} \cap X_{V(T_t)} \cap X_{V(T_{j,t'})} \rvert \leq g_0(k)$ by induction on $k$.
Therefore, $\lvert Y^{(i',0,0)} \cap \overline{I_j} \cap X_{V(T_t)} \cap X_{V(T_{j,t'})} \rvert = \lvert Y^{(i',-1,w_0)} \cap \overline{I_j} \cap X_{V(T_t)} \cap X_{V(T_{j,t'})} \rvert \leq g_0(w_0)$.
\end{proof}

\begin{claim} \label{claim:I_jsize_prep}
Let $i,i' \in \mathbb{N}_0$ with $i'\leq i$, and let $t$ and $t'$ be the nodes of $T$ with $\sigma_T(t)=i$ and $\sigma_T(t')=i'$. 
Let $j \in [\lvert V \rvert-1]$.
Let $a,b$ be integers such that $I_j = \bigcup_{\alpha=a}^bV_\alpha$.
Then the following hold:
	\begin{itemize}
		\item For every $k \in [0,s+1]$, 
		\begin{align*}
		& \lvert Y^{(i',0,k+1)} \cap (\bigcup_{\alpha=a-(s+3)+(k+1)}^{b+(s+3)-(k+1)}V_\alpha) \cap X_{V(T_t)} \cap X_{V(T_{j,t'})} \rvert \\
		\leq\; & f_1(\lvert Y^{(i',0,k)} \cap (\bigcup_{\alpha=a-(s+3)+k}^{b+(s+3)-k}V_\alpha) \cap X_{V(T_t)} \cap X_{V(T_{j,t'})} \rvert) + \lvert X_{\partial T_{j,t'}} \cap X_{V(T_t)} \cap \overline{I_j} \rvert + 2w_0.
		\end{align*}
		\item If $i'<i$, then $\lvert Y^{(i',0,s+2)} \cap (\bigcup_{\alpha=a-1}^{b+1}V_\alpha) \cap X_{V(T_t)} \cap X_{V(T_{j,t'})} \rvert \leq g_1(s+2)$.
	\end{itemize}
\end{claim}

\begin{proof}
We may assume that $t \in V(T_{t'})$, for otherwise $X_{V(T_t)} \cap X_{V(T_{j,t'})}=\emptyset$ and we are done.

We first prove Statement~1.
For every $k \in [0,s+1]$, 
\begin{align*}
(Y^{(i',0,k+1)}-Y^{(i',0,k)}) \cap \overline{I_j} \cap X_{V(T_t)} 
& \;\subseteq W_2^{(i',0,k)} \cap \overline{I_j} \cap X_{V(T_t)} \\
& \;\subseteq A_{L^{(i',0,k)}}(Y^{(i',0,k)}) \cap Z_{t'} \cap \overline{I_j} \cap X_{V(T_t)}.
\end{align*}
So $(Y^{(i',0,k+1)}-Y^{(i',0,k)}) \cap \overline{I_j} \cap X_{V(T_t)} \cap X_{V(T_{j,t})} \subseteq A_{L^{(i',0,k)}}(Y^{(i',0,k)}) \cap \overline{I_j} \cap X_{V(T_t)} \cap X_{V(T_{j,t})}$.
By \cref{claim:boundaryZ}, 
\begin{align*}
& A_{L^{(i',0,k)}}(Y^{(i',0,k)}) \cap (\bigcup_{\alpha=a-(s+3)+(k+1)}^{b+(s+3)-(k+1)}V_\alpha) \cap X_{V(T_t)}-X_t \\
\subseteq\; & A_{L^{(i',0,k)}}(Y^{(i',0,k)} \cap X_{V(T_t)} \cap (\bigcup_{\alpha=a-(s+3)+k}^{b+(s+3)-k}V_\alpha)) \cap (\bigcup_{\alpha=a-(s+3)+k+1}^{b+(s+3)-(k+1)}V_\alpha) \cap X_{V(T_t)}-X_t, 
\end{align*}
and 
\begin{align*}
A_{L^{(i',0,k)}}(Y^{(i',0,k)} \cap X_{V(T_t)}) \cap X_{V(T_{j,t'})} \subseteq N^{\geq s}_G(Y^{(i',0,k)} \cap X_{V(T_t)} \cap X_{V(T_{j,t'})}) \cup X_{\partial T_{j,t'}} \cup X_{t'}.
\end{align*}
Therefore, 
\begin{align*}
& (Y^{(i',0,k+1)}-Y^{(i',0,k)}) \cap (\bigcup_{\alpha=a-(s+3)+(k+1)}^{b+(s+3)-(k+1)}V_\alpha) \cap X_{V(T_t)} \cap X_{V(T_{j,t'})} \\
\subseteq\; & \Big(A_{L^{(i',0,k)}}(Y^{(i',0,k)}) \cap X_{V(T_{j,t'})} \cap (\bigcup_{\alpha=a-(s+3)+(k+1)}^{b+(s+3)-(k+1)}V_\alpha) \cap X_{V(T_t)}-X_t\Big) \cup (X_t \cap \overline{I_j}) \\
\subseteq\; & \Big( A_{L^{(i',0,k)}}(Y^{(i',0,k)} \cap X_{V(T_t)} \cap (\bigcup_{\alpha=a-(s+3)+k}^{b+(s+3)-k}V_\alpha)) \cap (\bigcup_{\alpha=a-(s+3)+(k+1)}^{b+(s+3)-(k+1)}V_\alpha) \cap X_{V(T_{j,t'})} \cap X_{V(T_t)}-X_t\Big)\\ 
& \quad \cup (X_t \cap \overline{I_j}) \\
\subseteq\; & \Big(  (N_G^{\geq s}(Y^{(i',0,k)} \cap (\bigcup_{\alpha=a-(s+3)+k}^{b+(s+3)-k}V_\alpha) \cap X_{V(T_t)} \cap X_{V(T_{j,t'})}) \cup X_{\partial T_{j,t'}} \cup X_{t'})  \\
 & \quad \cap (\bigcup_{\alpha=a-(s+3)+k+1}^{b+(s+3)-(k+1)}V_\alpha) \cap X_{V(T_t)}-X_t \Big) \cup (X_t \cap \overline{I_j}),
\end{align*}
 where the last inequality follows from \cref{claim:boundaryZ}.
Hence 
\begin{align*}
& \lvert (Y^{(i',0,k+1)}-Y^{(i',0,k)}) \cap (\bigcup_{\alpha=a-(s+3)+(k+1)}^{b+(s+3)-(k+1)}V_\alpha) \cap X_{V(T_t)} \cap X_{V(T_{j,t'})} \rvert \\
\leq\; & \lvert N_G^{\geq s}(Y^{(i',0,k)} \cap (\bigcup_{\alpha=a-(s+3)+k}^{b+(s+3)-k}V_\alpha) \cap X_{V(T_t)} \cap X_{V(T_{j,t'})}) \cap  (\bigcup_{\alpha=a-(s+3)+k+1}^{b+(s+3)-(k+1)}V_\alpha) \cap X_{V(T_t)}-X_t \rvert \\
& \quad + \lvert (X_{\partial T_{j,t'}} \cup X_{t'}) \cap  (\bigcup_{\alpha=a-(s+3)+k+1}^{b+(s+3)-(k+1)}V_\alpha ) \cap X_{V(T_t)}-X_t \rvert + \lvert X_t \cap \overline{I_j} \rvert \\
\leq \;& f(\lvert Y^{(i',0,k)} \cap (\bigcup_{\alpha=a-(s+3)+k}^{b+(s+3)-k}V_\alpha) \cap X_{V(T_t)} \cap X_{V(T_{j,t'})} \rvert) + \lvert X_{\partial T_{j,t'}} \cap X_{V(T_t)} \cap \overline{I_j} \rvert + \lvert X_{t'} \cap \overline{I_j} \rvert + \lvert X_t \cap \overline{I_j} \rvert \\
\leq \; &  f(\lvert Y^{(i',0,k)} \cap (\bigcup_{\alpha=a-(s+3)+k}^{b+(s+3)-k}V_\alpha) \cap X_{V(T_t)} \cap X_{V(T_{j,t'})} \rvert) + \lvert X_{\partial T_{j,t'}} \cap X_{V(T_t)} \cap \overline{I_j} \rvert + 2w_0.
\end{align*}
So 
\begin{align*}
& \lvert Y^{(i',0,k+1)} \cap (\bigcup_{\alpha=a-(s+3)+(k+1)}^{b+(s+3)-(k+1)}V_\alpha) \cap X_{V(T_t)} \cap X_{V(T_{j,t'})} \rvert \\
=\; & \lvert Y^{(i',0,k)} \cap (\bigcup_{\alpha=a-(s+3)+(k+1)}^{b+(s+3)-(k+1)}V_\alpha) \cap X_{V(T_t)} \cap X_{V(T_{j,t'})} \rvert \\
& \quad + \lvert (Y^{(i',0,k+1)}-Y^{(i',0,k)}) \cap (\bigcup_{\alpha=a-(s+3)+(k+1)}^{b+(s+3)-(k+1)}V_\alpha) \cap X_{V(T_t)} \cap X_{V(T_{j,t'})} \rvert \\
\leq\; &  \lvert Y^{(i',0,k)} \cap (\bigcup_{\alpha=a-(s+3)+k}^{b+(s+3)-k}V_\alpha) \cap X_{V(T_t)} \cap X_{V(T_{j,t'})} \rvert\\
& \quad +f(\lvert Y^{(i',0,k)} \cap (\bigcup_{\alpha=a-(s+3)+k}^{b+(s+3)-k}V_\alpha) \cap X_{V(T_t)} \cap X_{V(T_{j,t'})} \rvert) + \lvert X_{\partial T_{j,t'}} \cap X_{V(T_t)} \cap \overline{I_j} \rvert + 2w_0 \\
=\; & f_1(\lvert Y^{(i',0,k)} \cap (\bigcup_{\alpha=a-(s+3)+k}^{b+(s+3)-k}V_\alpha) \cap X_{V(T_t)} \cap X_{V(T_{j,t'})} \rvert) + \lvert X_{\partial T_{j,t'}} \cap X_{V(T_t)} \cap \overline{I_j} \rvert + 2w_0.
\end{align*}
This proves Statement~1.

Now we prove Statement~2 of this claim.
Assume $i'<i$.
By \cref{claim:I_jbasic1} and Statement~1 of this claim, for every $k \in [0,s+1]$, 
\begin{align*}
& \lvert Y^{(i',0,k+1)} \cap (\bigcup_{\alpha=a-(s+3)+(k+1)}^{b+(s+3)-(k+1)}V_\alpha) \cap X_{V(T_t)} \cap X_{V(T_{j,t'})} \rvert \\
\leq\; & f_1(\lvert Y^{(i',0,k)} \cap (\bigcup_{\alpha=a-(s+3)+k}^{b+(s+3)-k}V_\alpha) \cap X_{V(T_t)} \cap X_{V(T_{j,t'})} \rvert) + \lvert X_{\partial T_{j,t'}} \cap X_{V(T_t)} \cap \overline{I_j} \rvert + 2w_0\\
 \leq\; & f_1(\lvert Y^{(i',0,k)} \cap (\bigcup_{\alpha=a-(s+3)+k}^{b+(s+3)-k}V_\alpha) \cap X_{V(T_t)} \cap X_{V(T_{j,t'})} \rvert) + 3w_0.
\end{align*}
Then it is easy to verify that for every $k \in [0,s+2]$, $\lvert Y^{(i',0,k)} \cap (\bigcup_{\alpha=a-(s+3)+k)}^{b+(s+3)-k}V_\alpha) \cap X_{V(T_t)} \cap X_{V(T_{j,t'})} \rvert \leq g_1(k)$ by induction on $k$, where the base case $k=0$ follows from \cref{claim:-1jumpsize}.  
Then Statement~2 follows from the case $k=s+2$.
\end{proof}

\subsection{Claims related to Stage $(i, \geq 1,\star)$}
\label{i1star}

We first recall the relevant parts of the algorithm.

\medskip
\noindent\textbf{\boldmath Stage $(0,-1,0)$: Initialization:}
	\begin{itemize}
		\item Let $D^{(0,0,0)} := \emptyset$. 
		\item Let $S^{(0,0)}_{j,t} := \emptyset$, $S^{(0,1)}_{j,t}:=\emptyset$ and $S^{(0,2)}_{j,t} := \emptyset$ for every $j \in [\lvert V \rvert-1]$ and $t \in V(T)$.
		\item Other sets are defined.
See \cref{subsec:alog} for details. 
	\end{itemize}
For $i=0,1,2,\dots$, let $t$ be the node of $T$ with $\sigma_T(t)=i$ and perform the following steps:\\[1ex]
\hspace*{4mm} 
\noindent\textbf{\boldmath Stage: Building Fences:}
	See \cref{subsec:alog} for details.\\		
\hspace*{4mm} 
\noindent\textbf{\boldmath Stage $(i,-1,\star)$:} 
	See \cref{subsec:alog} for details.\\
\hspace*{4mm}
\noindent\textbf{\boldmath Stage $(i,0,\star)$:} 
	$(Y^{(i,0,s+2)},L^{(i,0,s+2)})$ and other sets are defined.
	See \cref{subsec:alog} for details.\\
\hspace*{4mm}
\noindent\textbf{\boldmath Stage $(i,\geq 1,\star)$: Isolating the Fences of $T_{j,t}$}
	\begin{itemize}
				\item Let $W_3^{(i,-1)} := \bigcup_{j=1}^{\lvert \V \rvert-1}(X_{t} \cap I_j)$. 
				\item Let $W_4^{(i)} := \bigcup_{j=1}^{\lvert \V \rvert-1}(X_{V(T_{j,t})} \cap I_j)$. 
				\item For each $\ell \in [0,\lvert V(T) \rvert]$, define the following:
					\begin{itemize}
						\item Let $W_3^{(i,\ell)} := W_3^{(i,\ell-1)} \cup \bigcup_{j=1}^{\lvert \V \rvert-1}\bigcup_{q} (X_q \cap I_j)$, where the last union is over all nodes $q$ in $\partial T_{j,t}$ for which there exists a monochromatic path in $G[Y^{(i,\ell,s+2)} \cap I_j \cap W_4^{(i)}]$ from $W_3^{(i,\ell-1)}$ to 
						\begin{align*}
						X_q \cap & I_j-((D^{(i,\ell,0)}-X_t) \cup \{v \in D^{(i,\ell,0)} \cap X_t \cap X_{q'} \cap I_j: q' \in V(T_t)-\{t\}, \\
						& q'\text{ is a witness for }X_{q'} \cap I_j \subseteq W_3^{(i',\ell')}\text{ for some }i' \in [0,i-1]\text{ and }\ell' \in [0,\lvert V(T) \rvert]\})
						\end{align*}
						internally disjoint from $X_t \cup X_{\partial T_{j,t}}$.
						\item Let $(Y^{(i,\ell+1,0)}, L^{(i,\ell+1,0)})$ be a $(W_{3}^{(i,\ell)},0)$-progress of $(Y^{(i,\ell,s+2)},L^{(i,\ell,s+2)})$.
						\item For each $k \in [0,s+1]$, define the following:
										\begin{itemize}
											\item 
											\begin{minipage}[t]{100mm}
												\vspace*{-5ex}
\begin{align*}
\text{Let }&W_{0}^{(i,\ell+1,k)}  := \{v \in Y^{(i,\ell+1,k)} \cap W_{4}^{(i)}: L^{(i,\ell+1,k)}(v)=\{k+1\} \\
& \text{ and there exists a path } P \text{ in } G[Y^{(i,\ell+1,k)} \cap W_{4}^{(i)}] \text{ from }v \text{ to } W_{3}^{(i,\ell)} \\
& \text{ such that } L^{(i,\ell+1,k)}(q)=\{k+1\} \text{ for all }q \in V(P)\}.
												\end{align*}
											\end{minipage}
											\item Let $W_2^{(i,\ell+1,k)} := A_{L^{(i,\ell+1,k)}}(W_0^{(i,\ell+1,k)}) \cap W_4^{(i)}$. 
											\item Let $(Y^{(i,\ell+1,k+1)},L^{(i,\ell+1,k+1)})$ be the $(W_2^{(i,\ell+1,k)},k+1)$-progress of $(Y^{(i,\ell+1,k)},L^{(i,\ell+1,k)})$.
											\item Let $D^{(i,\ell,k+1)} := D^{(i,\ell,k)} \cup W_0^{(i,\ell+1,k)}$.
										\end{itemize}
						\item Let $D^{(i,\ell+1,0)}:=D^{(i,\ell,s+2)}$.
					\end{itemize}
	\end{itemize}
\hspace*{4mm}
\noindent\textbf{Stage: Adding Fake Edges:}
	See \cref{subsec:alog} for details.\\
\hspace*{4mm}
\noindent\textbf{Stage: Moving to the Next Node in the Tree}
	\begin{itemize}
				\item Let $(Y^{(i+1,-1,0)},L^{(i+1,-1,0)})$ be a $(X_{t'},0)$-progress of $(Y^{(i,\lvert V(T) \rvert+1,s+2)},L^{(i,\lvert V(T) \rvert+1,s+2)})$, where $t'$ is the node of $T$ with $\sigma_T(t')=i+1$. 
				\item Let $D^{(i+1,0,0)}:=D^{(i,\lvert V(T) \rvert,s+2)}$.
				\item For every node $t'$ with $t '\in V(T_t)-\{t\}$ and $j \in [\lvert V \rvert-1]$, 
					\begin{itemize}
						\item If $t' \in V(T_{j,t})-\partial T_{j,t}$, then define 
							\begin{itemize}
								\item $S^{(i+1,0)}_{j,t'} := S^{(i,0)}_{j,t'} \cup \bigcup_{\ell=0}^{\lvert V(T) \rvert}\bigcup_{k=0}^{s+1}W_0^{(i,\ell+1,k)}$,
								\item $S^{(i+1,1)}_{j,t'} := S^{(i,1)}_{j,t'} \cup \bigcup_{j=1}^{\lvert \V \rvert-1}\bigcup_q (X_q \cap I_j)$, where the last union is over all nodes $q$ of $T$ such that $q \in V(T_{t'})-\{t'\}$ and $q$ is a witness for $X_q \cap I_j \subseteq W_3^{(i,\ell')}$ for some $\ell' \in [0,\lvert V(T) \rvert]$, and 
								\item $S^{(i+1,2)}_{j,t'} := S^{(i,2)}_{j,t'} \cup \bigcup_{k=0}^{w_0-1}W_0^{(i,-1,k)} \cup \bigcup_{k=0}^{s+1}W_0^{(i,0,k)}$;
							\end{itemize}
						\item otherwise, define $S^{(i+1,0)}_{j,t'} := S^{(i,0)}_{j,t'}$, $S^{(i+1,1)}_{j,t'} := S^{(i,1)}_{j,t'}$ and $S^{(i+1,2)}_{j,t'} := S^{(i,2)}_{j,t'}$.
					\end{itemize}
	\end{itemize}
\hspace*{4mm}
\noindent\textbf{Stage: Building a New Fence:}
		See \cref{subsec:alog} for details.

\medskip

Now we prove related claims.

\begin{claim} \label{claim:W_3isolate}
Let $i \in \mathbb{N}_0$, and let $t \in V(T)$ with $i_t=i$. 
Let $j \in [\lvert \V \rvert-1]$ and $\ell \in [-1,\lvert V(T) \rvert]$.
If $k \in [0,s+1]$ and $P$ is a $c$-monochromatic path of color $k+1$ contained in $G[W_4^{(i)}]$ intersecting $W_3^{(i,\ell)}$, then $V(P) \subseteq Y^{(i,\ell+1,k)}$ and $A_{L^{(i,\ell+1,k+1)}}(V(P)) \cap W_4^{(i)}=\emptyset$.
\end{claim}

\begin{proof}
First suppose that $V(P) \not \subseteq Y^{(i,\ell+1,k)}$.
Since $W_3^{(i,\ell)} \subseteq Y^{(i,\ell+1,0)}$, $V(P) \cap W_3^{(i,\ell)}$ is a nonempty subset of $Y^{(i,\ell+1,k)}$.
So there exists a longest subpath $P'$ of $P$ contained in $G[Y^{(i,\ell+1,k)}]$ intersecting $W_3^{(i,\ell)}$.
Since $V(P') \subseteq V(P) \subseteq W_4^{(i)}$, $V(P') \subseteq W_0^{(i,\ell+1,k)}$.
In addition, $P \neq P'$, for otherwise $V(P)=V(P') \subseteq Y^{(i,\ell+1,k)}$, a contradiction.
So there exist $v \in V(P')$ and $u \in N_{P}(v)-V(P')$.
That is, $c(u)=k+1$ and $u \not \in Y^{(i,\ell+1,k)}$.
So $u \in A_{L^{(i,\ell+1,k)}}(\{v\}) \cap V(P)-W_0^{(i,\ell+1,k)} \subseteq A_{L^{(i,\ell+1,k)}}(W_0^{(i,\ell+1,k)}) \cap W_4^{(i)} \subseteq W_2^{(i,\ell+1,k)}$.
But $(Y^{(i,\ell+1,k+1)},L^{(i,\ell+1,k+1)})$ is a $(W_2^{(i,\ell+1,k)},k+1)$-progress, so $k+1 \not \in L^{(i,\ell+1,k+1)}(u)$.
Hence $c(u) \neq k+1$, a contradiction.

Now we suppose that $A_{L^{(i,\ell+1,k+1)}}(V(P)) \cap W_4^{(i)} \neq \emptyset$.
So there exists $z \in A_{L^{(i,\ell+1,k+1)}}(V(P)) \cap W_4^{(i)}$.
Note that we have shown that $V(P) \subseteq Y^{(i,\ell+1,k)}$, so $V(P) \subseteq W_0^{(i,\ell+1,k)}$ and $A_{L^{(i,\ell+1,k+1)}}(V(P)) \subseteq A_{L^{(i,\ell+1,k)}}(V(P))$.
In addition, $z \in A_{L^{(i,\ell+1,k+1)}}(V(P))$, so $z \not \in Y^{(i,\ell+1,k+1)} \supseteq Y^{(i,\ell+1,k)} \supseteq W_0^{(i,\ell+1,k)}$.
So $z \in A_{L^{(i,\ell+1,k+1)}}(V(P)) \cap W_4^{(i)}-W_0^{(i,\ell+1,k)} \subseteq A_{L^{(i,\ell+1,k)}}(V(P)) \cap W_4^{(i)}-W_0^{(i,\ell+1,k)} \subseteq W_2^{(i,\ell+1,k)}$.
Since $(Y^{(i,\ell+1,k+1)},L^{(i,\ell+1,k+1)})$ is a $(W_2^{(i,\ell+1,k)},k+1)$-progress of $(Y^{(i,\ell+1,k)},L^{(i,\ell+1,k)})$, we know $k+1 \not \in L^{(i,\ell+1,k+1)}(z)$, so $z \not \in A_{L^{(i,\ell+1,k+1)}}(V(P))$, a contradiction.
This proves the claim.
\end{proof}

\begin{claim} \label{claim:1jump1}
Let $i,i' \in \mathbb{N}_0$ with $i'<i$, and let $t \in V(T)$ with $i_t=i$. 
Let $j \in [\lvert V \rvert-1]$ and $\ell \in [0,\lvert V(T) \rvert]$.
If $Y^{(i',\ell,s+2)} \cap X_{V(T_{t})} \cap I_j-X_{t} \neq Y^{(i',\ell+1,s+2)}\cap X_{V(T_{t})} \cap I_j-X_{t}$, then $(W_3^{(i',\ell)}-W_3^{(i',\ell-1)}) \cap I_j \neq \emptyset$, and either 
	\begin{itemize}
		\item $\ell>0$, or 
		\item $\lvert X_t \cap I_j \cap D^{(i',\ell,s+2)} \rvert > \lvert X_t \cap I_j \cap D^{(i',\ell,0)} \rvert$, or 
		\item $\lvert X_t \cap I_j \cap S^{(i'+1,0)}_{j,t} \rvert > \lvert X_t \cap I_j \cap S^{(i',0)}_{j,t} \rvert$, or 
		\item $\lvert X_t \cap I_j \cap S^{(i'+1,1)}_{j,t} \rvert > \lvert X_t \cap I_j \cap S^{(i',1)}_{j,t} \rvert$.
	\end{itemize}
\end{claim}

\begin{proof}
Let $t'$ be the node of $T$ with $i_{t'}=i'$.
Since $Y^{(i',\ell,s+2)} \cap X_{V(T_{t})} \cap I_j-X_{t} \neq Y^{(i',\ell+1,s+2)}\cap X_{V(T_{t})} \cap I_j-X_{t}$, $t'$ is an ancestor of $t$.
	
Suppose that $(W_3^{(i',\ell)}-W_3^{(i',\ell-1)}) \cap I_j=\emptyset$.
Then for every $k \in [0,s+1]$ and every monochromatic path $P$ of color $k+1$ in $G[Y^{(i',\ell+1,k)} \cap W_4^{(i')}]$ intersecting $W_3^{(i',\ell)} \cap I_j$, $P$ is a $c$-monochromatic path  in $G[W_4^{(i')}]$ intersecting $W_3^{(i',\ell)} \cap I_j \subseteq W_3^{(i',\ell-1)} \cap I_j$, so \cref{claim:W_3isolate} implies that $V(P) \subseteq Y^{(i',\ell,k)} \cap W_4^{(i')}$, and hence $V(P) \subseteq W_0^{(i',\ell,k)}$.
This implies that $W_0^{(i',\ell+1,k)} \cap I_j \subseteq W_0^{(i',\ell,k)} \cap I_j$ for every $k \in [0,s+1]$.
Since no edge of $G$ is between $I_j$ and $I_{j'}$ for any $j' \neq j$, for every $k \in [0,s+1]$, 
\begin{align*}
 W_2^{(i',\ell+1,k)} \cap I_j 
= \;& A_{L^{(i',\ell+1,k)}}(W_0^{(i',\ell+1,k)}) \cap W_4^{(i')} \cap I_j \\
= \;& A_{L^{(i',\ell+1,k)}}(W_0^{(i',\ell+1,k)} \cap I_j) \cap W_4^{(i')} \cap I_j \\
\subseteq \;& A_{L^{(i',\ell+1,k)}}(W_0^{(i',\ell,k)} \cap I_j) \cap W_4^{(i')} \cap I_j \\\
\subseteq \;& A_{L^{(i',\ell+1,k)}}(W_0^{(i',\ell,k)}) \cap W_4^{(i')} \cap I_j \\
\subseteq \;& A_{L^{(i',\ell,k)}}(W_0^{(i',\ell,k)}) \cap W_4^{(i')} \cap I_j \\
= \;& W_2^{(i',\ell,k)} \cap I_j.
\end{align*} 
Hence $(Y^{(i',\ell+1,s+2)}-Y^{(i',\ell,s+2)}) \cap I_j \subseteq \bigcup_{k=0}^{s+1}(W_2^{(i',\ell+1,k)} \cap I_j) \subseteq \bigcup_{k=0}^{s+1}(W_2^{(i',\ell,k)} \cap I_j) \subseteq Y^{(i',\ell,s+2)} \cap I_j$.
So $Y^{(i',\ell+1,s+2)} \cap I_j \subseteq Y^{(i',\ell,s+2)} \cap I_j$, a contradiction.

Hence $(W_3^{(i',\ell)}-W_3^{(i',\ell-1)}) \cap I_j \neq \emptyset$.

Suppose to the contrary that $\ell = 0$, $\lvert X_t \cap I_j \cap D^{(i',\ell,s+2)} \rvert \leq \lvert X_t \cap I_j \cap D^{(i',\ell,0)} \rvert$, $\lvert X_t \cap I_j \cap S^{(i'+1,0)}_{j,t} \rvert \leq \lvert X_t \cap I_j \cap S^{(i',0)}_{j,t} \rvert$, and $\lvert X_t \cap I_j \cap S^{(i'+1,1)}_{j,t} \rvert \leq \lvert X_t \cap I_j \cap S^{(i',1)}_{j,t} \rvert$. 
Since $D^{(i',\ell,s+2)} \supseteq D^{(i',\ell,0)}$, $S^{(i'+1,0)}_{j,t} \supseteq S^{(i',0)}_{j,t}$ and $S^{(i'+1,1)}_{j,t} \supseteq S^{(i',1)}_{j,t}$, we know $X_t \cap I_j \cap D^{(i',\ell,s+2)} = X_t \cap I_j \cap D^{(i',\ell,0)}$, $X_t \cap I_j \cap S^{(i'+1,0)}_{j,t} = X_t \cap I_j \cap S^{(i',0)}_{j,t}$ and $X_t \cap I_j \cap S^{(i'+1,1)}_{j,t} = X_t \cap I_j \cap S^{(i',1)}_{j,t}$.

Since $(W_3^{(i',\ell)}-W_3^{(i',\ell-1)}) \cap I_j \neq \emptyset$, there exists $q \in \partial T_{j,t'}$ such that $q$ is a witness for $X_q \cap I_j \subseteq W_3^{(i',\ell)}$ and $X_q \cap I_j \not \subseteq W_3^{(i',\ell-1)}$.
So $X_q \cap I_j$ is a nonempty subset of $W_3^{(i',\ell)} \cap I_j$ with $X_q \cap I_j \not \subseteq W_3^{(i',\ell-1)}$, and there exists a $c$-monochromatic path $P$ in $G[Y^{(i',\ell,s+2)} \cap I_j \cap W_4^{(i')}]$  from $W_3^{(i',\ell-1)}$ to $X_q \cap I_j-((D^{(i',\ell,0)}-X_{t'}) \cup \{v \in D^{(i',\ell,0)} \cap X_{t'} \cap X_{q'} \cap I_j: q' \in V(T_{t'})-\{t'\}, q'$ is a witness for $X_{q'} \cap I_j \subseteq W_3^{(i'',\ell')}$ for some $i'' \in [0,i'-1]$ and $\ell' \in [0,\lvert V(T) \rvert]\})$ and internally disjoint from $X_{t'} \cup X_{\partial T_{j,t'}}$.
Furthermore, choose such a node $q$ such that $q \in V(T_t)-\{t\}$ if possible.
Note that there exists $k \in [0,s+1]$ such that the color of $P$ is $k+1$.

We first suppose that $q \in V(T_t)-\{t\}$.
Since $V(P) \subseteq Y^{(i',\ell,s+2)} \cap W_4^{(i')} \subseteq Y^{(i',\ell+1,k)} \cap W_4^{(i')}$, $V(P) \subseteq W_0^{(i',\ell+1,k)} \subseteq D^{(i',\ell,k+1)} \subseteq D^{(i',\ell,s+2)}$.
Since $X_t \cap I_j \cap D^{(i',\ell,s+2)}= X_t \cap I_j \cap D^{(i',\ell,0)}$, $V(P) \cap X_t \subseteq D^{(i',\ell,0)}$.
Let $z$ be the end of $P$ belonging to $X_q \cap I_j -((D^{(i',\ell,0)}-X_{t'})  \cup \{v \in D^{(i',\ell,0)} \cap X_{t'} \cap X_{q'} \cap I_j: q' \in V(T_{t'})-\{t'\}, q'$ is a witness for $X_{q'} \cap I_j \subseteq W_3^{(i'',\ell')}$ for some $i'' \in [0,i'-1]$ and $\ell' \in [0,\lvert V(T) \rvert]\})$, such that $z \not \in W_3^{(i',\ell-1)}$ if possible.

Suppose $z \not \in D^{(i',\ell,0)}$.
Since $X_t \cap V(P) \subseteq D^{(i',\ell,0)}$, $z \not \in X_t$.
Since $\ell=0$, $P$ has one end in $W_3^{(i',-1)}$.
Since $t'$ is an ancestor of $t$ and $q \in V(T_t)$, $V(P)$ intersects $X_t$, and there exists a subpath $P''$ of $P$ from $X_t \cap V(P) \subseteq D^{(i',\ell,0)}$ to $z \in X_{V(T_t)}-(X_t \cup D^{(i',\ell,0)})$ and internally disjoint from $X_t$.
So there exist $a \in V(P'') \cap D^{(i',\ell,0)}$ and $b \in N_{P''}(a)-D^{(i',\ell,0)}$.
Note that $b \in V(P'')-D^{(i',\ell,0)} \subseteq V(P)-D^{(i',\ell,0)}$.
Since $X_t \cap V(P) \subseteq D^{(i',\ell,0)}$, $b \in N_{P''}(a)-(D^{(i',\ell,0)} \cup X_t)$.
Let $I^*=\{i'' \in [0,i'-1]: a \in \bigcup_{\ell'=0}^{\lvert V(T) \rvert}W_0^{(i'',\ell'+1,k)}\}$.
Since $\ell=0$ and $a \in D^{(i',\ell,0)}$, $a \in D^{(i',0,0)}=D^{(i'-1,\lvert V(T) \rvert,s+2)}$.
Since $c(a)=k+1$, $a \in \bigcup_{i''=0}^{i'-1}\bigcup_{\ell'=0}^{\lvert V(T) \rvert}W_0^{(i'',\ell'+1,k)}$, so $I^* \neq \emptyset$.
Note that for every $i'' \in I^*$, there exists $\ell_{i''} \in [0,\lvert V(T) \rvert]$ such that either $\ell_{i''}=0$ and $a \in W_0^{(i'',\ell_{i''}+1,k)}$, or $\ell_{i''}>0$ and $a \in W_0^{(i'',\ell_{i''}+1,k)}-W_0^{(i'',\ell_{i''},k)}$, so there exists a $c$-monochromatic path $P_{i''}$  of color $k+1$ in $G[Y^{(i'',\ell_{i''}+1,k)} \cap W_4^{(i'')}]$ from $a$ to $W_3^{(i'',\ell_{i''})}$ and internally disjoint from $W_3^{(i'',\ell_{i''})}$.
Since $b \not \in D^{(i',\ell,0)}$, $b \not \in Y^{(i'',\ell_{i''}+1,k)} \cap W_4^{(i'')}$ for every $i'' \in I^*$.
For every $i'' \in I^*$, since $A_{L^{(i'',\ell_{i''}+1,k+1)}}(V(P_{i''})) \cap W_4^{(i'')}=\emptyset$ by \cref{claim:W_3isolate}, we have $b \not \in W_4^{(i'')}$ (since $b \in W_4^{(i'')}$ implies that $b \not \in Y^{(i'',\ell_{i''}+1,k)} \supseteq V(P_{i''})$ and hence $b \in A_{L^{(i'',\ell_{i''}+1,k+1)}}(V(P_{i''})) \cap W_4^{(i'')}$, a contradiction).
Since $P''$ is from $X_t$ to $X_q$ and internally disjoint from $X_t \cup X_{t'} \cup X_{\partial T_{j,t'}}$, and since $b \not \in W_4^{(i'')}$, we know for every $i'' \in I^*$, $q \not \in V(T_{j,t_{i''}})$, where $t_{i''}$ is the node of $T$ with $i_{t_{i''}}=i''$.
Since $q \in \partial T_{j,t'}$, for every $i'' \in I^*$, if some node in $\partial T_{j,t_{i''}}$ belongs to the path in $T$ from $t'$ to the parent of $q$, then this node must be $t'$.
This together with the fact $q \not \in V(T_{j,t_{i''}})$ for every $i'' \in I^*$, where $t_{i''}$ is the node of $T$ with $i_{t_{i''}}=i''$, we have $t \not \in V(T_{j,t_{i''}})$ for every $i'' \in I^*$.
If $a \in S^{(i''+1,0)}_{j,t}-S^{(i'',0)}_{j,t}$ for some $i'' \in [0,i'-1]$, then $a \in \bigcup_{\ell'=0}^{\lvert V(T) \rvert}W_0^{(i'',\ell'+1,k)}$, so $i'' \in I^*$, but since $a \in \bigcup_{\ell'=0}^{\lvert V(T) \rvert}W_0^{(i'',\ell'+1,k)} \subseteq W_4^{(i'')}$ implies that $t \in V(T_{j,t_{i''}})$, where $t_{i''}$ is the node of $T$ with $i_{t_{i''}}=i''$, a contradiction.
So $a \not \in S^{(i''+1,0)}_{j,t}-S^{(i'',0)}_{j,t}$ for every $i'' \in [0,i'-1]$.
Hence $a \not \in S^{(i',0)}_{j,t}$.
On the other hand, since $t \in V(T_{j,t'})-\partial T_{j,t'}$ and $a \in V(P) \subseteq W_0^{(i',\ell+1,k)}$, we know $a \in S^{(i'+1,0)}_{j,t}$.
Since $X_t \cap I_j \cap S^{(i'+1,0)}_{j,t} = X_t \cap I_j \cap S^{(i',0)}_{j,t}$, we know $a \not \in X_t$.
Recall that $b \not \in W_4^{(i'')}$ for every $i'' \in I^*$.
Since $ab \in E(G)$ and $V(P) \subseteq I_j$, we know for every $i'' \in I^*$, $a \in X_{\partial T_{j,t_{i''}}}$, where $t_{i''}$ is the node of $T$ with $i_{t_{i''}}=i''$.
So for every $i'' \in I^*$, there exists $q_{i''} \in \partial T_{j,t_{i''}}$ such that $a \in X_{q_{i''}} \cap I_j$, where $t_{i''}$ is a node of $T$ with $i_{t_{i''}}=i''$.
If there exists $i'' \in I^*$ with $q_{i''} \in V(T_{t'})-\{t'\}$, then $t' \in V(T_{t_{i''}})$, where $t_{i''}$ is the node of $T$ with $i_{t_{i''}}=i''$, and since $q \in \partial T_{j,t'}$ and $q_{i''} \in V(T_{t'})-\{t'\}$, no node in $\partial T_{j,t_{i''}}$ is contained in the path in $T$ from $t_{i''}$ to the parent of $q$ for every $i'' \in I^*$, so we have $b \in W_4^{(i'')}$ since $b \in W_4^{(i')}$, a contradiction.
Hence $q_{i''} \not \in V(T_{t'})-\{t'\}$ for every $i'' \in I^*$.
Since $I^* \neq \emptyset$, $q_{i''} \not \in V(T_{t'})-\{t'\}$ for some $i'' \in I^*$, so $a \in X_{t'} \cap X_t$, since $a \in X_{q_{i''}} \cap X_{V(T_t)}$.
This contradicts that $a \not \in X_t$.

Hence $z \in D^{(i',\ell,0)}$.
Since $z \not \in D^{(i',\ell,0)}-X_{t'}$, $z \in X_{t'}$.
So there does not exist $q' \in V(T_{t'})-\{t'\}$ such that $z \in X_{q'} \cap I_j$ and $q'$ is a witness for $X_{q'} \cap I_j \subseteq W_3^{(i'',\ell')}$ for some $i'' \in [0,i'-1]$ and $\ell' \in [0,\lvert V(T) \rvert]$.
Since $\ell=0$, $W_3^{(i',\ell-1)}=W_3^{(i',-1)} \subseteq X_{t'}$.
Since $t'$ is an ancestor of $t$ and $z \in X_q$, $z \in X_{t'} \cap X_t$.
Since $z \in X_q \cap I_j$ and $q \in V(T_t)-\{t\}$ and $q$ is a witness for $X_q \cap I_j \subseteq W_3^{(i',\ell)}$, and $t \in V(T_{j,t'})-\partial T_{j,t'}$, we know $z \in S_{j,t}^{(i'+1,1)}$.
Since $X_t \cap I_j \cap S_{j,t}^{(i'+1,1)}=X_t \cap I_j \cap S_{j,t}^{(i',1)}$, $z \in S_{j,t}^{(i',1)}$.
So there exist $i_z \in [0,i'-1]$, and a node $q_z \in V(T_{t})-\{t\}$ such that $z \in X_{q_z} \cap I_j$ and $q_z$ is a witness for $X_{q_z} \cap I_j \subseteq W_3^{(i_z,\ell_z)}$ for some $\ell_z \in [0,\lvert V(T) \rvert]$. 
Note that $q_z \in V(T_t) \subseteq V(T_{t'})-\{t'\}$.
So $q_z$ is a node $q' \in V(T_{t'})-\{t'\}$ such that $z \in X_{q'} \cap I_j$ and $q'$ is a witness for $X_{q'} \cap I_j \subseteq W_3^{(i'',\ell')}$ for some $i'' \in [0,i'-1]$ and $\ell' \in [0,\lvert V(T) \rvert]$, a contradiction.

Therefore $q \not \in V(T_t)-\{t\}$.
So $(W_3^{(i',\ell)}-W_3^{(i',\ell-1)}) \cap I_j \cap X_{V(T_t)} \subseteq (W_3^{(i',\ell)}-W_3^{(i',\ell-1)}) \cap I_j \cap X_t$ by the choice of $q$.
Hence $Y^{(i',\ell,s+2)} \cap X_{V(T_{t})} \cap I_j-X_{t} = Y^{(i',\ell+1,0)}\cap X_{V(T_{t})} \cap I_j-X_{t}$.
Since $Y^{(i',\ell,s+2)} \cap X_{V(T_{t})} \cap I_j-X_{t} \neq Y^{(i',\ell+1,s+2)}\cap X_{V(T_{t})} \cap I_j-X_{t}$, there exists $k' \in [0,s+1]$ such that $Y^{(i',\ell+1,k'+1)} \cap X_{V(T_{t})} \cap I_j-X_{t} \neq Y^{(i',\ell+1,k')}\cap X_{V(T_{t})} \cap I_j-X_{t}$.
So $W_2^{(i',\ell+1,k')} \cap X_{V(T_t)} \cap I_j-X_t \neq \emptyset$.
Hence $t \in V(T_{j,t'})-\partial T_{j,t'}$, and there exists $x \in W_0^{(i',\ell+1,k')} \cap X_t \cap I_j \neq \emptyset$ such that there exists a monochromatic path $Q_x$ of color $k'+1$ in $G[Y^{(i',\ell+1,k')} \cap W_4^{(i')}]$ from $x$ to $W_3^{(i',\ell)}$ and internally disjoint from $X_t \cup W_3^{(i',\ell)}$, and there exists a monochromatic path $Q'_x$ of color $k'+1$ with $Q_x \subseteq Q'_x \subseteq W_0^{(i',\ell+1,k')}$ in $G[Y^{(i',\ell+1,k')} \cap W_4^{(i')}]$ from $W_3^{(i',\ell)}$ to a vertex $x'$ and internally disjoint from $W_3^{(i',\ell)}$, and there exists $x'' \in N_G(x') \cap (Y^{(i',\ell+1,k'+1)}-Y^{(i',\ell+1,k')}) \cap X_{V(T_t)}-X_t \neq \emptyset$ with $x'' \in A_{L^{(i',\ell+1,k')}}(x')$ and $c(x'') \neq c(x')=c(x)=k'+1$. 
Note that $x \in W_0^{(i',\ell+1,k')} \subseteq D^{(i',\ell,k'+1)} \subseteq D^{(i',\ell,s+2)}$.
Recall $X_t \cap I_j \cap D^{(i',\ell,s+2)} = X_t \cap I_j \cap D^{(i',\ell,0)}$.
Since $\ell=0$, $X_t \cap I_j \cap D^{(i',\ell,s+2)} = X_t \cap I_j \cap D^{(i',\ell,0)} = X_t \cap I_j \cap D^{(i',0,0)} = X_t \cap I_j \cap D^{(i'-1,\lvert V(T) \rvert,s+2)}$.
So $x \in D^{(i'-1,\lvert V(T) \rvert,s+2)}$.

Let $I' = \{i'' \in [0,i'-1]: x \in \bigcup_{\ell'=0}^{\lvert V(T) \rvert}W_0^{(i'',\ell'+1,k')}\}$.
Since $x \in D^{(i'-1,\lvert V(T) \rvert,s+2)}$, $I' \neq \emptyset$.
So for every $i'' \in I'$, there exist $t_{i''}' \in V(T)$ with $\sigma_T(t_{t''}')=i''$, $\ell_{i''} \in [0,\lvert V(T) \rvert]$, a witness $q'_{i''} \in \{t'_{i''}\} \cup \partial T_{j,t'_{i''}}$ for $X_{q'_{i''}} \cap I_j \subseteq W_3^{(i'',\ell_{i''})}$, and a monochromatic path $P'_{i''}$ of color $k'+1$ in $G[Y^{(i'',\ell_{i''}+1,k')} \cap W_4^{(i'')}]$ from $X_{q'_{i''}} \cap I_j$ to $x$.
Note that if $q'_{i''}$ is in the path in $T$ from $t'$ to $t$, then $q'_{i''}=t'$, since $t \in V(T_{t',j})-\partial T_{t',j}$.

Suppose $V(Q_x') \cup \{x''\} \subseteq W_4^{(i'')}$ for some $i'' \in I'$.
If $V(Q_x') \subseteq Y^{(i'',\ell_{i''}+1,k')}$, then $P'_{i''} \cup Q_x'$ is a connected monochromatic subgraph of color $k'+1$ contained in $G[Y^{(i'',\ell_{i''}+1,k')} \cap W_4^{(i'')}]$ intersecting $X_{q'_{i''}} \cap I_j \subseteq W_3^{(i'',\ell_{i''})}$ and $x'$, so $x'' \in W_2^{(i'',\ell_{i''}+1,k')} \subseteq Y^{(i'',\ell_{i''}+1,k'+1)} \subseteq Y^{(i',0,0)}$, contradicting that $x'' \in Y^{(i',\ell+1,k'+1)}-Y^{(i',\ell+1,k')}$.
So $V(Q_x') \not \subseteq Y^{(i'',\ell_{i''}+1,k')}$.
But this contradicts \cref{claim:W_3isolate}, since $P'_{i''} \cup Q_x'$ is a connected  $c$-monochromatic subgraph  contained in $G[W_4^{(i'')}]$ intersecting $W_3^{(i'',\ell_{i''})}$.

Hence $V(Q_x') \cup \{x''\} \not \subseteq W_4^{(i'')}$ for every $i'' \in I'$.

Suppose that $x \in S^{(i',0)}_{j,t}$.
So there exists $i'' \in [0,i'-1]$ such that $x \in S^{(i''+1,0)}_{j,t}-S^{(i'',0)}_{j,t} \subseteq \bigcup_{\ell'=0}^{\lvert V(T) \rvert}W_0^{(i'',\ell'+1,k')}$ (since $c(x)=k'+1$) and $t \in V(T_{j,t''})-\partial T_{j,t''}$, where $t''$ is the node of $T$ with $\sigma_T(t'')=i''$.
Hence $i'' \in I'$ and no node in $\partial T_{j,t''}$ is in the path in $T$ from $t''$ to the parent of $t$.
Since $i''<i'<i$ and $t \in V(T_{j,t''})-\partial T_{j,t''}$, for every $q^* \in \partial T_{j,t'}$, we know that $\partial T_{j,t''}$ is disjoint from the path in $T$ from $t'$ to the parent of $q^*$, so $\partial T_{j,t''}$ is disjoint from the path in $T$ from $t''$ to the parent of $q^*$.
So $W_4^{(i')} \subseteq W_4^{(i'')}$.
Hence $V(Q_x') \cup \{x''\} \subseteq W_4^{(i')} \subseteq W_4^{(i'')}$, a contradiction.

Hence $x \not \in S^{(i',0)}_{j,t}$.
However, $x \in W_0^{(i',\ell+1,k')}$ and $t \in V(T_{j,t'})-\partial T_{j,t'}$, so $x \in S^{(i'+1,0)}_{j,t}$.
Therefore, $X_t \cap I_j \cap S^{(i'+1,0)}_{j,t} \neq X_t \cap I_j \cap S^{(i',0)}_{j,t}$, a contradiction.
This proves the claim.
\end{proof}

\begin{claim} \label{claim:1jump2simple}
Let $i,i' \in \mathbb{N}_0$ with $i'<i$.
Let $t,t'$ be nodes of $T$ with $\sigma_T(t)=i$ and $\sigma_T(t')=i'$.
Let $j \in [\lvert V \rvert-1]$.
Let $\ell \in [-1,\lvert V(T) \rvert]$ such that there exists a witness $q \in \partial T_{j,t'} \cap V(T_t)-\{t\}$ for $X_q \cap I_j \subseteq W_3^{(i',\ell+1)}$ and $X_q \cap I_j \not \subseteq W_3^{(i',\ell)}$.
Then either:
	\begin{itemize}
		\item $\lvert X_t \cap I_j \cap S_{j,t}^{(i'+1,0)} \rvert > \lvert X_t \cap I_j \cap S_{j,t}^{(i',0)} \rvert$, or
		\item $\lvert X_t \cap I_j \cap S_{j,t}^{(i'+1,1)} \rvert > \lvert X_t \cap I_j \cap S_{j,t}^{(i',1)} \rvert$, or
		\item for every  $c$-monochromatic path $P$ in $G[Y^{(i',\ell+1,s+2)} \cap I_j \cap W_4^{(i')}]$  from $W_3^{(i',\ell)}$ to 
		\begin{align*}
&		X_q \cap I_j-((D^{(i',\ell+1,0)}-X_{t'}) \cup \{v \in D^{(i',\ell+1,0)} \cap X_{t'} \cap X_{q'} \cap I_j: q' \in V(T_{t'})-\{t'\},\\
&		\quad\quad\quad q'\text{ is a witness for }X_{q'} \cap I_j \subseteq W_3^{(i'',\ell')}\text{ for some }i'' \in [0,i'-1]\text{ and }\ell' \in [0,\lvert V(T) \rvert]\})
		\end{align*}
		 and internally disjoint from $X_{t'} \cup X_{\partial T_{j,t'}}$, we have $V(P) \subseteq X_{V(T_t)}-X_t$. 
	\end{itemize}
\end{claim}

\begin{proof}
Note that $t'$ is an ancestor of $t$, since $q \in \partial T_{j,t'} \cap V(T_t)-\{t\}$.
We may assume that $\lvert X_t \cap I_j \cap S_{j,t}^{(i'+1,0)} \rvert \leq \lvert X_t \cap I_j \cap S_{j,t}^{(i',0)} \rvert$ and $\lvert X_t \cap I_j \cap S_{j,t}^{(i'+1,1)} \rvert \leq \lvert X_t \cap I_j \cap S_{j,t}^{(i',1)} \rvert$, for otherwise we are done.
So $X_t \cap I_j \cap S_{j,t}^{(i'+1,0)} = X_t \cap I_j \cap S_{j,t}^{(i',0)}$ and $X_t \cap I_j \cap S_{j,t}^{(i'+1,1)} = X_t \cap I_j \cap S_{j,t}^{(i',1)}$.

Since there exists a witness $q \in \partial T_{j,t'} \cap V(T_t)-\{t\}$ for $X_q \cap I_j \subseteq W_3^{(i',\ell+1)}$ and $X_q \cap I_j \not \subseteq W_3^{(i',\ell)}$, there exists a  $c$-monochromatic path $P$ in $G[Y^{(i',\ell+1,s+2)} \cap I_j \cap W_4^{(i')}]$  from $W_3^{(i',\ell)}$ to a vertex 
\begin{align*}
z \in X_q \cap I_j-
& ((D^{(i',\ell+1,0)}- X_{t'}) \cup \{ v \in D^{(i',\ell+1,0)} \cap X_{t'} \cap X_{q'} \cap I_j: q' \in V(T_{t'})-\{t'\}, \\
& q'
\text{ is a witness for }X_{q'} \cap I_j \subseteq W_3^{(i'',\ell')} \text{ for some }i'' \in [0,i'-1]\text{ and }\ell' \in [0,\lvert V(T) \rvert ] \} )
\end{align*} 
and internally disjoint from $X_{t'} \cup X_{\partial T_{j,t'}}$.
Furthermore, choose $P$ such that $V(P) \not \subseteq X_{V(T_t)}-X_t$ if possible.
We may assume that $V(P) \not \subseteq X_{V(T_t)}-X_t$, for otherwise we are done.

Let $k \in [0,s+1]$ be the number such that $P$ is of color $k+1$.
By \cref{claim:W_3isolate}, $V(P) \subseteq Y^{(i',\ell+1,k)}$.
Since $V(P) \not \subseteq X_{V(T_t)}-X_t$ and $q \in V(T_t)-\{t\}$, there exists a vertex $z'$ in $V(P) \cap X_t$ such that the subpath of $P$ between $z$ and $z'$ is contained in $G[X_{V(T_t)}]$ and internally disjoint from $X_t$.
If $\ell \geq 0$, then let $\ell_1=\ell$; if $\ell=-1$, then let $\ell_1=0$.
Since $\ell_1 \geq \ell$ and $V(P) \subseteq Y^{(i',\ell+1,k)}$, $V(P) \subseteq Y^{(i',\ell_1+1,k)}$.
Since $\ell_1 \geq 0$, $V(P) \subseteq W_0^{(i',\ell_1+1,k)}$.
So $z' \in S_{j,t}^{(i'+1,0)} \cap X_t \cap I_j$.
Since $X_t \cap I_j \cap S_{j,t}^{(i'+1,0)} = X_t \cap I_j \cap S_{j,t}^{(i',0)}$, $z' \in S_{j,t}^{(i',0)} \cap X_t \cap I_j$.
Hence there exists $i_1 \in [0,i'-1]$ such that $z' \in S_{j,t}^{(i_1+1,0)}-S_{j,t}^{(i_1,0)}$.
Note that such $i_1$ exists since $S_{j,t}^{(0,0)}=\emptyset$.
Let $t_1$ be the node of $T$ with $\sigma_T(t_1)=i_1$. 
Since $z' \in S_{j,t}^{(i_1+1,0)}-S_{j,t}^{(i_1,0)}$, $t \in V(T_{j,t_1})-\partial T_{j,t_1}$, and there exists $\ell_{z'} \in [0,\lvert V(T) \rvert]$ such that $z' \in W_0^{(i_1,\ell_{z'}+1,k)}$.
Hence there exists a  $c$-monochromatic path $P_{z'}$ in $G[Y^{(i_1,\ell_{z'}+1,k)} \cap W_4^{(i_1)}]$ from $z'$ to $W_3^{(i_1,\ell_{z'})}$.

Since $t \in V(T_{j,t_1}) - \partial T_{j,t_1}$, $\partial T_{j,t_1}$ is disjoint from the path in $T$ from $t_1$ to $t'$.
So $V(T_{j,t'}) \subseteq V(T_{j,t_1})$ as $i_1<i'$.
So $P \cup P_{z'}$ is a  $c$-monochromatic subgraph of color $k+1$ contained in $W_4^{(i_1)}$ containing $z$ and intersecting $W_3^{(i_1,\ell_{z'})}$.
By \cref{claim:W_3isolate}, the path $P_z$ contained in $P \cup P_{z'}$ from $W_3^{(i_1,\ell_{z'})}$ to $z$ satisfies $V(P_z) \subseteq Y^{(i_1,\ell_{z'}+1,k)}$.
Hence $V(P_z) \subseteq W_0^{(i_1,\ell_{z'}+1,k)}$.
So $z \in D^{(i_1,\ell_{z'},k+1)} \subseteq D^{(i',\ell+1,0)}$.
Since $z \in X_q \cap I_j-(D^{(i',\ell+1,0)}-X_{t'})$, $z \in X_{t'}$.
Since $t$ belongs to the path in $T$ between $t'$ and $q$, $z \in X_t$.
Since $q \in \partial T_{j,t'} \cap V(T_t)-\{t\}$, $t \in V(T_{j,t'})-\partial T_{j,t'}$.
Since $z \in X_q \cap I_j$, $z \in S_{j,t}^{(i'+1,1)}$.
Since $S_{j,t}^{(i'+1,1)} \cap X_t \cap I_j = S_{j,t}^{(i',1)} \cap X_t \cap I_j$, $z \in S_{j,t}^{(i',1)}$.
So there exists $i_1' \in [0,i'-1]$ such that $z \in S_{j,t}^{(i_1'+1,1)}-S_{j,t}^{(i_1',1)}$.
Note that $i_1'$ exists since $S_{j,t}^{(0,1)}=\emptyset$.

Let $t_1'$ be the node of $T$ with $\sigma_T(t_1')=i_1'$.
Since $z \in S_{j,t}^{(i_1'+1,1)}-S_{j,t}^{(i_1',1)}$, $t \in V(T_{j,t_1'})-\partial T_{j,t_1'}$ and there exists $q_z \in V(T_t)-\{t\}$ such that $z \in X_{q_z} \cap I_j$ and $q_z$ is a witness for $X_{q_z} \cap I_j \subseteq W_3^{(i_1',\ell_1')}$ for some $\ell_1' \in [0,\lvert V(T) \rvert]$.
Note that $q_z \in V(T_{t'})-\{t'\}$ since $q_z \in V(T_t)-\{t\}$.
However, the existence of $q_z$ contradicts that $z \in X_q \cap I_j-\{v \in D^{(i',\ell+1,0)} \cap X_{t'} \cap X_{q'} \cap I_j: q' \in V(T_{t'})-\{t'\}, q'$ is a witness for $X_{q'} \cap I_j \subseteq W_3^{(i'',\ell')}$ for some $i'' \in [0,i'-1]$ and $\ell' \in [0,\lvert V(T) \rvert]\}$.
This proves the claim.
\end{proof}

\begin{claim} \label{claim:1jump2simple2}
Let $i,i' \in \mathbb{N}_0$ with $i'<i$.
Let $t,t'$ be nodes of $T$ with $\sigma_T(t)=i$ and $\sigma_T(t')=i'$. 
Let $j \in [\lvert V \rvert-1]$.
Let $\ell \in [-1,\lvert V(T) \rvert-1]$ such that there exists no witness $q \in \partial T_{j,t'} \cap V(T_t)-\{t\}$ for $X_q \cap I_j \subseteq W_3^{(i',\ell+1)}$ and $X_q \cap I_j \not \subseteq W_3^{(i',\ell)}$.
If $Y^{(i',\ell+2,s+2)} \cap X_{V(T_t)} \cap I_j - X_t \neq Y^{(i',\ell+1,s+2)} \cap X_{V(T_t)} \cap I_j - X_t$, then either:
	\begin{itemize}
		\item $\lvert X_t \cap I_j \cap S_{j,t}^{(i'+1,0)} \rvert > \lvert X_t \cap I_j \cap S_{j,t}^{(i',0)} \rvert$, or
		\item $\lvert X_t \cap I_j \cap S_{j,t}^{(i'+1,1)} \rvert > \lvert X_t \cap I_j \cap S_{j,t}^{(i',1)} \rvert$. 
	\end{itemize}
\end{claim}

\begin{proof}
Since $Y^{(i',\ell+2,s+2)} \cap X_{V(T_t)} \cap I_j - X_t \neq Y^{(i',\ell+1,s+2)} \cap X_{V(T_t)} \cap I_j - X_t$, $t'$ is an ancestor of $t$.
Suppose to the contrary that $\lvert X_t \cap I_j \cap S_{j,t}^{(i'+1,0)} \rvert \leq \lvert X_t \cap I_j \cap S_{j,t}^{(i',0)} \rvert$ and $\lvert X_t \cap I_j \cap S_{j,t}^{(i'+1,1)} \rvert \leq \lvert X_t \cap I_j \cap S_{j,t}^{(i',1)} \rvert$.
So $X_t \cap I_j \cap S_{j,t}^{(i'+1,0)} = X_t \cap I_j \cap S_{j,t}^{(i',0)}$ and $X_t \cap I_j \cap S_{j,t}^{(i'+1,1)} = X_t \cap I_j \cap S_{j,t}^{(i',1)}$.

Since there exists no witness $q \in \partial T_{j,t'} \cap V(T_t)-\{t\}$ for $X_q \cap I_j \subseteq W_3^{(i',\ell+1)}$ and $X_q \cap I_j \not \subseteq W_3^{(i',\ell)}$, $W_3^{(i',\ell+1)} \cap X_{V(T_t)} \cap I_j - X_t = W_3^{(i',\ell)} \cap X_{V(T_t)} \cap I_j-X_t$.
This implies that $Y^{(i',\ell+2,0)} \cap X_{V(T_t)} \cap I_j-X_t = Y^{(i',\ell+1,s+2)} \cap X_{V(T_t)} \cap I_j -X_t$.
Since $Y^{(i',\ell+2,s+2)} \cap X_{V(T_t)} \cap I_j - X_t \neq Y^{(i',\ell+1,s+2)} \cap X_{V(T_t)} \cap I_j - X_t$, we have $Y^{(i',\ell+2,0)} \cap X_{V(T_t)} \cap I_j-X_t \neq Y^{(i',\ell+2,s+2)} \cap X_{V(T_t)} \cap I_j-X_t$.
Hence there exists $k \in [0,s+1]$ such that $Y^{(i',\ell+2,k+1)} \cap X_{V(T_t)} \cap I_j-X_t \neq Y^{(i',\ell+2,k)} \cap X_{V(T_t)} \cap I_j-X_t$.
So there exist $x' \in W_2^{(i',\ell+2,k)} \cap W_4^{(i')} \cap X_{V(T_t)} \cap I_j-(X_t \cup Y^{(i',\ell+2,k)})$, $x \in N_G(x') \cap W_0^{(i',\ell+2,k)}$ and a monochromatic path $Q$ in $G[Y^{(i',\ell+2,k)} \cap W_4^{(i')}]$ of color $k+1$ from $x$ to $W_3^{(i',\ell+1)}$, such that $x' \in A_{L^{(i',\ell+2,k)}}(x)$.
Note that $x \in X_{V(T_t)}$.
Let $q_Q \in \partial T_{j,t'}$ be a witness for $X_{q_Q} \cap I_j \subseteq W_3^{(i',\ell+1)}$ such that one end of $Q$ is in $X_{q_Q} \cap I_j$.
Since $x' \in X_{V(T_t)}-X_t$, $t \not \in \partial T_{j,t'}$.

Suppose that $q_Q \in V(T_t)$.
Since $q_Q \in \partial T_{j,t'}$, $q_Q \in V(T_t)-\{t\}$.
Since there exists no witness $q \in \partial T_{j,t'} \cap V(T_t)-\{t\}$ for $X_q \cap I_j \subseteq W_3^{(i',\ell+1)}$ and $X_q \cap I_j \not \subseteq W_3^{(i',\ell)}$, we know $X_{q_Q} \cap I_j \subseteq W_3^{(i',\ell)}$.
So by \cref{claim:W_3isolate}, $V(Q) \subseteq W_0^{(i',\ell+1,k)}$ and $x' \in W_2^{(i',\ell+1,k)} \subseteq Y^{(i',\ell+1,k)} \subseteq Y^{(i',\ell+2,k)}$, a contradiction.

So $q_Q \not \in V(T_t)$.
Hence there exists a vertex $x^* \in X_t \cap I_j$ such that the subpath of $Q$ between $x^*$ and $x$ is contained in $G[X_{V(T_t)}]$.
Since $x^* \in X_t \cap V(Q) \subseteq X_t \cap I_j \cap W_0^{(i',\ell+2,k)}$ and $t \in V(T_{j,t'})-\partial T_{j,t'}$, $x^* \in S_{j,t}^{(i'+1,0)} \cap X_t \cap I_j$.
Since $X_t \cap I_j \cap S_{j,t}^{(i'+1,0)} = X_t \cap I_j \cap S_{j,t}^{(i',0)}$, there exists $i_2 \in [0,i'-1]$ such that $x^* \in S_{j,t}^{(i_2+1,0)}-S_{j,t}^{(i_2,0)}$.
So the node $t_2$ of $T$ with $\sigma_T(t_2)=i_2$ satisfies that $t \in V(T_{j,t_2})-\partial T_{j,t_2}$, and there exists $\ell_{x^*} \in [0,\lvert V(T) \rvert]$ such that $x^* \in W_0^{(i_2,\ell_{x^*}+1,k)}$, since $c(x^*)=k+1$.
Hence there exists a monochromatic path $Q_{x^*}$ in $G[Y^{(i_2,\ell_{x^*}+1,k)} \cap W_4^{(i_2)}]$ from $x^*$ to $W_3^{(i_2,\ell_{x^*})}$.

Since $t \in V(T_{j,t_2})-\partial T_{j,t_2}$, $\partial T_{j,t_2}$ is disjoint from the path in $T$ from $t_2$ to $t'$.
Since $i_2<i'$, $V(T_{j,t'}) \subseteq V(T_{j,t_2})$.
So $Q \cup Q_{x^*}$ is a  $c$-monochromatic subgraph of color
$k+1$ contained in $W_4^{(i_2)}$ containing $x$ and intersects $W_3^{(i_2,\ell_{x^*})}$.
By \cref{claim:W_3isolate}, the path $Q_{x}$ contained in $Q \cup Q_{x^*}$ from $W_3^{(i_2,\ell_{x^*})}$ to $x$ satisfies $V(Q_{x}) \subseteq Y^{(i_2,\ell_{x^*}+1,k)}$ and $A_{L^{(i_2,\ell_{x^*}+1,k+1)}}(V(Q_{x})) \cap W_4^{(i_2)}=\emptyset$.
But $x' \in A_{L^{(i',\ell+2,k)}}(V(Q_x)) \cap W_4^{(i')} \subseteq A_{L^{(i_2,\ell_{x^*}+1,k+1)}}(V(Q_x)) \cap W_4^{(i_2)} = \emptyset$, a contradiction.
This proves the claim.
\end{proof}

\begin{claim} \label{claim:1jump2}
Let $i,i' \in \mathbb{N}_0$ with $i'<i$, and let $t$ be the node of $T$ with $\sigma_T(t)=i$. 
Let $j \in [\lvert V \rvert-1]$.
If $Y^{(i',0,s+2)} \cap X_{V(T_{t})} \cap I_j-X_{t} = Y^{(i',1,s+2)}\cap X_{V(T_{t})} \cap I_j-X_{t}$, then either:
	\begin{itemize}
		\item $Y^{(i',0,s+2)} \cap X_{V(T_t)} \cap I_j -X_t = Y^{(i',\lvert V(T) \rvert+1, s+2)} \cap X_{V(T_t)} \cap I_j-X_t$, or
		\item $\lvert X_t \cap I_j \cap S_{j,t}^{(i'+1,0)} \rvert > \lvert X_t \cap I_j \cap S_{j,t}^{(i',0)} \rvert$, or
		\item $\lvert X_t \cap I_j \cap S_{j,t}^{(i'+1,1)} \rvert > \lvert X_t \cap I_j \cap S_{j,t}^{(i',1)} \rvert$.
	\end{itemize}
\end{claim}

\begin{proof}
We may assume that $Y^{(i',0,s+2)} \cap X_{V(T_t)} \cap I_j -X_t \neq Y^{(i',\lvert V(T) \rvert+1, s+2)} \cap X_{V(T_t)} \cap I_j-X_t$, for otherwise we are done.
Since $Y^{(i',0,s+2)} \cap X_{V(T_{t})} \cap I_j-X_{t} = Y^{(i',1,s+2)}\cap X_{V(T_{t})} \cap I_j-X_{t}$, $Y^{(i',1,s+2)} \cap X_{V(T_t)} \cap I_j -X_t \neq Y^{(i',\lvert V(T) \rvert+1, s+2)} \cap X_{V(T_t)} \cap I_j-X_t$.
So there exists a minimum $\ell \in [0,\lvert V(T) \rvert-1]$ such that $Y^{(i',\ell+1,s+2)} \cap X_{V(T_t)} \cap I_j -X_t \neq Y^{(i',\ell+2, s+2)} \cap X_{V(T_t)} \cap I_j-X_t$.
By the minimality of $\ell$, $Y^{(i',\ell,s+2)} \cap X_{V(T_t)} \cap I_j -X_t = Y^{(i',\ell+1, s+2)} \cap X_{V(T_t)} \cap I_j-X_t$.

Let $t'$ be the node of $T$ with $\sigma_T(t')=i'$.
Since $Y^{(i',0,s+2)} \cap X_{V(T_t)} \cap I_j -X_t \neq Y^{(i',\lvert V(T) \rvert+1, s+2)} \cap X_{V(T_t)} \cap I_j-X_t$, $t \in V(T_{j,t'})-\partial T_{j,t'}$.

Suppose to the contrary that $\lvert X_t \cap I_j \cap S_{j,t}^{(i'+1,0)} \rvert \leq \lvert X_t \cap I_j \cap S_{j,t}^{(i',0)} \rvert$ and $\lvert X_t \cap I_j \cap S_{j,t}^{(i'+1,1)} \rvert \leq \lvert X_t \cap I_j \cap S_{j,t}^{(i',1)} \rvert$.
So $X_t \cap I_j \cap S_{j,t}^{(i'+1,0)} = X_t \cap I_j \cap S_{j,t}^{(i',0)}$ and $X_t \cap I_j \cap S_{j,t}^{(i'+1,1)} = X_t \cap I_j \cap S_{j,t}^{(i',1)}$.

We first suppose that there exist $\ell_0 \in [-1,\ell]$ and a witness $q \in \partial T_{j,t'} \cap V(T_t)$ for $X_q \cap I_j \subseteq W_3^{(i',\ell_0+1)}$ and $X_q \cap I_j \not \subseteq W_3^{(i',\ell_0)}$. 
Choose $\ell_0$ so that $\ell_0$ is as small as possible.
So there exists a  $c$-monochromatic path $P$ in $G[Y^{(i',\ell_0+1,s+2)} \cap I_j \cap W_4^{(i')}]$ from $W_3^{(i',\ell_0)}$ to a node $z \in X_q \cap I_j-((D^{(i',\ell_0+1,0)}-X_{t'}) \cup \{v \in D^{(i',\ell_0+1,0)} \cap X_{t'} \cap X_{q'} \cap I_j: q' \in V(T_{t'})-\{t'\}, q'$ is a witness for $X_{q'} \cap I_j \subseteq W_3^{(i'',\ell')}$ for some $i'' \in [0,i'-1]$ and $\ell' \in [0,\lvert V(T) \rvert]\})$ internally disjoint from $X_{t'} \cup X_{\partial T_{j,t'}}$.
Since $t \not \in \partial T_{j,t'}$, by \cref{claim:1jump2simple}, $V(P) \subseteq X_{V(T_t)}-X_t$.
In particular, $\ell_0 \geq 0$.

Let $k \in [0,s+1]$ be the number such that $P$ is of color $k+1$.
If $V(P)$ intersects $W_3^{(i',\ell_0-1)}$, then by \cref{claim:W_3isolate}, $V(P) \subseteq Y^{(i',\ell_0,k)}$, so $q$ is a witness for $X_q \cap I_j \subseteq W_3^{(i',\ell_0)}$, a contradiction.
So $V(P) \cap W_3^{(i',\ell_0-1)} = \emptyset$.
Then there exists a witness $q' \in V(T_t)-\{t\}$ for $X_{q'} \cap I_j \subseteq W_3^{(i',\ell_0)}$ with $X_{q'} \cap I_j \not \subseteq W_3^{(i',\ell_0-1)}$ such that $V(P)$ is from $X_{q'}$ to $X_q$, contradicting the minimality of $\ell_0$.

Therefore, there do not exist $\ell_0 \in [-1,\ell]$ and a witness $q \in \partial T_{j,t'} \cap V(T_t)$ for $X_q \cap I_j \not \subseteq W_3^{(i',\ell_0+1)}$ and $X_q \cap I_j \not \subseteq W_3^{(i',\ell_0)}$.
In particular, there exists no witness $q \in \partial T_{j,t'} \cap V(T_t)$ for $X_q \cap I_j \subseteq W_3^{(i',\ell+1)}$ and $X_q \cap I_j \not \subseteq W_3^{(i',\ell)}$.
By \cref{claim:1jump2simple2}, either $\lvert X_t \cap I_j \cap S_{j,t}^{(i'+1,0)} \rvert > \lvert X_t \cap I_j \cap S_{j,t}^{(i',0)} \rvert$, or $\lvert X_t \cap I_j \cap S_{j,t}^{(i'+1,1)} \rvert > \lvert X_t \cap I_j \cap S_{j,t}^{(i',1)} \rvert$, a contradiction.
This proves the claim.
\end{proof}

\begin{claim} \label{claim:1jumptimes}
Let $i \in \mathbb{N}_0$, and let $t$ be the node of $T$ with $\sigma_T(t)=i$.
Let $j \in [\lvert V \rvert-1]$. 
Let $S := \{i' \in [0, i-1]: Y^{(i',\lvert V(T) \rvert+1,s+2)} \cap X_{V(T_{t})} \cap I_j-X_{t} \neq Y^{(i',0,s+2)}\cap X_{V(T_{t})} \cap I_j-X_{t}\}$.
Then $\lvert S \rvert \leq 3w_0$.
\end{claim}

\begin{proof}
Let $S_1:=\{i' \in [0, i-1]: Y^{(i',1,s+2)} \cap X_{V(T_{t})} \cap I_j-X_{t} \neq Y^{(i',0,s+2)}\cap X_{V(T_{t})} \cap I_j-X_{t}\}$, and let $S_2:=\{i' \in [0,i-1]: \lvert X_t \cap I_j \cap S_{j,t}^{(i'+1,0)} \rvert+\lvert X_t \cap I_j \cap S_{j,t}^{(i'+1,1)} \rvert + \lvert X_t \cap I_j \cap D^{(i',0,s+2)} \rvert > \lvert X_t \cap I_j \cap S_{j,t}^{(i',0)} \rvert +\lvert X_t \cap I_j \cap S_{j,t}^{(i',1)} \rvert + \lvert X_t \cap I_j \cap D^{(i',0,0)} \rvert\}$.
By \cref{claim:1jump2}, $S \subseteq S_1 \cup S_2$.
By \cref{claim:1jump1}, $S_1 \subseteq S_2$. 
So $S \subseteq S_2$.
Since $\lvert X_t \cap I_j \rvert \leq w_0$, $\lvert S \rvert \leq \lvert S_2 \rvert \leq 3w_0$.
\end{proof}

\begin{claim} \label{claim:I_jsize}
Let $i,i' \in \mathbb{N}_0$ with $i'< i$, and let $t$ and $t'$ be the nodes of $T$ with $\sigma_T(t)=i$ and $\sigma_T(t')=i'$. 
Let $j \in [\lvert V \rvert-1]$.
Let $a,b$ be integers such that $I_j = \bigcup_{\alpha=a}^bV_\alpha$.
Then the following hold:
	\begin{itemize}
		\item For every $\ell \in [0,\lvert V(T) \rvert+1]$, 
		\begin{align*}
	& \lvert Y^{(i',\ell+1,s+2)} \cap (\bigcup_{\alpha=a-1}^{b+1}V_\alpha) \cap X_{V(T_t)} \cap X_{V(T_{j,t'})} \rvert \\
	\leq\; & g_2(s+2,\lvert Y^{(i',\ell,s+2)} \cap (\bigcup_{\alpha=a-1}^{b+1}V_\alpha) \cap X_{V(T_t)} \cap X_{V(T_{j,t'})} \rvert).
		\end{align*}
		\item Then $\lvert Y^{(i',\lvert V(T) \rvert+1,s+2)} \cap (\bigcup_{\alpha=a-1}^{b+1}V_\alpha) \cap X_{V(T_t)} \cap X_{V(T_{j,t'})} \rvert \leq g_3(4w_0)$. 
	\end{itemize}
\end{claim}

\begin{proof}
We may assume that $t \in V(T_{t'})$, for otherwise $X_{V(T_t)} \cap X_{V(T_{j,t'})}=\emptyset$ and we are done.

We first prove Statement~1 of this claim.
Note that for every $\ell \in [0,\lvert V(T) \rvert]$, 
\begin{align*}
& \lvert Y^{(i',\ell+1,0)} \cap (\bigcup_{\alpha=a-1}^{b+1}V_\alpha) \cap X_{V(T_{j,t'})} \cap X_{V(T_t)} \rvert \\
\leq \; & \lvert Y^{(i',\ell,s+2)} \cap (\bigcup_{\alpha=a-1}^{b+1}V_\alpha) \cap X_{V(T_{j,t'})} \cap X_{V(T_t)} \rvert + \lvert (W_3^{(i',\ell)}-X_{t'}) \cap (\bigcup_{\alpha=a-1}^{b+1}V_\alpha) \cap X_{V(T_{j,t'})} \cap X_{V(T_t)} \rvert \\
\leq\; & \lvert Y^{(i',\ell,s+2)} \cap (\bigcup_{\alpha=a-1}^{b+1}V_\alpha) \cap X_{V(T_{j,t'})} \cap X_{V(T_t)} \rvert + \lvert X_{\partial T_{j,t'}} \cap \overline{I_j} \cap X_{V(T_t)} \rvert \\
\leq\; & \lvert Y^{(i',\ell,s+2)} \cap (\bigcup_{\alpha=a-1}^{b+1}V_\alpha) \cap X_{V(T_{j,t'})} \cap X_{V(T_t)} \rvert + w_0 \\
=\; & g_2(0,\lvert Y^{(i',\ell,s+2)} \cap (\bigcup_{\alpha=a-1}^{b+1}V_\alpha) \cap X_{V(T_{j,t'})} \cap X_{V(T_t)} \rvert)
\end{align*} 
by \cref{claim:I_jbasic1}.
For every $\ell \in [0,\lvert V(T) \rvert]$ and $k \in [0,s+1]$, by \cref{claim:boundaryZ}, we know 
\begin{align*}
& W_2^{(i',\ell+1,k)} \cap (\bigcup_{\alpha=a-1}^{b+1}V_\alpha) \cap X_{V(T_t)} \\
=\; & A_{L^{(i',\ell+1,k)}}(W_0^{(i',\ell+1,k)}) \cap W_4^{(i')} \cap (\bigcup_{\alpha=a-1}^{b+1}V_\alpha) \cap X_{V(T_t)}\\ \subseteq \; & A_{L^{(i',\ell+1,k)}}(W_0^{(i',\ell+1,k)}) \cap X_{V(T_{j,t'})} \cap I_j \\
\subseteq \; & (A_{L^{(i',\ell+1,k)}}(W_0^{(i',\ell+1,k)} \cap X_{V(T_{j,t'})}) \cup X_{\partial T_{j,t'}} \cup X_{t'}) \cap X_{V(T_{j,t'})} \cap I_j \\
\subseteq \; & (A_{L^{(i',\ell+1,k)}}(W_0^{(i',\ell+1,k)} \cap I_j \cap X_{V(T_{j,t'})}) \cup X_{\partial T_{j,t'}} \cup X_{t'}) \cap X_{V(T_{j,t'})} \cap I_j \\
\subseteq \; & (N_G^{\geq s}(W_0^{(i',\ell+1,k)} \cap (\bigcup_{\alpha=a-1}^{b+1}V_\alpha) \cap X_{V(T_{j,t'})}) \cup X_{\partial T_{j,t'}} \cup X_{t'}) \cap X_{V(T_{j,t'})} \cap I_j.
\end{align*}
So for every $\ell \in [0,\lvert V(T) \rvert]$ and $k \in [0,s+1]$, 
\begin{align*}
& (Y^{(i',\ell+1,k+1)}-Y^{(i',\ell+1,k)}) \cap (\bigcup_{\alpha=a-1}^{b+1}V_\alpha) \cap X_{V(T_t)} \\
\subseteq \; & W_2^{(i',\ell+1,k)} \cap (\bigcup_{\alpha=a-1}^{b+1}V_\alpha) \cap X_{V(T_t)} \\
\subseteq \; & (N_G^{\geq s}(W_0^{(i',\ell+1,k)} \cap (\bigcup_{\alpha=a-1}^{b+1}V_\alpha) \cap X_{V(T_{j,t'})}) \cup X_{\partial T_{j,t'}} \cup X_{t'}) \cap X_{V(T_{j,t'})} \cap I_j \cap (\bigcup_{\alpha=a-1}^{b+1}V_\alpha) \cap X_{V(T_t)} \\\subseteq \; &  (N_G^{\geq s}(W_0^{(i',\ell+1,k)} \cap (\bigcup_{\alpha=a-1}^{b+1}V_\alpha) \cap X_{V(T_{j,t'})} \cap X_{V(T_t)}) \cup X_t \cup X_{\partial T_{j,t'}} \cup X_{t'}) \cap X_{V(T_{j,t'})} \cap I_j \cap X_{V(T_t)} \\
\subseteq \; & N_G^{\geq s}(Y^{(i',\ell+1,k)} \cap (\bigcup_{\alpha=a-1}^{b+1}V_\alpha) \cap X_{V(T_{j,t'})} \cap X_{V(T_t)}) \cup (X_t \cap I_j) \cup (X_{\partial T_{j,t'}} \cap I_j \cap X_{V(T_t)}) \cup (X_{t'} \cap I_j).
\end{align*}
Hence for every $\ell \in [0,\lvert V(T) \rvert]$ and $k \in [0,s+1]$, 
\begin{align*}
& \lvert (Y^{(i',\ell+1,k+1)}-Y^{(i',\ell+1,k)}) \cap (\bigcup_{\alpha=a-1}^{b+1}V_\alpha) \cap X_{V(T_t)} \rvert \\
\leq\; & \lvert N_G^{\geq s}(Y^{(i',\ell+1,k)} \cap (\bigcup_{\alpha=a-1}^{b+1}V_\alpha) \cap X_{V(T_{j,t'})} \cap X_{V(T_t)}) \rvert + \lvert X_t \cap I_j \rvert + \lvert X_{\partial T_{j,t'}} \cap I_j \cap X_{V(T_t)} \rvert + \lvert X_{t'} \cap I_j \rvert \\
\leq\; & f_1(Y^{(i',\ell+1,k)} \cap (\bigcup_{\alpha=a-1}^{b+1}V_\alpha) \cap X_{V(T_{j,t'})} \cap X_{V(T_t)}) + 3w_0
\end{align*} 
by \cref{claim:I_jbasic1}.
So for every $\ell \in [0,\lvert V(T) \rvert]$ and $k \in [0,s+1]$, 
\begin{align*}
 \lvert Y^{(i',\ell+1,k+1)} \cap (\bigcup_{\alpha=a-1}^{b+1}V_\alpha) \cap X_{V(T_{j,t'})} \cap X_{V(T_t)} \rvert
 \leq\; & f_2(Y^{(i',\ell+1,k)} \cap (\bigcup_{\alpha=a-1}^{b+1}V_\alpha) \cap X_{V(T_{j,t'})} \cap X_{V(T_t)})+3w_0.
\end{align*}
Then for every $\ell \in [0,\lvert V(T) \rvert]$, it is easy to verify 
by induction on $k \in [0,s+2]$ that 
\begin{align*}
\lvert Y^{(i',\ell+1,k)} \cap (\bigcup_{\alpha=a-1}^{b+1}V_\alpha) \cap X_{V(T_{j,t'})} \cap X_{V(T_t)} \rvert \leq g_2(k,\lvert Y^{(i',\ell,s+2)} \cap (\bigcup_{\alpha=a-1}^{b+1}V_\alpha) \cap X_{V(T_{j,t'})} \cap X_{V(T_t)} \rvert).
\end{align*}
Hence 
\begin{align*}
 \lvert Y^{(i',\ell+1,s+2)} \cap (\bigcup_{\alpha=a-1}^{b+1}V_\alpha) \cap X_{V(T_{j,t'})} \cap X_{V(T_t)} \rvert 
\leq\; & g_2(s+2,\lvert Y^{(i',\ell,s+2)} \cap (\bigcup_{\alpha=a-1}^{b+1}V_\alpha) \cap X_{V(T_{j,t'})} \cap X_{V(T_t)} \rvert).
\end{align*}

Now we prove Statement~2 of this claim.
Let 
\begin{align*}
S:=\{\ell \in [0,\lvert V(T) \rvert]: \lvert Y^{(i',\ell+1,s+2)} \cap (\bigcup_{\alpha=a-1}^{b+1}V_\alpha) \cap X_{V(T_t)} \rvert > \lvert Y^{(i',\ell,s+2)} \cap (\bigcup_{\alpha=a-1}^{b+1}V_\alpha) \cap X_{V(T_t)} \rvert\}.
\end{align*}
For every $\ell \in S$, either $(Y^{(i',\ell+1,s+2)}-Y^{(i',\ell,s+2)}) \cap X_t \cap (\bigcup_{\alpha=a-1}^{b+1}V_\alpha) \neq \emptyset$, or $(Y^{(i',\ell+1,s+2)}-Y^{(i',\ell,s+2)}) \cap (\bigcup_{\alpha=a-1}^{b+1}V_\alpha) \cap X_{V(T_t)}-X_t \neq \emptyset$.
Note that $(Y^{(i',\ell+1,s+2)}-Y^{(i',\ell,s+2)}) \cap \bigcup_{\alpha=a-1}^{b+1}V_\alpha \subseteq (Y^{(i',\ell+1,s+2)}-Y^{(i',\ell,s+2)}) \cap W_4^{(i')} \cap \bigcup_{\alpha=a-1}^{b+1}V_\alpha \subseteq (Y^{(i',\ell+1,s+2)}-Y^{(i',\ell,s+2)}) \cap I_j$.
So for every $\ell \in S$, either $(Y^{(i',\ell+1,s+2)}-Y^{(i',\ell,s+2)}) \cap X_t \cap (\bigcup_{\alpha=a-1}^{b+1}V_\alpha) \neq \emptyset$, or $(Y^{(i',\ell+1,s+2)}-Y^{(i',\ell,s+2)}) \cap I_j \cap X_{V(T_t)}-X_t \neq \emptyset$.

Let 
\begin{align*}
S_1 & := \{\ell \in [0,\lvert V(T) \rvert]: (Y^{(i',\ell+1,s+2)}-Y^{(i',\ell,s+2)}) \cap X_t \cap (\bigcup_{\alpha=a-1}^{b+1}V_\alpha) \neq \emptyset\} \text{ and}\\
S_2 & :=\{\ell \in [0,\lvert V(T) \rvert]:(Y^{(i',\ell+1,s+2)}-Y^{(i',\ell,s+2)}) \cap I_j \cap X_{V(T_t)}-X_t \neq \emptyset\}.
\end{align*}
So $S \subseteq S_1 \cup S_2$.
Note that $\lvert S_1 \rvert \leq \lvert X_t \cap (\bigcup_{\alpha=a-1}^{b+1}V_\alpha) \rvert \leq w_0$.
Let 
\begin{align*}
S_3 := \{\ell \in [0,\lvert V(T) \rvert]: (W_3^{(i',\ell)}-W_3^{(i',\ell-1)}) \cap I_j \cap X_{V(T_t)} \neq \emptyset\}.
\end{align*}
Since $(W_3^{(i',\ell)}-W_3^{(i',-1)}) \cap I_j \cap X_{V(T_t)} \subseteq X_{\partial T_{j,t'}} \cap I_j \cap X_{V(T_t)}$ for every $\ell \in [0,\lvert V(T) \rvert]$, we know $\lvert S_3 \rvert \leq \lvert X_{\partial T_{j,t'}} \cap I_j \cap X_{V(T_t)} \rvert \leq w_0$ by \cref{claim:I_jbasic1} since $i'<i$.
Let 
\begin{align*}
S_4 := 
& \{\ell \in [0,\lvert V(T) \rvert]: \lvert X_t \cap I_j \cap S_{j,t}^{(i'+1,0)} \rvert + \lvert X_t \cap I_j \cap S_{j,t}^{(i'+1,1)} \rvert > \lvert X_t \cap I_j \cap S_{j,t}^{(i',0)} \rvert + \lvert X_t \cap I_j \cap S_{j,t}^{(i',1)} \rvert\}.
\end{align*}
For every $\ell \in S_2-S_3$, we know $(W_3^{(i',\ell)}-W_3^{(i',\ell-1)}) \cap I_j \cap X_{V(T_t)} = \emptyset$, so there exists no witness $q \in \partial T_{j,t'} \cap V(T_t)-\{t\}$ for $X_q \cap I_j \subseteq W_3^{(i',\ell)}$ and $X_q \cap I_j \not \subseteq W_3^{(i',\ell-1)}$, and hence $\ell \in S_4$ by \cref{claim:1jump2simple2}.
So $S \subseteq S_1 \cup S_2 \subseteq S_1 \cup S_3 \cup S_4$.
Therefore, $\lvert S \rvert \leq \lvert S_1 \rvert + \lvert S_3 \rvert + \lvert S_4 \rvert \leq w_0 + w_0 + 2w_0 = 4w_0$.

By Statement 2 of \cref{claim:I_jsize_prep} and Statement 1 of this claim, it is easy to verify by induction on $k \in [\lvert S \rvert]$ that $\lvert Y^{(i',\ell_k,s+2)} \cap (\bigcup_{\alpha=a-1}^{b+1}V_\alpha) \cap X_{V(T_t)} \cap X_{V(T_{j,t'})} \rvert \leq g_3(k)$, where we denote the elements of $S$ by $\ell_1<\ell_2<\dots<\ell_{\lvert S \rvert}$.
Therefore, $\lvert Y^{(i',\lvert V(T) \rvert+1,s+2)} \cap (\bigcup_{\alpha=a-1}^{b+1}V_\alpha) \cap X_{V(T_t)} \cap X_{V(T_{j,t'})} \rvert \leq g_3(\lvert S \rvert) \leq g_3(4w_0)$.
\end{proof}

\begin{claim} \label{claim:isolation}
Let $j \in [\lvert \V \rvert-1]$, and let $M$ be a $c$-monochromatic component such that $S_M \cap I_j^\circ \neq \emptyset$, where $S_M$ is the $s$-segment containing $V(M)$ whose level equals the color of $M$.
Let $i \in \mathbb{N}_0$ and $t$ be a node of $T$ with $\sigma_T(t)=i$ such that $V(M) \cap X_{t^*} \neq \emptyset$ for some witness $t^* \in \partial T_{j,t} \cup \{t\}$ for $X_{t^*} \cap I_j \subseteq W_3^{(i,\ell)}$ for some $\ell \in [-1,\lvert V(T) \rvert]$.
Then $V(M) \cap X_{V(T_{j,t})} \subseteq Y^{(i,\lvert V(T) \rvert+1,s+2)}$ and $A_{L^{(i,\lvert V(T) \rvert+1,s+2)}}(V(M) \cap X_{V(T_{j,t})}) \cap X_{V(T_{j,t})} = \emptyset$.
\end{claim}

\begin{proof}
We may assume that there exists no $t' \in V(T)$ with $t \in V(T_{t'})-\{t'\}$ and $V(T_{j,t'}) \cap \partial T_{j,t} \neq \emptyset$ such that $V(M) \cap X_{t''} \neq \emptyset$ for some witness $t'' \in \partial T_{j,t'} \cup \{t'\}$ for $X_{t''} \cap I_j \subseteq W_3^{(i',\ell')}$ for some $\ell' \in [-1,\lvert V(T) \rvert]$, where $i'$ is the integer $\sigma_T(t')$, for otherwise we may replace $t$ by $t'$ due to the facts that $T_{j,t} \subseteq T_{j,t'}$ and $Y^{(i',\lvert V(T) \rvert+1,s+2)} \subseteq Y^{(i,\lvert V(T) \rvert+1,s+2)}$.

Suppose to the contrary that either $V(M) \cap X_{V(T_{j,t})} \not \subseteq Y^{(i,\lvert V(T) \rvert+1,s+2)}$, or $A_{L^{(i,\lvert V(T) \rvert+1,s+2)}}(V(M) \cap X_{V(T_{j,t})}) \cap X_{V(T_{j,t})} \neq \emptyset$.
Since $S_M \cap I_j^\circ \neq \emptyset$, $V(M) \cap X_{V(T_{j,t})} \subseteq W_4^{(i)}$.
By \cref{claim:W_3isolate}, some component $Q$ of $G[V(M) \cap X_{V(T_{j,t})}]$ is disjoint from $W_3^{(i,\lvert V(T) \rvert)}$.
Since $V(M)$ intersects $X_{t^*} \subseteq X_{\partial T_{j,t}} \cup X_t$, $V(Q)$ intersects $X_{\partial T_{j,t}} \cup X_t$.
Since $X_t \cap I_j \subseteq W_3^{(i,-1)}$, $V(Q) \cap X_t = \emptyset$.
So $V(Q)$ intersects $X_{\partial T_{j,t}}$.
Since $X_{t^*} \cap I_j \subseteq W_3^{(i,\lvert V(T) \rvert)}$, there exist a node $t' \in \partial T_{j,t}$ with $X_{t'} \cap V(Q) \neq \emptyset$ and $X_{t'} \cap I_j \not \subseteq W_3^{(i,\lvert V(T) \rvert)}$ and a path $P$ in $M$ internally disjoint from $W_3^{(i,\lvert V(T) \rvert)} \cup V(Q)$ passing through a vertex in $V(M) \cap W_3^{(i,\lvert V(T) \rvert)}$, a vertex in $X_{t'}-V(Q)$ and a vertex in $X_{t'} \cap V(Q)$ in the order listed.
Furthermore, choose $Q$ such that $P$ is as short as possible.

Let $P_0,P_1,\dots,P_m$ (some some $m \in \mathbb{N}_0)$ be the maximal subpaths of $P$ contained in $G[X_{V(T_{j,t})}]$ internally disjoint from $X_{\partial T_{j,t}} \cup X_t$, where $P_0$ intersects $W_3^{(i,\lvert V(T) \rvert)}$ and $P$ passes through\linebreak $P_0,P_1,\dots,P_m$ in the order listed.
So $P_m$ intersects $Q$.
Since $P$ is internally disjoint from $W_3^{(i,\lvert V(T) \rvert)} \supseteq X_{t}$, there exist (not necessarily distinct) $t_1,t_2,\dots,t_{m+1} \in \partial T_{j,t}$ such that for every $\ell \in [m]$, $P_\ell$ is from $X_{t_\ell}$ to $X_{t_{\ell+1}}$.
So $V(Q) \cap X_{t_{m+1}} \neq \emptyset$.

Since $V(Q) \cap W_3^{(i,\lvert V(T) \rvert)}=\emptyset$, there exists a minimum $\ell^* \in [m+1]$ such that $X_{t_{\ell^*}} \cap I_j \not\subseteq W_3^{(i,\lvert V(T) \rvert)}$.
So $X_{t_{\ell^*-1}} \cap I_j \subseteq W_3^{(i,\lvert V(T) \rvert)}$.
Let $t_0=t'$.
Note that if there exists $q \in [-1,\lvert V(T) \rvert]$ such that $I_j \cap W_3^{(i,q)} = I_j \cap W_3^{(i,q+1)}$, then $I_j \cap W_3^{(i,q)} = I_j \cap W_3^{(i,\lvert V(T) \rvert)}$.
This implies that if $q \in [-1,\lvert V(T) \rvert]$ is the minimum such that $I_j \cap W_3^{(i,q)} = I_j \cap W_3^{(i,\lvert V(T) \rvert)}$, then there exist at least $q$ nodes $t'' \in \partial T_{j,t}$ such that $X_{t''} \cap I_j \subseteq W_3^{(i,q)}$.
This together with $X_{t_{\ell^*}} \cap I_j \not \subseteq W_3^{(i,\lvert V(T) \rvert)}$ imply that $X_{t_{\ell^*-1}} \cap I_j \subseteq W_3^{(i,\lvert V(T) \rvert-1)}$.

Since $P_{\ell^*-1}$ is contained in $G[W_4^{(i)}]$ and intersects $X_{t_{\ell^*-1}}$, and $X_{t_{\ell^*-1}} \cap I_j \subseteq W_3^{(i,\lvert V(T) \rvert-1)}$, \cref{claim:W_3isolate} implies that $V(P_{\ell^*-1}) \subseteq Y^{(i,\lvert V(T) \rvert,k)}$, where $k+1$ is the color of $M$.
Let $v$ be the vertex in $V(P_{\ell^*-1}) \cap X_{t_{\ell^*}}$.
Since $X_{t_{\ell^*}} \cap I_j \not \subseteq W_3^{(i,\lvert V(T) \rvert)}$, $v \in (D^{(i,\lvert V(T) \rvert,0)}-X_t) \cup \{u \in D^{(i,\lvert V(T) \rvert,0)} \cap X_t \cap X_{q'} \cap I_j: q' \in V(T_t)-\{t\}, q'$ is a witness for $X_{q'} \cap I_j \subseteq W_3^{(i'',\ell')}$ for some $i'' \in [0,i-1]$ and $\ell' \in [0,\lvert V(T) \rvert]\}$.

Suppose there exists $q' \in V(T_t)-\{t\}$ such that $v \in D^{(i,\lvert V(T) \rvert,0)} \cap X_t \cap X_{q'} \cap I_j$ and $q'$ is a witness for $X_{q'} \cap I_j \subseteq W_3^{(i',\ell')}$ for some $i'' \in [0,i-1]$ and $\ell' \in [0,\lvert V(T) \rvert]$.
Then there exists a node $t''$ of $T$ with $\sigma_T(t'')=i'' \in [0,i-1]$ and with $t \in V(T_{t''})-\{t''\}$ such that $q' \in \partial T_{j,t''}$ is a witness for $X_{q'} \cap I_j \subseteq W_3^{(i'',\ell')}$ for some $\ell' \in [0,\lvert V(T) \rvert]$.
Since $q' \in (V(T_t)-\{t\}) \cap \partial T_{j,t''}$, we know $\partial T_{j,t''}$ is disjoint from the path in $T$ between $t''$ and $t$.
So $t_{\ell^*} \in \partial T_{j,t} \cap V(T_{j,t''})$.
Hence $t''$ is a node of $T$ with $t \in V(T_{t''})-\{t''\}$ and $V(T_{j,t''}) \cap \partial T_{j,t} \neq \emptyset$ such that $v \in V(M) \cap X_{q'} \neq \emptyset$ and $q'$ is a witness for $X_{q'} \cap I_j \subseteq W_3^{(i'',\ell')}$ for some $\ell' \in [-1,\lvert V(T) \rvert]$, a contradiction.

Therefore no such $q'$ exists, and hence $v \in D^{(i,\lvert V(T) \rvert,0)}-X_t$.
So $v \in W_0^{(i',\ell'+1,k)}$ for some $i' \in [0,i],\ell' \in [0,\lvert V(T) \rvert]$ with $(i',\ell')$ lexicographically at most $(i,\lvert V(T) \rvert-1)$.
We assume that $(i',\ell')$ is lexicographically minimum.
So there exists a monochromatic path $P'$ in $G[Y^{(i',\ell'+1,k)} \cap W_4^{(i')}]$ from $v$ to $W_3^{(i',\ell')}$ internally disjoint from $W_3^{(i',\ell')}$.
The minimality of $(i',\ell')$ implies that $v \not \in D^{(i',\ell',0)}$.
So if $i'=i$, then $t_{\ell^*}$ is a witness for $X_{t_{\ell^*}} \cap I_j \subseteq W_3^{(i,\ell'+1)} \subseteq W_3^{(i,\lvert V(T) \rvert)}$, a contradiction.
Hence $i'<i$.

Let $u$ be the end of $P'$ in $W_3^{(i',\ell')}$.
Since $P'$ is in $G[Y^{(i',\ell'+1,k)} \cap W_4^{(i')}]$ from $v$ to $W_3^{(i',\ell')}$ internally disjoint from $W_3^{(i',\ell')}$, there exists $q^* \in \partial T_{j,t''} \cup \{t''\}$, where $t''$ is the node with $\sigma_T(t'')=i'$, such that $u \in X_{q^*} \cap I_j$ and $q^*$ is a witness for $X_{q^*} \cap I_j \subseteq W_3^{(i',\ell')}$.
Since $i'<i$ and $v \in (X_{V(T_t)}-X_t) \cap W_4^{(i')}$, we know that $t \in V(T_{t''})-\{t''\}$ and $\partial T_{j,t''}$ is disjoint from the path in $T$ between $t''$ and $t$.
So $V(T_{j,t''}) \cap \partial T_{j,t} \neq \emptyset$.
Since $v \in V(M)$ and $P'$ is between $u$ and $v$, $u \in V(M)$.
So $u \in V(M) \cap X_{q^*}$ and $q^* \in \partial T_{j,t''} \cup \{t''\}$ is a witness for $X_{q^*} \cap I_j \subseteq W_3^{(i',\ell')}$ with $i'<i$, a contradiction.
This proves the claim.
\end{proof}

\subsection{The size of the first type of monochromatic component}

The goal of this subsection is to prove \cref{claim:centralcomponentsize}, which shows that one kind of monochromatic component has bounded size. In this subsection, we use some of the properties about fake edges. 
We first recall part of the algorithm that is related to the properties that we need at this point.

\medskip
\noindent\textbf{\boldmath Stage $(0,-1,0)$: Initialization:}
See \cref{subsec:alog} for details.

\medskip
For $i=0,1,2,\dots$, let $t$ be the node of $T$ with $\sigma_T(t)=i$ and perform the following steps:\\[1ex]
\hspace*{4mm} 
\textbf{\boldmath Stage: Building Fences:}
Let $E^{(0)}_{j,t} := \emptyset$ for every $j \in [\lvert \V \rvert-1]$ and $t \in V(T)$.
Other sets are defined.
See \cref{subsec:alog} for details.\\
\hspace*{4mm} 
\noindent\textbf{\boldmath Stage $(i,-1,\star)$:}
See \cref{subsec:alog} for details.\\
\hspace*{4mm}
\noindent\textbf{\boldmath Stage $(i,\geq 1,\star)$:}
$(Y^{(i,\lvert V(T) \rvert+1,s+2)},L^{(i,\lvert V(T) \rvert+1,s+2)})$, $W_3^{(i,\ell)}$ for $\ell \in [0,\lvert V(T) \rvert]$, and other sets are defined.
See \cref{subsec:alog} for details.\\
\hspace*{4mm}
\noindent\textbf{Stage: Adding Fake Edges}
	\begin{itemize}
				\item For each $j \in [\lvert \V \rvert-1]$ and $t' \in V(T)-(V(T_t)-\{t\})$, let $E^{(i+1)}_{j,t'}:=E^{(i)}_{j,t'}$.
				\item For each $j \in [\lvert \V \rvert-1]$, let $E^{(i,0)}_{j,t}:=E^{(i)}_{j,t}$.
				\item For each $j \in [\lvert \V \rvert-1]$ and $t' \in V(T_t)-\{t\}$, let $E^{(i+1)}_{j,t'}:=E^{(i,\lvert V(G) \rvert)}_{j,t}$, where every $\ell \in [0,\lvert V(G) \rvert-1]$, let $E^{(i,\ell+1)}_{j,t}$ be the union of $E^{(i,\ell)}_{j,t}$ and the set consisting of the 2-element sets $\{u,v\}$ satisfying the following.
				\begin{itemize}
					\item $\{u,v\} \subseteq Y^{(i,\lvert V(T) \rvert+1,s+2)}$.
					\item $L^{(i,\lvert V(T) \rvert+1,s+2)}(u) = L^{(i,\lvert V(T) \rvert+1,s+2)}(v)$.
					\item There exists an $s$-segment $S$ in $\Se_{j}^\circ$ whose level equals the color of $u$ and $v$ such that $\{u,v\} \subseteq S$.
					\item There exists $t'' \in \partial T_{j,t}$ such that:
					\begin{itemize}
						\item $\{u,v\} \subseteq X_{t''}$,
						\item $V(M) \cap X_{t''} \neq \emptyset$, where $M$ is the monochromatic $E^{(i,\ell)}_{j,t}$-pseudocomponent in \linebreak $G[Y^{(i,\lvert V(T) \rvert+1,s+2)}]$ such that $\sigma(M)$ is the $(\ell+1)$-th smallest among all monochromatic $E^{(i,\ell)}_{j,t}$-pseudocomponents in  $G[Y^{(i,\lvert V(T) \rvert+1,s+2)}]$,
						\item $M$ is contained in some $s$-segment in $\Se_{j}^\circ$ whose level equals the color of $M$,
						\item $t''$ is a witness for $X_{t''} \cap I_j \subseteq W_3^{(i,\ell')}$ for some $\ell' \in [0,\lvert V(T) \rvert]$,
						\item $A_{L^{(i,\lvert V(T) \rvert+1,s+2)}}(V(M)) \cap X_{V(T_{t''})}-X_{t''} \neq \emptyset$,
						\item $A_{L^{(i,\lvert V(T) \rvert+1,s+2)}}(V(M_u)) \cap X_{V(T_{t''})}-X_{t''} \neq \emptyset$,
						\item $A_{L^{(i,\lvert V(T) \rvert+1,s+2)}}(V(M_v)) \cap X_{V(T_{t''})}-X_{t''} \neq \emptyset$, and
						\item $\sigma(M) < \min\{\sigma(M_u),\sigma(M_v)\}$, 
						
					\end{itemize}
					where $M_u$ and $M_v$ are the monochromatic $E^{(i,\ell)}_{j,t}$-pseudocomponents in $G[Y^{(i,\lvert V(T) \rvert+1,s+2)}]$ containing $u$ and $v$, respectively.
				\item Some other properties.
					See \cref{subsec:alog} for details.
				\end{itemize}					
	\end{itemize}
\hspace*{4mm}
\noindent\textbf{Stage: Moving to the Next Node in the Tree:}
See \cref{subsec:alog} for details.\\
\hspace*{4mm}
\noindent\textbf{Stage: Building a New Fence:}
See \cref{subsec:alog} for details.

\medskip

Now we prove some claims.

\begin{claim} \label{claim:0jumptimesinnersimple}
Let $i \in \mathbb{N}_0$ and let $t$ be a node of $T$ with $\sigma_T(t)=i$.
Let $j \in [\lvert \V \rvert-1]$.
If $i'$ is an integer in $[0,i-1]$ such that either $\lvert Y^{(i',0,0)} \cap I_j^\circ \cap X_{V(T_t)}-X_t \rvert \neq \lvert Y^{(i',0,s+2)} \cap I_j^\circ \cap X_{V(T_t)}-X_t \rvert$ or $\lvert Y^{(i',-1,0)} \cap I_j^\circ \cap X_{V(T_t)}-X_t \rvert \neq \lvert Y^{(i',-1,w_0)} \cap I_j^\circ \cap X_{V(T_t)}-X_t \rvert$, then $\lvert S_{j,t}^{(i'+1,2)} \cap X_t \cap I_j \rvert > \lvert S_{j,t}^{(i',2)} \cap X_t \cap I_j \rvert$.
\end{claim}

\begin{proof}
Assume that $i'$ is an integer in $[0,i-1]$ such that either $\lvert Y^{(i',0,0)} \cap I_j^\circ \cap X_{V(T_t)}-X_t \rvert \neq \lvert Y^{(i',0,s+2)} \cap I_j^\circ \cap X_{V(T_t)}-X_t \rvert$ or $\lvert Y^{(i',-1,0)} \cap I_j^\circ \cap X_{V(T_t)}-X_t \rvert \neq \lvert Y^{(i',-1,w_0)} \cap I_j^\circ \cap X_{V(T_t)}-X_t \rvert$.
Let $t'$ be the node of $T$ with $\sigma_T(t')=i'$.
Then either $W_2^{(i',-1,k,\alpha)} \cap I_j^\circ \cap X_{V(T_t)}-(X_t \cup Y^{(i',-1,k,\alpha)}) \neq \emptyset$ for some $k \in [0,w_0-1]$ and $\alpha \in [0,s+1]$, or $W_2^{(i',0,\alpha)} \cap I_j^\circ \cap X_{V(T_t)}-(X_t \cup Y^{(i',0,\alpha)}) \neq \emptyset$ for some $\alpha \in [0,s+1]$.
For the former, define $\ell_0 := -1$, and let $\beta(\ell_0):=(-1,k,\alpha)$ and $\beta'(\ell_0):=(-1,k,\alpha+1)$; for the latter, define $\ell_0:=0$, and let $\beta(\ell_0):=(0,\alpha)$ and $\beta'(\ell_0):=(0,\alpha+1)$.
So there exists a  $c$-monochromatic path $P$ of color $\alpha+1$ contained in $G[W_0^{(i',\ell_0,q)}]+E_{j,t'}^{(i')}$ (for some $q$) intersecting $X_{t'}$ such that $A_{L^{(i',\beta(\ell_0))}}(V(P)) \cap (Y^{(i',\beta'(\ell_0))}-Y^{(i',\beta(\ell_0))}) \cap Z_{t'} \cap I_j^\circ \cap X_{V(T_t)}-X_t \neq \emptyset$. 

Note that $t \in V(T_{t'})$.
So $V(P) \cap X_t \neq \emptyset$, since every element of $E_{j,t'}^{(i')}$ is contained in some bag of $(T,\X)$.
Since $A_{L^{(i',\beta(\ell_0))}}(V(P)) \cap I_j^\circ \neq \emptyset$, $V(P) \subseteq I_j$.
So there exist $v \in X_{V(T_t)}$, $u \in A_{L^{(i',\beta(\ell_0))}}(\{v\}) \cap (Y^{(i',\beta'(\ell_0))}-Y^{(i',\beta(\ell_0))}) \cap Z_{t'} \cap I_j^\circ \cap X_{V(T_t)}-X_t \neq \emptyset$, and a subpath $P'$ of $P$ contained in $G[W_0^{(i',\ell_0,q)}]+E_{j,t'}^{(i')}$ from $X_t$ to $v$ internally disjoint from $X_t$.
Note that $V(P') \subseteq X_{V(T_t)}$.

Suppose there exist $i_0 \in [0,i'-1]$ with $T_{j,t'} \subseteq T_{j,t_0}$ and a $c$-monochromatic path $Q_0$ containing $v$ intersecting $X_{t''_0}$ for some witness $t''_0 \in \partial T_{j,t_0} \cup \{t_0\}$ for $X_{t''_0} \cap I_j \subseteq W_3^{(i_0,\ell_0)}$ for some $\ell_0 \in [-1,\lvert V(T) \rvert]$, where $t_0$ is the node of $T$ with $\sigma_T(t_0)=i_0$.
Since $u \in I_j^\circ \cap A_{L^{(i',\beta(\ell_0))}}(\{v\})$, $Q_0$ is contained in an $s$-segment in $\Se_j^\circ$ whose level equals the color of $Q_0$.
By \cref{claim:isolation}, $A_{L^{(i_0,\lvert V(T) \rvert+1,s+2)}}(V(Q_0)) \cap X_{V(T_{j,t_0})}=\emptyset$.
Since $u \in Z_{t'} \cap I_j^\circ$, $u \in X_{V(T_{j,t'})}$.
Since $i_0<i'$ and $T_{j,t'} \subseteq T_{j,t_0}$, $u \in A_{L^{(i',\beta(\ell_0))}}(\{v\}) \cap X_{V(T_{j,t'})} \subseteq A_{L^{(i_0,\lvert V(T) \rvert+1,s+2)}}(V(Q_0)) \cap X_{V(T_{j,t_0})}=\emptyset$, a contradiction.

Hence there do not exist $i_0 \in [0,i'-1]$ with $T_{j,t'} \subseteq T_{j,t_0}$ and a $c$-monochromatic path $Q_0$ containing $v$ intersecting $X_{t''_0}$ for some witness $t''_0 \in \partial T_{j,t_0} \cup \{t_0\}$ for $X_{t''_0} \cap I_j \subseteq W_3^{(i_0,\ell_0)}$ for some $\ell_0 \in [-1,\lvert V(T) \rvert]$, where $t_0$ is the node of $T$ with $\sigma_T(t_0)=i_0$.

Suppose that $P' \not \subseteq G[W_0^{(i',\ell_0,q)}]$.
So $E(P') \cap E_{j,t'}^{(i')} \neq \emptyset$.
Let $P_1$ be the maximal subpath of $P'-E_{j,t'}^{(i')}$ containing $v$.
So $P_1$ is a $c$-monochromatic path contained in $G[W_0^{(i',\ell_0,q)}]$, and one end of $P_1$ is contained in an element $e_1$ in $E_{j,t'}^{(i')}$.
So there exists the lexicographically minimal pair $(i_1,\ell_1)$ such that $e_1 \in E_{j,t_1}^{(i_1,\ell_1)}$, where $t_1$ is the node of $T$ with $\sigma_T(t_1)=i_1$.
Note that $i_1<i'$.
By the definition of $E_{j,t_1}^{(i_1,\ell_1)}$, there exists $t_1'' \in \partial T_{j,t_1}$ such that $e_1 \subseteq X_{t_1''}$ and $t_1''$ is a witness for $X_{t_1''} \cap I_j \subseteq W_3^{(i_1,\ell')}$ for some $\ell' \in [0,\lvert V(T) \rvert]$.
Since $V(P') \subseteq X_{V(T_t)}$ and $P'$ is internally disjoint from $X_t$ and $e_1 \in E(P')$, we know $t_1'' \in V(T_t)-\{t\} \subseteq V(T_{t'})-\{t'\}$.
Since $i_1<i'$ and $t_1'' \in \partial T_{j,t_1} \cap V(T_{t'})-\{t'\}$, $T_{j,t'} \subseteq T_{j,t_1}$.
But $P_1$ is a $c$-monochromatic path containing $v$ intersecting $X_{t''_1}$ for some witness $t''_1$ for $X_{t''_1} \cap I_j \subseteq W_3^{(i_1,\ell')}$ for some $\ell' \in [0,\lvert V(T) \rvert]$, a contradiction.

So $P' \subseteq G[W_0^{(i',\ell_0,q)}]$.
Let $x$ be an end of $P'$ in $X_t$.
Since $Z_{t'} \cap I_j^\circ \cap X_{V(T_t)}-X_t \neq \emptyset$, $t \in V(T_{j,t'})-\partial T_{j,t'}$.
Since $V(P') \subseteq W_0^{(i',\ell_0,q)}$, we know $x \in S_{j,t}^{(i'+1,2)}$. 

We may assume that $x \in S_{j,t}^{(i',2)}$, for otherwise we are done.
So there exists $i'' \in [0,i'-1]$ with $t \in V(T_{j,t_{i''}})-\partial T_{j,t_{i''}}$, where $t_{i''}$ is the node of $T$ with $\sigma_T(t_{i''})=i''$, such that $x \in S_{j,t}^{(i''+1,2)}-S_{j,t}^{(i'',2)}$.
Hence there exist $\ell'' \in \{0,-1\}$ and $k'' \in [0,w_0-1]$ such that $x \in W_0^{(i'',\ell'',k'')}$. 
So there exists a monochromatic path $Q$ in $G[Y^{(i'',\ell'',k'')}]+E_{j,t_{i''}}^{(i'')}$ from $X_{t_{i''}}$ to $x$ internally disjoint from $X_{t_{i''}}$.
Since $x \in V(P') \cap Y^{(i'',\ell'',k'')}$, there exists a maximal subpath $P^*$ of $P'$ contained in $G[Y^{(i'',\ell'',k'')}]+E_{j,t_{i''}}^{(i'')}$ containing $x$.

If $P^*=P'$, then let $a=v$ and $b=u$; otherwise there exist $a \in V(P^*)$ and $b \in N_{P'}(a) \cap V(P')-V(P^*)$.
Note that for the former, $b=u \in X_{V(T_t)}-X_t$; for the latter, $b \neq x$, so $b \in V(P')-X_t \subseteq X_{V(T_t)}-X_t$.

Since $P' \subseteq G[W_0^{(i',\ell_0,q)}]$, $E(P') \subseteq E(G)$.
So $b \not \in Y^{(i'',\ell'',k'')}$ by the fact $i''<i'$ and the maximality of $P^*$.
Hence $b \in Y^{(i',\beta'(\ell_0))}-Y^{(i'',\ell'',k'')}$.
For each $i''' \in [i'',i']$, let $t_{i'''}$ be the node of $T$ with $\sigma_T(t_{i'''})=i'''$.
Since $Q \cup P^* \subseteq G[Y^{(i'',\ell'',k'')}]+E_{j,t_{i''}}^{(i'')}$ intersects $X_{t_i''}$ and contains $a$, we know $a \in W_0^{(i'',\ell'',k'')}$.
So if $b \in Z_{t_{i''}}$, then $b \in W_2^{(i'',-1,k'',c(a)-1)} \cup W_2^{(i'',0,k'')}$, and hence $c(a) \neq c(b)$ since $a \in W_0^{(i'',\ell'',k'')}$, a contradiction.
So $b \not \in Z_{t_{i''}}$.
In particular, $b \in I_j^\circ - X_{V(T_{j,t''})}$.
Since $b \in X_{V(T_t)}-X_t$ and $t \in V(T_{j,t_{i''}})-\partial T_{j,t_{i''}}$, there exists $z \in \partial T_{j,t_{i''}} \cap V(T_t)-\{t\}$ such that $b \in X_{V(T_z)}-X_z$.
Since $V(T_t) \subseteq V(T_{t'})-\{t'\}$, $b \not \in X_{V(T_{j,t_{i'''}})}$ for every $i''' \in [i'',i']$.
However, since $b \in I_j^\circ$ and $b \in Y^{(i',\beta'(\ell_0))}$, there exists a minimum $i^* \in [i'',i']$ such that $b \in X_{V(T_{j,t_{i^*}})}$, a contradiction.
This proves the claim.
\end{proof}

\begin{claim} \label{claim:0jumptimesinner}
Let $i \in \mathbb{N}_0$ and let $t$ be the node of $T$ with $\sigma_T(t)=i$.
Let $j \in [\lvert \V \rvert-1]$.
Let $S := \{i' \in [0,i-1]: \lvert Y^{(i',-1,0)} \cap I_j^\circ \cap X_{V(T_t)}-X_t \rvert \neq \lvert Y^{(i',0,s+2)} \cap I_j^\circ \cap X_{V(T_t)}-X_t \rvert\}$.
Then $\lvert S \rvert \leq w_0$.
\end{claim}

\begin{proof}
Let $i' \in S$.
Since $\lvert Y^{(i',-1,0)} \cap I_j^\circ \cap X_{V(T_t)}-X_t \rvert \neq \lvert Y^{(i',0,s+2)} \cap I_j^\circ \cap X_{V(T_t)}-X_t \rvert$, we know either $\lvert Y^{(i',-1,0)} \cap I_j^\circ \cap X_{V(T_t)}-X_t \rvert \neq \lvert Y^{(i',-1,w_0)} \cap I_j^\circ \cap X_{V(T_t)}-X_t \rvert$ or $\lvert Y^{(i',0,0)} \cap I_j^\circ \cap X_{V(T_t)}-X_t \rvert \neq \lvert Y^{(i',0,s+2)} \cap I_j^\circ \cap X_{V(T_t)}-X_t \rvert$.
By \cref{claim:0jumptimesinnersimple}, $\lvert S_{j,t}^{(i'+1,2)} \cap X_t \cap I_j \rvert > \lvert S_{j,t}^{(i',2)} \cap X_t \cap I_j \rvert$.
Hence $\lvert S \rvert \leq \lvert X_t \cap I_j \rvert \leq w_0$.
\end{proof}

\begin{claim} \label{claim:sizeinner}
Let $i \in \mathbb{N}_0$ and let $t$ be the node of $T$ with $\sigma_T(t)=i$.
Let $j \in [\lvert \V \rvert-1]$.
Then $\lvert Y^{(i,-1,0)} \cap I_j^\circ \cap X_{V(T_t)} \rvert \leq \eta_1$.
\end{claim}

\begin{proof}
Let $S_0 := \{i' \in [0,i-1]: \lvert Y^{(i',0,s+2)} \cap I_j^\circ \cap X_{V(T_t)}-X_t \rvert > \lvert Y^{(i',-1,0)} \cap I_j^\circ \cap X_{V(T_t)}-X_t \rvert\}$.
By \cref{claim:0jumptimesinner}, $\lvert S_0 \rvert \leq w_0$.
Let $S_2 := \{i' \in [0,i-1]: \lvert Y^{(i',\lvert V(T) \rvert+1,s+2)} \cap I_j^\circ \cap X_{V(T_t)}-X_t \rvert > \lvert Y^{(i',0,s+2)} \cap I_j^\circ \cap X_{V(T_t)}-X_t \rvert\}$.
Since $I_j^\circ \subseteq I_j$, $\lvert S_2 \rvert \leq 3w_0$ by \cref{claim:1jumptimes}.

For every $i' \in [0,i]$, let $t_{i'}$ be the node of $T$ with $\sigma_T(t_{i'})=i'$.

For every $i' \in S_0$, $(Y^{(i',0,s+2)}-Y^{(i',-1,0)}) \cap I_j^\circ \cap X_{V(T_t)}-X_t \subseteq X_{V(T_{j,t_{i'}})} \cap I_j \cap X_{V(T_t)}-X_t$.
So for every $i' \in S_0$, $\lvert (Y^{(i',0,s+2)}-Y^{(i',-1,0)}) \cap I_j^\circ \cap X_{V(T_t)}-X_t \rvert \leq \lvert Y^{(i',0,s+2)} \cap I_j \cap X_{V(T_t)} \cap X_{V(T_{j,t_{i'}})} \rvert \leq g_1(s+2)$ by Statement~2 in \cref{claim:I_jsize_prep}.

For every $i' \in S_2$, $(Y^{(i',\lvert V(T) \rvert+1,s+2)}-Y^{(i',0,s+2)}) \cap I_j^\circ \cap X_{V(T_t)}-X_t \subseteq X_{V(T_{j,t_{i'}})} \cap I_j \cap X_{V(T_t)}-X_t$.
So for every $i' \in S_2$, $\lvert (Y^{(i',\lvert V(T) \rvert+1,s+2)}-Y^{(i',0,s+2)}) \cap I_j^\circ \cap X_{V(T_t)}-X_t \rvert \leq \lvert Y^{(i',\lvert V(T) \rvert+1,s+2)} \cap I_j \cap X_{V(T_t)} \cap X_{V(T_{j,t_{i'}})} \rvert \leq g_3(4w_0)$ by Statement~2 in \cref{claim:I_jsize}.

Note that for every $i' \in [0,i-1]$, $(Y^{(i'+1,-1,0)}-Y^{(i',\lvert V(T) \rvert+1,s+2)}) \cap X_{V(T_t)} \subseteq X_t$.
Hence $\lvert Y^{(i,-1,0)} \cap I_j^\circ \cap X_{V(T_t)}-X_t \rvert \leq \lvert Y^{(0,-1,0)} \cap I_j^\circ \cap X_{V(T_t)}-X_t \rvert + \lvert S_0 \rvert \cdot g_1(s+2) + \lvert S_2 \rvert \cdot g_3(4w_0) \leq 0+w_0g_1(s+2)+3w_0g_3(4w_0)$.
Therefore, $\lvert Y^{(i,-1,0)} \cap I_j^\circ \cap X_{V(T_t)} \rvert \leq w_0g_1(s+2)+3w_0g_3(4w_0) + \lvert X_t \cap I_j \rvert \leq w_0g_1(s+2)+3w_0g_3(4w_0)+w_0=\eta_1$.
\end{proof}

\begin{claim} \label{claim:0jumptimesmiddle}
Let $i \in \mathbb{N}_0$, and let $t$ be the node of $T$ with $\sigma_T(t)=i$.
Let $j \in [\lvert V \rvert-1]$.
Let $Z_j$ be the set obtained from the $j$-th belt by deleting $I_{j-1,1} \cup I_{j,0}$. 
Let $S:=\{i' \in [0,i-1]: \lvert Y^{(i',-1,0)} \cap X_{V(T_t)} \cap Z_j-X_t \rvert < \lvert Y^{(i',0,s+2)} \cap X_{V(T_t)} \cap Z_j-X_t \rvert\}$.
Then $\lvert S \rvert \leq 5w_0(2w_0+1)+2w_0$.
\end{claim}

\begin{proof}
Let $\overline{Z_j} := Z_j \cup I^\circ_{j-1} \cup I^\circ_j$.
For every $i' \in {\mathbb N}_0$, let $t_{i'}$ be the node of $T$ with $\sigma_T(t_{i'})=i'$.

For every $i' \in S$, since $\lvert Y^{(i',-1,0)} \cap X_{V(T_t)} \cap Z_j-X_t \rvert < \lvert Y^{(i',0,s+2)} \cap X_{V(T_t)} \cap Z_j-X_t \rvert$, there exist $\ell_{i'} \in \{-1,0\}$ and $k_{i'} \in [0,w_0-1]$ such that $W_0^{(i',\ell_{i'},k_{i'})} \neq \emptyset$ and $A_{L^{(i',\ell_{i'},k_{i'})}}(W_0^{(i',\ell_{i'},k_{i'})}) \cap Z_j \cap Y^{(i',0,s+2)} \cap X_{V(T_t)}-X_t \neq \emptyset$, so there exist $v_{i'} \in W_0^{(i',\ell_{i'},k_{i'})} \cap X_{V(T_t)}$ and $u_{i'} \in A_{L^{(i',\ell_{i'},k_{i'})}}(\{v_i\}) \cap Z_j \cap Y^{(i',0,s+2)} \cap X_{V(T_t)}-X_t \neq \emptyset$.
For every $i' \in S$, since $i'<i$, there exists a monochromatic path $P_{i'}$ in $G[Y^{(i',\ell_{i'},k_{i'})}]+R^{(i')}$ from $X_{t_{i'}}$ to $v_{i'} \in X_{V(T_t)}$ intersecting $N_G[Z_j]$ internally disjoint from $X_{t_{i'}}$, and $A_{L^{(i',\ell,k)}}(V(P_{i'})) \cap (Y^{(i',0,s+2)}-Y^{(i',-1,0)}) \cap Z_j \neq \emptyset$, where $R^{(i')}=E_{j,t_{i'}}^{(i')}$ if $\ell_{i'}=-1$, and $R^{(i')}=\emptyset$ if $\ell_{i'}=0$.
Since $P_{i'}$ is a monochromatic path, $V(P_{i'}) \subseteq \overline{Z_j}$ by the definition of $E_{j,t_{i'}}^{(i')}$.
For every $i' \in S$, let $P'_{i'}$ be the subpath of $P_{i'}$ from $v_{i'}$ to $X_t$ internally disjoint from $X_t$, and let $x_{i'} \in V(P_{i'}) \cap X_t$.
Note that $V(P'_{i'}) \subseteq X_{V(T_t)}$ and $x_{i'} \in X_t \cap \overline{Z_j}$.

Let $S' = \{i' \in [0,i-1]: \lvert Y^{(i',-1,0)} \cap (I_{j-1}^{\circ} \cup I_j^\circ) \cap X_{V(T_t)}-X_t \rvert \neq \lvert Y^{(i',0,s+2)} \cap (I_{j-1}^{\circ} \cup I_j^\circ) \cap X_{V(T_t)}-X_t \rvert\}$.
By \cref{claim:0jumptimesinner}, $\lvert S' \rvert \leq 2w_0$.

Suppose to the contrary that $\lvert S \rvert \geq 5w_0(2w_0+1)+2w_0+1$.
Since $\lvert X_t \cap \overline{Z_j} \rvert \leq w_0$, there exist $i_1,i_2,i_3 \in S$ with $i_1<i_2<i_3$, $\ell_{i_1}=\ell_{i_2}=\ell_{i_3}$ and $x_{i_1}=x_{i_2}=x_{i_3}$ such that $S' \cap [i_1,i_3]=\emptyset$.
Let $x=x_{i_1}$.
For each $\alpha \in \{2,3\}$, let $Q_\alpha$ be the maximal subpath of $P_{i_\alpha}' \cap G[Y^{(i_{\alpha-1},\ell_{i_{\alpha-1}},k_{i_{\alpha-1}})}]+R^{(i_{\alpha-1})}$ containing $x$, and let $a_\alpha$ be the end of $Q_\alpha$ different from $x$ if possible.
For each $\alpha \in [3]$, if $a_\alpha=v_{i_\alpha}$, then let $b_\alpha=u_{i_\alpha}$; otherwise, let $b_\alpha$ be the neighbor of $a_\alpha$ in $P'_{i_\alpha}$ not in $Q_\alpha$.

Suppose that there exists $\alpha_0 \in [2]$ such that $\{a_{\alpha_0+1},b_{\alpha_0+1}\} \in E(G) \cup R^{(i_{\alpha_0})}$.
Then $b_{\alpha_0+1} \not \in Y^{(i_{\alpha_0},\ell_{i_{\alpha_0}},k_{i_{\alpha_0}})}$ by the maximality of $Q_{\alpha_0+1}$.
If $\{a_{\alpha_0+1},b_{\alpha_0+1}\} \in R^{(i_{\alpha_0})}$, then $R^{(i_{\alpha_0})} = E_{j,t_{i_{\alpha_0}}}^{(i_{\alpha_0})}$ and $\{a_{\alpha_0+1},b_{\alpha_0+1}\} \subseteq Y^{(i_{\alpha_0}-1,\lvert V(T) \rvert+1,s+2)} \subseteq Y^{(i_{\alpha_0},\ell_{i_{\alpha_0}},k_{i_{\alpha_0}})}$, a contradiction.
So $\{a_{\alpha_0+1},b_{\alpha_0+1}\} \in E(G)$.
Note that $P_{i_{\alpha_0}} \cup Q_{\alpha_0+1}$ contains a monochromatic path in $G[Y^{(i_{\alpha_0},\ell_{i_{\alpha_0}},k_{i_{\alpha_0}})}]+R^{(i_{\alpha_0})}$ from $X_{t_{i_{\alpha_0}}}$ to $a_{\alpha_0+1}$.
So $a_{\alpha_0+1} \in W_0^{(i_{\alpha_0},\ell_{i_{\alpha_0}},k_{i_{\alpha_0}})}$ and $b_{\alpha_0+1} \in A_{L^{(i_{\alpha_0},\ell_{i_{\alpha_0}},k_{i_{\alpha_0}})}}(W_0^{(i_{\alpha_0},\ell_{i_{\alpha_0}},k_{i_{\alpha_0}})}) \cap X_{V(T_t)}-X_t$.
If $b_{\alpha_0+1} \in Z_{t_{\alpha_0}}$, then $c(a_{\alpha_0+1}) \neq c(b_{\alpha_0+1})$ and $b_{\alpha_0+1} \in Y^{(i_{\alpha_0}+1,-1,0)} \subseteq Y^{(i_{\alpha_0+1},\ell_{\alpha_0+1},k_{\alpha_0+1})}$, so the former implies that $b_{\alpha_0+1}=u_{i_{\alpha_0+1}}$ and the latter implies that $b_{\alpha_0+1} \neq u_{i_{\alpha_0+1}}$, a contradiction.
So $b_{\alpha_0+1} \not \in Z_{t_{\alpha_0}}$.
So $b_{\alpha_0+1} \in I_{j-1}^{\circ} \cup I_j^\circ$.
Note that $u_{i_{\alpha_0+1}} \in Z_j$, so $u_{i_{\alpha_0+1}} \neq b_{\alpha_0+1}$.
Hence $b_{\alpha_0+1} \in V(P_{i_{\alpha_0+1}}') \subseteq Y^{(i_{\alpha_0+1},\ell_{i_{\alpha_0+1}},k_{i_{\alpha_0+1}})}$.
So $b_{\alpha_0+1} \in (Y^{(i_{\alpha_0+1},\ell_{i_{\alpha_0+1}},k_{i_{\alpha_0+1}})}-Y^{(i_{\alpha_0},-1,0)}) \cap (I_{j-1}^{\circ} \cup I_j^\circ) \cap X_{V(T_t)}-X_t$.
Hence $[i_1,i_{\alpha_0+1}] \cap S' \neq \emptyset$, a contradiction.

Hence for each $\alpha \in [2]$, $\{a_{\alpha+1},b_{\alpha+1}\} \in R^{(i_{\alpha+1})} - R^{(i_\alpha)}$.
In particular, $\ell_{i_\alpha}=-1$ and $R^{(i_\alpha)}=E_{j,t_{i_\alpha}}^{(i_\alpha)}$ for every $\alpha \in [3]$.
For each $\alpha \in \{2,3\}$, since $\{a_\alpha,b_\alpha\} \in E_{j,t_{i_\alpha}}^{(i_\alpha)}-E_{j,t_{i_{\alpha-1}}}^{(i_{\alpha-1})}$, there exist $\beta_\alpha \in [i_{\alpha-1},i_\alpha-1]$ with $\{a_\alpha,b_\alpha\} \in E_{j,t_{i_\alpha}}^{(\beta_\alpha+1)}-E_{j,t_{i_\alpha}}^{(\beta_\alpha)}$ and a witness $q_\alpha \in \partial T_{j,t_{\beta_\alpha}}$ for $X_{q_\alpha} \cap I_j \subseteq W_3^{(\beta_\alpha,\ell')}$ for some $\ell' \in [0,\lvert V(T) \rvert]$ such that $\{a_\alpha,b_\alpha\} \subseteq X_{q_\alpha}$.
For each $\alpha \in \{2,3\}$, since $V(P'_{i_\alpha}) \subseteq X_{V(T_t)}$ and $P'_{i_\alpha}$ is internally disjoint from $X_{V(T_t)}$, $q_\alpha \in V(T_t)-\{t\}$.
So $T_{j,t_{i'}} \subseteq T_{j,t_{\beta_2}}$ for every $i' \in [\beta_2,i_3]$.

Since $P_{i_1}$ is a monochromatic path intersecting $X_{t_{i_1}}$ and $x \in X_{V(T_t)}$, $V(P_{i_1}) \cap X_{t_{\beta_3}} \neq \emptyset$.
So there exists a subpath $P''_{i_1}$ of $P_{i_1}$ from $X_{t_{\beta_3}}$ to $x \in X_{V(T_t)}$ internally disjoint from $X_{t_{\beta_3}}$.
Hence there exists $k^* \in [0,w_0-1]$ such that $V(P''_{i_1}) \cap X_{t_{\beta_3}} \subseteq W_0^{(\beta_3,-1,k^*)}$.
Let $Q^*$ be the maximal subpath of $P_{i_3}'$ contained in $G[Y^{(\beta_3,-1,k^*)}]+E_{j,t_{\beta_3}}^{(\beta_3)}$ containing $x$.
Let $a^*$ be the end of $Q^*$ other than $x$ if possible.
If $Q^*=P_{i_3}'$, then $a^*=v_{i_3}$, and we let $b^*=u_{i_3}$; otherwise, let $b^*$ be the vertex in $V(P_{i_3}')-V(Q^*)$ adjacent to $a^*$ in $P_{i_3}'$.

Suppose that $\{a^*,b^*\} \in E(G) \cup E_{j,t_{\beta_3}}^{(\beta_3)}$.
Then $b^* \not \in Y^{(\beta_3,-1,k^*)}$ by the maximality of $Q^*$, since $\beta_3<i_3$.
If $\{a^*,b^*\} \in E_{j,t_{\beta_3}}^{(\beta_3)}$, then $\{a^*,b^*\} \subseteq Y^{(\beta_3-1,\lvert V(T) \rvert+1,s+2)} \subseteq Y^{(\beta_3,-1,k^*)}$, a contradiction.
So $\{a^*,b^*\} \in E(G)$.
Note that $P_{i_1} \cup Q^*$ contains a monochromatic path in $G[Y^{(\beta_3,-1,k^*)}]+E_{j,t_{\beta_3}}^{(\beta_3)}$ from $V(P''_{i_1}) \cap X_{t_{\beta_3}}$ to $a^*$.
So $a^* \in W_0^{(\beta_3,-1,k^*)}$ and $b^* \in A_{L^{(\beta_3,-1,k^*)}}(W_0^{(\beta_3,-1,k^*)}) \cap X_{V(T_t)}-X_t$.
Since $\beta_3 <i_3$, if $b^* \in Z_{t_{\beta_3}}$, then $c(a^*) \neq c(b^*)$ and $b^* \in Y^{(\beta_3+1,-1,0)} \subseteq Y^{(i_3,\ell_3,k_3)}$, so the former implies that $b^*=u_{i_3}$ and the latter implies that $b^* \neq u_{i_3}$, a contradiction.
Hence $b^* \not \in Z_{t_{\beta_3}}$.
So $b^* \in I_{j-1}^{\circ} \cup I_j^\circ$.
Note that $u_{i_3} \in Z_j$, so $u_{i_3} \neq b^*$.
Hence $b^* \in V(P_{i_3}') \subseteq Y^{(i_3,-1,k_{i_3})}$.
So $b^* \in (Y^{(i_3,-1,k_{i_3})}-Y^{(\beta_3,-1,0)}) \cap (I_{j-1}^{\circ} \cup I_j^\circ) \cap X_{V(T_t)}-X_t$.
Hence $[i_1,i_3] \cap S' \supseteq [\beta_3,i_3] \cap S' \neq \emptyset$, a contradiction.

Hence $\{a^*,b^*\} \in E_{j,t_{i_3}}^{(i_3)}-E_{j,t_{\beta_3}}^{(\beta_3)}=E_{j,t_{i_3}}^{(i_3)}-E_{j,t_{i_3}}^{(\beta_3)}$.
So there exist $\gamma \in [\beta_3,i_3-1]$ with $\{a_\alpha,b_\alpha\} \in E_{j,t_{i_3}}^{(\gamma+1)}-E_{j,t_{i_3}}^{(\gamma)}$ and a witness $q^* \in \partial T_{j,t_{\gamma}}$ for $X_{q^*} \cap I_j \subseteq W_3^{(\gamma,\ell')}$ for some $\ell' \in [0,\lvert V(T) \rvert]$ such that $\{a^*,b^*\} \subseteq X_{q^*}$.
However, $T_{j,t_\gamma} \subseteq T_{j,t_{\beta_2}}$, $\beta_2<\gamma$, $q_2 \in \partial T_{j,t_{\beta_2}}$ is a witness for $X_{q_2} \cap I_j \subseteq W_3^{(\beta_2,\ell')}$ for some $\ell' \in [0,\lvert V(T) \rvert]$, and $Q_2 \cup Q^*$ is a monochromatic path in $G[Y^{(\beta_3,\lvert V(T) \rvert+1,s+2)}]+E_{j,t_{\beta_3}}^{(\beta_3)}$ intersects $X_{q_2}$ and $\{a^*,b^*\}$, so some $E_{j,t_{\beta_3}}^{(\beta_3)}$-pseudocomponent in $G[Y^{(\beta_3,\lvert V(T) \rvert+1,s+2)}]$ intersects $X_{q_2}$ and $\{a^*,b^*\}$, contradicting that $\{a^*,b^*\} \in E_{j,t_{i_3}}^{(\gamma+1)}$.
This proves the claim.
\end{proof}

\begin{claim} \label{claim:jumpsizemiddle}
Let $i,i' \in \mathbb{N}_0$ with $i'<i$, and let $t$ be a node of $T$ with $\sigma_T(t)=i$.
Let $j \in [\lvert V \rvert-1]$.
Let $Z_j$ be the set obtained from the $j$-th belt by deleting $I_{j-1,1} \cup I_{j,0}$. 
Then $\lvert Y^{(i',\lvert V(T) \rvert+1,s+2)} \cap Z_j \cap X_{V(T_t)} \rvert \leq g_5(\lvert Y^{(i',-1,0)} \cap Z_j \cap X_{V(T_t)} \rvert)$.
\end{claim}

\begin{proof}
Note that $N_G[Z_j] \subseteq Z_j \cup I_{j-1}^\circ \cup I_j^\circ$.
For every $k \in [0,w_0-1]$ and $q \in [0,s+1]$, 
\begin{align*}
       & \lvert (Y^{(i',-1,k,q+1)}-Y^{(i',-1,k,q)}) \cap Z_j \cap X_{V(T_t)}-X_t \rvert \\
\leq\; & \lvert A_{L^{(i',-1,k,q)}}(W_1^{(i',-1,k,q)}) \cap Z_j \cap X_{V(T_t)}-X_t \rvert \\
 \leq\; &  \lvert A_{L^{(i',-1,k,q)}}(W_1^{(i',-1,k,q)} \cap (Z_j \cup I_{j-1}^\circ \cup I_j^\circ) \cap X_{V(T_t)}) \cap Z_j \cap X_{V(T_t)}-X_t \rvert \\
 \leq\; & \lvert N_G^{\geq s}(W_1^{(i',-1,k,q)} \cap (Z_j \cup I_{j-1}^\circ \cup I_j^\circ) \cap X_{V(T_t)}) \rvert \\
 \leq\; & \lvert N_G^{\geq s}(Y^{(i',-1,k)} \cap (Z_j \cup I_{j-1}^\circ \cup I_j^\circ) \cap X_{V(T_t)}) \rvert \\
 \leq\; & f(\lvert Y^{(i',-1,k)} \cap (Z_j \cup I_{j-1}^\circ \cup I_j^\circ) \cap X_{V(T_t)} \rvert) \\
 \leq\; & f(\lvert Y^{(i',-1,k)} \cap Z_j \cap X_{V(T_t)}\rvert + \lvert Y^{(i',-1,k)} \cap (I_{j-1}^\circ \cup I_j^\circ) \cap X_{V(T_t)} \rvert) \\
 \leq\; & f(\lvert Y^{(i',-1,k)} \cap Z_j \cap X_{V(T_t)}\rvert + \lvert Y^{(i,-1,0)} \cap (I_{j-1}^\circ \cup I_j^\circ) \cap X_{V(T_t)} \rvert) \\
 \leq\; & f(\lvert Y^{(i',-1,k)} \cap Z_j \cap X_{V(T_t)}\rvert + 2\eta_1),
\end{align*} where the last inequality follows from \cref{claim:sizeinner}.
So for every $k \in [0,w_0-1]$, 
\begin{align*}
& \lvert Y^{(i',-1,k+1)} \cap Z_j \cap X_{V(T_t)} \rvert \\
\leq \; & \lvert X_t \cap Z_j \rvert + \lvert Y^{(i',-1,k,s+2)} \cap Z_j \cap X_{V(T_t)}-X_t \rvert \\
\leq\; & \lvert X_t \cap Z_j \rvert + \lvert Y^{(i',-1,k,0)} \cap Z_j \cap X_{V(T_t)}-X_t \rvert + (s+2)\cdot f(\lvert Y^{(i',-1,k)} \cap Z_j \cap X_{V(T_t)}\rvert + 2\eta_1) \\
\leq\; & w_0+\lvert Y^{(i',-1,k)} \cap Z_j \cap X_{V(T_t)}-X_t \rvert + (s+2)\cdot f(\lvert Y^{(i',-1,k)} \cap Z_j \cap X_{V(T_t)}\rvert + 2\eta_1) \\
\leq\; & (s+2)\cdot f_1(\lvert Y^{(i',-1,k)} \cap Z_j \cap X_{V(T_t)}\rvert + 2\eta_1).
\end{align*}

Hence it is easy to verify that for every $k \in [0,w_0]$, $\lvert Y^{(i',-1,k)} \cap Z_j \cap X_{V(T_t)} \rvert \leq g_4(k,\lvert Y^{(i',-1,0)} \cap Z_j \cap X_{V(T_t)} \rvert)$ by induction on $k$.
Therefore $\lvert Y^{(i',0,0)} \cap Z_j \cap X_{V(T_t)} \rvert = \lvert Y^{(i',-1,w_0)} \cap Z_j \cap X_{V(T_t)} \rvert \leq g_4(w_0,\lvert Y^{(i',-1,0)} \cap Z_j \cap X_{V(T_t)} \rvert)$.

For every $k \in [0,s+1]$, 
\begin{align*}
& \lvert (Y^{(i',0,k+1)}-Y^{(i',0,k)}) \cap Z_j \cap X_{V(T_t)}-X_t \rvert \\
\leq\; & \lvert A_{L^{(i',0,k)}}(W_0^{(i',0,k)}) \cap Z_j \cap X_{V(T_t)}-X_t \rvert\\
 \leq\; & \lvert A_{L^{(i',0,k)}}(W_0^{(i',0,k)} \cap (Z_j \cup I_{j-1}^\circ \cup I_j^\circ) \cap X_{V(T_t)}) \cap Z_j \cap X_{V(T_t)}-X_t \rvert \\
 \leq\; & \lvert N_G^{\geq s}(W_0^{(i',0,k)} \cap (Z_j \cup I_{j-1}^\circ \cup I_j^\circ) \cap X_{V(T_t)}) \rvert \\
 \leq\; & \lvert N_G^{\geq s}(Y^{(i',0,k)} \cap (Z_j \cup I_{j-1}^\circ \cup I_j^\circ) \cap X_{V(T_t)}) \rvert \\
 \leq\; & f(\lvert Y^{(i',0,k)} \cap (Z_j \cup I_{j-1}^\circ \cup I_j^\circ) \cap X_{V(T_t)} \rvert) \\
 \leq\; & f(\lvert Y^{(i',0,k)} \cap Z_j \cap X_{V(T_t)}\rvert + \lvert Y^{(i',0,k)} \cap (I_{j-1}^\circ \cup I_j^\circ) \cap X_{V(T_t)} \rvert) \\
 \leq\; & f(\lvert Y^{(i',0,k)} \cap Z_j \cap X_{V(T_t)}\rvert + \lvert Y^{(i,-1,0)} \cap (I_{j-1}^\circ \cup I_j^\circ) \cap X_{V(T_t)} \rvert) \\
 \leq\; & f(\lvert Y^{(i',0,k)} \cap Z_j \cap X_{V(T_t)}\rvert + 2\eta_1),
\end{align*}
 where the last inequality follows from \cref{claim:sizeinner}.
So for every $k \in [0,s+1]$, 
\begin{align*}
& \lvert Y^{(i',0,k+1)} \cap Z_j \cap X_{V(T_t)} \rvert \\
\leq\; & \lvert X_t \cap Z_j \rvert + \lvert Y^{(i',0,k+1)} \cap Z_j \cap X_{V(T_t)}-X_t \rvert \\
\leq\; & \lvert X_t \cap Z_j \rvert + \lvert Y^{(i',0,k)} \cap Z_j \cap X_{V(T_t)}-X_t \rvert + f(\lvert Y^{(i',0,k)} \cap Z_j \cap X_{V(T_t)}\rvert + 2\eta_1) \\
\leq\; & f_1(\lvert Y^{(i',0,k)} \cap Z_j \cap X_{V(T_t)}\rvert + 2\eta_1).
\end{align*}
Therefore, $\lvert Y^{(i',0,s+2)} \cap Z_j \cap X_{V(T_t)} \rvert \leq g_4(s+2,\lvert Y^{(i',0,0)} \cap Z_j \cap X_{V(T_t)}\rvert)$. 

Note that 
\begin{align*}
& (Y^{(i',\lvert V(T) \rvert+1,s+2)}-Y^{(i',0,s+2)}) \cap Z_j \cap X_{V(T_t)} \\
\subseteq \; & (Y^{(i',\lvert V(T) \rvert+1,s+2)}-Y^{(i',0,s+2)}) \cap ((I_{j-1} \cap X_{V(T_{j-1,t'})}) \cup (I_j \cap X_{V(T_{j,t'})})) \cap X_{V(T_t)} \\\subseteq \; & Y^{(i',\lvert V(T) \rvert+1,s+2)} \cap ((I_{j-1} \cap X_{V(T_{j-1,t'})}) \cup (I_j \cap X_{V(T_{j,t'})})) \cap X_{V(T_t)},
\end{align*} 
where $t'$ is the node of $T$ with $\sigma_T(t')=i'$.
So $\lvert (Y^{(i',\lvert V(T) \rvert+1,s+2)}-Y^{(i',0,s+2)}) \cap Z_j \cap X_{V(T_t)} \rvert \leq \lvert Y^{(i',\lvert V(T) \rvert+1,s+2)} \cap ((I_{j-1} \cap X_{V(T_{j-1,t'})}) \cup (I_j \cap X_{V(T_{j,t'})})) \cap X_{V(T_t)} \rvert \leq 2g_3(4w_0)$ by Statement~2 in \cref{claim:I_jsize}.

Therefore, 
\begin{align*}
& \lvert Y^{(i',\lvert V(T) \rvert+1,s+2)} \cap Z_j \cap X_{V(T_t)} \rvert \\
\leq\; & \lvert Y^{(i',0,s+2)} \cap Z_j \cap X_{V(T_t)} \rvert + 2g_3(4w_0) \\
\leq\; & g_4(s+2,\lvert Y^{(i',0,0)} \cap Z_j \cap X_{V(T_t)}\rvert) + 2g_3(4w_0) \\ 
\leq\; & g_4(s+2,g_4(w_0,\lvert Y^{(i',-1,0)} \cap Z_j \cap X_{V(T_t)} \rvert)) + 2g_3(4w_0) \\ 
=\; & g_5(\lvert Y^{(i',-1,0)} \cap Z_j \cap X_{V(T_t)}\rvert).\qedhere
\end{align*}
\end{proof}

\begin{claim} \label{claim:sizemiddle}
Let $i \in \mathbb{N}_0$, and let $t$ be a node of $T$ with $\sigma_T(t)=i$.
Let $j \in [\lvert V \rvert-1]$.
Let $Z_j$ be the set obtained from the $j$-th belt by deleting $I_{j-1,1} \cup I_{j,0}$. 
Then $\lvert Y^{(i,-1,0)} \cap Z_j \cap X_{V(T_t)} \rvert \leq \eta_2$.
\end{claim}

\begin{proof}
Let 
$$S_1:=\{i' \in [0,i-1]: \lvert Y^{(i',-1,0)} \cap X_{V(T_t)} \cap Z_j-X_t \rvert < \lvert Y^{(i',0,s+2)} \cap X_{V(T_t)} \cap Z_j-X_t \rvert\}.$$
By \cref{claim:0jumptimesmiddle}, $\lvert S_1 \rvert \leq 5w_0(2w_0+1)+2w_0$.
Let $$S_2:=\{i' \in [0,i-1]: \lvert Y^{(i',0,s+2)} \cap X_{V(T_t)} \cap Z_j-X_t \rvert < \lvert Y^{(i',\lvert V(T) \rvert+1,s+2)} \cap X_{V(T_t)} \cap Z_j-X_t \rvert\}.$$
Note that for every $i' \in S_2$, $(Y^{(i',\lvert V(T) \rvert+1,s+2)}-Y^{(i',0,s+2)}) \cap X_{V(T_t)} \cap Z_j-X_t \subseteq (Y^{(i',\lvert V(T) \rvert+1,s+2)}-Y^{(i',0,s+2)}) \cap X_{V(T_t)} \cap Z_j \cap (I_{j-1} \cup I_j)-X_t$, where $I_0=\emptyset$. So 
\begin{align*}
 \lvert S_2 \rvert \leq\; & \lvert \{i' \in [0,i-1]: \lvert Y^{(i',0,s+2)} \cap X_{V(T_t)} \cap I_{j-1}-X_t \rvert < \lvert Y^{(i',\lvert V(T) \rvert+1,s+2)} \cap X_{V(T_t)} \cap I_{j-1}-X_t \rvert\} \rvert \\
 &  + \lvert \{i' \in [0,i-1]: \lvert Y^{(i',0,s+2)} \cap X_{V(T_t)} \cap I_j-X_t \rvert < \lvert Y^{(i',\lvert V(T) \rvert+1,s+2)} \cap X_{V(T_t)} \cap I_j-X_t \rvert\} \rvert\\ 
 \leq\; & 6w_0
\end{align*} 
by \cref{claim:1jumptimes}.
Let $$S_3:=\{i' \in [0,i-1]: \lvert Y^{(i',-1,0)} \cap X_{V(T_t)} \cap Z_j-X_t \rvert < \lvert Y^{(i',\lvert V(T) \rvert+1,s+2)} \cap X_{V(T_t)} \cap Z_j-X_t \rvert\}.$$
So $\lvert S_3 \rvert \leq \lvert S_1 \rvert + \lvert S_2 \rvert \leq 5w_0(2w_0+1)+8w_0$.
Let $$S:=\{i' \in [0,i-1]: \lvert Y^{(i',-1,0)} \cap X_{V(T_t)} \cap Z_j-X_t \rvert < \lvert Y^{(i'+1,-1,0)} \cap X_{V(T_t)} \cap Z_j-X_t \rvert\}.$$
Since for every $i' \in [0,i-1]$, $Y^{(i'+1,-1,0)} \cap Z_j \cap X_{V(T_t)} \subseteq (Y^{(i',\lvert V(T) \rvert+1,s+2)} \cap Z_j \cap X_{V(T_t)}) \cup (X_t \cap Z_j)$, we know $\lvert S \rvert \leq \lvert S_3 \rvert + \lvert X_t \cap Z_j \rvert \leq 5w_0(2w_0+1)+9w_0$.

By \cref{claim:jumpsizemiddle}, for every $i' \in S$, $\lvert Y^{(i'+1,-1,0)} \cap Z_j \cap X_{V(T_t)} \rvert \leq \lvert Y^{(i',\lvert V(T) \rvert+1,s+2)} \cap Z_j \cap X_{V(T_t)} \rvert + \lvert X_t \cap Z_j \rvert \leq g_5(\lvert Y^{(i',-1,0)} \cap Z_j \cap X_{V(T_t)} \rvert) + \lvert X_t \cap Z_j \rvert \leq g_5(\lvert Y^{(i',-1,0)} \cap Z_j \cap X_{V(T_t)} \rvert) + w_0$.

Note that $\lvert Y^{(0,-1,0)} \cap Z_j \cap X_{V(T_t)} \rvert \leq \lvert X_t \cap Z_j \rvert \leq w_0 = g_6(0)$.
Then it is easy to verify that $\lvert Y^{(i,-1,0)} \cap Z_j \cap X_{V(T_t)} \rvert \leq g_6(\lvert S \rvert)$ by induction on the elements in $S$.
Since $\lvert S \rvert \leq 5w_0(2w_0+1)+9w_0$, $\lvert Y^{(i,-1,0)} \cap Z_j \cap X_{V(T_t)} \rvert \leq g_6(5w_0(2w_0+1)+9w_0)=\eta_2$.
\end{proof}

\begin{claim} \label{claim:boundarysize}
Let $t \in V(T)$.
Let $j \in [\lvert \V \rvert-1]$.
Then $\lvert X_{\partial T_{j,t}} \cap \overline{I_j} \rvert \leq \eta_3$.
\end{claim}

\begin{proof}
Let $p$ be the parent of $t$ if $t \neq r^*$; otherwise, let $F_{j,p}$ be the set mentioned in the algorithm when $t=r^*$.
Note that $\partial T_{j,t} \subseteq F_{j,p}$, so by \cref{claim:fencedownside}, for every $q \in \partial T_{j,t}$, $Y^{(i,-1,0)} \cap \overline{I_j} \cap X_{V(T_q)}-X_q \neq \emptyset$.
Since $X_{V(T_{q'})}-X_{q'}$ is disjoint from $X_{V(T_{q''})}-X_{q''}$ for any distinct $q',q''$ of $\partial T_{j,t}-\{t^*\}$. where $t^*$ is the node in $T_{j,t}$ with $\sigma_T(t^*)$ minimum, we have $\lvert \partial T_{j,t} \rvert \leq \lvert Y^{(i,-1,0)} \cap \overline{I_j} \rvert$.
Hence $\lvert \partial T_{j,t} \rvert \leq \lvert Y^{(i,-1,0)} \cap \overline{I_j} \cap X_{V(T_t)}-X_t \rvert+1 \leq \lvert Y^{(i,-1,0)} \cap \overline{I_j} \cap X_{V(T_t)} \rvert+1 \leq \lvert Y^{(i,-1,0)} \cap (Z_j \cup Z_{j+1} \cup I_j^\circ) \cap X_{V(T_t)} \rvert+1 \leq \eta_1+2\eta_2+1$ by \cref{claim:sizeinner,claim:sizemiddle}, where $Z_j$ and $Z_{j+1}$ are the sets obtained from the $j$-th belt by deleting $I_{j-1,1} \cup I_{j,0}$ and from the $(j+1)$-th belt by deleting $I_{j,1} \cup I_{j+1,0}$, respectively.
Therefore, $\lvert X_{\partial T_{j,t}} \cap \overline{I_j} \rvert \leq \lvert \partial T_{j,t} \rvert \cdot w_0 \leq (\eta_1+2\eta_2+1)w_0=\eta_3$.
\end{proof}

\begin{claim} \label{claim:sizeI_ji}
Let $i \in \mathbb{N}_0$ and let $t \in V(T)$ with $\sigma_T(t)=i$.
Let $j \in [\lvert \V \rvert-1]$.
Then $\lvert Y^{(i,0,0)} \cap I_j \cap X_{V(T_t)} \rvert \leq g_8(w_0)$. 
\end{claim}

\begin{proof}
Note that for every $k \in [0,w_0-1]$,  $W_0^{(i,-1,k)} \cap X_{V(T_t)} \subseteq \bigcup_{j'=1}^{\lvert \V \rvert-1}\bigcup_{S \in S_{j'}^\circ}S$.
So 
 $$(Y^{(i,-1,k+1)}-Y^{(i,-1,k)}) \cap X_{V(T_t)} \subseteq N_G[W_0^{(i,-1,k)}] \cap X_{V(T_t)} \subseteq N_G[\bigcup_{j'=1}^{\lvert \V \rvert-1}\bigcup_{S \in S_{j'}^\circ}S] \subseteq \bigcup_{j'=1}^{\lvert \V \rvert-1}\overline{I_{j'}^\circ} \subseteq \bigcup_{j'=1}^{\lvert \V \rvert-1}I_{j'}.$$
So for every $k \in [0,w_0-1]$ and $q \in [0,s+1]$, 
\begin{align*}
 (Y^{(i,-1,k,q+1)}-Y^{(i,-1,k,q)}) \cap I_j \cap X_{V(T_t)}-X_t 
& \subseteq A_{L^{(i,-1,k,1)}}(Y_1^{(i,-1,k,q)} \cap \overline{I_{j}^\circ}) \cap I_j \cap X_{V(T_t)}-X_t \\
& \subseteq N_G^{\geq s}(Y^{(i,-1,k,q)} \cap \overline{I_j^\circ}) \cap I_j \cap X_{V(T_t)}-X_t \\
& \subseteq N_G^{\geq s}(Y^{(i,-1,k,q)} \cap \overline{I_j^\circ} \cap X_{V(T_t)}).
\end{align*}
Hence for every $k \in [0,w_0-1]$ and $q \in [0,s+1]$, $\lvert (Y^{(i,-1,k,q+1)}-Y^{(i,-1,k,q)}) \cap I_j \cap X_{V(T_t)}-X_t \rvert \leq \lvert N_G^{\geq s}(Y^{(i,-1,k,q)} \cap \overline{I_j^\circ} \cap X_{V(T_t)}) \rvert \leq f(\lvert Y^{(i,-1,k,q)} \cap I_j \cap X_{V(T_t)} \rvert)$, so
\begin{align*}
 \lvert (Y^{(i,-1,k,q+1)}-Y^{(i,-1,k,q)}) \cap I_j \cap X_{V(T_t)} \rvert 
 \leq\; & \lvert (Y^{(i,-1,k,q+1)}-Y^{(i,-1,k,q)}) \cap I_j \cap X_{V(T_t)}-X_t \rvert + \lvert X_t \cap I_j \rvert \\
 \leq\; &  f(\lvert Y^{(i,-1,k,q)} \cap I_j \cap X_{V(T_t)} \rvert) + w_0.
\end{align*}
So 
\begin{align*}
 \lvert Y^{(i,-1,k,q+1)} \cap I_j \cap X_{V(T_t)} \rvert 
\leq\; & \lvert (Y^{(i,-1,k,q+1)}-Y^{(i,-1,k,q)}) \cap I_j \cap X_{V(T_t)} \rvert + \lvert Y^{(i,-1,k,q)} \cap I_j \cap X_{V(T_t)} \rvert \\
\leq\; & f(\lvert Y^{(i,-1,k,q)} \cap I_j \cap X_{V(T_t)} \rvert) + w_0 + \lvert Y^{(i,-1,k,q)} \cap I_j \cap X_{V(T_t)} \rvert \\
=\; & f_1(\lvert Y^{(i,-1,k,q)} \cap I_j \cap X_{V(T_t)} \rvert) + w_0.
\end{align*}

Note that for every $k \in [0,w_0-1]$, $\lvert Y^{(i,-1,k,0)} \cap I_j \cap X_{V(T_t)} \rvert \leq g_7(0,\lvert Y^{(i,-1,k,0)} \cap I_j \cap X_{V(T_t)} \rvert)$.
So it is easy to verify that for every $k \in [0,w_0-1]$, for every $q \in [0,s+2]$, $\lvert Y^{(i,-1,k,q)} \cap I_j \cap X_{V(T_t)} \rvert \leq g_7(q,\lvert Y^{(i,-1,k,0)} \cap I_j \cap X_{V(T_t)} \rvert)$ by induction on $q$.
That is, for every $k \in [0,w_0-1]$, $\lvert Y^{(i,-1,k+1)} \cap I_j \cap X_{V(T_t)} \rvert = \lvert Y^{(i,-1,k,s+2)} \cap I_j \cap X_{V(T_t)} \rvert \leq g_7(s+2,\lvert Y^{(i,-1,k,0)} \cap I_j \cap X_{V(T_t)} \rvert) = g_7(s+2,\lvert Y^{(i,-1,k)} \cap I_j \cap X_{V(T_t)} \rvert)$.

Note that $\lvert Y^{(i,-1,0)} \cap I_j \cap X_{V(T_t)} \rvert \leq \eta_1+2\eta_2=g_8(0)$ by \cref{claim:sizeinner,claim:sizemiddle}.
So it is easy to verify that for every $k \in [0,w_0]$, $\lvert Y^{(i,-1,k)} \cap I_j \cap X_{V(T_t)} \rvert \leq g_8(k)$ by induction on $k$.
Therefore, $\lvert Y^{(i,0,0)} \cap I_j \cap X_{V(T_t)} \rvert = \lvert Y^{(i,-1,w_0)} \cap I_j \cap X_{V(T_t)} \rvert \leq g_8(w_0)$.
\end{proof}

\begin{claim} \label{claim:sizebelt}
Let $i \in \mathbb{N}_0$ and let $t \in V(T)$ with $\sigma_T(t)=i$.
Let $j \in [\lvert \V \rvert-1]$.
Then $\lvert Y^{(i,0,s+2)} \cap B_j \cap X_{V(T_t)} \rvert \leq \eta_4$, where $B_j$ is the $j$-th belt.
\end{claim}

\begin{proof}
Let $a_j$ and $b_j$ be the integers such that $B_j = \bigcup_{\alpha=a_j}^{b_j}V_\alpha$. 
For every $k \in [0,s+2]$, let $R_k := \bigcup_{\alpha=a_j-(s+2)+k}^{b_j+(s+2)-k}V_\alpha$.
Let $Z_j := B_j-(I_{j-1,1} \cup I_{j,0})$.
Note that $B_j = R_{s+2} \subseteq R_{k+1} \subseteq I_{j-1}^\circ \cup Z_j \cup I_j^\circ=R_0$ for every $k \in [0,s+1]$.

Note that for every $k \in [0,s+1]$, 
\begin{align*}
 (Y^{(i,0,k+1)}-Y^{(i,0,k)}) \cap R_{k+1} \cap X_{V(T_t)}-X_t
 \subseteq \; & A_{L^{(i,0,k)}}(W_0^{(i,0,k)}) \cap R_{k+1} \cap X_{V(T_t)}-X_t \\
 \subseteq \; & A_{L^{(i,0,k)}}(W_0^{(i,0,k)} \cap R_k) \cap R_{k+1} \cap X_{V(T_t)}-X_t \\
 \subseteq \; & A_{L^{(i,0,k)}}(Y^{(i,0,k)} \cap R_k) \cap R_{k+1} \cap X_{V(T_t)}-X_t \\
 \subseteq \; & N_G^{\geq s}(Y^{(i,0,k)} \cap R_k \cap X_{V(T_t)}),
\end{align*}
so $\lvert (Y^{(i,0,k+1)}-Y^{(i,0,k)}) \cap R_{k+1} \cap X_{V(T_t)}-X_t \rvert \leq \lvert N_G^{\geq s}(Y^{(i,0,k)} \cap R_k \cap X_{V(T_t)}) \rvert \leq f(\lvert Y^{(i,0,k)} \cap R_k \cap X_{V(T_t)} \rvert)$, and hence
\begin{align*}
& \lvert Y^{(i,0,k+1)} \cap R_k \cap X_{V(T_t)} \rvert \\
\leq\; & \lvert Y^{(i,0,k)} \cap R_k \cap X_{V(T_t)} \rvert + \lvert (Y^{(i,0,k+1)}-Y^{(i,0,k)}) \cap R_{k+1} \cap X_{V(T_t)}-X_t \rvert + \lvert X_t \cap R_{k+1} \rvert \\
\leq\; & \lvert Y^{(i,0,k)} \cap R_k \cap X_{V(T_t)} \rvert + f(\lvert Y^{(i,0,k)} \cap R_k \cap X_{V(T_t)} \rvert) + w_0 = f_1(\lvert Y^{(i,0,k)} \cap R_k \cap X_{V(T_t)} \rvert)+w_0.
\end{align*}

Recall that $(Y^{(i,0,0)}-Y^{(i,-1,0)}) \cap B_j \cap X_{V(T_t)} \subseteq N_G[\bigcup_{S \in \Se_{j-1}^\circ \cup \Se_j^\circ}S] \cap X_{V(T_t)} \subseteq (I_{j-1} \cup I_j) \cap X_{V(T_t)}$.
So 
\begin{align*}
& \lvert Y^{(i,0,0)} \cap R_0 \cap X_{V(T_t)} \rvert \\
=\; & \lvert Y^{(i,0,0)} \cap (I_{j-1}^\circ \cup Z_j \cup I_j^\circ) \cap X_{V(T_t)} \rvert \\
\leq\; & \lvert Y^{(i,0,0)} \cap I_{j-1} \cap X_{V(T_t)} \rvert + \lvert Y^{(i,0,0)} \cap (Z_j-(I_{j-1} \cup I_j)) \cap X_{V(T_t)} \rvert + \lvert Y^{(i,0,0)} \cap I_j \cap X_{V(T_t)} \rvert \\
\leq\; & g_8(w_0) + \lvert Y^{(i,-1,0)} \cap Z_j \cap X_{V(T_t)} \rvert + g_8(w_0) \\
\leq\; & 2g_8(w_0)+\eta_2\\
=\; & g_9(0)
\end{align*} 
by \cref{claim:sizemiddle,claim:sizeI_ji}.
Hence it is easy to verify that for every $k \in [0,s+2]$, $\lvert Y^{(i,0,k)} \cap R_k \cap X_{V(T_t)} \rvert \leq g_9(k)$.
In particular, $\lvert Y^{(i,0,s+2)} \cap B_j \cap X_{V(T_t)} \rvert = \lvert Y^{(i,0,s+2)} \cap R_{s+2} \cap X_{V(T_t)} \rvert \leq g_9(s+2)=\eta_4$
\end{proof}

\begin{claim} \label{claim:centralcomponentsize}
Let $M$ be a  $c$-monochromatic component. 
If $V(M) \cap \bigcup_{j=1}^{\lvert \V \rvert-1}I_j^\circ=\emptyset$, then $\lvert V(M) \rvert \leq \eta_4$.
\end{claim}

\begin{proof}
Let $t$ be the node of $T$ with minimum $\sigma_T(t)$ such that $V(M) \cap X_{V(T_t)} \neq \emptyset$.
Let $i = \sigma_T(t)$.
By the minimality of $i$, $V(M) \subseteq X_{V(T_t)}$ and $V(M) \cap X_t \neq \emptyset$. 
We claim that $V(M) \subseteq Y^{(i,0,s+2)} \cap X_{V(T_t)}$.

Suppose to the contrary that $V(M) \not \subseteq Y^{(i,0,s+2)}$.
Let $k \in [0,s+1]$ such that $c(v)=k+1$ for every $v \in V(M)$.
Since $X_t \subseteq Y^{(i,-1,0)}$, $V(M) \cap Y^{(i,-1,0)} \cap X_t \neq \emptyset$. 
So $V(M) \cap Y^{(i,0,k)} \cap X_t \neq \emptyset$. 
Let $M'$ be the union of all components of $M[Y^{(i,0,k)}]$ intersecting $X_t$. 
Since $V(M) \subseteq X_{V(T_t)}$, $V(M') \subseteq W_0^{(i,0,k)}$.
Since $V(M) \not \subseteq Y^{(i,0,s+2)}$, $V(M) \not \subseteq Y^{(i,0,k)}$, so there exists $v \in V(M)-V(M')$ adjacent in $G$ to $V(M')$.
So $v \in A_{L^{(i,0,k)}}(V(M'))$.
Since $v \in V(M)$, and $V(M)$ is disjoint from $\bigcup_{j=1}^{\lvert \V \rvert-1}I_{j}^\circ$, $v \in Z_t$.
So $v \in W_2^{(i,0,k)}$.
Since $(Y^{(i,0,k+1)},L^{(i,0,k+1)})$ is a $(W_2^{(i,0,k)},k+1)$-progress of $(Y^{(i,0,k)},L^{(i,0,k)})$, $k+1 \not \in L^{(i,0,k+1)}(v)$.
But $c(v)=k+1$, a contradiction.

Therefore, $V(M) \subseteq Y^{(i,0,s+2)} \cap X_{V(T_t)}$.
Since $M$ is a monochromatic component disjoint from $\bigcup_{j=1}^{\lvert \V \rvert-1}I_{j}^\circ$, there exists $j^* \in [\lvert \V \rvert-1]$ such that $V(M) \subseteq Z_{j^*}$, where $Z_{j^*}$ is the set obtained from the $j^*$-th belt by deleting $I_{j^*-1,1} \cup I_{j^*,0}$.
So $\lvert V(M) \rvert \leq \lvert Y^{(i,0,s+2)} \cap X_{V(T_t)} \cap Z_{j^*} \rvert \leq \eta_4$ by \cref{claim:sizebelt}.
\end{proof}

\subsection{The size of precolored sets}
\label{subsec:PrecoloredSets}

\cref{claim:centralcomponentsize} shows that the size of a $c$-monochromatic component is bounded if it is disjoint from $\bigcup_{j=1}^{\lvert \V \rvert-1}I_j^\circ$.
To finish the proof of the lemma, it suffices to show that the size of any $c$-monochromatic component intersecting $\bigcup_{j=1}^{\lvert \V \rvert-1}I_j^\circ$ is bounded.
The remaining claims are dedicated to this task.

The goal of this subsection is to show \cref{claim:sizebeltfull} which ensures that the size of the precolored set in any ``interesting region'' is bounded.

\begin{claim} \label{claim:sizebeltfullweakweak}
Let $i \in \mathbb{N}_0$ and let $t \in V(T)$ with $\sigma_T(t)=i$.
Let $j \in [\lvert \V \rvert-1]$ and $\ell \in [0,\lvert V(T) \rvert]$.
Then $\lvert Y^{(i,\ell+1,s+2)} \cap I_{j} \cap X_{V(T_t)} \rvert \leq f_{s+2}(\lvert Y^{(i,\ell+1,0)} \cap I_{j} \cap X_{V(T_t)} \rvert + 2\eta_4)-2\eta_4$.
\end{claim}

\begin{proof}
Let $B_j$ be the $j$-th belt.
Note that $\lvert Y^{(i,\lvert V(T) \rvert+1,s+2)} \cap (B_{j}-(I_{j-1} \cup I_{j})) \cap X_{V(T_t)} \rvert = \lvert Y^{(i,0,s+2)} \cap (B_{j}-(I_{j-1} \cup I_{j})) \cap X_{V(T_t)} \rvert \leq \eta_4$ by \cref{claim:sizebelt}.
So for every $\ell \in [0,\lvert V(T) \rvert]$ and $k \in [0,s+1]$, $\lvert (Y^{(i,\ell+1,k+1)}-Y^{(i,\ell+1,k)}) \cap I_{j} \cap X_{V(T_t)} \rvert \leq \lvert N_G^{\geq s}(Y^{(i,\ell+1,k)} \cap (I_{j} \cup (B_{j}-(I_{j-1} \cup I_{j})) \cup (B_{j+1}-(I_{j} \cup I_{j+1}))) \cap X_{V(T_t)}) \rvert \leq f(\lvert Y^{(i,\ell+1,k)} \cap I_{j} \cap X_{V(T_t)} \rvert + 2\eta_4)$.
Hence for every $\ell \in [0,\lvert V(T) \rvert]$ and $k \in [0,s+1]$, $\lvert Y^{(i,\ell+1,k+1)} \cap I_{j} \cap X_{V(T_t)} \rvert \leq f_1(\lvert Y^{(i,\ell+1,k)} \cap I_{j} \cap X_{V(T_t)} \rvert + 2\eta_4)-2\eta_4$.
Then it is easy to show that for every $\ell \in [0,\lvert V(T) \rvert]$ and $k \in [0,s+1]$, $\lvert Y^{(i,\ell+1,k)} \cap I_{j} \cap X_{V(T_t)} \rvert \leq f_k(\lvert Y^{(i,\ell+1,0)} \cap I_{j} \cap X_{V(T_t)} \rvert + 2\eta_4)-2\eta_4$ by induction on $k$.
Therefore, $\lvert Y^{(i,\ell+1,s+2)} \cap I_{j} \cap X_{V(T_t)} \rvert \leq f_{s+2}(\lvert Y^{(i,\ell+1,0)} \cap I_{j} \cap X_{V(T_t)} \rvert + 2\eta_4)-2\eta_4$.
\end{proof}

\begin{claim} \label{claim:sizebeltfullweak}
Let $i \in \mathbb{N}_0$ and let $t \in V(T)$ with $\sigma_T(t)=i$.
Let $j \in [\lvert \V \rvert-1]$.
Then $\lvert Y^{(i,\lvert V(T) \rvert+1,s+2)} \cap ((B_j \cup B_{j+1})-(I_{j-1} \cup I_{j+1})) \cap X_{V(T_t)} \rvert \leq \eta_5/3$, where $B_{j'}$ is the $j'$-th belt for every $j' \in [\lvert \V \rvert]$.
\end{claim}

\begin{proof}
By \cref{claim:sizebelt}, $\lvert Y^{(i,0,s+2)} \cap (B_j \cup B_{j+1}) \cap X_{V(T_t)} \rvert \leq 2\eta_4 = g_{10}(0)$.

For every $\ell \in [0,\lvert V(T) \rvert]$, by \cref{claim:boundarysize}, 
\begin{align*}
 & \lvert Y^{(i,\ell+1,0)} \cap ((B_j \cup B_{j+1})-(I_{j-1} \cup I_{j+1})) \cap X_{V(T_t)} \rvert \\ 
\leq\; & \lvert Y^{(i,\ell,s+2)} \cap ((B_j \cup B_{j+1})-(I_{j-1} \cup I_{j+1})) \cap X_{V(T_t)} \rvert + \lvert X_{\partial T_{j,t}} \cap \overline{I_j} \rvert \\
\leq\; & \lvert Y^{(i,\ell,s+2)} \cap ((B_j \cup B_{j+1})-(I_{j-1} \cup I_{j+1}) \cap X_{V(T_t)} \rvert + \eta_3.
\end{align*}

For every $\ell \in [0,\lvert V(T) \rvert]$, $(Y^{(i,\ell+1,s+2)}-Y^{(i,\ell+1,0)}) \cap ((B_j \cup B_{j+1})-(I_{j-1} \cup I_{j+1})) \cap X_{V(T_t)} \subseteq (Y^{(i,\ell+1,s+2)}-Y^{(i,\ell+1,0)}) \cap I_j \cap X_{V(T_t)}$.
So for every $\ell \in [0,\lvert V(T) \rvert]$, by \cref{claim:sizebeltfullweakweak},
\begin{align*}
 & \lvert Y^{(i,\ell+1,s+2)} \cap ((B_j \cup B_{j+1})-(I_{j-1} \cup I_{j+1})) \cap X_{V(T_t)} \rvert \\
\leq\; & \lvert Y^{(i,\ell+1,0)} \cap ((B_j \cup B_{j+1})-(I_{j-1} \cup I_{j+1})) \cap X_{V(T_t)} \rvert + \lvert Y^{(i,\ell+1,s+2)} \cap I_{j} \cap X_{V(T_t)} \rvert \\
\leq\; & \lvert Y^{(i,\ell+1,0)} \cap ((B_j \cup B_{j+1})-(I_{j-1} \cup I_{j+1})) \cap X_{V(T_t)} \rvert + f_{s+2}(\lvert Y^{(i,\ell+1,0)} \cap I_j \cap X_{V(T_t)} \rvert + 2\eta_4)-2\eta_4 \\
\leq\; & f_{s+3}(\lvert Y^{(i,\ell+1,0)} \cap ((B_j \cup B_{j+1})-(I_{j-1} \cup I_{j+1})) \cap X_{V(T_t)} \rvert + 2\eta_4) - 2\eta_4 \\
\leq\; & f_{s+3}(\lvert Y^{(i,\ell,s+2)} \cap ((B_j \cup B_{j+1})-(I_{j-1} \cup I_{j+1})) \cap X_{V(T_t)} \rvert + \eta_3+2\eta_4).
\end{align*}

By \cref{claim:boundarysize}, $\lvert X_{\partial T_{j,t}} \cap \overline{I_j} \rvert \leq \eta_3$, so there are at most $\eta_3$ numbers $\ell \in [0,\lvert V(T) \rvert]$ such that $Y^{(i,\ell+1,0)} \cap ((B_j \cup B_{j+1})-(I_{j-1} \cup I_{j+1})) \cap X_{V(T_t)} \neq Y^{(i,\ell,0)} \cap ((B_j \cup B_{j+1})-(I_{j-1} \cup I_{j+1})) \cap X_{V(T_t)}$.
Hence it is straight forward to verify by induction on $\ell$ that $\lvert Y^{(i,\lvert V(G) \rvert+1,s+2)} \cap ((B_j \cup B_{j+1})-(I_{j-1} \cup I_{j+1})) \cap X_{V(T_t)} \rvert \leq g_{10}(\eta_3)=\eta_5/3$.
\end{proof}

\begin{claim} \label{claim:sizebeltfull}
Let $i \in \mathbb{N}_0$ and let $t \in V(T)$ with $\sigma_T(t)=i$.
Let $j \in [\lvert \V \rvert-1]$.
Then $\lvert Y^{(i,\lvert V(T) \rvert+1,s+2)} \cap (B_j \cup B_{j+1}) \cap X_{V(T_t)} \rvert \leq \eta_5$, where $B_{j'}$ is the $j'$-th belt for every $j' \in [\lvert \V \rvert]$.
\end{claim}

\begin{proof}
By \cref{claim:sizebeltfullweak}, $\lvert Y^{(i,\lvert V(T) \rvert+1,s+2)} \cap (B_j \cup B_{j+1}) \cap X_{V(T_t)} \rvert \leq \sum_{j'=j-1}^{j+1} \lvert Y^{(i,\lvert V(T) \rvert+1,s+2)} \cap ((B_{j'} \cup B_{j'+1})-(I_{j'-1} \cup I_{j'+1})) \cap X_{V(T_t)} \rvert \leq 3 \cdot \eta_5/3 = \eta_5$.
\end{proof}

\subsection{The sets $K$ and $K^*$}
\label{subsec:KKstar}

This subsection introduces two kinds of sets for a given monochromatic component. They are helpful for bounding the size of monochromatic components.

Given any $c$-monochromatic component $M$, there uniquely exists a node $r_M$ of $T$ such that $V(M) \cap X_{r_M} \neq \emptyset$ and $V(M) \subseteq X_{V(T_{r_M})}$, and we define $i_M:=\sigma_T(r_M)$.
Recall that every monochromatic component is contained in some $s$-segment.
For every $c$-monochromatic component $M$, let $S_M$ be the $s$-segment containing $V(M)$ whose level equals the color of $M$.
Recall that for every node $t$ of $T$, we denote $\sigma_T(t)$ by $i_t$.

Given any $c$-monochromatic component $M$ with $S_M \cap I_j^\circ \neq \emptyset$ for some $j \in [\lvert \V \rvert-1]$, define the following.
	\begin{itemize}
		\item Define $K(M)$ to be the subset of $V(T)$ constructed by repeatedly applying the following process until no more nodes can be added:
			\begin{itemize}
				\item $r_M \in K(M)$.
				\item For every $t \in K(M)$, if there exists $t' \in \partial T_{j,t}$ such that $V(M) \cap X_{V(T_{t'})}-X_{t'} \neq \emptyset$, then adding $t'$ into $K(M)$.
			\end{itemize}
		\item Define $K^*(M) := \{r_M\} \cup \{t' \in K(M)-\{r_M\}: t$ is the node in $K(M)$ such that $t' \in \partial T_{j,t}$, and $\lvert V(M) \cap Y^{(i_{t'},\lvert V(T) \rvert+1,s+2)} \cap X_{V(T_{t'})} \rvert > \lvert V(M) \cap Y^{(i_t, \lvert V(T) \rvert+1,s+2)} \cap X_{V(T_{t'})} \rvert\}$.
	\end{itemize}
Note that $\lvert V(M) \rvert = \lvert V(M) \cap X_{V(T_{j,r_M})} \rvert+ \sum_{t \in K(M)-\{r_M\}} \lvert V(M) \cap X_{V(T_{j,t})}-X_t \rvert$.

\begin{claim} \label{claim:bddbyK*}
Let $j \in [\lvert \V \rvert-1]$.
Let $M$ be a $c$-monochromatic component  with $S_M \cap I_j^\circ \neq \emptyset$.
Then $\lvert V(M) \rvert \leq \lvert K^*(M) \rvert \cdot \eta_5$.
\end{claim}

\begin{proof}
For every $t \in K(M)$, since $t$ is a witness for $X_t \cap I_j \subseteq W_3^{(i_t,-1)}$, by \cref{claim:isolation}, $\lvert V(M) \cap X_{V(T_{j,t})} \rvert = \lvert V(M) \cap Y^{(i_t,\lvert V(T) \rvert+1,s+2)} \cap X_{V(T_{j,t})} \rvert$.
So \begin{align*}
\lvert V(M) \rvert 
& = \lvert V(M) \cap X_{V(T_{j,r_M})} \rvert+ \sum_{t \in K(M)-\{r_M\}} \lvert V(M) \cap X_{V(T_{j,t})}-X_t \rvert \\
& = \lvert V(M) \cap Y^{(i_{r_M},\lvert V(T) \rvert+1,s+2)} \cap X_{V(T_{j,r_M})} \rvert+ 
\!\!\!\!\!\!\sum_{t \in K(M)-\{r_M\}}\!\!\!\!\!\!\!\!  \lvert V(M) \cap Y^{(i_t,\lvert V(T) \rvert+1,s+2)} \cap X_{V(T_{j,t})}-X_t \rvert.
\end{align*}
Note that for every $t \in K(M)-\{r_M\}$, there exists $k_t \in K(M)$ with $t \in V(T_{k_t})$ such that $\lvert V(M) \cap Y^{(i_{k_t},\lvert V(T) \rvert+1,s+2)} \cap X_{V(T_{j,t})}-X_t \rvert = \lvert V(M) \cap Y^{(i_t,\lvert V(T) \rvert+1,s+2)} \cap X_{V(T_{j,t})}-X_t \rvert$ (since $t$ is a candidate for $k_t$). 
Choose such a node $k_t$ such that $\sigma_T(k_t)$ is as small as possible. 
Thus $k_t \in K^*(M)$ for each $t \in K(M)-\{r_M\}$.
Let $K' = \{t' \in V(T): t'=k_t$ for some $t \in K(M)-\{r_M\}\}$.
Note that $K' \subseteq K^*(M)$, and for distinct $t_1,t_2 \in K(M)-\{r_M\}$, $X_{V(T_{j,t_1})}-X_{t_1}$ and $X_{V(T_{j,t_2})}-X_{t_2}$ are disjoint.
So 
\begin{align*}
\lvert V(M) \rvert 
& = \lvert V(M) \cap Y^{(i_{r_M},\lvert V(T) \rvert+1,s+2)} \cap X_{V(T_{j,r_M})} \rvert+ \!\!\!\sum_{t \in K(M)-\{r_M\}} \!\!\!\!\!\! \lvert V(M) \cap Y^{(i_t,\lvert V(T) \rvert+1,s+2)} \cap X_{V(T_{j,t})}-X_t \rvert \\ 
& = \lvert V(M) \cap Y^{(i_{r_M},\lvert V(T) \rvert+1,s+2)} \cap X_{V(T_{j,r_M})} \rvert + 
\!\!\!\sum_{t \in K(M)-\{r_M\}} \!\!\!\!\!\! \lvert V(M) \cap Y^{(i_{k_t},\lvert V(T) \rvert+1,s+2)} \cap X_{V(T_{j,t})}-X_t \rvert \\
& \leq \lvert V(M) \cap Y^{(i_{r_M},\lvert V(T) \rvert+1,s+2)} \cap X_{V(T_{r_M})} \rvert+\sum_{q \in K'}\lvert V(M) \cap Y^{(i_q,\lvert V(T) \rvert+1,s+2)} \cap X_{V(T_{q})} \rvert\\ 
& \leq \sum_{q \in K^*(M)}\lvert V(M) \cap Y^{(i_q,\lvert V(T) \rvert+1,s+2)} \cap X_{V(T_{q})} \rvert \\
& \leq \lvert K^*(M) \rvert \cdot \eta_5
\end{align*} by \cref{claim:sizebeltfull}.
\end{proof}

\begin{claim} \label{claim:downopenobvious}
Let $j \in [\lvert \V \rvert-1]$.
Let $M$ be a $c$-monochromatic component  with $S_M \cap I_j^\circ \neq \emptyset$.
Let $t \in V(T)$ with $V(M) \cap X_t \neq \emptyset$.
Then either $V(M) \cap X_{V(T_t)} \subseteq Y^{(i_t,-1,0)}$, or there exists a monochromatic component $M'$ in $G[Y^{(i_t,-1,0)}]$ with $M' \subseteq M$, $V(M') \cap X_t \neq \emptyset$ and $A_{L^{(i_t,-1,0)}}(V(M')) \cap X_{V(T_t)}-X_t \neq \emptyset$.
\end{claim}

\begin{proof}
We may assume $V(M) \cap X_{V(T_t)} \not \subseteq Y^{(i_t,-1,0)}$, for otherwise we are done.
Since $X_t \subseteq Y^{(i_t,-1,0)}$, there exists a path $P$ in $M$ from $X_t$ to a vertex $v \in V(M) \cap X_{V(T_t)}-Y^{(i_t,-1,0)}$ internally disjoint from $X_t$ such that $V(P-v) \subseteq Y^{(i_t,-1,0)}$.
Hence the monochromatic component $M'$ in $G[Y^{(i_t,-1,0)}]$ containing $P-v$ satisfies that $M' \subseteq M$, $V(M') \cap X_t \neq \emptyset$ and $v \in A_{L^{(i_t,-1,0)}}(V(M')) \cap X_{V(T_t)}-X_t \neq \emptyset$.
\end{proof}

\begin{claim} \label{claim:downopennotallzero}
Let $j \in [\lvert \V \rvert-1]$.
Let $M$ be a $c$-monochromatic component  with $S_M \cap I_j^\circ \neq \emptyset$.
Let $t \in K(M)$.
If $K^*(M) \cap V(T_t)-\{t\} \neq \emptyset$, then there exists a monochromatic component $M'$ in $G[Y^{(i_t,-1,0)}]$ with $M' \subseteq M$, $V(M') \cap X_t \neq \emptyset$ and $A_{L^{(i_t,-1,0)}}(V(M')) \cap X_{V(T_t)}-X_t \neq \emptyset$.
\end{claim}

\begin{proof}
Suppose to the contrary.
Since $t \in K(M)$, $V(M) \cap X_t \neq \emptyset$.
So by \cref{claim:downopenobvious}, $V(M) \cap X_{V(T_t)} \subseteq Y^{(i_t,-1,0)}$.
Since $K^*(M) \cap V(T_t)-\{t\} \neq \emptyset$, there exists $t_1 \in K^*(M) \cap V(T_t)-\{t\}$.
Since $t \in K(M)$ and $t_1 \in K^*(M) \cap V(T_t)-\{t\}$, $t \neq r_M$. 
So there exists $t_2 \in K(M)$ with $t_1 \in \partial T_{j,t_2}$ such that $V(M) \cap Y^{(i_{t_1},\lvert V(T) \rvert+1,s+2)} \cap X_{V(T_{t_1})} \neq V(M) \cap Y^{(i_{t_2},\lvert V(T) \rvert+1,s+2)} \cap X_{V(T_{t_1})}$.
Since $t,t_2 \in K(M)$ and $t_1 \in \partial T_{j,t_2} \cap X_{V(T_t)}-X_t$, $t_2 \in V(T_t)$.
So $i_{t_1}>i_{t_2} \geq i_t$ and $X_{V(T_{t_1})} \subseteq X_{V(T_t)}$.
Hence $V(M) \cap Y^{(i_{t_1},\lvert V(T) \rvert+1,s+2)} \cap X_{V(T_{t_1})} \subseteq Y^{(i_t,-1,0)} \subseteq Y^{(i_{t_2},-1,0)}$.
So $V(M) \cap Y^{(i_{t_1},\lvert V(T) \rvert+1,s+2)} \cap X_{V(T_{t_1})} = V(M) \cap Y^{(i_{t_2},\lvert V(T) \rvert+1,s+2)} \cap X_{V(T_{t_1})}$, a contradiction.
\end{proof}

\subsection{Signatures}

This subsection introduces signatures and pseudosignatures, which record the size of the gates of monochromatic components and pseudocomponents during the algorithm.
Note that a monochromatic component increases its size during the algorithm only when we try to color the gates of other monochromatic pseudocomponents.
So intuitively, signatures and pseudosignatures more or less determine the number of times that a monochromatic component can increase its size during the algorithm.
Proving this intuition is complicated and is the main motivation for later subsections.

Recall that $\sigma$ is a linear order of $V(G)$ such that for any distinct vertices $u,v$, if $\sigma_T(r_u)<\sigma_T(r_v)$, then $\sigma(u)<\sigma(v)$; and for every subgraph $H$ of $G$, $\sigma(H)$ is defined to be $\min\{\sigma(v): v \in V(H)\}$.

For any $i \in [0,\lvert V(T) \rvert-1]$, $\ell \in [-1,\lvert V(T) \rvert+1]$, $k \in [0,s+2]$ and $j \in [\lvert \V \rvert-1]$, we define the \mathdef{$(i,\ell,k,j)$}{signature} to be the sequence $(a_1,a_2,\dots,a_{\lvert V(G) \rvert})$ such that for every $\alpha \in [\lvert V(G) \rvert]$, 
\begin{itemize}
	\item if $\sigma(M)=\alpha$ for some monochromatic component $M$ in $G[Y^{(i,\ell,k)}]$ intersecting $X_t$ (where $t$ is the node of $T$ with $i_t=i$) with $A_{L^{(i,\ell,k)}}(V(M)) \cap X_{V(T_t)}-X_t \neq \emptyset$ and contained in some $s$-segment in $\Se_j^\circ$ whose level equals the color of $M$, then define $a_\alpha = \lvert A_{L^{(i,\ell,k)}}(V(M)) \cap X_{V(T_t)}-X_t \rvert$,
	\item otherwise, define $a_\alpha=0$.
\end{itemize}
In the case $a_\alpha>0$, we say that $M$ \defn{defines $a_\alpha$}.

\begin{claim} \label{claim:signaturenumbernonzero}
	Let $i \in [0,\lvert V(T) \rvert-1]$, $\ell \in [-1,\lvert V(T) \rvert+1]$, $k \in [0,s+2]$ and $j \in [\lvert \V \rvert-1]$.
	Then the $(i,\ell,k,j)$-signature has at most $w_0$ nonzero entries, and every entry is at most $f(\eta_5)$.
\end{claim}

\begin{proof}
	Let $t$ be the node of $T$ with $i_t=i$.
	Since $\bigcup_{S \in \Se_j^\circ}S \subseteq I_j$ and $\lvert X_t \cap I_j \rvert \leq w_0$, there are at most $w_0$ monochromatic components $M$ in $G[Y^{(i,\ell,k)}]$ intersecting $X_t$ with $A_{L^{(i,\ell,k)}}(V(M)) \cap X_{V(T_t)}-X_t \neq \emptyset$ and contained in some $s$-segment in $\Se_j^\circ$ whose level equals the color of $M$.
	So the $(i,\ell,k,j)$-signature has at most $w_0$ nonzero entries.
	
	Let $a_\alpha$ be a nonzero entry of the $(i,\ell,k,j)$-signature.
	Let $M$ be the monochromatic component defining $a_\alpha$.
	So $a_\alpha = \lvert A_{L^{(i,\ell,k)}}(V(M)) \cap X_{V(T_t)}-X_t \rvert \leq \lvert N_G^{\geq s}(Y^{(i,\ell,k)} \cap \overline{I_j} \cap X_{V(T_t)}) \cap X_{V(T_t)} \rvert \leq \lvert N_G^{\geq s}(Y^{(i,\lvert V(T) \rvert+1,s+2)} \cap \overline{I_j} \cap X_{V(T_t)}) \cap X_{V(T_t)} \rvert \leq f(\eta_5)$ by \cref{claim:sizebeltfull}.
\end{proof}

For any $i \in [0,\lvert V(T) \rvert-1]$, $\ell \in [-1,\lvert V(T) \rvert+1]$, $k \in [0,s+2]$ and $j \in [\lvert \V \rvert-1]$, we define the \mathdef{$(i,\ell,k,j)$}{pseudosignature} to be the sequence $(p_1,p_2,\dots,p_{\lvert V(G) \rvert})$ such that for every $\alpha \in [\lvert V(G) \rvert]$, $p_\alpha$ is a sequence of length $\lvert V(G) \rvert$ such that 
\begin{itemize}
	\item if $\sigma(M)=\alpha$ for some monochromatic $E^{(i)}_{j,t}$-pseudocomponent $M$ in $G[Y^{(i,\ell,k)}]$ intersecting $X_t$ (where $t$ is the node of $T$ with $i_t=i$) with $A_{L^{(i,\ell,k)}}(V(M)) \cap X_{V(T_t)}-X_t \neq \emptyset$ and contained in some $s$-segment in $\Se_j^\circ$ whose level equals the color of $M$, then for every $\beta \in [\lvert V(G) \rvert]$, 
		\begin{itemize}
			\item the $\beta$-th entry of $p_\alpha$ is the $\beta$-th entry of the $(i,\ell,k,j)$-signature if the $\beta$-th entry of the $(i,\ell,k,j)$-signature is positive and the monochromatic component defining the $\beta$-th entry of the $(i,\ell,k,j)$-signature is contained in $M$, 
			\item otherwise, the $\beta$-th entry of $p_\alpha$ is 0,
		\end{itemize}
	\item otherwise, $p_\alpha$ is a zero sequence.
\end{itemize}
In the case $p_\alpha$ is not a zero sequence, we say that $M$ \defn{defines $p_\alpha$}. Note that the $(i,\ell,k,j)$-pseudosignature has $\lvert V(G) \rvert^2$ entries.

\begin{claim} \label{claim:pseudosignaturenumbernonzero}
Let $i \in [0,\lvert V(T) \rvert-1]$, $\ell \in [-1,\lvert V(T) \rvert+1]$, $k \in [0,s+2]$ and $j \in [\lvert \V \rvert-1]$.
Let $(p_1,p_2,\dots,p_{\lvert V(G) \rvert})$ be the $(i,\ell,k,j)$-pseudosignature.
Then for every $\alpha \in [\lvert V(G) \rvert]$, there exist at most one index $\beta$ such that the $\alpha$-th entry of $p_\beta$ is nonzero.
Furthermore, the $(i,\ell,k,j)$-pseudosignature has at most $w_0$ nonzero entries, and every entry is at most $f(\eta_5)$.
\end{claim}

\begin{proof}
Every monochromatic component in $G[Y^{(i,\ell,k)}]$ is contained in at most one \linebreak $E_{j,t}^{(i)}$-pseudocomponent in $G[Y^{(i,\ell,k)}]$.
So for every $\alpha \in [\lvert V(G) \rvert]$, there exist at most one index $\beta$ such that the $\alpha$-th entry of $p_\beta$ is nonzero.
Hence the number of nonzero entries and the sum of the entries of the $(i,\ell,k,j)$-pseudosignature equal the number of nonzero entries and the sum of the entries of the $(i,\ell,k,j)$-signature, respectively.
Then this claim follows from \cref{claim:signaturenumbernonzero}.
\end{proof}

\subsection{The pseudocomponent with the smallest $\sigma$-value}

The goal of this subsection is to prove \cref{claim:1stdone}, which roughly says that during the algorithm, the pseudocomponent with the smallest $\sigma$-value, among the pseudocomponents with non-empty gates, will not grow in the future.
To prove \cref{claim:1stdone} and other later claims, we need to use another property for fake edges that we have not used in previous subsections.
We first recall the part of the algorithm relevant to this subsection.

\medskip
\noindent\textbf{\boldmath Stage $(0,-1,0)$: Initialization:}
See \cref{subsec:alog} for details.

\medskip
For $i=0,1,2,\dots$, let $t$ be the node of $T$ with $\sigma_T(t)=i$ and perform the following steps:\\[1ex]
\hspace*{4mm} 
\textbf{\boldmath Stage: Building Fences:}
Let $E^{(0)}_{j,t} := \emptyset$ for every $j \in [\lvert \V \rvert-1]$ and $t \in V(T)$.
Other sets are defined.
See \cref{subsec:alog} for details.\\
\hspace*{4mm} 
\noindent\textbf{\boldmath Stage $(i,-1,\star)$:}
See \cref{subsec:alog} for details.\\
\hspace*{4mm}
\noindent\textbf{\boldmath Stage $(i,\geq 1,\star)$:}
$(Y^{(i,\lvert V(T) \rvert+1,s+2)},L^{(i,\lvert V(T) \rvert+1,s+2)})$, $W_3^{(i,\ell)}$ for $\ell \in [0,\lvert V(T) \rvert]$, and other sets are defined.
See \cref{subsec:alog} for details.\\
\hspace*{4mm}
\noindent\textbf{Stage: Adding Fake Edges}
	\begin{itemize}
				\item For each $j \in [\lvert \V \rvert-1]$ and $t' \in V(T)-(V(T_t)-\{t\})$, let $E^{(i+1)}_{j,t'}:=E^{(i)}_{j,t'}$.
				\item For each $j \in [\lvert \V \rvert-1]$, let $E^{(i,0)}_{j,t}:=E^{(i)}_{j,t}$.
				\item For each $j \in [\lvert \V \rvert-1]$ and $t' \in V(T_t)-\{t\}$, let $E^{(i+1)}_{j,t'}:=E^{(i,\lvert V(G) \rvert)}_{j,t}$, where every $\ell \in [0,\lvert V(G) \rvert-1]$, let $E^{(i,\ell+1)}_{j,t}$ be the union of $E^{(i,\ell)}_{j,t}$ and the set consisting of the 2-element sets $\{u,v\}$ satisfying the following.
				\begin{itemize}
					\item $\{u,v\} \subseteq Y^{(i,\lvert V(T) \rvert+1,s+2)}$.
					\item $L^{(i,\lvert V(T) \rvert+1,s+2)}(u) = L^{(i,\lvert V(T) \rvert+1,s+2)}(v)$.
					\item There exists an $s$-segment $S$ in $\Se_{j}^\circ$ whose level equals the color of $u$ and $v$ such that $\{u,v\} \subseteq S$.
					\item There exists $t'' \in \partial T_{j,t}$ such that:
					\begin{itemize}
						\item $\{u,v\} \subseteq X_{t''}$,
						\item $V(M) \cap X_{t''} \neq \emptyset$, where $M$ is the monochromatic $E^{(i,\ell)}_{j,t}$-pseudocomponent in \linebreak $G[Y^{(i,\lvert V(T) \rvert+1,s+2)}]$ such that $\sigma(M)$ is the $(\ell+1)$-th smallest among all monochromatic $E^{(i,\ell)}_{j,t}$-pseudocomponents in  $G[Y^{(i,\lvert V(T) \rvert+1,s+2)}]$,
						\item $M$ is contained in some $s$-segment in $\Se_{j}^\circ$ whose level equals the color of $M$,
						\item $t''$ is a witness for $X_{t''} \cap I_j \subseteq W_3^{(i,\ell')}$ for some $\ell' \in [0,\lvert V(T) \rvert]$,
						\item $A_{L^{(i,\lvert V(T) \rvert+1,s+2)}}(V(M)) \cap X_{V(T_{t''})}-X_{t''} \neq \emptyset$,
						\item $A_{L^{(i,\lvert V(T) \rvert+1,s+2)}}(V(M_u)) \cap X_{V(T_{t''})}-X_{t''} \neq \emptyset$,
						\item $A_{L^{(i,\lvert V(T) \rvert+1,s+2)}}(V(M_v)) \cap X_{V(T_{t''})}-X_{t''} \neq \emptyset$, and
						\item $\sigma(M) < \min\{\sigma(M_u),\sigma(M_v)\}$, 
						
					\end{itemize}
					where $M_u$ and $M_v$ are the monochromatic $E^{(i,\ell)}_{j,t}$-pseudocomponents in $G[Y^{(i,\lvert V(T) \rvert+1,s+2)}]$ containing $u$ and $v$, respectively.
				\item There do not exist $q \in V(T)$ with $T_{j,t} \subseteq T_{j,q}$ and $i_q<i_t$, a witness $q' \in \partial T_{j,q}$ for $X_{q'} \cap I_j \subseteq W_3^{(i_{q},\ell')}$ for some $\ell' \in [0,\lvert V(T) \rvert]$, and a monochromatic $E_{j,q}^{(i_{q})}$-pseudocomponent in $G[Y^{(i_{q},\lvert V(T) \rvert+1,s+2)}]$ intersecting $X_{q'}$ and $\{u,v\}$.
				\item Some other properties.
					See \cref{subsec:alog} for details.
				\end{itemize}					
	\end{itemize}
\hspace*{4mm}
\noindent\textbf{Stage: Moving to the Next Node in the Tree:}
See \cref{subsec:alog} for details.\\
\hspace*{4mm}
\noindent\textbf{Stage: Building a New Fence:}
See \cref{subsec:alog} for details.

\medskip

Now we prove some claims.

For $i \in {\mathbb N}_0$, $j \in [\lvert \V \rvert-1]$, $\ell \in [-1,\lvert V(T) \rvert+1]$ and $k \in [0,s+2]$, we say a monochromatic subgraph $M$ in $G[Y^{(i,\ell,k)}]$ is \mathdef{$\Se_j^\circ$}{related} if $V(M)$ is contained in some $s$-segment in $\Se_j^\circ$ whose level equals the color of $M$.

\begin{claim} \label{claim:orderpreserving}
Let $i_1,i_2 \in {\mathbb N}_0$.
Let $j \in [\lvert \V \rvert-1]$.
Let $\ell_1,\ell_2 \in [-1,\lvert V(T) \rvert+1]$.
Let $k_1,k_2 \in [0,s+2]$.
Let $t_1,t_2$ be nodes of $T$ with $\sigma_T(t_1)=i_1$ and $\sigma_T(t_2)=i_2$.
Assume that $t_2 \in V(T_{t_1})$ and $(i_1,\ell_1,k_1)$ is lexicographically smaller than $(i_2,\ell_2,k_2)$.
Let $M_2$ be an $\Se_j^\circ$-related monochromatic $E^{(i_2)}_{j,t_2}$-pseudocomponent in $G[Y^{(i_2,\ell_2,k_2)}]$ intersecting $X_{t_1}$.
Then there exist a monochromatic $E^{(i_1)}_{j,t_1}$-pseudocomponent $M_1$ in $G[Y^{(i_1,\ell_1,k_1)}]$ such that $M_1 \subseteq M_2$, $\sigma(M_1)=\sigma(M_2)$ and $V(M_1) \cap X_{t_1} \neq \emptyset$.
Furthermore, if $A_{L^{(i_2,\ell_2,k_2)}}(V(M_2)) \cap X_{V(T_{t_2})}-X_{t_2} \neq \emptyset$, then $A_{L^{(i_1,\ell_1,k_1)}}(V(M_1)) \cap X_{V(T_{t_1})}-X_{t_1} \neq \emptyset$.
\end{claim}

\begin{proof}
Let $v_M$ be the vertex of $M_2$ such that $\sigma(v_M)=\sigma(M_2)$.
Since $V(M_2) \cap X_{t_1} \neq \emptyset$, $t_1 \in V(T_{r_{v_M}})$.
So there exists a monochromatic component $M'$ in $G[Y^{(i_1,\ell_1,k_1)}]$ containing $v_M$.
Hence there exists a monochromatic $E^{(i_1)}_{j,t_1}$-pseudocomponent $M_1$ in $G[Y^{(i_1,\ell_1,k_1)}]$ containing $M'$.
So $\sigma(M_1) \leq \sigma(M') \leq \sigma(v_M)=\sigma(M_2)$.
Since $t_2 \in V(T_{t_1})$, $E_{j,t_1}^{(i_1)} \subseteq E_{j,t_2}^{(i_1)} \subseteq E_{j,t_2}^{(i_2)}$.
Since $(i_1,\ell_1,k_1)$ is lexicographically smaller than $(i_2,\ell_2,k_2)$, $M_1 \subseteq M_2$.
Hence $\sigma(M_1)=\sigma(M_2)$.

Suppose $V(M_1) \cap X_{t_1} = \emptyset$.
Since $t_1 \in V(T_{r_{v_M}})$ and $v_M \in V(M_1)$, $V(M_1) \cap X_{V(T_{t_1})}=\emptyset$.
Since $V(M_2) \cap X_{t_1} \neq \emptyset$, there exists a path $P$ in $M_2$ from $X_{t_1}$ to $V(M_1)$ internally disjoint from $V(M_1) \cup X_{t_1}$.
Let $u$ be the vertex in $V(P) \cap V(M_1)$.
Since $V(P) \subseteq Y^{(i_{t_2},\ell_2,k_2)}$ and $t_2 \in V(T_{t_1})$ and $P$ is internally disjoint from $X_{V(T_{t_1})}$, the neighbor of $u$ in $P$ is in $Y^{(i_1,-1,0)} \subseteq Y^{(i_1,\ell_1,k_1)}$.
Since $M_1$ is a monochromatic $E^{(i_1)}_{j,t_1}$-pseudocomponent, the edge of $P$ incident with $u$ belongs to $E_{j,t_2}^{(i_2)}-E_{j,t_1}^{(i_1)}$.
But $u \in V(M_1)$, so $u \not \in X_{V(T_{t_1})}$.
Hence there exists no element in $E_{j,t_2}^{(i_2)}-E_{j,t_1}^{(i_1)}$ containing $u$, a contradiction.
Therefore $V(M_1) \cap X_{t_1} \neq \emptyset$.

Now we assume that $A_{L^{(i_2,\ell_2,k_2)}}(V(M_2)) \cap X_{V(T_{t_2})}-X_{t_2} \neq \emptyset$.
Suppose that $A_{L^{(i_1,\ell_1,k_1)}}(V(M_1)) \cap X_{V(T_{t_1})}-X_{t_1} = \emptyset$.
Since $(i_1,\ell_1,k_1)$ is lexicographically smaller than $(i_2,\ell_2,k_2)$ and $X_{V(T_{t_2})}-X_{t_2} \subseteq X_{V(T_{t_1})}-X_{t_1}$, $V(M_2) \supset V(M_1)$.
Since $A_{L^{(i_1,\ell_1,k_1)}}(V(M_1)) \cap X_{V(T_{t_1})}-X_{t_1} = \emptyset$, there exists $e \in E_{j,t_2}^{(i_2)}-E_{j,t_1}^{(i_1)}$ such that $e \cap V(M_1) \neq \emptyset \neq e-V(M_1)$.
Since $t_2 \in V(T_{t_1})$, $E_{j,t_2}^{(i_1)}=E_{j,t_1}^{(i_1)}$, so $E_{j,t_2}^{(i_2)}-E_{j,t_1}^{(i_1)} \subseteq E_{j,t_2}^{(i_2)}-E_{j,t_2}^{(i_1)}$.
So there exists $i_3 \in [i_1,i_2-1]$ and $t_3 \in V(T)$ with $\sigma_T(t_3)=i_3$ and $t_2 \in V(T_{t_3})$ such that $e \in E_{j,t_2}^{(i_3+1)}-E_{j,t_2}^{(i_3)}$.
We may assume $i_3$ is as small as possible.
Hence there exist $t_4 \in \partial T_{j,t_3}$ and $\alpha \in [0,\lvert V(G) \rvert]$ such that $t_4$ is the witness for $e \in E_{j,t_2}^{(i_3,\alpha+1)}-E_{j,t_2}^{(i_3,\alpha)}$. 

Since $A_{L^{(i_1,\ell_1,k_1)}}(V(M_1)) \cap X_{V(T_{t_1})}-X_{t_1} = \emptyset$, by the minimality of $i_3$, $M_1$ is a monochromatic $E_{j,t_2}^{(i_{t_3},\alpha)}$-pseudocomponent in $G[Y^{(i_{t_3},\lvert V(T) \rvert+1,s+2)}]$ and $A_{L^{(i_3,\lvert V(T) \rvert+1,s+2)}}(V(M_1)) \cap X_{V(T_{t_4})}-X_{t_4} \neq \emptyset$. 
Since $i_3 \geq i_1$, $t_4 \in V(T_{t_1})$.
So $A_{L^{(i_1,\ell_1,k_1)}}(V(M_1)) \cap X_{V(T_{t_1})}-X_{t_1} \supseteq A_{L^{(i_3,\lvert V(T) \rvert+1,s+2)}}(V(M_1)) \cap X_{V(T_{t_4})}-X_{t_4} \neq \emptyset$, a contradiction.
This proves the claim.
\end{proof}

\begin{claim} \label{claim:zerosignature}
Let $t_1,t_2,t_3$ be nodes of $T$ such that $T_{t_3} \subseteq T_{t_2} \subset T_{t_1}$.
For $\alpha \in [3]$, let $(a^{(\alpha)}_1,a^{(\alpha)}_2,\dots,a^{(\alpha)}_{\lvert V(G) \rvert})$ be the $(i_{t_\alpha},-1,0,j)$-pseudosignature.
Let $\alpha^*$ be the largest index such that $a^{(1)}_{\alpha^*}$ is a nonzero sequence.
Then the following hold.
	\begin{itemize}
		\item For every $\beta \in [\alpha^*-1]$, if $a^{(1)}_\beta$ is a zero sequence, then $a^{(2)}_\beta$ is a zero sequence.
		\item For every $\beta \in [\alpha^*]$, if $a^{(2)}_{\beta}$ is a zero sequence, then $a^{(3)}_{\beta}$ is a zero sequence.
	\end{itemize}
\end{claim}

\begin{proof}
We first prove the first statement.
Suppose there exists $\beta \in [\alpha^*-1]$ such that $a^{(2)}_\beta$ is a nonzero sequence.
Let $M_2$ be the monochromatic $E_{j,t_2}^{(i_{t_2})}$-pseudocomponent in $G[Y^{(i_{t_2},-1,0)}]$ defining $a^{(2)}_\beta$.
Since $a^{(1)}_{\alpha^*}$ is a nonzero sequence, there exists a monochromatic $E_{j,t_1}^{(i_{t_1})}$-pseudocomponent $M_1'$ in $G[Y^{(i_{t_1},-1,0)}]$ defining $a^{(1)}_{\alpha^*}$.
Since $\sigma(M_2) = \beta < \alpha^* =\sigma(M_1')$ and $V(M_1') \cap X_{t_1} \neq \emptyset \neq V(M_2) \cap X_{t_2}$, $V(M_2) \cap X_{t_1} \neq \emptyset$.
By \cref{claim:orderpreserving}, there exists a monochromatic $E_{j,t_1}^{(i_{t_1},-1,0)}$-pseudocomponent $M_1$ such that $M_1 \subseteq M_2$, $\sigma(M_1)=\sigma(M_2)=\beta$, $V(M_1) \cap X_{t_1} \neq \emptyset$, and $A_{L^{(i_{t_1},-1,0)}}(V(M_1)) \cap X_{V(T_{t_1})}-X_{t_1} \neq \emptyset$.
So $a^{(1)}_\beta$ is a nonzero sequence, a contradiction.

Now we prove the second statement.
Suppose to the contrary that $a^{(2)}_{\beta}$ is a zero sequence and $a^{(3)}_{\beta}$ is a nonzero sequence.
If $\beta=\alpha^*$, then $a^{(1)}_\beta=a^{(1)}_{\alpha^*}$ is a nonzero sequence; if $\beta \in [\alpha^*-1]$, then $a^{(1)}_\beta$ is a nonzero sequence by the first statement of this claim.
Since $a^{(3)}_{\beta}$ is a nonzero sequence, there exists a monochromatic $E_{j,t_3}^{(i_{t_3})}$-pseudocomponent $Q_3$ in $G[Y^{(i_{t_3},-1,0)}]$ defining $a^{(3)}_{\beta}$.
Since $a^{(1)}_{\beta}$ is a nonzero sequence, there exists a monochromatic $E_{j,t_1}^{(i_{t_1})}$-pseudocomponent $Q_1$ in $G[Y^{(i_{t_1},-1,0)}]$ defining $a^{(1)}_{\beta}$.
Hence $\sigma(Q_1)=\beta=\sigma(Q_3)$.
So $Q_1 \subseteq Q_3$.
Hence $V(Q_3) \cap X_{t_1} \neq \emptyset \neq V(Q_3) \cap X_{t_3}$.
So $V(Q_3) \cap X_{t_2} \neq \emptyset$.
By \cref{claim:orderpreserving}, there exists a monochromatic $E_{j,t_2}^{(i_{t_2},-1,0)}$-pseudocomponent $Q_2$ such that $Q_2 \subseteq Q_3$, $\sigma(Q_2)=\sigma(Q_3)=\beta$, $V(Q_2) \cap X_{t_2} \neq \emptyset$ and $A_{L^{(i_{t_2},-1,0)}}(V(Q_2)) \cap X_{V(T_{t_2})}-X_{t_2} \neq \emptyset$.
So $a^{(2)}_{\beta}$ is a nonzero sequence, a contradiction.
\end{proof}

\begin{claim} \label{claim:pseudocomptouchingW3}
Let $j \in [\lvert \V \rvert-1]$.
Let $t \in V(T)$.
Let $M$ be an $\Se_j^\circ$-related monochromatic $E_{j,t}^{(i_t)}$-pseudocomponent in $G[Y^{(i_t,-1,0)}]$ intersecting $X_t$.
Let $t^* \in V(T)$ be the node with $t \in V(T_{j,t^*})$ such that some monochromatic $E_{j,t^*}^{(i_{t^*})}$-pseudocomponent in $G[Y^{(i_{t^*},\lvert V(T) \rvert+1,s+2)}]$ intersecting $V(M)$ intersects $X_{t'}$ for some witness $t' \in (\partial T_{j,t^*}) \cup \{t^*\}$ for $X_{t'} \cap I_j \subseteq W_3^{(i_{t^*},\ell)}$ for some $\ell \in [-1,\lvert V(T) \rvert]$, and subject to this, $i_{t^*}$ is minimum.
	Then for every monochromatic $E_{j,t^*}^{(i_{t^*})}$-pseudocomponent $M_1$ in $G[Y^{(i_{t^*},\lvert V(T) \rvert+1,s+2)}]$ intersecting $V(M)$,
	\begin{itemize}
		\item $M_1$ intersects $X_{t''}$ for some witness $t'' \in (\partial T_{j,t^*}) \cup \{t^*\}$ for $X_{t''} \cap I_j \subseteq W_3^{(i_{t^*},\ell)}$ for some $\ell \in [-1,\lvert V(T) \rvert]$, and 
		\item for every node $q \in (\partial T_{j,t^*}) \cup \{t^*\}$ with $V(M_1) \cap X_q \neq \emptyset$, $q$ is a witness for $X_q \cap I_j \subseteq W_3^{(i_{t^*},\ell)}$ for some $\ell \in [-1,\lvert V(T) \rvert]$.
	\end{itemize}
\end{claim}

\begin{proof}
Note that $t^*$ exists since $t$ is a candidate.
So $i_{t^*} \leq i_t$.

If there exist distinct nodes $q_1,q_2 \in (\partial T_{j,t^*}) \cup \{t^*\}$ where $q_1$ is a witness for $X_{q_1} \cap I_j \subseteq W_3^{(i_{t^*},\ell_1)}$ for some $\ell_1 \in [-1,\lvert V(T) \rvert]$ such that there exists a path $P$ in $M[V(M) \cap X_{V(T_{j,t^*})}]$ from $X_{q_1}$ to $X_{q_2}$ internally disjoint from $X_{\partial T_{j,t^*}} \cup X_{t^*}$,
then by the minimality of $i_{t^*}$ and the construction of $E_{j,t}^{(i_{t})}$, $E(P) \cap E_{j,t}^{(i_t)} = \emptyset$, so $P$ is a path in $G$, and hence by \cref{claim:isolation}, $V(P) \subseteq Y^{(i_{t^*},\lvert V(T) \rvert+1,s+2)}$.
Hence, since $\lvert \partial T_{j,t^*} \rvert \leq \lvert V(T) \rvert-1$ and some monochromatic $E_{j,t^*}^{(i_{t^*})}$-pseudocomponent in $G[Y^{(i_{t^*},\lvert V(T) \rvert+1,s+2)}]$ intersecting $V(M)$ intersects $X_{t''}$ for some witness $t'' \in (\partial T_{j,t^*}) \cup \{t^*\}$ for $X_{t''} \cap I_j \subseteq W_3^{(i_{t^*},\ell)}$ for some $\ell \in [-1,\lvert V(T) \rvert]$, to prove this claim, it suffices to prove that if $M_1$ is a component of $M[V(M) \cap X_{V(T_{j,t^*})}]-\{e \in E(M): e \subseteq X_{t'}$ for some $t' \in \partial T_{j,t^*} \cup \{t^*\}\}$ such that $V(M_1) \cap X_{t''} \neq \emptyset$ for some witness $t'' \in (\partial T_{j,t^*}) \cup \{t^*\}$ for $X_{t''} \cap I_j \subseteq W_3^{(i_{t^*},\ell)}$ for some $\ell \in [-1,\lvert V(T) \rvert]$, then for every node $q \in (\partial T_{j,t^*}) \cup \{t^*\}$ with $V(M_1) \cap X_q \neq \emptyset$, $q$ is a witness for $X_q \cap I_j \subseteq W_3^{(i_{t^*},\ell)}$ for some $\ell \in [-1,\lvert V(T) \rvert]$.

Let $M_1$ be a component of $M[V(M) \cap X_{V(T_{j,t^*})}]-\{e \in E(M): e \subseteq X_{t'}$ for some $t' \in \partial T_{j,t^*} \cup \{t^*\}\}$ such that $V(M_1) \cap X_{t''} \neq \emptyset$ for some witness $t'' \in (\partial T_{j,t^*}) \cup \{t^*\}$ for $X_{t''} \cap I_j \subseteq W_3^{(i_{t^*},\ell)}$ for some $\ell \in [-1,\lvert V(T) \rvert]$.
Since $M$ is $\Se_j^\circ$-related, $V(M_1) \subseteq W_4^{(i_{t^*})}$.
By the minimality of $i_{t^*}$, $E(M_1) \cap E_{j,t}^{(i_t)} = \emptyset$.

Suppose to the contrary that there exists $q \in (\partial T_{j,t^*}) \cup \{t^*\}$ such that $V(M_1) \cap X_q \neq \emptyset$, and $q$ is not a witness for $X_q \cap I_j \subseteq W_3^{(i_{t^*},\alpha)}$ for any $\alpha \in [-1,\lvert V(T) \rvert]$.
So $\ell \leq \lvert \partial T_{j,t^*} \rvert-1 \leq \lvert V(T) \rvert-1$.
We may choose $q$ and $t''$ such that there exists $v \in V(M_1) \cap X_q$ such that there exists a monochromatic path $P^*$ in $M_1$ from $X_{t''}$ to $X_q$ internally disjoint from $X_{t^*} \cup X_{\partial T_{j,t^*}}$. 
Recall that $P^*$ is a monochromatic path in $G[Y^{(i_{t^*},\lvert V(T) \rvert+1,s+2)}]$ by the minimality of $i_{t^*}$ and \cref{claim:isolation}.

Let $k+1$ be the color of $M_1$.
By \cref{claim:W_3isolate}, $V(P^*) \subseteq Y^{(i_{t^*},\ell+1,k)}$. 
Since $q$ is not a witness for $X_q \cap I_j \subseteq W_3^{(i_{t^*},\ell+1)}$, either $v \in D^{(i_{t^*},\ell+1,0)}-X_{t^*}$, or $v \in D^{(i_{t^*},\ell+1,0)} \cap X_{t^*} \cap X_{q'} \cap I_j$ for some witness $q' \in V(T_{t^*})-\{t^*\}$ for $X_{q'} \cap I_j \subseteq W_3^{(i',\ell')}$ for some $i' \in [0,i_{t^*}-1]$ and $\ell' \in [0,\lvert V(T) \rvert]$.
The latter cannot hold by the minimality of $i_{t^*}$.

Hence $v \in D^{(i_{t^*},\ell+1,0)}-X_{t^*}$.
By the minimality of $i_{t^*}$ and \cref{claim:isolation}, $v \not \in D^{(i_{t^*},0,0)}$.
So there exists $\ell^* \in [0,\lvert V(T) \rvert-1]$ such that $v \in D^{(i_{t^*},\ell^*+1,0)}-D^{(i_{t^*},\ell^*,0)}$.
So $v \in W_0^{(i_{t^*},\ell^*+1,k)}$.
Hence there exists a path in $G[Y^{(i_{t^*},\ell^*+1,k)} \cap W_4^{(i_{t^*})}]$ with color $k+1$ from $v$ to $W_3^{(i_{t^*},\ell^*)}$.
So $q$ is a witness for $X_q \cap I_j \subseteq W_3^{(i_{t^*},\ell')}$ for some $\ell' \in [-1,\ell^*+1]$, a contradiction.
This proves the claim.
\end{proof}

\begin{claim} \label{claim:1stdone}
Let $i \in {\mathbb N}_0$ and $t \in V(T)$ with $\sigma_T(t)=i$.
Let $j \in [\lvert \V \rvert-1]$.
Let $(p_1,p_2,\dots,p_{\lvert V(G) \rvert})$ be the $(i,-1,0,j)$-pseudosignature.
Let $\alpha \in [\lvert V(G) \rvert]$ be the smallest index such that $p_\alpha$ is not a zero sequence.
Let $M$ be the monochromatic $E_{j,t}^{(i)}$-pseudocomponent defining $p_\alpha$.
Let $i' \in {\mathbb N}_0$ and $t' \in V(T_t)-\{t\}$ with $\sigma_T(t')=i'$.
Let $M'$ be the monochromatic $E_{j,t'}^{(i')}$-pseudocomponent in $G[Y^{(i',-1,0)}]$ containing $M$.
Then $V(M') = V(M) \subseteq Y^{(i,-1,0)}$.
\end{claim}

\begin{proof}
Since $V(M) \subseteq Y^{(i,-1,0)}$, it suffices to show that $V(M')=V(M)$.
Suppose to the contrary that $V(M') \neq V(M)$. 
We choose $i,i'$ such that $i'-i$ is minimum among all counterexamples.

Since $V(M') \neq V(M)$, there exists the minimum $i'' \in [i,i']$ such that $V(M'') \neq V(M)$, where $M''$ is the monochromatic $E_{j,t''}^{(i'')}$-pseudocomponent in $G[Y^{(i'',-1,0)}]$ containing $M$ and $t''$ is the node with $\sigma_T(t'')=i''$.
Note that $i''>i$.
Let $p$ be the parent of $t''$.
So $p \in V(T_t)$.

For every node $z$ with $z \in V(T_t)$, let $M_z$ be the monochromatic $E_{j,z}^{(i_z)}$-pseudocomponent in $G[Y^{(i_z,-1,0)}]$ containing $M$.

Suppose there exist $z \in V(T)$ with $T_p \subseteq T_z \subseteq T_t$ and an $\Se_j^\circ$-related monochromatic $E_{j,z}^{(i_z)}$-pseudocomponent $M_z'$ in $G[Y^{(i_z,-1,0)}]$ with $A_{L^{(i_z,-1,0)}}(V(M_z')) \cap X_{V(T_z)}-X_z \neq \emptyset$ intersecting $X_z$ such that $\sigma(M_z')<\sigma(M_z)$.
By the minimality of $i''$, $V(M_z)=V(M)$.
Since $M_z$ contains $M$, $V(M_z) \cap X_t \neq \emptyset$.
Since $\sigma(M_z')<\sigma(M_z)$ and $V(M_z') \cap X_z \neq \emptyset$, $V(M_z') \cap X_t \neq \emptyset$.
By \cref{claim:orderpreserving}, there exists $\alpha' \in [\alpha-1]$ such that $p_{\alpha'}$ is not a zero sequence, a contradiction.

So for every $z \in V(T)$ with $T_p \subseteq T_z \subseteq T_t$, there exists no $\Se_j^\circ$-related monochromatic $E_{j,z}^{(i_z)}$-pseudocomponent $M_z'$ in $G[Y^{(i_z,-1,0)}]$ with $A_{L^{(i_z,-1,0)}}(V(M_z')) \cap X_{V(T_z)}-X_z \neq \emptyset$ intersecting $X_z$ such that $\sigma(M_z')<\sigma(M_z)$.

Hence $V(Q)=V(M_p)=V(M)$, where $Q$ is the monochromatic $E_{j,p}^{(i_p)}$-pseudocomponent in $G[Y^{(i_p,\lvert V(T) \rvert+1,s+2)}]$ containing $M_p$.
Since $V(M'') \neq V(M)$ and $t'' \in V(T_{j,p})$, by the minimality of $i''$ and \cref{claim:isolation}, there exists $e \in E(M'') \cap E_{j,t''}^{(i'')}-E_{j,p}^{(i_p)}$ such that $e \cap V(Q) \neq \emptyset \neq e-V(Q)$.

Since $E_{j,t''}^{(i_p)}=E_{j,p}^{(i_p)}$, $V(Q)$ intersects an element in $E_{j,t''}^{(i'')}-E_{j,t''}^{(i_p)}$. 
Then since $t''$ is a child of $p$, there exist $q \in \partial T_{j,p}$ and an $\Se_j^\circ$-related monochromatic $E_{j,t''}^{(i_p)}$-pseudocomponent $Q'$ in $G[Y^{(i_p,\lvert V(T) \rvert+1,s+2)}]$ with $A_{L^{(i_p,\lvert V(T) \rvert+1,s+2)}}(V(Q')) \cap X_{V(T_p)}-X_p \neq \emptyset$ intersecting $X_q$ such that $\sigma(Q')<\sigma(Q)$.
Since $V(Q)=V(M)$, $\sigma(Q')<\sigma(Q)=\sigma(M)$.
Since $V(Q) \cap X_t \neq \emptyset$ and $V(Q') \cap X_q \neq \emptyset$, $V(Q') \cap X_t \neq \emptyset$.
Let $Q''$ be the $E_{j,t''}^{(i'')}$-pseudocomponent in $G[Y^{(i'',-1,0)}]$ containing $Q'$.
Then $\sigma(Q'') \leq \sigma(Q')<\sigma(M)$.
By \cref{claim:orderpreserving}, there exists $\alpha' \in [\alpha-1]$ such that $p_{\alpha'}$ is not a zero sequence and is defined by a subgraph of $Q''$, a contradiction.
This proves the claim.
\end{proof}

\subsection{Nice pairs and outer-safe pairs}

This subsection introduces the notions of nice pairs and outer-safe pairs. Roughly speaking, they describe the property that the fake edges correctly predict the existence of monochromatic paths constructed later in the algorithm.

For any $j \in [\lvert \V \rvert-1]$, $t \in V(T)$ and $\Se_j^\circ$-related monochromatic $E_{j,t}^{(i_t)}$-pseudocomponent $M$ in $G[Y^{(i_t,-1,0)}]$ intersecting $X_t$, we say that $(t,M)$ is a \defn{nice pair} if the following hold:
	\begin{itemize}
		\item If $t^* \in V(T)$ is the node of $T$ such that $t \in V(T_{j,t^*})$ and some $E_{j,t^*}^{(i_{t^*})}$-pseudocomponent in $G[Y^{(i_{t^*},\lvert V(T) \rvert+1,s+2)}]$ intersects $V(M)$ and $X_q$ for some witness $q \in \partial T_{j,t^*} \cup \{t^*\}$ for $X_q \cap I_j \subseteq W_3^{(i_{t^*},\ell)}$ for some $\ell \in [-1,\lvert V(T) \rvert]$, and subject to this, $i_{t^*}$ is minimum, then either:
			\begin{itemize}
				\item $t = t^*$, or 
				\item $t \neq t^*$, and for every $c$-monochromatic path $P$ in $G$ intersecting $V(M)$ with $V(P) \subseteq X_{V(T_{t^*})}$ internally disjoint from $X_{V(T_t)}$ and for every vertex $u_P \in V(P) \cap X_{V(T_{j,t^*})}$, there exists a monochromatic $E_{j,t^*}^{(i_{t^*},\lvert V(G) \rvert)}$-pseudocomponent $M'$ in $G[Y^{(i_{t^*},\lvert V(T) \rvert+1,s+2)}]$ such that:
					\begin{itemize}
						\item $V(M') \cap V(M) \neq \emptyset \neq V(M') \cap X_{V(T_t)}$, 
						\item $M'$ contains $u_P$, and 
						\item if $A_{L^{(i_t,-1,0)}}(V(M)) \cap X_{V(T_t)}-X_t \neq \emptyset$, then $A_{L^{(i_{t^*},\lvert V(T) \rvert+1,s+2)}}(V(M')) \cap X_{V(T_t)}-X_t \neq \emptyset$.
					\end{itemize}
			\end{itemize}
	\end{itemize}

\begin{claim} \label{claim:underassumptouch}
Let $j \in [\lvert \V \rvert-1]$.
Let $t \in V(T)$ and let $M$ be an $\Se_j^\circ$-related monochromatic $E_{j,t}^{(i_t)}$-pseudocomponent in $G[Y^{(i_t,-1,0)}]$ intersecting $X_t$.
Let $t^* \in V(T)$ be the node of $T$ such that $t \in V(T_{j,t^*})$ and some $E_{j,t^*}^{(i_{t^*})}$-pseudocomponent in $G[Y^{(i_{t^*},\lvert V(T) \rvert+1,s+2)}]$ intersects $V(M)$ and $X_q$ for some witness $q \in \partial T_{j,t^*} \cup \{t^*\}$ for $X_q \cap I_j \subseteq W_3^{(i_{t^*},\ell)}$ for some $\ell \in [-1,\lvert V(T) \rvert]$, and subject to this, $i_{t^*}$ is minimum.
If $t^* \neq t$, $A_{L^{(i_t,-1,0)}}(V(M)) \cap X_{V(T_t)}-X_t \neq \emptyset$ and $(t,M)$ is a nice pair, then every monochromatic $E_{j,t^*}^{(i_{t^*},\lvert V(G) \rvert)}$-pseudocomponent $M'$ in $G[Y^{(i_{t^*},\lvert V(T) \rvert+1,s+2)}]$ intersecting $V(M)$ intersects $X_{V(T_t)}$ and satisfies $A_{L^{(i_{t^*},\lvert V(T) \rvert+1,s+2)}}(V(M')) \cap X_{V(T_t)}-X_t \neq \emptyset$.
\end{claim}

\begin{proof}
Let $S_1 = \{Q: Q$ is a monochromatic $E_{j,t^*}^{(i_{t^*},\lvert V(G) \rvert)}$-pseudocomponent in $G[Y^{(i_{t^*},\lvert V(T) \rvert+1,s+2)}]$ such that $V(Q)$ intersects $V(M)$ and $X_{V(T_t)}$, and $A_{L^{(i_{t^*},\lvert V(T) \rvert+1,s+2)}}(V(Q)) \cap X_{V(T_t)}-X_t \neq \emptyset\}$.
Let $S_2 = \{Q \not \in S_1: Q$ is a monochromatic $E_{j,t^*}^{(i_{t^*},\lvert V(G) \rvert)}$-pseudocomponent in $G[Y^{(i_{t^*},\lvert V(T) \rvert+1,s+2)}]$ intersecting $V(M)\}$.
By \cref{claim:isolation,claim:pseudocomptouchingW3}, since $(t,M)$ is a nice pair, $S_1 \neq \emptyset$. 

Suppose that the claim does not hold.
That is, $S_2 \neq \emptyset$.
By the construction of $E_{j,t}^{(i_t)}$ and \cref{claim:isolation,claim:pseudocomptouchingW3}, there exist monochromatic $E_{j,t^*}^{(i_{t^*},\lvert V(G) \rvert)}$-pseudocomponents $M_1$ and $M_2$ in $G[Y^{(i_{t^*},\lvert V(T) \rvert+1,s+2)}]$ intersecting $V(M)$ such that $M_1 \in S_1$, $M_2 \in S_2$ and there exist $q \in \partial T_{j,t^*}$ with $V(M_1) \cap X_q \neq \emptyset \neq V(M_2) \cap X_q$ and a path $P$ in $M[V(M) \cap X_{V(T_q)}]$ from $V(M_1) \cap X_q$ to $V(M_2) \cap X_q$ internally disjoint from $X_{q} \cup V(M_1) \cup M_2$.
By the minimality of $i_{t^*}$, $P$ is a $c$-monochromatic path in $G$.
So $V(P) \cap A_{L^{(i_{t^*},\lvert V(T) \rvert+1,s+2)}}(V(M_1)) \cap X_{V(T_q)}-X_q \neq \emptyset \neq V(P) \cap A_{L^{(i_{t^*},\lvert V(T) \rvert+1,s+2)}}(V(M_2)) \cap X_{V(T_q)}-X_q$.
Since either $V(M_2) \cap X_{V(T_t)} = \emptyset$ or $A_{L^{(i_{t^*},\lvert V(T) \rvert+1,s+2)}}(V(M_2)) \cap X_{V(T_t)}-X_t = \emptyset$, $q \not \in V(T_t)$ and $P$ is internally disjoint from $X_{V(T_t)}$.
Let $u$ be the vertex in $V(P) \cap V(M_2) \cap X_q$.
Since $(t,M)$ is nice, there exists $M_3 \in S_1$ such that $M_3$ contains $u$. 
Since $u \in V(M_2) \cap V(M_3)$, $M_3=M_2$.
Hence $M_2=M_3 \in S_1$, a contradiction. 
\end{proof}

\begin{claim} \label{claim:underassumptouch2}
Let $j \in [\lvert \V \rvert-1]$.
Let $t \in V(T)$ and let $M$ be an $\Se_j^\circ$-related monochromatic $E_{j,t}^{(i_t)}$-pseudocomponent in $G[Y^{(i_t,-1,0)}]$ intersecting $X_t$.
Let $t^* \in V(T)$ be the node of $T$ such that $t \in V(T_{j,t^*})$ and some $E_{j,t^*}^{(i_{t^*})}$-pseudocomponent in $G[Y^{(i_{t^*},\lvert V(T) \rvert+1,s+2)}]$ intersects $V(M)$ and $X_q$ for some witness $q \in \partial T_{j,t^*} \cup \{t^*\}$ for $X_q \cap I_j \subseteq W_3^{(i_{t^*},\ell)}$ for some $\ell \in [-1,\lvert V(T) \rvert]$, and subject to this, $i_{t^*}$ is minimum.
If $t \neq t^*$, $A_{L^{(i_t,-1,0)}}(V(M)) \cap X_{V(T_t)}-X_t \neq \emptyset$ and $(t,M)$ is a nice pair, then there exists a monochromatic $E_{j,t^*}^{(i_{t^*},\lvert V(G) \rvert)}$-pseudocomponent $M^*$ in $G[Y^{(i_{t^*},\lvert V(T) \rvert+1,s+2)}]$ intersecting $V(M)$ such that $\sigma(M^*)=\sigma(M)$, $V(M^*) \cap X_{t} \neq \emptyset$ and satisfies $A_{L^{(i_{t^*},\lvert V(T) \rvert+1,s+2)}}(V(M^*)) \cap X_{V(T_t)}-X_t \neq \emptyset$.
\end{claim}

\begin{proof}
Let $v_M$ be the vertex of $M$ such that $\sigma(v_M)=\sigma(M)$.
If $v_M \in V(T_{t^*})-\{t^*\}$, then by \cref{claim:isolation,claim:pseudocomptouchingW3}, there exists a monochromatic $E_{j,t^*}^{(i_{t^*},\lvert V(G) \rvert)}$-pseudocomponent $M^*$ in $G[Y^{(i_{t^*},\lvert V(T) \rvert+1,s+2)}]$ containing $v_M$.
If $v_M \not \in V(T_{t^*})-\{t^*\}$, then by \cref{claim:orderpreserving}, there exists a monochromatic $E_{j,t^*}^{(i_{t^*},\lvert V(G) \rvert)}$-pseudocomponent $M^*$ in $G[Y^{(i_{t^*},\lvert V(T) \rvert+1,s+2)}]$ containing $v_M$.
Since $t \neq t^*$, $M^* \subseteq M$.
Hence $\sigma(M)=\sigma(M^*)$ and $V(M^*) \cap V(M) \neq \emptyset$.
So by \cref{claim:underassumptouch}, $V(M^*) \cap X_{V(T_t)} \neq \emptyset$ and $A_{L^{(i_{t^*},\lvert V(T) \rvert+1,s+2)}}(V(M^*)) \cap X_{V(T_t)}-X_t \neq \emptyset$.
Since $V(M) \cap X_t \neq \emptyset$, $t \in V(T_{r_M})$.
Since $v_M \in V(M^*)$, $V(M^*) \cap X_{r_M} \neq \emptyset$.
Since $V(M^*) \cap X_{V(T_t)} \neq \emptyset$, $V(M^*) \cap X_t \neq \emptyset$.
\end{proof}

For any $j \in [\lvert \V \rvert-1]$, $z \in V(T)$ and $\Se_j^\circ$-related monochromatic $E_{j,z}^{(i_z)}$-pseudocomponent $M$ in $G[Y^{(i_z,-1,0)}]$ intersecting $X_z$, we say that $(z,M)$ is an \defn{outer-safe pair} if the following hold:
	\begin{itemize}
		\item If $z^* \in V(T)$ is the node of $T$ such that $z \in V(T_{j,z^*})$ and some $E_{j,z^*}^{(i_{z^*})}$-pseudocomponent in $G[Y^{(i_{z^*},\lvert V(T) \rvert+1,s+2)}]$ intersects $V(M)$ and $X_q$ for some witness $q \in \partial T_{j,z^*} \cup \{z^*\}$ for $X_q \cap I_j \subseteq W_3^{(i_{z^*},\ell)}$ for some $\ell \in [-1,\lvert V(T) \rvert]$, and subject to this, $i_{z^*}$ is minimum, then there exist no $t' \in \partial T_{j,z^*} - V(T_z)$, $u,v \in X_{t'}$, a monochromatic path in $G[Y^{(i_z,-1,0)} \cap X_{V(T_{t'})}]$ but not in $G[Y^{(i_{z^*},\lvert V(T) \rvert+1,s+2)}]$ from $u$ to $v$ internally disjoint from $X_{t'}$ such that there exists a monochromatic $E_{j,z^*}^{(i_{z^*},\lvert V(G) \rvert)}$-pseudocomponent $M^*$ in $G[Y^{(i_{z^*},\lvert V(T) \rvert+1,s+2)}]$ with $\sigma(M^*)=\sigma(M)$ containing $u$ but not $v$.
	\end{itemize}
(Note that we do not require $A_{L^{(i_z,-1,0)}}(V(M)) \cap X_{V(T_z)}-X_z \neq \emptyset$.)

\begin{claim} \label{claim:underassumpnounder}
Let $j \in [\lvert \V \rvert-1]$.
Let $t \in V(T)$ and let $M$ be an $\Se_j^\circ$-related monochromatic $E_{j,t}^{(i_t)}$-pseudocomponent in $G[Y^{(i_t,-1,0)}]$ intersecting $X_t$. 
Let $t^* \in V(T)$ be the node of $T$ such that $t \in V(T_{j,t^*})$ and some $E_{j,t^*}^{(i_{t^*})}$-pseudocomponent in $G[Y^{(i_{t^*},\lvert V(T) \rvert+1,s+2)}]$ intersects $V(M)$ and $X_q$ for some witness $q \in \partial T_{j,t^*} \cup \{t^*\}$ for $X_q \cap I_j \subseteq W_3^{(i_{t^*},\ell)}$ for some $\ell \in [-1,\lvert V(T) \rvert]$, and subject to this, $i_{t^*}$ is minimum.
Assume that $t \neq t^*$.
Assume that $(z,M')$ is a nice pair and an outer-safe pair for every $z \in V(T)$ and every $\Se_j^\circ$-related monochromatic $E_{j,z}^{(i_z)}$-pseudocomponent $M'$ in $G[Y^{(i_z,-1,0)}]$ intersecting $X_z$ with $\sigma(M')<\sigma(M)$. 

Then there exist no $t' \in \partial T_{j,t^*} \cap V(T_t)$, $u,v \in X_{t'}$ and a monochromatic path in $G[Y^{(i_t,-1,0)} \cap X_{V(T_{t'})}]$ but not in $G[Y^{(i_{t^*},\lvert V(T) \rvert+1,s+2)}]$ from $u$ to $v$ internally disjoint from $X_{t'}$ such that $t'$ is a witness for $X_{t'} \cap I_j \subseteq W_3^{(i_{t^*},\ell)}$ for some $\ell \in [0,\lvert V(T) \rvert]$, and there exists a monochromatic $E_{j,t^*}^{(i_{t^*},\lvert V(G) \rvert)}$-pseudocomponent $M^*$ in $G[Y^{(i_{t^*},\lvert V(T) \rvert+1,s+2)}]$ with $\sigma(M^*)=\sigma(M)$ containing $u$ but not $v$.
\end{claim}

\begin{proof}
Suppose to the contrary that there exist $t' \in \partial T_{j,t^*} \cap V(T_t)$, $u',v' \in X_{t'}$, a monochromatic path $P'$ in $G[Y^{(i_t,-1,0)} \cap X_{V(T_{t'})}]$ from $u'$ to $v'$ internally disjoint from $X_{t'}$, and a monochromatic $E_{j,t^*}^{(i_{t^*},\lvert V(G) \rvert)}$-pseudocomponent $M^*$ in $G[Y^{(i_{t^*},\lvert V(T) \rvert+1,s+2)}]$ with $\sigma(M^*)=\sigma(M)$ such that $t'$ is a witness for $X_{t'} \cap I_j \subseteq W_3^{(i_{t^*},\ell)}$ for some $\ell \in [0,\lvert V(T) \rvert]$, $V(P') \not \subseteq Y^{(i_{t^*},\lvert V(T) \rvert+1,s+2)}$, and $M^*$ contains $u'$ but not $v'$.
We further assume that $\sigma(M)$ is minimum among all counterexamples.

Let $t_0$ be the node of $T$ such that $i_{t_0} \leq i_{t'}$ and $V(P') \not \subseteq Y^{(i_{t_0},-1,0)}$, and subject to these, $i_{t_0}$ is maximum.
Since $i_{t'} \geq i_t>i_{t^*}$ and $V(P') \not \subseteq Y^{(i_{t^*},\lvert V(T) \rvert+1,s+2)}$, we know $V(P') \not \subseteq Y^{(i_{t^*}+1,-1,0)}$, so $i_{t^*}<i_{t_0}$.
By the maximality of $i_{t_0}$, $t' \in V(T_{t_0}) \subseteq V(T_{t^*})$.
Since $t' \in \partial T_{j,t^*}$, $t' \not \in V(T_{j,t_0})-\partial T_{j,t_0}$.
Let $P_0$ be the maximal subpath of $P'[V(P') \cap Y^{(i_{t_0},-1,0)}]$ containing $u'$.
Let $R_0$ be the monochromatic $E_{j,t_0}^{(i_{t_0})}$-pseudocomponent in $G[Y^{(i_{t_0},-1,0)}]$ containing $P_0$.
So $M^* \subseteq R_0$.
Let $R_0'$ be the monochromatic $E_{j,t_0}^{(i_{t_0})}$-pseudocomponent in $G[Y^{(i_{t_0},-1,0)}]$ containing $v'$.

Since $V(P') \subseteq Y^{(i_{t},-1,0)}$ and $i_t \leq i_{t'}$, we know $i_{t_0}<i_t$ and $t_0 \neq t'$.
Since $t_0 \neq t'$ and $t' \not \in V(T_{j,t_0})-\partial T_{j,t_0}$, by the maximality of $i_{t_0}$, there exists an $\Se_j^\circ$-related monochromatic $E_{j,t_0}^{(i_{t_0})}$-pseudocomponent $M_0$ in $G[Y^{(i_{t_0},-1,0)}]$ intersecting $X_{t_0}$ such that there exist $\ell' \in [-1,0]$ and $k' \in [0,s+2]$ such that $\sigma(M_0) = \sigma(Z_0)<\min\{\sigma(R_0),\sigma(R_0')\}$ and $A_{L^{(i_{t_0},\ell',k')}}(V(Z_0)) \cap V(P') \cap N_G(V(P_0)) \cap Z_{t_0}-Y^{(i_{t_0},-1,0)} \neq \emptyset$, where $Z_0$ is the monochromatic $E_{j,t_0}^{(i_{t_0})}$-pseudocomponent in $G[Y^{(i_{t_0},\ell',k')}]$ containing $M_0$. 
So $V(Z_0) \cap X_{t'} \neq \emptyset$.

Since $M^* \subseteq R_0$, $\sigma(R_0) \leq \sigma(M^*)$.
Hence $\sigma(M_0) < \sigma(M^*)=\sigma(M)$.
Let $z^*$ be the node of $T$ such that $t_0 \in V(T_{j,z^*})$ and some $E_{j,z^*}^{(i_{z^*})}$-pseudocomponent in $G[Y^{(i_{z^*},\lvert V(T) \rvert+1,s+2)}]$ intersects $V(M_0)$ and $X_q$ for some witness $q \in \partial T_{j,z^*} \cup \{z^*\}$ for $X_q \cap I_j \subseteq W_3^{(i_{z^*},\ell)}$ for some $\ell \in [-1,\lvert V(T) \rvert]$, and subject to this, $i_{z^*}$ is minimum.
Since $t^*$ is a candidate for $z^*$, $i_{z^*} \leq i_{t^*}$.
Since $\sigma(M_0) < \sigma(M)$, $(t_0,M_0)$ is a nice pair by assumption.
By \cref{claim:underassumptouch2}, there exists a monochromatic $E_{j,z^*}^{(i_{z^*},\lvert V(G) \rvert)}$-pseudocomponent $M_0'$ in $G[Y^{(i_{z^*},\lvert V(T) \rvert+1,s+2)}]$ intersecting $V(M_0)$ such that $\sigma(M_0')=\sigma(M_0)$, $V(M_0') \cap X_{t_0} \neq \emptyset$ and $A_{L^{(i_{z^*},\lvert V(T) \rvert+1,s+2)}}(V(M_0')) \cap X_{V(T_{t_0})}-X_{t_0} \neq \emptyset$.

Since $V(M_0) \cap X_{t'} \neq \emptyset$ and $A_{L^{(i_{t_0},-1,0)}}(V(M_0)) \cap X_{V(T_{t'})}-X_{t'} \neq \emptyset$, by \cref{claim:isolation,claim:pseudocomptouchingW3} and the existence of $M_0'$, there exists a monochromatic $E_{j,z^*}^{(i_{z^*},\lvert V(G) \rvert)}$-pseudocomponent $M_1$ in \linebreak $G[Y^{(i_{z^*},\lvert V(T) \rvert+1,s+2)}]$ intersecting $M_0$ such that $V(M_1) \cap X_{t'} \neq \emptyset$ and $A_{L^{(i_{z^*},\lvert V(T) \rvert+1,s+2)}}(V(M_1)) \cap X_{V(T_{t'})}-X_{t'} \neq \emptyset$.
Suppose $M_1 \not \subseteq M_0'$.
Since $M_1 \cup M_0' \subseteq M_0$, there exists a path $Q$ in $M_0$ from $V(M_0')$ to $V(M_1)$ internally disjoint from $V(M_0')$.
By the construction of $E_{j,t_0}^{(i_{t_0})}$ and \cref{claim:isolation,claim:pseudocomptouchingW3}, $Q$ is in $G[Y^{(i_{t_0},-1,0)} \cap X_{V(T_{z^*})}]$.
Hence there exist $t'' \in \partial T_{j,z^*}$, $u'',v'' \in X_{t''}$, a monochromatic path $Q'$ in $G[Y^{(i_{t_0},-1,0)} \cap X_{V(T_{t''})}]$ from $u''$ to $v''$ internally disjoint from $X_{t''}$ but not in $G[Y^{(i_{z^*},\lvert V(T) \rvert+1,s+2)}]$ such that $M_0'$ contains $u''$ but not $v''$.
Since $\sigma(M_0)<\sigma(M)$, $(t_0,M_0)$ is an outer-safe pair, so $t'' \in V(T_{t_0})$.
Hence $M_0$ is a counterexample of this claim, contradicting that $\sigma(M)$ is minimum among all counterexamples. 

So $M_0'$ contains $M_1'$ for every monochromatic $E_{j,z^*}^{(i_{z^*},\lvert V(G) \rvert)}$-pseudocomponent $M'_1$ in \linebreak $G[Y^{(i_{z^*},\lvert V(T) \rvert+1,s+2)}]$ intersecting $M_0$ such that $V(M'_1) \cap X_{t'} \neq \emptyset$ and $A_{L^{(i_{z^*},\lvert V(T) \rvert+1,s+2)}}(V(M'_1)) \cap X_{V(T_{t'})}-X_{t'} \neq \emptyset$.
Hence there exists a monochromatic $E_{j,t^*}^{(i_{t^*},\lvert V(G) \rvert)}$-pseudocomponent $M_0^*$ in $G[Y^{(i_{t^*},\lvert V(T) \rvert+1,s+2)}]$ containing $M_0'$ such that $V(M_0^*) \cap X_{t'} \neq \emptyset$ and $A_{L^{(i_{t^*},\lvert V(T) \rvert+1,s+2)}}(V(M_0^*)) \cap X_{V(T_{t'})}-X_{t'} \neq \emptyset$.

Since $M_0^*$ is an $E_{j,t^*}^{(i_{t^*},\lvert V(G) \rvert)}$-pseudocomponent in $G[Y^{(i_{t^*},\lvert V(T) \rvert+1,s+2)}]$, there exists a nonnegative integer $\beta_0$ such that $M^*_0$ is the $(\beta_0+1)$-smallest among all monochromatic $E_{j,t^*}^{(i_{t^*},\beta_0)}$-pseudocomponents in $G[Y^{(i_{t^*},\lvert V(T) \rvert+1,s+2)}]$.
Let $R_u$ and $R_v$ be the monochromatic $E_{j,t^*}^{(i_{t^*},\beta_0)}$-pseudocomponents in $G[Y^{(i_{t^*},\lvert V(T) \rvert+1,s+2)}]$ containing $u'$ and $v'$, respectively.
By the existence of $P'$, $R_u \cup R_v \subseteq M$, so $\sigma(R_u)=\sigma(M^*)=\sigma(M) \leq \sigma(R_v)$.

Note that $\sigma(M_0^*) \leq \sigma(M_0')=\sigma(M_0)<\sigma(M) \leq \sigma(R_u)$.
Since $\{u',v'\} \not \in E_{j,t^*}^{(i_{t^*},\lvert V(G) \rvert)}$, $\{u',v'\} \not \in E_{j,t^*}^{(i_{t^*},\beta_0+1)}$.
Since $\sigma(M_0^*) < \sigma(R_u)$, by the construction of $E_{j,t^*}^{(i_{t^*},\beta_0+1)}$ and the minimality of $i_{t^*}$ we know $(V(R_u) \cup V(R_v)) \cap X_{t^*} \neq \emptyset$ and for every $\Se_j^\circ$-related monochromatic $E_{j,t^*}^{(i_{t^*},\beta_0)}$-pseudocomponent $M'$ in $G[Y^{(i_{t^*},\lvert V(T) \rvert+1,s+2)}]$ with $\sigma(M') \leq \sigma(M_0^*)$, $A_{L^{(i_{t^*},\lvert V(T) \rvert+1,s+2)}}(V(M')) \cap X_{V(T_{t^*})}-X_{t^*} \subseteq X_{V(T_{t'})}-(X_{t'} \cup Z_{t^*})$.

Since $(V(R_u) \cup V(R_v)) \cap X_{t^*} \neq \emptyset$, we know $V(M') \cap X_{t^*} \neq \emptyset$ for every $\Se_j^\circ$-related monochromatic $E_{j,t^*}^{(i_{t^*},\beta_0)}$-pseudocomponent $M'$ in $G[Y^{(i_{t^*},\lvert V(T) \rvert+1,s+2)}]$ with $\sigma(M') \leq \sigma(M_0^*)$.
Since $t' \in \partial T_{j,t^*}$, and $A_{L^{(i_{t^*},\lvert V(T) \rvert+1,s+2)}}(V(M')) \cap X_{V(T_{t^*})}-X_{t^*} \subseteq X_{V(T_{t'})}-(X_{t'} \cup Z_{t^*})$ for every $\Se_j^\circ$-related monochromatic $E_{j,t^*}^{(i_{t^*},\beta_0)}$-pseudocomponent $M'$ in $G[Y^{(i_{t^*},\lvert V(T) \rvert+1,s+2)}]$ with $\sigma(M') \leq \sigma(M_0')$, we know $V(M_0)=V(M_0^*)$.
Since $\sigma(M_0^*)=\sigma(M_0)=\sigma(Z_0)$ and $Z_{t_0} \cap X_{V(T_{t'})} \subseteq Z_{t^*}$, we know $V(Z_0)=V(M_0)=V(M_0^*)$ and 
\begin{align*}
\emptyset & \neq A_{L^{(i_{t_0},\ell',k')}}(V(Z_0)) \cap V(P') \cap N_G(V(P_0)) \cap Z_{t_0}-Y^{(i_{t_0},-1,0)} \\
& \subseteq A_{L^{(i_{t_0},\ell',k')}}(V(M_0^*)) \cap Z_{t_0} \cap X_{V(T_{t'})} -X_{t'} \\
& \subseteq A_{L^{(i_{t^*},\lvert V(T) \rvert+1,s+2)}}(V(M_0^*)) \cap Z_{t^*} \cap X_{V(T_{t'})} -X_{t'} \\
& \subseteq A_{L^{(i_{t^*},\lvert V(T) \rvert+1,s+2)}}(V(M_0^*)) \cap (X_{V(T_{t^*})}-X_{t^*}) \cap Z_{t^*} \cap X_{V(T_{t'})} -X_{t'} \\
& \subseteq (X_{V(T_{t'})}-(X_{t'} \cup Z_{t^*})) \cap Z_{t^*} \cap X_{V(T_{t'})} -X_{t'} = \emptyset,
\end{align*}
a contradiction.
\end{proof}

\begin{claim} \label{claim:underassumpoutersafeunique}
	Let $j \in [\lvert \V \rvert-1]$.
	Let $t \in V(T)$.
	Let $M$ be an $\Se_j$-related monochromatic $E_{j,t}^{(i_t)}$-pseudocomponent in $G[Y^{(i_t,-1,0)}]$ intersecting $X_t$. 
	Let $t^*$ be the node of $T$ with $t \in V(T_{j,t^*})$ such that some $E_{j,t^*}^{(i_{t^*})}$-pseudocomponent in $G[Y^{(i_{t^*},\lvert V(T) \rvert+1,s+2)}]$ intersecting $V(M)$ intersects $X_{t'}$ for some witness $t' \in \partial T_{j,t^*} \cup \{t^*\}$ for $X_{t'} \cap I_j \subseteq W_3^{(i_{t^*},\ell)}$ for some $\ell \in [-1,\lvert V(T) \rvert]$, and subject to this, $i_{t^*}$ is minimum.
	Assume that $(z,M_z)$ is a nice pair and an outer-safe pair for every $z \in V(T)$ and $\Se_j^\circ$-related monochromatic $E_{j,z}^{(i_z)}$-pseudocomponent $M_z$ in $G[Y^{(i_z,-1,0)}]$ intersecting $X_z$ with $\sigma(M_z)<\sigma(M)$. 
	If $(t,M)$ is an outer-safe pair, then the following hold.
	\begin{itemize}
		\item There uniquely exists a monochromatic $E_{j,t^*}^{(i_{t^*},\lvert V(G) \rvert)}$-pseudocomponent $M^*$ in $G[Y^{(i_{t^*},\lvert V(T) \rvert+1,s+2)}]$ intersecting $V(M)$.
		\item $V(M^*) \cap X_t \neq \emptyset$ and $\sigma(M^*)=\sigma(M)$.
		\item If $t^* \neq t$ and $A_{L^{(i_t,-1,0)}}(V(M)) \cap X_{V(T_t)}-X_t \neq \emptyset$, then $A_{L^{(i_{t^*},\lvert V(T) \rvert+1,s+2)}}(V(M^*)) \cap X_{V(T_t)}-X_t \neq \emptyset$.
	\end{itemize}
\end{claim}

\begin{proof}
By \cref{claim:isolation,claim:pseudocomptouchingW3,claim:orderpreserving}, there exists a monochromatic $E_{j,t^*}^{(i_{t^*},\lvert V(G) \rvert)}$-pseudocomponent $M^*$ in $G[Y^{(i_{t^*},\lvert V(T) \rvert+1,s+2)}]$ with $\sigma(M^*)=\sigma(M)$ such that $V(M^*)$ intersects $X_q$ for some witness $q \in \partial T_{j,t^*} \cup \{t^*\}$ for $X_q \cap I_j \subseteq W_3^{(i_{t^*},\ell)}$ for some $\ell \in [-1,\lvert V(T) \rvert]$.
	
Suppose there exists a monochromatic $E_{j,t^*}^{(i_{t^*},\lvert V(G) \rvert)}$-pseudocomponent $M'$ in $G[Y^{(i_{t^*},\lvert V(T) \rvert+1,s+2)}]$ intersecting $V(M)$ such that $M' \neq M^*$.
Then $t \neq t^*$, so $M' \cup M^* \subseteq M$, and hence there exists a monochromatic path $P$ in $G[Y^{(i_t,-1,0)}]$ from $V(M^*)$ to $V(M')$ internally disjoint from $V(M^*)$.
By tracing $P$ from $V(M^*)$, there exist $t' \in \partial T_{j,t^*}$, $u,v \in X_{t'}$, a monochromatic path $Q$ in $G[Y^{(i_t,-1,0)} \cap X_{V(T_{t'})}]$ from $u$ to $v$ such that $M^*$ contains $u$ but not $v$, and $V(Q) \cap A_{L^{(i_{t^*},\lvert V(T) \rvert+1,s+2)}}(V(M^*)) \cap X_{V(T_{t'})}-X_{t'} \neq \emptyset$.
Note that $V(Q) \not \subseteq Y^{(i_{t^*},\lvert V(T) \rvert + 1,s+2)}$, since $M^*$ contains $u$ but not $v$.
By \cref{claim:pseudocomptouchingW3}, since $V(M^*) \cap X_{t'} \neq \emptyset \neq V(M^*) \cap V(M)$, $t'$ is a witness for $X_{t'} \cap I_j \subseteq W_3^{(i_{t^*},\ell)}$ for some $\ell \in [0,\lvert V(T) \rvert]$.
Since $(t,M)$ is an outer-safe pair, $t' \in V(T_t)$.
However, by \cref{claim:underassumpnounder}, $t' \not \in V(T_t)$, a contradiction.
	
Hence $M^*$ is the unique monochromatic $E_{j,t^*}^{(i_{t^*},\lvert V(G) \rvert)}$-pseudocomponent in $G[Y^{(i_{t^*},\lvert V(T) \rvert+1,s+2)}]$ intersecting $V(M)$.
Recall that $\sigma(M^*)=\sigma(M)$.
Since $V(M) \cap X_t \neq \emptyset$ and $t \in V(T_{j,t^*})$, by \cref{claim:isolation,claim:pseudocomptouchingW3}, $V(M^*) \cap X_t \neq \emptyset$.
	
If $t^* \neq t$ and $A_{L^{(i_{t},-1,0)}}(V(M)) \cap X_{V(T_t)}-X_t \neq \emptyset$, then since $V(M) \cap X_t \subseteq V(M^*)$ by \cref{claim:isolation}, and $M^*$ is the unique monochromatic $E_{j,t^*}^{(i_{t^*},\lvert V(G) \rvert)}$-pseudocomponent in $G[Y^{(i_{t^*},\lvert V(T) \rvert+1,s+2)}]$ intersecting $V(M)$, $A_{L^{(i_{t^*},\lvert V(T) \rvert+1,s+2)}}(V(M^*)) \cap X_{V(T_t)}-X_t \neq \emptyset$.
\end{proof}

\begin{claim} \label{claim:underassumpunder1}
Let $j \in [\lvert \V \rvert-1]$.
Let $t \in V(T)$ and let $M$ be an $\Se_j^\circ$-related monochromatic $E_{j,t}^{(i_t)}$-pseudocomponent in $G[Y^{(i_t,-1,0)}]$ intersecting $X_t$. 
Let $t^* \in V(T)$ be the node of $T$ such that $t \in V(T_{j,t^*})$ and some $E_{j,t^*}^{(i_{t^*})}$-pseudocomponent in $G[Y^{(i_{t^*},\lvert V(T) \rvert+1,s+2)}]$ intersects $V(M)$ and $X_q$ for some witness $q \in \partial T_{j,t^*} \cup \{t^*\}$ for $X_q \cap I_j \subseteq W_3^{(i_{t^*},\ell)}$ for some $\ell \in [-1,\lvert V(T) \rvert]$, and subject to this, $i_{t^*}$ is minimum.
Assume that $t \neq t^*$. 
Assume that $(z,M')$ is a nice pair and an outer-safe pair for every $z \in V(T)$ and every $\Se_j^\circ$-related monochromatic $E_{j,z}^{(i_z)}$-pseudocomponent $M'$ in $G[Y^{(i_z,-1,0)}]$ intersecting $X_z$ with $\sigma(M')<\sigma(M)$. 

If $(t,M)$ is an outer-safe pair and there exists a $c$-monochromatic path $P$ in $G$ from $V(M)$ to a monochromatic $E_{j,t}^{(i_t)}$-pseudocomponent $M_0$ in $G[Y^{(i_t,-1,0)}]$ other than $M$ internally disjoint from every monochromatic $E_{j,t}^{(i_t)}$-pseudocomponent in $G[Y^{(i_t,-1,0)}]$ such that $V(P) \cap A_{L^{(i_t,-1,0)}}(V(M)) \cap X_{V(T_t)}-X_t \neq \emptyset$,  then there exist $t' \in \partial T_{j,t^*} \cap V(T_t)$, $u,v \in X_{t'}$ and a $c$-monochromatic path in $G[X_{V(T_{t'})}]$ but not in $G[Y^{(i_{t},-1,0)}]$ from $u$ to $v$ internally disjoint from $X_{t'}$ such that $t'$ is a witness for $X_{t'} \cap I_j \subseteq W_3^{(i_{t^*},\ell)}$ for some $\ell \in [0,\lvert V(T) \rvert]$, and there exists a monochromatic $E_{j,t^*}^{(i_{t^*},\lvert V(G) \rvert)}$-pseudocomponent $M^*$ in $G[Y^{(i_{t^*},\lvert V(T) \rvert+1,s+2)}]$ with $\sigma(M^*)=\sigma(M)$ containing $u$ but not $v$.
\end{claim}

\begin{proof}
Since $P$ is internally disjoint from every monochromatic $E_{j,t}^{(i_t)}$-pseudocomponent in $G[Y^{(i_t,-1,0)}]$ and $V(P) \cap A_{L^{(i_t,-1,0)}}(V(M)) \cap X_{V(T_t)}-X_t \neq \emptyset$, $V(P) \subseteq X_{V(T_t)}$.

Since $M \neq M_0$, there exist $t' \in \partial T_{j,t^*} \cap V(T_t)$, $u,v \in X_{t'}$, a subpath $P'$ of $P \cup M \cup M_0$ from $u$ to $v$ contained in $G[X_{V(T_{t'})}]$ internally disjoint from $X_{t'}$ but not contained in $G[Y^{(i_t,-1,0)}]$, and a monochromatic $E_{j,t^*}^{(i_{t^*},\lvert V(G) \rvert)}$-pseudocomponent $M^*$ in $G[Y^{(i_{t^*},\lvert V(T) \rvert+1,s+2)}]$ intersecting $V(M) \cap V(P)$ containing $u$ but not containing $v$.
Note that $t'$ is a witness for $X_{t'} \cap I_j \subseteq W_3^{(i_{t^*},\ell)}$ for some $\ell \in [0,\lvert V(T) \rvert]$ by \cref{claim:isolation,claim:pseudocomptouchingW3}.
To prove this claim, it suffices to prove that $\sigma(M^*)=\sigma(M)$.

By \cref{claim:isolation,claim:pseudocomptouchingW3,claim:orderpreserving}, there exists a monochromatic $E_{j,t^*}^{(i_{t^*},\lvert V(G) \rvert)}$-pseudocomponent $Q^*$ in $G[Y^{(i_{t^*},\lvert V(T) \rvert+1,s+2)}]$ with $\sigma(Q^*)=\sigma(M)$.
Suppose $\sigma(M^*) \neq \sigma(M)$.
So $M^* \neq Q^*$.
Since $M^* \cup Q^* \subseteq M$, by the construction of $E_{j,t}^{(i_t)}$ and \cref{claim:isolation,claim:pseudocomptouchingW3}, there exist $t'' \in \partial T_{j,t^*}$, a monochromatic path $Q'$ in $G[Y^{(i_t,-1,0)} \cap X_{V(T_{t''})}]$ from $V(Q^*)$ to a monochromatic $E_{j,t^*}^{(i_{t^*},\lvert V(G) \rvert)}$-pseudocomponent in $G[Y^{(i_{t^*},\lvert V(T) \rvert+1,s+2)}]$ intersecting $V(M)$ other than $Q^*$ internally disjoint from $X_{t''}$ such that $V(Q') \cap A_{L^{(i_{t^*},\lvert V(T) \rvert+1,s+2)}}(V(Q^*)) \cap X_{V(T_{t''})}-X_{t''} \neq \emptyset$, and $t''$ is a witness for $X_{t''} \cap I_j \subseteq W_3^{(i_{t^*},\ell)}$ for some $\ell \in [0,\lvert V(T) \rvert]$.
By \cref{claim:underassumpnounder}, $t'' \not \in V(T_t)$.
So $(t,M)$ is not an outer-safe pair, a contradiction.
Hence $\sigma(M^*)=\sigma(M)$.
This proves the claim.
\end{proof}

\subsection{Blockers}

From now on, we use all the properties of fake edges.
We first recall the relevant parts of the algorithm.

\medskip
\noindent\textbf{\boldmath Stage $(0,-1,0)$: Initialization:}
See \cref{subsec:alog} for details.

\medskip
For $i=0,1,2,\dots$, let $t$ be the node of $T$ with $\sigma_T(t)=i$ and perform the following steps:\\[1ex]
\hspace*{4mm} 
\textbf{\boldmath Stage: Building Fences:}
Let $E^{(0)}_{j,t} := \emptyset$ for every $j \in [\lvert \V \rvert-1]$ and $t \in V(T)$.
Other sets are defined.
See \cref{subsec:alog} for details.\\
\hspace*{4mm} 
\noindent\textbf{\boldmath Stage $(i,-1,\star)$:}
See \cref{subsec:alog} for details.\\
\hspace*{4mm}
\noindent\textbf{\boldmath Stage $(i,\geq 1,\star)$:}
$(Y^{(i,\lvert V(T) \rvert+1,s+2)},L^{(i,\lvert V(T) \rvert+1,s+2)})$, $W_3^{(i,\ell)}$ for $\ell \in [0,\lvert V(T) \rvert]$, and other sets are defined.
See \cref{subsec:alog} for details.\\
\hspace*{4mm}
\noindent\textbf{Stage: Adding Fake Edges}
	\begin{itemize}
				\item For each $j \in [\lvert \V \rvert-1]$ and $t' \in V(T)-(V(T_t)-\{t\})$, let $E^{(i+1)}_{j,t'}:=E^{(i)}_{j,t'}$.
				\item For each $j \in [\lvert \V \rvert-1]$, let $E^{(i,0)}_{j,t}:=E^{(i)}_{j,t}$.
				\item For each $j \in [\lvert \V \rvert-1]$ and $t' \in V(T_t)-\{t\}$, let $E^{(i+1)}_{j,t'}:=E^{(i,\lvert V(G) \rvert)}_{j,t}$, where every $\ell \in [0,\lvert V(G) \rvert-1]$, let $E^{(i,\ell+1)}_{j,t}$ be the union of $E^{(i,\ell)}_{j,t}$ and the set consisting of the 2-element sets $\{u,v\}$ satisfying the following.
				\begin{itemize}
					\item $\{u,v\} \subseteq Y^{(i,\lvert V(T) \rvert+1,s+2)}$.
					\item $L^{(i,\lvert V(T) \rvert+1,s+2)}(u) = L^{(i,\lvert V(T) \rvert+1,s+2)}(v)$.
					\item There exists an $s$-segment $S$ in $\Se_{j}^\circ$ whose level equals the color of $u$ and $v$ such that $\{u,v\} \subseteq S$.
					\item There exists $t'' \in \partial T_{j,t}$ such that:
					\begin{itemize}
						\item $\{u,v\} \subseteq X_{t''}$,
						\item $V(M) \cap X_{t''} \neq \emptyset$, where $M$ is the monochromatic $E^{(i,\ell)}_{j,t}$-pseudocomponent in \linebreak $G[Y^{(i,\lvert V(T) \rvert+1,s+2)}]$ such that $\sigma(M)$ is the $(\ell+1)$-th smallest among all monochromatic $E^{(i,\ell)}_{j,t}$-pseudocomponents in  $G[Y^{(i,\lvert V(T) \rvert+1,s+2)}]$,
						\item $M$ is contained in some $s$-segment in $\Se_{j}^\circ$ whose level equals the color of $M$,
						\item $t''$ is a witness for $X_{t''} \cap I_j \subseteq W_3^{(i,\ell')}$ for some $\ell' \in [0,\lvert V(T) \rvert]$,
						\item $A_{L^{(i,\lvert V(T) \rvert+1,s+2)}}(V(M)) \cap X_{V(T_{t''})}-X_{t''} \neq \emptyset$,
						\item $A_{L^{(i,\lvert V(T) \rvert+1,s+2)}}(V(M_u)) \cap X_{V(T_{t''})}-X_{t''} \neq \emptyset$,
						\item $A_{L^{(i,\lvert V(T) \rvert+1,s+2)}}(V(M_v)) \cap X_{V(T_{t''})}-X_{t''} \neq \emptyset$, and
						\item $\sigma(M) < \min\{\sigma(M_u),\sigma(M_v)\}$, 
						
					\end{itemize}
					where $M_u$ and $M_v$ are the monochromatic $E^{(i,\ell)}_{j,t}$-pseudocomponents in $G[Y^{(i,\lvert V(T) \rvert+1,s+2)}]$ containing $u$ and $v$, respectively.
				\item Let $M$, $M_u$ and $M_v$ be the monochromatic $E_{j,t}^{(i,\ell)}$-pseudocomponents as in the previous bullet.
						If $(V(M_u) \cup V(M_v)) \cap X_t \neq \emptyset$, then at least one of the following does not hold,
						\begin{itemize}
							\item for every monochromatic $E_{j,t}^{(i,\ell)}$-pseudocomponent $M'$ in $G[Y^{(i,\lvert V(T) \rvert+1,s+2)}]$ with $\sigma(M') \leq \sigma(M)$ such that $M'$ is contained in some $s$-segment in $\Se_{j}^\circ$ whose level equals the color of $M'$, $M'$ satisfies that $A_{L^{(i,\lvert V(T) \rvert+1,s+2)}}(V(M')) \cap X_{V(T_{t})}-X_{t} \subseteq X_{V(T_{t''})}-(X_{t''} \cup Z_t)$,
							\item there exists a parade $(t_{-\beta_{u,v,\ell}},t_{\beta_{u,v,\ell}+1},\dots,t_{-1},t_1,t_2,\dots,t_{\alpha_{u,v,\ell}})$ for some $\alpha_{u,v,\ell} \in {\mathbb N}$ and $\beta_{u,v,\ell} \in {\mathbb N}$ such that 
								\begin{itemize}
									\item $t_{-\beta_{u,v,\ell}} \in V(T_{t''})-\{t''\}$, 
									\item $\alpha_{u,v,\ell} = \psi_2(\alpha_{u,\ell},\alpha_{v,\ell})$ and $\beta_{u,v,\ell} = \psi_3(\alpha_{u,\ell},\alpha_{v,\ell})$, \\
									where $\alpha_{u,\ell}$ and $\alpha_{v,\ell}$ are the integers such that $\sigma(M_{u,\ell})$ and $\sigma(M_{v,\ell})$ are the $\alpha_{u,\ell}$-th and the $\alpha_{v,\ell}$-th smallest among all monochromatic $E_{j,t}^{(i,\ell)}$-pseudocomponents $M'$ in $G[Y^{(i,\lvert V(T) \rvert+1,s+2)}]$ contained in some $s$-segment in $\Se_j^\circ$ whose color equals the color of $M_u$ and $M_v$ with $A_{L^{(i,\lvert V(T) \rvert+1,s+2)}}(V(M')) \cap X_{V(T_{t''})}-X_{t''} \neq \emptyset$, respectively, 
									\item for any distinct $\ell_1,\ell_2 \in [-\beta_{u,v,\ell}, \alpha_{u,v,\ell}]-\{0\}$, $X_{t_{\ell_1}} \cap I_j - X_{t''}$ and $X_{t_{\ell_2}} \cap I_j - X_{t''}$ are disjoint non-empty sets,
									\item $\bigcup_{M''} (A_{L^{(i,\lvert V(T) \rvert+1,s+2)}}(V(M'')) \cap X_{V(T_{t})}-X_{t}) \subseteq (X_{V(T_{t_{\alpha_{u,v,\ell}}})}-X_{t_{\alpha_{u,v,\ell}}}) \cap I_j^\circ$, where the union is over all monochromatic $E_{j,t}^{(i,\ell)}$-pseudocomponents $M''$ in \linebreak $G[Y^{(i,\lvert V(T) \rvert+1,s+2)}]$ such that $V(M'')$ is contained in some $s$-segment in $\Se_j^\circ$ whose level equals the color of $M''$, $\sigma(M'') \leq \sigma(M)$, $V(M'') \cap X_{t''} \neq \emptyset$, and \linebreak $A_{L^{(i,\lvert V(T) \rvert+1,s+2)}}(V(M'')) \cap X_{V(T_{t''})}-X_{t''} \neq \emptyset$,
									\item there exists $x_\ell \in \{u,v\}$ such that $A_{L^{(i,\lvert V(T) \rvert+1,s+2)}}(V(M_{x_\ell})) \cap (X_{V(T_{t})}-X_{t}) \cap X_{V(T_{t_1})}-X_{t_1} = \emptyset$, and $A_{L^{(i,\lvert V(T) \rvert+1,s+2)}}(V(M_{x_\ell})) \cap (X_{V(T_{t})}-X_{t}) \cap X_{V(T_{t_{-\beta_{u,v,\ell}}})}-(X_{t_{-\beta_{u,v,\ell}}} \cup I_j^\circ) = \emptyset$, 
								\end{itemize}
							\item there exists a monochromatic $E_{j,t}^{(i)}$-pseudocomponent $M^*$ in $G[Y^{(i,\lvert V(T) \rvert+1,s+2)}]$ with $\sigma(M^*) \leq \sigma(M)$ such that $M^*$ is contained in some $s$-segment in $\Se_{j}^\circ$ whose level equals the color of $M^*$ satisfying that $A_{L^{(i,\lvert V(T) \rvert+1,s+2)}}(V(M^*)) \cap X_{V(T_{t})}-X_{t} \subseteq X_{V(T_{t''})}-(X_{t''} \cup Z_t)$.
						\end{itemize}
					\item There do not exist $q \in V(T)$ with $T_{j,t} \subseteq T_{j,q}$ and $i_q<i_t$, a witness $q' \in \partial T_{j,q}$ for $X_{q'} \cap I_j \subseteq W_3^{(i_{q},\ell')}$ for some $\ell' \in [0,\lvert V(T) \rvert]$, and a monochromatic $E_{j,q}^{(i_{q})}$-pseudocomponent in $G[Y^{(i_{q},\lvert V(T) \rvert+1,s+2)}]$ intersecting $X_{q'}$ and $\{u,v\}$.
				\end{itemize}
	\end{itemize}
\hspace*{4mm}
\noindent\textbf{Stage: Moving to the Next Node in the Tree:}
See \cref{subsec:alog} for details.\\
\hspace*{4mm}
\noindent\textbf{Stage: Building a New Fence:}
See \cref{subsec:alog} for details.

\bigskip

We now define terminology that is related to some conditions mentioned in the algorithm. For any $t \in V(T)$, $t' \in V(T_t)$, $j \in [\lvert \V \rvert-1]$, $\alpha_0 \in [0,w_0]$, $\alpha_1,\alpha_2 \in [w_0]$ with $\alpha_0<\alpha_1 \leq \alpha_2$, $k \in {\mathbb N}$, $k' \in {\mathbb N}_0$, $\ell \in [-1,\lvert V(T) \rvert+1]$, $\xi \in [0,s+2]$ and $\xi' \in [0,\lvert V(G) \rvert]$, we define an \mathdef{$(\alpha_0,\alpha_1,\alpha_2,k,k')$}{blocker} \emph{for $(j,t,\ell,\xi,\xi',t')$} to be a parade $(t_{-k'},t_{-k'+1},\dots,t_{-1},t_1,t_2,\dots,t_{k})$ such that the following hold:
	\begin{itemize}
		\item For every $\beta \in [w_0]$, let $M_\beta$ be the $\Se_j^\circ$-related monochromatic $E_{j,t}^{(i_t,\xi')}$-pseudocomponent in $G[Y^{(i_t,\ell,\xi)}]$ such that $\sigma(M_\beta)$ is the $\beta$-th smallest among all $\Se_j^\circ$-related monochromatic $E_{j,t}^{(i_t,\xi')}$-pseudocomponents $M$ in $G[Y^{(i_t,\ell,\xi)}]$ intersecting $X_{t'}$ with $A_{L^{(i_t,\ell,\xi)}}(V(M)) \cap X_{V(T_{t'})}-X_{t'} \neq \emptyset$ (if there are less than $\beta$ such $M$, then let $M_\beta=\emptyset$).
		\item If $k' \neq 0$, then $t_{-k'} \in V(T_{t'})-\{t'\}$.
		\item $t_1 \in V(T_{t'})-\{t'\}$.
		\item For every distinct $\ell_1,\ell_2 \in [-k',k]-\{0\}$, $X_{t_{\ell_1}} \cap I_j - X_{t'}$ and $X_{t_{\ell_2}} \cap I_j - X_{t'}$ are disjoint non-empty sets.
		\item $\bigcup_{\beta=1}^{\alpha_0}A_{L^{(i_t,\ell,\xi)}}(V(M_\beta)) \cap X_{V(T_t)}-X_t \subseteq (X_{V(T_{t_k})}-X_{t_k}) \cap I_j^\circ$.
		\item There exists $\alpha^* \in \{\alpha_1,\alpha_2\}$ such that $A_{L^{(i_t,\ell,\xi)}}(V(M_{\alpha^*})) \cap (X_{V(T_{t})}-X_{t}) \cap X_{V(T_{t_1})}-X_{t_1} = \emptyset$, and if $k' \neq 0$, then $A_{L^{(i_t,\ell,\xi)}}(V(M_{\alpha^*})) \cap (X_{V(T_{t})}-X_{t}) \cap X_{V(T_{t_{-k'}})}-(X_{t_{-k'}} \cup I_j^\circ) = \emptyset$.
	\end{itemize}

\begin{claim} \label{claim:blocker2-0}
Let $j \in [\lvert \V \rvert-1]$.
Let $t \in V(T)$ and let $t' \in \partial T_{j,t}$ such that $t'$ is a witness for $X_{t'} \cap I_j \subseteq W_3^{(i_t,\ell)}$ for some $\ell \in [0,\lvert V(T) \rvert]$.
Let $M'_1,M'_2,\dots,M'_{k'}$ (for some $k' \in {\mathbb N}$) be the $\Se_j^\circ$-related monochromatic $E_{j,t'}^{(i_{t'})}$-pseudocomponents in $G[Y^{(i_{t'},-1,0)}]$ such that for every $\alpha \in [k']$, $V(M'_\alpha) \cap X_{t'} \neq \emptyset$ and $A_{L^{(i_{t'},-1,0)}}(V(M'_\alpha)) \cap X_{V(T_{t'})}-X_{t'} \neq \emptyset$.
Assume that $\sigma(M'_1)<\sigma(M'_2)<\dots<\sigma(M'_{k'})$.
Let $V(M'_\gamma)=\emptyset$ for every $\gamma \geq k'+1$.
Let $\alpha_1' \in [w_0]$ and $\alpha_2' \in [\alpha_1'+1,k']$.

Let $M^{(0)}_1,M^{(0)}_2,\dots,M^{(0)}_{k^{(0)}}$ (for some $k^{(0)} \in {\mathbb N}$) be the $\Se_j^\circ$-related monochromatic \linebreak  $E_{j,t}^{(i_t)}$-pseudocomponents in $G[Y^{(i_t,\lvert V(T) \rvert+1,s+2)}]$ such that for every $\alpha \in [k^{(0)}]$, we have \linebreak $A_{L^{(i_t,\lvert V(T) \rvert+1,s+2)}}(V(M^{(0)}_\alpha)) \cap X_{V(T_{t})}-X_{t} \neq \emptyset$.
Assume that $\sigma(M^{(0)}_1)<\sigma(M^{(0)}_2)<\dots<\sigma(M^{(0)}_{k^{(0)}})$.
Let $\beta_0$ be the minimum such that $M^{(0)}_{\beta_0} \subseteq M'_{\alpha_1'}$.
Let $\beta_1 \in [\lvert V(G) \rvert]$ be the integer such that the monochromatic $E_{j,t}^{(i_t,\beta_1-1)}$-pseudocomponent in $G[Y^{(i_t,\lvert V(T) \rvert+1,s+2)}]$ containing $M^{(0)}_{\beta_0}$ has the $\beta_1$-th smallest $\sigma$-value among all monochromatic $E_{j,t}^{(i_t,\beta_1-1)}$-pseudocomponents in $G[Y^{(i_t,\lvert V(T) \rvert+1,s+2)}]$.
Let $\beta_2$ be the largest element in $[0,\beta_1-1]$ such that the monochromatic $E_{j,t}^{(i_t,\beta_2-1)}$-pseudocomponent $M$ in $G[Y^{(i_t,\lvert V(T) \rvert+1,s+2)}]$ with the $\beta_2$-th smallest $\sigma$-value among all monochromatic $E_{j,t}^{(i_t,\beta_2-1)}$-pseudocomponents in $G[Y^{(i_t,\lvert V(T) \rvert+1,s+2)}]$ satisfies that $M$ is $\Se_j^\circ$-related, $V(M) \cap X_{t'} \neq \emptyset$ and $A_{L^{(i_t,\lvert V(T) \rvert+1,s+2)}}(V(M)) \cap X_{V(T_{t'})}-X_{t'} \neq \emptyset$.
Let $M_1,M_2,\dots,M_k$ (for some $k \in {\mathbb N}$) be the $\Se_j^\circ$-related monochromatic $E_{j,t}^{(i_t,\beta_2-1)}$-pseudocomponents in $G[Y^{(i_t,\lvert V(T) \rvert+1,s+2)}]$ such that for every $\alpha \in [k]$, $V(M_\alpha) \cap X_{t'} \neq \emptyset$ and $A_{L^{(i_t,\lvert V(T) \rvert+1,s+2)}}(V(M_\alpha)) \cap X_{V(T_{t'})}-X_{t'} \neq \emptyset$.
Assume that $\sigma(M_1)<\sigma(M_2)<\dots<\sigma(M_k)$.
Let $V(M_\gamma)=\emptyset$ for every $\gamma \geq k+1$.

For $\gamma \in \{1,2\}$, let $\Q_\gamma = \{M_\ell: M_\ell \subseteq M'_{\alpha_\gamma'}\}$.
For every $Q \in \Q_1 \cup \Q_2$, let $\alpha_Q$ be the index such that $Q=M_{\alpha_Q}$.

Assume $V(M'_{\alpha_1'}) \cap X_t \neq \emptyset$, and for every $Q_1 \in \Q_1$ and $Q_2 \in \Q_2$, there exists an $(\beta_2,\alpha_{Q_1},\alpha_{Q_2}, \allowbreak \psi_{2}(\alpha_{Q_1},\alpha_{Q_2}),\psi_{3}(\alpha_{Q_1},\alpha_{Q_2}))$-blocker for $(j,t,\lvert V(T) \rvert+1,s+2,\beta_2-1,t')$.

If either 
	\begin{itemize}
		\item[(i)] $\alpha_1' \leq \alpha_Q$ for every $Q \in \Q_1 \cup \Q_2$, or
		\item[(ii)] for every $\beta \in [\alpha_1'-1]$, $(t',M'_\beta)$ is a nice pair, 
	\end{itemize}
then there exists an $(\alpha'_1-1,\alpha'_1,\alpha_2',\psi_2(\alpha'_1,\alpha'_2),\psi_3(\alpha'_1,\alpha'_2)-\psi_3(0,0))$-blocker for $(j,t',-1,0,0,t')$.
\end{claim}

\begin{proof}
We first prove that condition (ii) implies condition (i).
	
Suppose that condition (ii) holds but condition (i) does not hold.
So $\alpha_1' \geq 2$.
Let $\beta \in [\alpha_1'-1]$.
Let $t^* \in V(T)$ be the node of $T$ such that $t' \in V(T_{j,t^*})$ and some $E_{j,t^*}^{(i_{t^*})}$-pseudocomponent in $G[Y^{(i_{t^*},\lvert V(T) \rvert+1,s+2)}]$ intersects $V(M'_\beta)$ and $X_q$ for some witness $q \in \partial T_{j,t^*} \cup \{t^*\}$ for $X_q \cap I_j \subseteq W_3^{(i_{t^*},\ell)}$ for some $\ell \in [-1,\lvert V(T) \rvert]$, and subject to this, $i_{t^*}$ is minimum.
Note that $t$ is a candidate for $t^*$, so $t' \neq t^*$.
Since $V(M'_{\alpha_1'}) \cap X_t \neq \emptyset$, $V(M'_\beta) \cap X_t \neq \emptyset$.
By \cref{claim:orderpreserving}, there exists $\alpha^{(0)}_\beta \in [\beta_0-1]$ such that $M^{(0)}_{\alpha^{(0)}_\beta} \subseteq M'_\beta$ and $\sigma(M^{(0)}_{\alpha^{(0)}_\beta})=\sigma(M'_\beta)$.
By \cref{claim:isolation,claim:pseudocomptouchingW3}, there exists a monochromatic $E_{j,t^*}^{(i_{t^*},\lvert V(G) \rvert)}$-pseudocomponent $R_\beta$ in $G[Y^{(i_{t^*},\lvert V(T) \rvert+1,s+2)}]$ intersecting $M^{(0)}_{\alpha^{(0)}_\beta}$ with $\sigma(R_\beta)=\sigma(M^{(0)}_{\alpha^{(0)}_\beta})$. 
Since $R_\beta$ intersects $V(M'_\beta)$, by \cref{claim:underassumptouch}, $R_\beta$ intersects $X_{V(T_{t'})}$ and satisfies $A_{L^{(i_{t^*},\lvert V(T) \rvert+1,s+2)}}(V(R_\beta)) \cap X_{V(T_{t'})}-X_{t'} \neq \emptyset$.
Since $R_\beta$ intersects $X_{V(T_{t'})}$ and $\sigma(R_\beta)=\sigma(M^{(0)}_{\alpha^{(0)}_\beta})$ and $V(M^{(0)}_{\alpha^{(0)}_\beta}) \cap X_t \neq \emptyset$, we have $R_\beta$ intersects $X_{t'}$.
Since $\alpha^{(0)}_\beta<\beta_0$, $R_\beta$ is contained in a monochromatic $E_{j,t}^{(i_{t},\beta_1-2)}$-pseudocomponent $R'_\beta$ in $G[Y^{(i_{t},\lvert V(T) \rvert+1,s+2)}]$.
So $R_\beta' \supseteq M^{(0)}_{\alpha^{(0)}_\beta}$.
Since $A_{L^{(i_{t^*},\lvert V(T) \rvert+1,s+2)}}(V(R_\beta)) \cap X_{V(T_{t'})}-X_{t'} \neq \emptyset$ and $A_{L^{(i_{t^*},\lvert V(T) \rvert+1,s+2)}}(V(R_\beta)) \cap (X_{V(T_{t'})}-X_{t'}) \cap Z_{t^*} = \emptyset$ and $t' \in V(T_{j,t^*})$, $A_{L^{(i_{t},\lvert V(T) \rvert+1,s+2)}}(V(R'_\beta)) \cap X_{V(T_{t'})}-X_{t'} \neq \emptyset$.
Since $V(R'_\beta) \cap X_{t'} \supseteq V(R_\beta) \cap X_{t'} \neq \emptyset$ and $A_{L^{(i_{t},\lvert V(T) \rvert+1,s+2)}}(V(R'_\beta)) \cap X_{V(T_{t'})}-X_{t'} \neq \emptyset$, by the maximality of $\beta_2$, there exists $\alpha_\beta \in [\beta_2-1]$ such that $R'_\beta = M_{\alpha_\beta}$.
Since $\alpha^{(0)}_\beta<\beta_0$, $\alpha_\beta<\alpha_Q$ for every $Q \in \Q_1 \cup \Q_2$.
Hence there exists a mapping $\iota$ that maps each element $\beta \in [\alpha_1'-1]$ to $\alpha_\beta$.
Note that $\iota$ is an injection.
So there are at least $\alpha_1'-1$ different numbers strictly smaller than $\min_{Q \in \Q_1 \cup \Q_2}\alpha_Q$.
Hence $\alpha_1' \leq \min_{Q \in \Q_1 \cup Q_2}\alpha_Q$ and condition (i) holds, a contradiction.

Therefore, to prove this claim, it suffices to assume the case that condition (i) holds.
	
For every $Q_1 \in \Q_1$ and $Q_2 \in \Q_2$, let $\kappa_{Q_1,Q_2} = \psi_{2}(\alpha_{Q_1},\alpha_{Q_2})$ and $\xi_{Q_1,Q_2} = \psi_{3}(\alpha_{Q_1},\alpha_{Q_2})$, and let $t^{Q_1,Q_2} = (t^{Q_1,Q_2}_{-\xi_{Q_1,Q_2}},\dots,t^{Q_1,Q_2}_{-1},t^{Q_1,Q_2}_1,\dots,t^{Q_1,Q_2}_{\kappa_{Q_1,Q_2}})$ be a $(\beta_2,\alpha_{Q_1},\alpha_{Q_2},\kappa_{Q_1,Q_2},\xi_{Q_1,Q_2})$-blocker for $(j,t,\lvert V(T) \rvert+1,s+2,\beta_2-1,t')$ such that $\sum_xi_x$ is maximum, where the sum is over all entries of $t^{Q_1,Q_2}$. 
We also denote $\kappa_{Q_1,Q_2}$ by $\kappa_{\alpha_{Q_1},\alpha_{Q_2}}$, and denote $\xi_{Q_1,Q_2}$ by $\xi_{\alpha_{Q_1},\alpha_{Q_2}}$.
For every $Q_1 \in \Q_1$ and $Q_2 \in \Q_2$, let $\alpha^*_{Q_1,Q_2} \in \{\alpha_{Q_1},\alpha_{Q_2}\}$ such that $A_{L^{(i_t,\lvert V(T) \rvert+1,s+2)}}(V(M_{\alpha^*_{Q_1,Q_2}})) \cap (X_{V(T_{t})}-X_{t}) \cap X_{V(T_{t_{-\xi_{Q_1,Q_2}}})}-(X_{t_{-\xi_{Q_1,Q_2}}} \cup I_j^\circ) = \emptyset$, and $A_{L^{(i_t,\lvert V(T) \rvert+1,s+2)}}(V(M_{\alpha^*_{Q_1,Q_2}})) \cap (X_{V(T_{t})}-X_{t}) \cap X_{V(T_{t_1})}-X_{t_1} = \emptyset$, and subject to this, $\alpha^*_{Q_1,Q_2}=\alpha_{Q_2}$ if possible.

Let $M_1'',M_2'',\dots,M_k''$  (for some $k'' \in {\mathbb N}$) be the $\Se_j^\circ$-related monochromatic \linebreak $E_{j,t}^{(i_{t},\lvert V(G) \rvert)}$-pseudocomponents in $G[Y^{(i_{t},\lvert V(T) \rvert+1,s+2)}]$ such that for every $\alpha \in [k'']$, $V(M''_\alpha) \cap X_{t'} \neq \emptyset$ and $A_{L^{(i_{t'},-1,0)}}(V(M''_\alpha)) \cap X_{V(T_{t'})}-X_{t'} \neq \emptyset$.
Assume that $\sigma(M''_1)<\sigma(M''_2)<\dots<\sigma(M'_{k''})$.
Let $V(M''_\gamma)=\emptyset$ for every $\gamma \geq k''+1$.
Let $b''$ be the sequence 
$$(\lvert A_{L^{(i_t,\lvert V(T) \rvert+1,s+2)}}(V(M''_1)) \cap X_{V(T_{t'})}-(X_{t'} \cup I_j^\circ) \rvert, \dots, \lvert A_{L^{(i_t,\lvert V(T) \rvert+1,s+2)}}(V(M''_{w_0})) \cap X_{V(T_{t'})}-(X_{t'} \cup I_j^\circ) \rvert).$$

Let $m_1 := \min_{Q \in \Q_1} \alpha_Q$.
Let $m_2 := \min_{Q \in \Q_2} \alpha_Q$.
Since $\alpha_1' \leq \alpha_Q$ for every $Q \in \Q_1 \cup \Q_2$, $\alpha'_1 \leq \min\{m_1,m_2\}$.
So $\psi_2(\alpha_1',\alpha_2') \leq \psi_2(m_1,m_2)$.

We first assume that for every $Q_2 \in \Q_2$, $\alpha^*_{M_{m_1},Q_2} = \alpha_{Q_2}$.

Note that for all $Q_2 \in \Q_2$, $\kappa_{m_1,m_2}+\xi_{m_1,m_2} \leq \kappa_{m_1,\alpha_{Q_2}} + \xi_{m_1,\alpha_{Q_2}}$, and by the choices of $t^{M_{m_1},M_{m_2}}$ and $t^{M_{m_1},Q_2}$, $(t^{M_{m_1},M_{m_2}}_{\xi_{m_1,m_2}},\dots,t^{M_{m_1},M_{m_2}}_{-1}, t^{M_{m_1},M_{m_2}}_1,\dots,t^{M_{m_1},M_{m_2}}_{\kappa_{m_1,m_2}})$ is a suffix of $(t^{M_{m_1},Q_2}_{-\xi_{m_1,\alpha_{Q_2}}}, \dots, t^{M_{m_1},Q_2}_{-1}, \allowbreak t^{M_{m_1},Q_2}_1,\allowbreak \dots,t^{M_{m_1},Q_2}_{\kappa_{m_1,\alpha_{Q_2}}})$.
Since for every $Q_2 \in \Q_2$, $\alpha^*_{M_{m_1},Q_2} = \alpha_{Q_2}$, we know 
\begin{align*}
\bigcup_{Q_2 \in \Q_2}A_{L^{(i_{t},\lvert V(T) \rvert+1,s+2)}}(V(Q_2)) \cap (X_{V(T_{t'})} - X_{t'}) \cap X_{V(T_{t^{M_{m_1},M_{m_2}}_{1}})} - X_{t^{M_{m_1},M_{m_2}}_{1}} & = \emptyset \text{ and }\\
\bigcup_{Q_2 \in \Q_2}A_{L^{(i_{t},\lvert V(T) \rvert+1,s+2)}}(V(Q_2)) \cap (X_{V(T_{t'})} - X_{t'}) \cap X_{V(T_{t^{M_{m_1},M_{m_2}}_{-\xi_{m_1,m_2}}})} - (X_{t^{M_{m_1},M_{m_2}}_{-\xi_{m_1,m_2}}} \cup I_j^\circ) & = \emptyset.
\end{align*}
Since for each vertex in $(Y^{(i_{t'},-1,0)}-Y^{(i_t,\lvert V(T) \rvert+1,s+2)}) \cap I_j \cap X_{V(T_{t'})}$, it is in $X_{V(T_{t'})} \cap I_j - (X_{t'} \cup I_j^\circ)$, and when it gets colored, it decreases the lexicographic order of $b''$.
So by \cref{claim:sizebeltfull}, we loss at most $\eta_5 \cdot (f(\eta_5))^{w_0-1} \leq (f(\eta_5))^{w_0} \leq \psi_3(0,0)$ nodes in the blocker. 
That is, 
\begin{align*}
A_{L^{(i_{t'},-1,0)}}(V(M'_{\alpha_2'})) \cap (X_{V(T_{t'})} - X_{t'}) \cap (X_{V(T_{t^{M_{m_1},M_{m_2}}_{1}})} - X_{t^{M_{m_1},M_{m_2}}_{1}}) & = \emptyset\text{ and }\\
A_{L^{(i_{t'},-1,0)}}(V(M'_{\alpha_2'})) \cap (X_{V(T_{t'})} - X_{t'}) \cap X_{V(T_{t^{M_{m_1},M_{m_2}}_{-\xi_{m_1,m_2}+\psi_3(0,0)}})} - (X_{t^{M_{m_1},M_{m_2}}_{-\xi_{m_1,m_2}+\psi_3(0,0)}} \cup I_j^\circ) & = \emptyset.
\end{align*}
So $(t^{M_{m_1},M_{m_2}}_{-\xi_{m_1,m_2}+\psi_3(0,0)},\dots,t^{M_{m_1},M_{m_2}}_{-1}, t^{M_{m_1},M_{m_2}}_{1}, \dots,t^{M_{m_1},M_{m_2}}_{\kappa_{m_1,m_1}})$ contains a subsequence that is an $(\alpha'_1-1,\alpha'_1,\alpha_2',\psi_2(\alpha'_1,\alpha'_2),\psi_3(\alpha'_1,\alpha'_2)-\psi_3(0,0))$-blocker for $(j,t',-1,0,0,t')$.

Hence we may assume that there exists $Q_2^* \in \Q_2$ such that $m_1 = \alpha^*_{M_{m_1},Q_2^*} \neq \alpha_{Q_2^*}$.
Since for every $Q_1 \in \Q_1-\{M_{m_1}\}$, $\psi_{2}(\alpha_{Q_1},\alpha_{Q_2^*}) > \psi_2(m_1,\alpha_{Q_2^*}) + \psi_3(m_1,\alpha_{Q_2^*})$, we know $\alpha^*_{Q_1,Q_2^*}=\alpha_{Q_1}$, for otherwise $\alpha^*_{M_{m_1},Q_2^*}=\alpha_{Q_2^*}$ by the choice of $\alpha^*_{M_{m_1},Q_2^*}$.
So $(t^{M_{m_1},Q^*_2}_{-\xi_{m_1,\alpha_{Q_2^*}}},\dots,t^{M_{m_1},Q^*_2}_{-1},t^{M_{m_1},Q_2^*}_1,\dots,t^{M_{m_1},Q^*_2}_{\kappa_{m_1,\alpha_{Q_2^*}}})$ is a suffix of $(t^{Q_1,Q_2^*}_{-\xi_{\alpha_{Q_1},\alpha_{Q_2^*}}},\dots,t^{Q_1,Q_2^*}_{-1},t^{Q_1,Q_2^*}_1,\dots,t^{Q_1,Q_2^*}_{\kappa_{\alpha_{Q_1},\alpha_{Q_2^*}}})$ for every $Q_1 \in \Q_1$.
Note that for each vertex in $(Y^{(i_{t'},-1,0)}-Y^{(i_t,\lvert V(T) \rvert+1,s+2)}) \cap I_j \cap X_{V(T_{t'})}$, it is in $X_{V(T_{t'})} - (X_{t'} \cup I_j^\circ)$, and when it gets colored, it decreases the lexicographic order of $b''$.
So by \cref{claim:sizebeltfull}, we lose at most $\eta_5 \cdot (f(\eta_5))^{w_0-1} \leq (f(\eta_5))^{w_0} \leq \psi_3(0,0)$ nodes in the blocker. 
Hence $(t^{M_{m_1},Q_2^*}_{-\xi_{m_1,\alpha_{Q_2^*}}+\psi_3(0,0)},\dots,t^{M_{m_1},Q^*_2}_{\kappa_{m_1,\alpha_{Q_2^*}}})$ contains a subsequence that is an $(\alpha'_1-1,\alpha'_1,\alpha_2',\psi_2(\alpha'_1,\alpha'_2),\psi_3(\alpha'_1,\alpha'_2)-\psi_3(0,0))$-blocker for $(j,t',-1,0,0,t')$.
This proves the claim.
\end{proof}

\begin{claim} \label{claim:blocker2}
Let $j \in [\lvert \V \rvert-1]$.
Let $t \in V(T)$.
Let $M_1,M_2,\dots,M_k$ (for some $k \in {\mathbb N}$) be the $\Se_j^\circ$-related monochromatic $E_{j,t}^{(i_t)}$-pseudocomponents in $G[Y^{(i_t,-1,0)}]$ such that for every $\alpha \in [k]$, $V(M_\alpha) \cap X_t \neq \emptyset$ and $A_{L^{(i_t,-1,0)}}(V(M_\alpha)) \cap X_{V(T_t)}-X_t \neq \emptyset$.
Assume that $\sigma(M_1)<\sigma(M_2)<\dots<\sigma(M_k)$.
Let $V(M_\gamma)=\emptyset$ for every $\gamma \geq k+1$.
Let $\alpha_1,\alpha_2 \in [2,k]$ with $\alpha_1<\alpha_2$.

Assume that the following hold.
	\begin{itemize}
		\item for every $\alpha \in [2,\alpha_1]$ and $\alpha' \in [\alpha+1,k]$, if:
			\begin{itemize}
				\item either $\alpha \leq \alpha_1-1$ or $(\alpha,\alpha')=(\alpha_1,\alpha_2)$,
				\item $M_\alpha$ and $M_{\alpha'}$ have the same color, and 
				\item the $s$-segment in $\Se_j^\circ$ containing $V(M_\alpha)$ whose level equals the color of $M_\alpha$ contains $V(M_{\alpha'})$,
			\end{itemize}
			then either:
			\begin{itemize}
				\item there exists an $(\alpha-1,\alpha,\alpha',\psi_2(\alpha,\alpha'), \psi_3(\alpha,\alpha')-\psi_3(0,0))$-blocker for $(j,t,-1,0,0,t)$, or  
				\item there exist $t^* \in V(T)$ with $i_{t^*}<i_t$ and $t \in V(T_{j,t^*})$, a witness $q' \in \partial T_{j,t^*}$ for $X_{q'} \cap I_j \subseteq W_3^{(i_{t^*},\ell)}$ for some $\ell \in [0,\lvert V(T) \rvert]$, and a monochromatic $E_{j,t^*}^{(i_{t^*})}$-pseudocomponent in $G[Y^{(i_{t^*},\lvert V(T) \rvert+1,s+2)}]$ intersecting $X_{q'}$ and $V(M_\alpha) \cup V(M_{\alpha'})$. 
			\end{itemize}
		\item $(z,M_z)$ is a nice pair and an outer-safe pair for every $z \in V(T)$ and monochromatic $E_{j,z}^{(i_z)}$-pseudocomponent $M_z$ in $G[Y^{(i_z,-1,0)}]$ with $V(M_z) \cap X_z \neq \emptyset$ and $\sigma(M_z)<\sigma(M_{\alpha_1})$.
	\end{itemize}
Then either:
	\begin{itemize}
		\item there exists no $c$-monochromatic path $P$ in $G$ from $V(M_{\alpha_1}) \cap V(P)$ to $V(M_{\alpha_2}) \cap V(P)$ internally disjoint from $\bigcup_{\gamma=1}^{w_0}V(M_{\gamma})$ such that $V(P) \cap A_{L^{(i_t,-1,0)}}(V(M_{\alpha_1})) \cap X_{V(T_t)}-X_t \neq \emptyset$, or 
		\item $(t,M_{\alpha_1})$ is not an outer-safe pair. 
	\end{itemize}
\end{claim}

\begin{proof}
Suppose that this claim does not hold.
Among all counterexamples, we choose $(t,M_{\alpha_1})$ such that $\sigma(M_{\alpha_1})$ is as small as possible, and subject to this, $i_t$ is as large as possible.

Hence $(t,M_{\alpha_1})$ is an outer-safe pair, and there exists a $c$-monochromatic path $P$ in $G$ from $V(M_{\alpha_1}) \cap V(P)$ to $V(M_{\alpha_2}) \cap V(P)$ internally disjoint from $\bigcup_{\gamma=1}^{w_0}V(M_{\gamma})$ such that \linebreak $V(P) \cap A_{L^{(i_t,-1,0)}}(V(M_{\alpha_1})) \cap X_{V(T_t)}-X_t \neq \emptyset$. 

Since for every $\alpha \in [2,\alpha_1-1]$, $(t,M_\alpha)$ is an outer-safe pair, applying this claim to $M_\alpha$ implies that for every $\alpha \in [2,\alpha_1-1]$ and $\alpha' \in [\alpha+1,k]$, there exists no $c$-monochromatic path $P'$ in $G$ from $V(M_{\alpha}) \cap V(P')$ to $V(M_{\alpha'}) \cap V(P')$ internally disjoint from $\bigcup_{\gamma=1}^{w_0}V(M_{\gamma})$ such that $V(P') \cap A_{L^{(i_t,-1,0)}}(V(M_{\alpha})) \cap X_{V(T_t)}-X_t \neq \emptyset$. 
So we have the following, where the case $\alpha=1$ follows from \cref{claim:1stdone}.
	\begin{itemize}
		\item[(a)] For every $\alpha \in [\alpha_1-1]$ and $\alpha' \in [\alpha+1,k]$, there exists no $c$-monochromatic path $P'$ in $G$ from $V(M_{\alpha}) \cap V(P')$ to $V(M_{\alpha'}) \cap V(P')$ internally disjoint from $\bigcup_{\gamma=1}^{w_0}V(M_{\gamma})$ such that $V(P') \cap A_{L^{(i_t,-1,0)}}(V(M_{\alpha})) \cap X_{V(T_t)}-X_t \neq \emptyset$. 
	\end{itemize}

Suppose there exist $\alpha \in [2,\alpha_1]$ and $\alpha' \in [\alpha+1,k]$ with either $\alpha \leq \alpha_1-1$ or $(\alpha,\alpha')=(\alpha_1,\alpha_2)$, such that $M_\alpha$ and $M_{\alpha'}$ have the same color, and the $s$-segment in $\Se_j^\circ$ containing $V(M_\alpha)$ whose level equals the color of $M_\alpha$ contains $V(M_{\alpha'})$, and there exists no $(\alpha-1,\alpha,\alpha',\psi_2(\alpha,\alpha'), \psi_3(\alpha,\alpha')-\psi_3(0,0))$-blocker for $(j,t,-1,0,0,t)$, and for every $t^* \in V(T)$ with $i_{t^*}<i_t$ and $t \in V(T_{j,t^*})$, if there exist a witness $q' \in \partial T_{j,t^*}$ for $X_{q'} \cap I_j \subseteq W_3^{(i_{t^*},\ell)}$ for some $\ell \in [0,\lvert V(T) \rvert]$ and a monochromatic $E_{j,t^*}^{(i_{t^*})}$-pseudocomponent in $G[Y^{(i_{t^*},\lvert V(T) \rvert+1,s+2)}]$ intersecting $X_{q'}$ and $V(M_\alpha) \cup V(M_{\alpha'})$, then $t \in \partial T_{j,t^*}$.
By our assumption, such $t^*$ exists, and we choose $t^*$ such that $i_{t^*}$ is as small as possible.
So $t \in \partial T_{j,t^*}$.
By \cref{claim:isolation,claim:pseudocomptouchingW3}, $t$ is a witess for $X_t \cap I_j \subseteq W_3^{(i_{t^*},\ell)}$ for some $\ell \in [0,\lvert V(T) \rvert]$.
Let $t_1^*$ be the node of $T$ with $t \in V(T_{j,t_1^*})$ such that some monochromatic $E_{j,t_1^*}^{(i_{t_1^*})}$-pseudocomponent in $G[Y^{(i_{t_1^*},\lvert V(T) \rvert+1,s+2)}]$ intersecting $X_{q'}$ and $V(M_{1})$ for some witness $q' \in \partial T_{j,t_1^*}$ for $X_{q'} \cap I_j \subseteq W_3^{(i_{t_1^*},\ell)}$ for some $\ell \in [0,\lvert V(T) \rvert]$, and subject to this, $i_{t_1^*}$ is minimum.
Since $V(M_1) \cap X_t \neq \emptyset$, $t^*$ is a candidate of $t_1^*$.
So $i_{t_1^*} \leq i_{t^*}$ and hence $t_1^* \neq t$.
Since $\sigma(M_1)<\sigma(M_{\alpha_1})$, $(t,M_1)$ is a nice pair, so by \cref{claim:underassumptouch2}, there exists a monochromatic $E_{j,t_1^*}^{(i_{t_1^*},\lvert V(G) \rvert)}$-pseudocomponent $M^*$ in $G[Y^{(i_{t_1^*},\lvert V(T) \rvert+1,s+2)}]$ intersecting $V(M_1)$ such that $\sigma(M^*)=\sigma(M_1)$, $V(M^*) \cap X_{t} \neq \emptyset$ and satisfies $A_{L^{(i_{t_1^*},\lvert V(T) \rvert+1,s+2)}}(V(M^*)) \cap X_{V(T_t)}-X_t \neq \emptyset$.
Since $i_{t^*}<i_t$ and $\alpha \neq \alpha'$, there exists no element in $E_{j,t^*}^{(i_{t^*},\lvert V(G) \rvert)}$ intersecting both $M_\alpha$ and $M_{\alpha'}$.
Since $M_\alpha \cap X_t \neq \emptyset \neq M_{\alpha'} \cap X_{t}$ and $t \in \partial T_{j,t^*}$ is a witness for $X_t \cap I_j \subseteq W_3^{(i_{t^*},\ell)}$ for some $\ell \in [0,\lvert V(T) \rvert]$, by the construction of $E_{j,t^*}^{(i_{t^*},\lvert V(G) \rvert)}$ and the existence of $M_1$ and the minimality of $t^*$, we know $V(M_{\alpha} \cup M_{\alpha'}) \cap X_{t^*} \neq \emptyset$ (so $V(M_{\alpha}) \cap X_{t^*} \neq \emptyset$ since $\alpha<\alpha'$) and there exist blockers for $(j,t^*,\lvert V(T) \rvert+1,s+2,\beta,t)$ (for some $\beta \in [0,\lvert V(G) \rvert]$) as stated in the assumption of \cref{claim:blocker2-0} (where the $M'_{\alpha'_1}$ and $M'_{\alpha'_2}$ in \cref{claim:blocker2-0} are replaced by $M_\alpha$ and $M_{\alpha'}$, respectively).
Since $\alpha \leq \alpha_1$, the assumption of this claim implies that condition (ii) in \cref{claim:blocker2-0} holds. 
So there is an $(\alpha-1,\alpha,\alpha',\psi_2(\alpha,\alpha'),\psi_3(\alpha,\alpha')-\psi_3(0,0))$-blocker for $(j,t,-1,0,0,t)$ by \cref{claim:blocker2-0}, a contradiction.

Therefore, we have 
	\begin{itemize}
		\item[(b)] for every $\alpha \in [2,\alpha_1]$ and $\alpha' \in [\alpha+1,k]$, if:
		\begin{itemize}
			\item either $\alpha \leq \alpha_1-1$ or $(\alpha,\alpha')=(\alpha_1,\alpha_2)$,
			\item $M_\alpha$ and $M_{\alpha'}$ have the same color, and 
			\item the $s$-segment in $\Se_j^\circ$ containing $V(M_\alpha)$ whose level equals the color of $M_\alpha$ contains $V(M_{\alpha'})$,
		\end{itemize}
		then either:
		\begin{itemize}
			\item there exists an $(\alpha-1,\alpha,\alpha',\psi_2(\alpha,\alpha'), \psi_3(\alpha,\alpha')-\psi_3(0,0))$-blocker for $(j,t,-1,0,0,t)$, or 
			\item there exist $t^* \in V(T)$ with $i_{t^*}<i_t$ and $t \in V(T_{j,t^*})-\partial T_{j,t^*}$, a witness $q' \in \partial T_{j,t^*}$ for $X_{q'} \cap I_j \subseteq W_3^{(i_{t^*},\ell)}$ for some $\ell \in [0,\lvert V(T) \rvert]$, and a monochromatic $E_{j,t^*}^{(i_{t^*})}$-pseudocomponent in $G[Y^{(i_{t^*},\lvert V(T) \rvert+1,s+2)}]$ intersecting $X_{q'}$ and $V(M_\alpha) \cup V(M_{\alpha'})$.
		\end{itemize}
	\end{itemize}

Now we prove that:
	\begin{itemize}
		\item[(c)] there exist no $t^* \in V(T)$ with $i_{t^*}<i_t$ and $t \in V(T_{j,t^*})-\partial T_{j,t^*}$, a witness $q' \in \partial T_{j,t^*}$ for $X_{q'} \cap I_j \subseteq W_3^{(i_{t^*},\ell)}$ for some $\ell \in [0,\lvert V(T) \rvert]$, and a monochromatic $E_{j,t^*}^{(i_{t^*})}$-pseudocomponent in $G[Y^{(i_{t^*},\lvert V(T) \rvert+1,s+2)}]$ intersecting $X_{q'}$ and $V(M_{\alpha_1}) \cup V(M_{\alpha_2})$.
	\end{itemize}

Suppose to the contrary that (c) does not hold.
So there exist $t^* \in V(T)$ with $i_{t^*}<i_t$ and $t \in V(T_{j,t^*})-\partial T_{j,t^*}$, a witness $q' \in \partial T_{j,t^*}$ for $X_{q'} \cap I_j \subseteq W_3^{(i_{t^*},\ell)}$ for some $\ell \in [0,\lvert V(T) \rvert]$, and a monochromatic $E_{j,t^*}^{(i_{t^*})}$-pseudocomponent in $G[Y^{(i_{t^*},\lvert V(T) \rvert+1,s+2)}]$ intersecting $X_{q'}$ and $V(M_{\alpha_1}) \cup V(M_{\alpha_2})$.
We choose $t^*$ such that $i_{t^*}$ is as small as possible.
For each $\gamma \in [2]$, let $t^*_{\alpha_\gamma}$ be the node of $T$ with $t \in V(T_{j,t_{\alpha_\gamma}^*})$ such that some monochromatic $E_{j,t_{\alpha_\gamma}^*}^{(i_{t_{\alpha_\gamma}^*})}$-pseudocomponent in $G[Y^{(i_{t_{\alpha_\gamma}^*},\lvert V(T) \rvert+1,s+2)}]$ intersecting $X_{q'}$ and $V(M_{\alpha_\gamma})$ for some witness $q' \in \partial T_{j,t_{\alpha_\gamma}^*}$ for $X_{q'} \cap I_j \subseteq W_3^{(i_{t_{\alpha_\gamma}^*},\ell)}$ for some $\ell \in [0,\lvert V(T) \rvert]$, and subject to this, $i_{t_{\alpha_\gamma}^*}$ is minimum.
Note that $t^*$ is a candidate for $t^*_{\alpha_\gamma}$ for some $\gamma \in [2]$.
So $i_{t^*} \geq \min\{i_{t^*_{\alpha_1}},i_{t^*_{\alpha_2}}\}$.
Let $t_0 \in \{t^*_{\alpha_1},t^*_{\alpha_2}\}$ with $i_{t_0} = \min\{i_{t^*_{\alpha_1}},i_{t^*_{\alpha_2}}\}$.
If $i_{t_0}<i_{t^*}$, then since $t \in V(T_{j,t^*})-\partial T_{j,t^*}$, $t \in V(T_{j,t_0})-\partial T_{j,t_0}$, so $t_0$ is a candidate for $t^*$.
Hence $t^*=t^*_{\alpha_\gamma}$ for some $\gamma \in [2]$.
By \cref{claim:isolation,claim:pseudocomptouchingW3} and by tracing along $P$, we have $t^*=t^*_{\alpha_1}=t^*_{\alpha_2}$.
Since $\alpha_1 \neq \alpha_2$ and $(t,M_{\alpha_1})$ is outer-safe, by (b) and \cref{claim:underassumpunder1}, there exist $t' \in \partial T_{j,t^*_{\alpha_1}} \cap V(T_t)$, $u,v \in X_{t'}$, a $c$-monochromatic path $P'$ in $G[X_{V(T_{t'})}]$ but not in $G[Y^{(i_t,-1,0)}]$ from $u$ to $v$ internally disjoint from $X_{t'}$ such that $t'$ is a witness for $X_{t'} \cap I_j \subseteq W_3^{(i_{t^*},\ell)}$ for some $\ell \in [0,\lvert V(T) \rvert]$, and there exists a monochromatic $E_{j,t^*}^{(i_{t^*},\lvert V(G) \rvert)}$-pseudocomponent $M^*$ in $G[Y^{(i_{t^*},\lvert V(T) \rvert+1,s+2)}]$ with $\sigma(M^*)=\sigma(M_{\alpha_1})$ containing $u$ but not containing $v$.
Since $t \in V(T_{j,t^*})-\partial T_{j,t^*}$, $t' \in V(T_t)-\{t\}$.
Let $M_0$ be the monochromatic $E_{j,t'}^{(i_{t'})}$-pseudocomponent in $G[Y^{(i_{t'},-1,0)}]$ containing $M^*$.
Since $t' \in \partial T_{j,t^*}$, by \cref{claim:isolation,claim:pseudocomptouchingW3}, $t^*$ is the node of $T$ with $t' \in V(T_{j,t^*})$ such that there exist a witness $q'' \in \partial T_{j,t^*}$ for $X_{q''} \cap I_j \subseteq W_3^{(i_{t^*},\ell)}$ for some $\ell \in [0,\lvert V(T) \rvert]$, and a monochromatic $E_{j,t^*}^{(i_{t^*})}$-pseudocomponent in $G[Y^{(i_{t^*},\lvert V(T) \rvert+1,s+2)}]$ intersecting $X_{q''}$ and $V(M_0)$, and subject to this, $i_{t^*}$ is minimum.

Note that there exists a tuple $(q,q^*,Q,u_q,v_q,P_q,Q^*)$ satisfying the following statement, since $(t',t^*,M_0,u,v,P',M^*)$ is a candidate for $(q,q^*,Q,u_q,v_q,P_{q},Q^*)$:
$q \in \partial T_{j,q^*} \cap V(T_t)-\{t\}$, $Q$ is a monochromatic $E_{j,q}^{(i_q)}$-pseudocomponent in $G[Y^{(i_q,-1,0)}]$ intersecting $X_q$, $u_q,v_q \in X_{q}$, $P_q$ is a $c$-monochromatic path in $G[X_{V(T_{q})}]$ but not in $G[Y^{(i_{q^*},\lvert V(T) \rvert+1,s+2)}]$ from $u_q$ to $v_q$ internally disjoint from $X_{q}$, and $Q^*$ is a monochromatic $E_{j,q^*}^{(i_{q^*},\lvert V(G) \rvert)}$-pseudocomponent in $G[Y^{(i_{q^*},\lvert V(T) \rvert+1,s+2)}]$ with $\sigma(Q^*) \leq \sigma(M_{\alpha_1})$ containing $u_q$ but not containing $v_q$ such that $Q \supseteq Q^*$, $T_{j,t^*} \subseteq T_{j,q^*}$ and $q$ is a witenss for $X_q \cap I_j \subseteq W_3^{(i_{q^*},\ell')}$ for some $\ell' \in [0,\lvert V(T) \rvert]$, and $q^*$ is the node with $i_{q^*}<i_{q}$ and $q \in V(T_{j,q^*})$ such that there exist a witness $q'' \in \partial T_{j,q^*}$ for $X_{q''} \cap I_j \subseteq W_3^{(i_{q^*},\ell)}$ for some $\ell \in [0,\lvert V(T) \rvert]$ and a monochromatic $E_{j,q^*}^{(i_{q^*})}$-pseudocomponent in $G[Y^{(i_{q^*},\lvert V(T) \rvert+1,s+2)}]$ intersecting $X_{q''}$ and $V(Q)$, and subject to this, $i_{q^*}$ is minimum.
We further choose them such that $\sigma(Q^*)$ is as small as possible.
Note that $\sigma(Q) \leq \sigma(Q^*) \leq \sigma(M_{\alpha_1})$.

Suppose $V(P_q) \subseteq Y^{(i_q,-1,0)}$.
Since $\sigma(Q) \leq \sigma(M_{\alpha_1})$, we can apply \cref{claim:underassumpnounder} by taking $(t,M)$ to be $(q,Q)$.
By applying \cref{claim:underassumpnounder} by taking $(t,M)$ to be $(q,Q)$, the non-existence of the node $t'$ mentioned in the conclusion of \cref{claim:underassumpnounder} and the existence of $q$ here imply that $\sigma(Q^*)>\sigma(Q)$.
So by \cref{claim:isolation,claim:pseudocomptouchingW3,claim:orderpreserving}, there exists a monochromatic $E_{j,q^*}^{(i_{q^*})}$-pseudocomponent $Q''$ in $G[Y^{(i_{q^*},\lvert V(T) \rvert+1,s+2)}]$ with $\sigma(Q'')=\sigma(Q)<\sigma(Q^*)$.
Since $Q''$ and $Q^*$ are disjoint subgraphs of $Q$, by \cref{claim:isolation,claim:pseudocomptouchingW3}, there exist $q'' \in \partial T_{j,q^*}$, $u_{q''},v_{q''} \in X_{q''}$, a monochromatic path $P_{q''}$ in $G[Y^{(i_{q},-1,0)} \cap X_{V(T_{q''})}]$ but not in $G[Y^{(i_{q^*},\lvert V(T) \rvert+1,s+2)}]$ from $u_{q''}$ to $v_{q''}$ internally disjoint from $X_{q''}$ such that $Q''$ contains $u_{q''}$ but not $v_{q''}$.
Since $\sigma(Q)<\sigma(Q^*) \leq \sigma(M_{\alpha_1})$, $(q,Q)$ is outer-safe, so $q'' \in V(T_q)$.
Since $q \in \partial T_{j,q^*}$, $q''=q$.
Since $\sigma(Q'')<\sigma(Q^*)$, $Q''$ contradicts the choice of $Q^*$.

So $V(P_q) \not \subseteq Y^{(i_q,-1,0)}$.
By the assumption of this claim, $(q,Q)$ is a nice pair.
Hence, by \cref{claim:underassumptouch2} (taking $(t,M)=(q,Q)$) and \cref{claim:1stdone}, since $V(P_q) \not \subseteq Y^{(i_q,-1,0)}$, there exists a monochromatic $E_{j,q^*}^{(i_{q^*},\lvert V(G) \rvert)}$-pseudocomponent $Q_1^*$ in $G[Y^{(i_{q^*},\lvert V(T) \rvert +1,s+2)}]$ such that $\sigma(Q_1^*)<\sigma(Q^*)$, $V(Q_1^*) \cap X_{q} \neq \emptyset$ and $A_{L^{(i_{q^*},\lvert V(T) \rvert+1,s+2)}}(V(Q_1^*)) \cap X_{V(T_{q})}-X_{q} \neq \emptyset$.
Let $R^*$ be the monochromatic $E_{j,q}^{(i_{q})}$-pseudocomponent in $G[Y^{(i_{q},-1,0)}]$ containing $Q^*$.
Let $R^*_1,\dots,R^*_\beta$ (for some $\beta \in {\mathbb N}$) be the monochromatic $E_{j,q}^{(i_{q})}$-pseudocomponents $M'$ in $G[Y^{(i_{q},-1,0)}]$ intersecting $X_{q}$ with $A_{L^{(i_{q},-1,0)}}(V(M')) \cap X_{V(T_{q})}-X_{q} \neq \emptyset$ and $\sigma(M')<\sigma(R^*)$.
Assume $\sigma(R^*_1)<\sigma(R^*_2)<\dots<\sigma(R^*_\beta)$.
Let $R^*_{\beta+1}=R^*$.
Since $u_q \in V(Q^*)$ and $v_q \not \in V(Q^*)$, $\{u_q,v_q\} \not \in E_{j,q^*}^{(i_{q^*},\lvert V(G) \rvert)}$.
By the existence of $Q_1^*$ and the construction of $E_{j,q^*}^{(i_{q^*},\lvert V(G) \rvert)}$, $V(Q^*) \cap X_{q^*} \neq \emptyset$.
Since $\sigma(Q^*) \leq \sigma(M_{\alpha_1})$, for every $\alpha \in [\beta]$, $R^*_\alpha$ contains $M_\alpha$ for some $\alpha \in [\alpha_1-1]$ by \cref{claim:orderpreserving}.
So by the assumption of this claim and \cref{claim:blocker2-0}, the assumptions of this claim hold when $t$ is replaced by $q$ and $M_{\alpha_1}$ is replaced by $R^*_{\beta+1}=R^*$. 
Since $\sigma(R^*) \leq \sigma(M_{\alpha_1})$ and $i_{q}>i_t$, this claim is applicable to $q$ and $R^*$, and by applying it, the existence of $P_q$ implies that $(q,R^*)$ is not an outer-safe pair.
So $\sigma(R^*)=\sigma(M_{\alpha_1}) \geq \sigma(Q^*)$.
So $\sigma(R^*)=\sigma(Q^*)$ and $R^*=Q \supseteq M_{\alpha_1}$.
Hence $t^*=q^*$ by \cref{claim:isolation,claim:pseudocomptouchingW3}.
Since $(q,R^*)$ is not an outer-safe pair, there exist $t'' \in \partial T_{j,q^*} -V(T_{q})$, $u'',v'' \in X_{t''}$, a monochromatic path $P''$ in $G[Y^{(i_{q},-1,0)} \cap X_{V(T_{t''})}]$ but not in $G[Y^{(i_{q^*},\lvert V(T) \rvert+1,s+2)}]$ from $u''$ to $v''$ internally disjoint from $X_{t''}$, and a monochromatic $E_{j,q^*}^{(q^*,\lvert V(G) \rvert)}$-pseudocomponent $M''$ in $G[Y^{(i_{q^*},\lvert V(T) \rvert+1,s+2)}]$ with $\sigma(M'')=\sigma(R^*) = \sigma(M_{\alpha_1})$ containing $u''$ but not $v''$.
Since $q \in V(T_t)$ and $V(P'') \subseteq Y^{(i_q,-1,0)} \cap X_{V(T_{t''})}$, if $t'' \not \in V(T_t)-\{t\}$, then $V(P'') \subseteq Y^{(i_t,-1,0)}$, so it contradicts that $(t,M_{\alpha_1})$ is an outer-safe pair. 
So $t'' \in \partial T_{j,q^*} \cap (V(T_t)-\{t\})-V(T_{q})$.
In particular, $t'' \neq q$.

If $V(P'') \not \subseteq Y^{(i_{t''},-1,0)}$, then by \cref{claim:1stdone,claim:underassumptouch2}, the fact that $\{u'',v''\} \not \in E_{j,q^*}^{(i_{q^*},\lvert V(G) \rvert)}$, the existence of $P_q$ and $P_{t''}$, and the construction of $E_{j,q^*}^{(i_{q^*},\lvert V(G) \rvert)}$, we know $A_{L^{(i_{q^*},\lvert V(T) \rvert+1,s+2)}}(V(Q_1^*)) \cap X_{V(T_{q^*})}-X_{q^*} \subseteq (X_{V(T_{q})}-X_{q}) \cap (X_{V(T_{t''})}-X_{t''})$, so $t''=q$, a contradiction.
Hence $V(P'') \subseteq Y^{(i_{t''},-1,0)}$.
Since $\sigma(M'')=\sigma(R^*)=\sigma(Q^*)$, $M''=Q^*$.
Let $M'''$ be the monochromatic $E_{j,t''}^{(i_{t''})}$-pseudocomponent in $G[Y^{(i_{t''},-1,0)}]$ containing $M''$.
Since $\sigma(M'')=\sigma(Q^*)=\sigma(M_{\alpha_1})$, $M'''$ contains $M_{\alpha_1}$, so $t^*=q^*$ is the node in $T$ with $t'' \in \partial T_{j,t^*}$ such that there exist a witness $q' \in \partial T_{j,t^*}$ for $X_{q'} \cap I_j \subseteq W_3^{(i_{t^*},\ell)}$ for some $\ell \in [0,\lvert V(T) \rvert]$ and a monochromatic $E_{j,t^*}^{(i_{t^*})}$-pseudocomponent in $G[Y^{(i_{t^*},\lvert V(T) \rvert+1,s+2)}]$ intersecting $X_{q'}$ and $V(M''')$.

Since $\sigma(M'')=\sigma(Q^*)=\sigma(M_{\alpha_1})$, we can apply \cref{claim:underassumpnounder} by taking $(t,M)$ to be $(t'',M''')$.
By applying \cref{claim:underassumpnounder} by taking $(t,M)$ to be $(t'',M''')$, the non-existence of the node $t'$ mentioned in \cref{claim:underassumpnounder} and the existence of $t''$ here imply that $\sigma(Q^*)=\sigma(M'')>\sigma(M''')$.
So by \cref{claim:isolation,claim:pseudocomptouchingW3,claim:orderpreserving}, there exists a monochromatic $E_{j,q^*}^{(i_{q^*},\lvert V(G) \rvert)}$-pseudocomponent $Q'''$ in $G[Y^{(i_{q^*},\lvert V(T) \rvert+1,s+2)}]$ with $\sigma(Q''')=\sigma(M''')<\sigma(M'')=\sigma(Q^*)$.
Hence there exist $q''' \in \partial T_{j,q^*}$, $u_{q'''},v_{q'''} \in X_{q'''}$, a monochromatic path $P_{q'''}$ in $G[Y^{(i_{t''},-1,0)} \cap X_{V(T_{q'''})}]$ but not in $G[Y^{(i_{q^*},\lvert V(T) \rvert+1,s+2)}]$ from $u_{q'''}$ to $v_{q'''}$ internally disjoint from $X_{q'''}$ such that $Q'''$ contains $u_{q'''}$ but not $v_{q'''}$.
Since $\sigma(Q''')<\sigma(Q^*) \leq \sigma(M_{\alpha_1})$, $(t'',M''')$ is outer-safe, so $q''' \in V(T_{t''})$.
Hence $q'''=t''$.
Since $\sigma(Q''')<\sigma(Q^*)$, $Q'''$ contradicts the choice of $Q^*$.

This proves (c).

By the existence of $P$, $M_{\alpha_1}$ and $M_{\alpha_2}$ have the same color, and the $s$-segment in $\Se_j^\circ$ containing $V(M_{\alpha_1})$ whose level equals the color of $M_{\alpha_1}$ contains $V(M_{\alpha_2})$.
By (b) and (c), there exists an $(\alpha_1-1,\alpha_1,\alpha_2,\kappa,\xi)$-blocker $(t_{-\xi},t_{-\xi+1}, \dots, t_{-1}, t_1,t_2,\dots,t_\kappa)$ for $(j,t,-1,0,0,t)$, where $\kappa = \psi_2(\alpha_1,\alpha_2)$ and $\xi=\psi_3(\alpha_1,\alpha_2)-\psi_3(0,0)$.
So $\bigcup_{\alpha=1}^{\alpha_1-1}A_{L^{(i_t,-1,0)}}(V(M_\alpha)) \cap X_{V(T_t)}-X_t \subseteq (X_{V(T_{t_\kappa})}-X_{t_\kappa}) \cap I_j^\circ$.
Let $\alpha^*$ be the element in $\{\alpha_1,\alpha_2\}$ such that $A_{L^{(i_t,-1,0)}}(V(M_{\alpha^*})) \cap (X_{V(T_t)}-(X_t \cup I_j^\circ)) \cap X_{V(T_{t_{-\xi}})}-X_{t_{-\xi}} = \emptyset$ (if $\xi \neq 0$), and $A_{L^{(i_t,-1,0)}}(V(M_{\alpha^*})) \cap (X_{V(T_t)}-X_t) \cap X_{V(T_{t_1})}-X_{t_1} = \emptyset$.

Now we prove
	\begin{itemize}
		\item[(d)] $t_{\eta_5^3} \not \in V(T_{j,t})$.
	\end{itemize}

Suppose to the contrary that $t_{\eta_5^3} \in V(T_{j,t})$.

Suppose $(Y^{(i_t,-1,\alpha_1-1)}-Y^{(i_t,-1,0)}) \not \subseteq X_{V(T_{t_{\kappa-\alpha_1 \eta_5^2}})}$.
By \cref{claim:sizebeltfull}, since $\bigcup_{\alpha=1}^{\alpha_1-1}A_{L^{(i_t,-1,0)}}(V(M_\alpha)) \cap X_{V(T_t)}-X_t \subseteq (X_{V(T_{t_\kappa})}-X_{t_\kappa}) \cap I_j^\circ$ and $\{(X_{t_\gamma}-X_{t}) \cap I_j: \gamma \in [-\xi,\kappa]-\{0\}\}$ consists of pariwise disjoint members, there exist $\alpha \in [\alpha_1-1]$, $\alpha' \in [\alpha+1,w_0]$ and a monochromatic path $P'$ in $G[Y^{(i_t,-1,\alpha_1-1)}]$ from $V(M_\alpha)$ to $V(M_{\alpha'})$ internally disjoint from $\bigcup_{\gamma=1}^{w_0}V(M_\gamma)$ such that all internal vertices of $P'$ are in $Y^{(i_t,-1,0)} \cup (I_j^\circ \cap X_{V(T_{t_{\kappa-\alpha_1 \eta_5^2}})}-X_{t_{\kappa-\alpha_1 \eta_5^2}})$ and $V(P') \cap A_{L^{(i_t,-1,0)}}(V(M_\alpha)) \cap X_{V(T_{t})}-X_t \neq \emptyset$, contradicting (a).

Hence $(Y^{(i_t,-1,\alpha_1-1)}-Y^{(i_t,-1,0)}) \subseteq X_{V(T_{t_{\kappa-\alpha_1 \eta_5^2}})}$.
Since $A_{L^{(i_t,-1,0)}}(V(M_{\alpha^*})) \cap (X_{V(T_t)}-X_t) \cap (X_{V(T_{t_1})}-X_{t_1}) = \emptyset$ and $\kappa-\alpha_1\eta_5^2 \geq 2$, $M_{\alpha^*}$ is a monochromatic $E_{j,t}^{(i_t)}$-pseudocomponent in $G[Y^{(i_t,-1,\alpha_1-1)}]$.
So $V(P) \not \subseteq Y^{(i_t,-1,\alpha_1-1)}$.
Let $M'$ be the monochromatic $E_{j,t}^{(i_t)}$-pseudocomponent in $G[Y^{(i_t,-1,\alpha_1-1)}]$ containing $M_{\alpha_1}$.
Since $V(P) \not \subseteq Y^{(i_t,-1,\alpha_1-1)}$, there exists \linebreak $u \in V(P) \cap A_{L^{(i_{t^*},-1,\alpha_1-1)}}(V(M')) \cap X_{V(T_t)}-X_t$.
Then either $u$ is colored in $Y^{(i_t,-1,\alpha_1)}$ by a color different from $c(P)$, or $u \not \in Y^{(i_t,\lvert V(T) \rvert+1,s+2)}$.
The former cannot hold, so the latter holds.
So $u \not \in Z_t$.
Hence there exists $q_u \in \partial T_{j,t}$ with $u \in I_j^\circ \cap X_{V(T_{q_u})}-X_{q_u}$.

Let $M''$ be the monochromatic $E_{j,t}^{(i_t)}$-pseudocomponent in $G[Y^{(i_t,\lvert V(T) \rvert+1,s+2)}]$ containing $M_{\alpha^*}$.
Since $u \not \in Y^{(i_t,\lvert V(T) \rvert+1,s+2)}$, there exists $v \in V(P) \cap A_{L^{(i_t,\lvert V(T) \rvert+1,s+2)}}(V(M'')) \cap X_{V(T_t)}-X_t$.
We choose $v$ such that $\sigma(v)$ is as small as possible.
Since $P$ is internally disjoint from $X_t$, and $A_{L^{(i_t,-1,0)}}(M_{\alpha^*}) \cap X_{V(T_t)}-X_t \subseteq V(G)-(X_{V(T_{t_1})}-X_{t_1})$, by \cref{claim:sizebeltfull}, $v \in V(G)-X_{V(T_{t_{\eta_5^2}})}$.
Note that $v \not \in Z_t$.
Hence there exists $q_v \in \partial T_{j,t}$ with $v \in X_{V(T_{q_v})}-X_{q_v}$.

By \cref{claim:sizebeltfull}, $\lvert Y^{(i_t,\lvert V(T) \rvert+1,s+2)} \cap I_j \rvert \leq \eta_5$. 
Since $t_{\eta_5^3} \in V(T_{j,t})$, there exist $\eta_5^2+1 \leq \beta_1 < \beta_2 \leq \eta_5^3-\alpha_1\eta_5^2$ with $\beta_2-\beta_1 \geq \frac{\eta_5^3-(\alpha_1+1)\eta_5^2}{\eta_5+1} > \eta_5$ such that $Y^{(i_t,\lvert V(T) \rvert+1,s+2)} \cap X_{V(T_{t_{\beta_1}})}-X_{V(T_{t_{\beta_2}})} = \emptyset$.
Since $v \in V(G)-X_{V(T_{t_{\eta_5^2}})}$, $q_v \in V(T)-V(T_{t_{\beta_1}})$.

Suppose that $q_u \in V(T_{t_{\beta_1}})$.
So $\alpha^*=\alpha_2$.
Since $Y^{(i_t,\lvert V(T) \rvert+1,s+2)} \cap X_{V(T_{t_{\beta_1}})}-X_{V(T_{t_{\beta_2}})} = \emptyset$ and $X_{q_u} \cap I_j \subseteq Y^{(i_t,\lvert V(T) \rvert+1,s+2)}$, $q_u \in V(T_{t_{\beta_2}})$.
By the existence of $P$ and \cref{claim:isolation,claim:pseudocomptouchingW3}, there exist subpaths of $P$ contained in $G[Y^{(i_t,\lvert V(T) \rvert+1,s+2)}]$ such that for each $\beta \in [\beta_1,\beta_2]$, some of those subpaths intersects $X_{t_\beta}-X_t$.
But it contradicts \cref{claim:sizebeltfull} since $\beta_2-\beta_1>\eta_5$.

Hence $q_u \in V(T)-V(T_{t_{\beta_1}})$.
Let $M^*$ be the monochromatic $E_{j,q_u}^{(i_{q_u})}$-pseudocomponent in $G[Y^{(i_{q_u},-1,0)}]$ containing $M_{\alpha_1}$.
Note that $u \in V(P) \cap A_{L^{(i_{q_u},-1,0)}}(V(M^*)) \cap X_{V(T_{q_u})}-X_{q_u}$ since $u \in I_j^\circ \cap X_{V(T_{q_u})}-X_{q_u}$.
By \cref{claim:1stdone}, $M^*$ does not have the smallest $\sigma$-value among all $\Se_j^\circ$-related monochromatic $E_{j,q_u}^{(i_{q_u})}$-pseudocomponents $M'''$ in $G[Y^{(i_{q_u},-1,0)}]$ intersecting $X_{q_u}$ with $A_{L^{(i_{q_u},-1,0)}}(V(M''')) \cap X_{V(T_{q_u})}-X_{q_u} \neq \emptyset$.
So there exists an $\Se_j^\circ$-related monochromatic $E_{j,q_u}^{(i_{q_u})}$-pseudocomponents $Q^*$ in $G[Y^{(i_{q_u},-1,0)}]$ intersecting $X_{q_u}$ with $A_{L^{(i_{q_u},-1,0)}}(V(Q^*)) \cap X_{V(T_{q_u})}-X_{q_u} \neq \emptyset$ such that $\sigma(Q^*)<\sigma(M^*)$.
By \cref{claim:orderpreserving}, there exists $\alpha \in [\alpha_1-1]$ such that $M_\alpha \subseteq Q^*$.
By \cref{claim:isolation,claim:pseudocomptouchingW3}, there exists a monochromatic $E_{j,t}^{(i_t,\lvert V(G) \rvert)}$-pseudocomponent $Q$ in $G[Y^{(i_t,\lvert V(T) \rvert+1,s+2)}]$ containing $M_\alpha$ and a vertex $u_0 \in X_{q_u}$ such that $A_{L^{(i_{t},\lvert V(T) \rvert+1,s+2)}}(\{u_0\}) \cap X_{V(T_{q_u})}-X_{q_u} \neq \emptyset$.
Since $A_{L^{(i_t,-1,0)}}(V(M_\alpha)) \cap X_{V(T_t)}-X_t \subseteq (X_{V(T_{t_\kappa})}-X_{t_\kappa}) \cap I_j^\circ$, $u_0 \not \in V(M_\alpha)$.
Since $Q$ contains $M_\alpha$ and $u_0$, and $Y^{(i_t,-1,\alpha_1-1)}-Y^{(i_t,-1,0)} \subseteq X_{V(T_{t_{\kappa-\alpha_1 \eta_5^2}})}$, $V(Q) \cap X_{t_\beta} \neq \emptyset$ for every $\beta \in [\beta_1,\kappa-\alpha_1 \eta_5^2]$.
But $\kappa-\alpha_1 \eta_5^2-\beta_1>\eta_5$, contradicting \cref{claim:sizebeltfull}.

This proves (d) (That is, $t_{\eta_5^3} \not \in V(T_{j,t})$).

So there exists $t' \in \partial T_{j,t}$ such that $t_{\eta_5^3} \in V(T_{t'})-\{t'\}$.
Since $$\bigcup_{\alpha=1}^{\alpha_1-1}A_{L^{(i_t,-1,0)}}(V(M_\alpha)) \cap X_{V(T_t)}-X_t \subseteq (X_{V(T_{t_\kappa})}-X_{t_\kappa}) \cap I_j^\circ,$$ $M_{\alpha_1}$ is a monochromatic $E_{j,t}^{(i_t)}$-pseudocomponent in $G[Y^{(i_t,\lvert V(T) \rvert+1,s+2)}]$ and \linebreak $V(P) \cap A_{L^{(i_t,\lvert V(T) \rvert+1,s+2)}}(V(M_{\alpha_1})) \cap X_{V(T_t)}-X_t$ is a nonempty subset of $I_j^\circ$. 

Since the process for adding fake edges in one interface is irrelevant with the process for adding fake edges in other interfaces, to simplify the notation, we may without loss of generality assume that for every $\ell \in [0,\lvert V(G) \rvert]$, the monochromatic $E_{j,t}^{(i_t,\ell)}$-pseudocomponent $M'$ in $G[Y^{(i_t,\lvert V(T) \rvert+1,s+2)}]$ in which $\sigma(M')$ is the $(\ell+1)$-th smallest among all monochromatic $E_{j,t}^{(i_t,\ell)}$-pseudocomponents in $G[Y^{(i_t,\lvert V(T) \rvert+1,s+2)}]$ is $\Se_j^\circ$-related, unless there are less than $\ell+1$ $\Se_j^\circ$-related monochromatic $E_{j,t}^{(i_t,\ell)}$-pseudocomponents in $G[Y^{(i_t,\lvert V(T) \rvert+1,s+2)}]$.

For every $\gamma \in [w_0]$ and $\gamma' \in [0,\lvert V(G) \rvert]$, let $M^{(\gamma')}_\gamma$ be the $\Se_j^\circ$-related monochromatic $E_{j,t}^{(i_t,\gamma')}$-pseudocomponent in $G[Y^{(i_t,\lvert V(T) \rvert+1,s+2)}]$ with the $\gamma$-th smallest $\sigma$-value among all $\Se_j^\circ$-related monochromatic $E_{j,t}^{(i_t,\gamma')}$-pseudocomponents $M'$ in $G[Y^{(i_t,\lvert V(T) \rvert+1,s+2)}]$ with $V(M') \cap X_{t'} \neq \emptyset$ and \linebreak $A_{L^{(i_t,\lvert V(T) \rvert+1,s+2)}}(V(M')) \cap X_{V(T_{t'})}-X_{t'} \neq \emptyset$ (if there are at most $\gamma-1$ such $M'$, then let $M^{(\gamma')}_\gamma = \emptyset$). 
By \cref{claim:orderpreserving}, each $M^{(\gamma')}_\gamma$ either contains $M_{\gamma''}$ for some $\gamma''$, or is disjoint from $X_t$.

We shall prove the following statements (i) and (ii) by induction on $\alpha_0$: for every $\alpha_0 \in [0,\alpha_1-2]$ and $\alpha \in [\alpha_0+1,\alpha_1-1]$,
	\begin{itemize}
		\item[(i)] there exists no $e \in E_{j,t}^{(i_t,\alpha_0)}-E_{j,t}^{(i_t)}$ such that $e \cap V(M_\alpha^{(\alpha_0)}) \neq \emptyset$, and 
		\item[(ii)] for every $\alpha' \in [\alpha+1,w_0]$, if $\alpha_0 \geq 1$, and $M^{(\alpha_0-1)}_\alpha$ and $M^{(\alpha_0-1)}_{\alpha'}$ have the same color, and the $s$-segment in $\Se_j^\circ$ containing $V(M^{(\alpha_0-1)}_\alpha)$ whose level equals the color of $M^{(\alpha_0-1)}_\alpha$ contains $V(M^{(\alpha_0-1)}_{\alpha'})$, then either:
			\begin{itemize}
				\item there exists an $(\alpha_0,\alpha,\alpha',\psi_{2}(\alpha,\alpha'), \psi_{3}(\alpha,\alpha'))$-blocker for $(j,t,\lvert V(T) \rvert+1,s+2,\alpha_0-1,t')$, or 
				\item there exist $t^* \in V(T)$ with $i_{t^*}<i_{t'}$ and $t' \in V(T_{j,t^*})$, a witness $q' \in \partial T_{j,t^*}$ for $X_{q'} \cap I_j \subseteq W_3^{(i_{t^*},\ell)}$ for some $\ell \in [0,\lvert V(T) \rvert]$, and a monochromatic $E_{j,t^*}^{(i_{t^*})}$-pseudocomponent in $G[Y^{(i_{t^*},\lvert V(T) \rvert+1,s+2)}]$ intersecting $X_{q'}$ and $V(M_\alpha^{(\alpha_0-1)}) \cup V(M_{\alpha'}^{(\alpha_0-1)})$.
			\end{itemize}
	\end{itemize}
Note that (i) and (ii) obviously hold for $\alpha_0=0$.
And for every $\alpha_0 \geq 1$, if (i) holds for all smaller $\alpha_0$, then (ii) for $\alpha_0$ implies (i) for $\alpha_0$ by the construction for $E_{j,t}^{(i_t,\alpha_0)}$ and \cref{claim:isolation,claim:pseudocomptouchingW3}.
So it suffices to show that if (i) holds for every smaller $\alpha_0$, then (ii) holds for $\alpha_0$.

Assume (i) holds for every smaller $\alpha_0$.
That is, there exist no $e \in E_{j,t}^{(i_t,\alpha_0-1)}-E_{j,t}^{(i_t)}$ and $\alpha'' \in [0,\alpha_1-1]$ such that $e \cap V(M_{\alpha''}^{(\alpha_0-1)}) \neq \emptyset$.
So $M^{(\alpha_0-1)}_\gamma = M_\gamma$ for every $\gamma \in [\alpha_1-1]$.
Let $\alpha'$ be a fixed element in $[\alpha+1,w_0]$.
Since $\alpha \leq \alpha_1-1$, $\psi_{2}(\alpha,\alpha')+\psi_{3}(\alpha,\alpha') \leq \psi_2(\alpha_1,\alpha_2)-4f_1(\eta_5)^{w_0+3}$. 
If $V(M^{(\alpha_0-1)}_{\alpha'}) \cap X_t = \emptyset$, then since $t_{(f_1(\eta_5))^{w_0+3}} \in V(T_{t'})$ and $(t_1,t_2,\dots,t_\kappa)$ is an $(\alpha_1-1,\alpha_1,\alpha_2,\kappa,0)$-blocker for $(j,t,-1,0,0,t)$, we know $A_{L^{(i_t,\lvert V(T) \rvert+1,s+2)}}(V(M^{(\alpha_0-1)}_{\alpha'})) \cap X_{V(T_{t_{(f_1(\eta_5))^{w_0+3}+2\eta_5}})} = \emptyset$ by \cref{claim:sizebeltfull}, so $(t_{(f_1(\eta_5))^{w_0+3}+3\eta_5+1},\dots,t_\kappa)$ contains a subsequence that is an $(\alpha_0,\alpha,\alpha',\psi_2(\alpha,\alpha'), \psi_{3}(\alpha,\alpha'))$-blocker for $(j,t,\lvert V(T) \rvert+1,s+2,\alpha_0-1,t')$. 
So we may assume that $V(M^{(\alpha_0-1)}_{\alpha'}) \cap X_t \neq \emptyset$.
Hence $M^{(\alpha_0-1)}_{\alpha'}$ contains $M_{\alpha''}$ for some $\alpha'' \in [\alpha',k]$ by \cref{claim:orderpreserving}.
Let $S = \{\beta \in [w_0]: M_\beta \subseteq M_{\alpha'}^{(\alpha_0-1)}\}$.
Note that for every $\beta \in S$, $\beta \geq \alpha' \geq \alpha+1$. 

If there exist $\beta \in S$ and $t^* \in V(T)$ with $i_{t^*}<i_t$ and $t \in V(T_{j,t^*})-\partial T_{j,t^*}$, a witness $q' \in \partial T_{j,t^*}$ for $X_{q'} \cap I_j \subseteq W_3^{(i_{t^*},\ell)}$ for some $\ell \in [0,\lvert V(T) \rvert]$, and a monochromatic $E_{j,t^*}^{(i_{t^*})}$-pseudocomponent in $G[Y^{(i_{t^*},\lvert V(T) \rvert+1,s+2)}]$ intersecting $X_{q'}$ and $V(M_\alpha) \cup V(M_{\beta})$, then $i_{t^*}<i_{t'}$ and $t' \in V(T_{j,t^*})$, so the second case for (ii) holds and we are done.
So by (b), we may assume that fore every $\beta \in S$, there exists an $(\alpha_0,\alpha,\beta,\kappa_\beta,\xi_\beta)$-blocker $q^\beta = (q_{-\xi_\beta}^\beta, q_{-\xi_\beta+1}^\beta,\dots,q^\beta_{-1},q^\beta_1,q^\beta_2,\dots,q^\beta_{\kappa_\beta})$ for $(j,t,-1,0,0,t)$, where $\kappa_\beta = \psi_{2}(\alpha,\beta)$ and $\xi_\beta= \psi_3(\alpha,\beta)-\psi_3(0,0)$. 
We further choose each $q^\beta$ such that the sequence $(i_{q_{-\xi_\beta}^\beta}, i_{q_{-\xi_\beta+1}^\beta},\dots,i_{q^\beta_{-1}},i_{q^\beta_1},i_{q^\beta_2},\dots,i_{q^\beta_{\kappa_\beta}})$ is lexicographically maximal.

Since $\beta \geq \alpha' \geq \alpha+1$ for every $\beta \in S$, we know $\kappa_\beta$ is identical for all $\beta \in S$, and $\xi_\beta$ is identical for all $\beta \in S$.
Since $q^\beta$ are chosen such that the sequence of the $\sigma_T$-value of the nodes in $q^\beta$ is lexicographically maximal, $q^\beta$ is identical for all $\beta \in S$.

Since $\alpha \leq \alpha_1-1$, for every $\beta \in S$, 
\begin{align*}
\kappa_\beta+\xi_\beta & \leq \psi_2(\alpha_1,\alpha_2)-(2f_1(\eta_5)^{w_0+3}+3\eta_5+\psi_1(20,20))\\
& = \kappa-(2f_1(\eta_5)^{w_0+3}+3\eta_5+\psi_1(20,20)),
\end{align*} 
so if $t_{2f_1(\eta_5)^{w_0+3}+3\eta_5+\psi_1(20,20)} \in V(T_{q^\beta_{1}})$, then $(q^\beta_{-\xi_\beta},\dots,q^\beta_{-1},t_{2f_1(\eta_5)^{w_0+3}+4\eta_5+\psi_1(20,20)+1}, \allowbreak \dots,t_\kappa)$
 contains a subsequence that is an $(\alpha_0,\alpha,\beta,\kappa_\beta,\xi_\beta)$-blocker for $(j,t,\allowbreak -1,0,0,t)$, contradicting the maximality of the $\sigma_T$-value of the nodes. 
Hence for every $\beta \in S$, we have $q^\beta_{1} \in V(T_{t_{2f_1(\eta_5)^{w_0+3}+3\eta_5+\psi_1(20,20)}})-\{t_{2f_1(\eta_5)^{w_0+3}+3\eta_5+\psi_1(20,20)}\}$.
Note that for every $\beta \in S$, $\psi_3(\alpha,\beta) \leq \psi_1(3,3)$.
So if there exists $\beta \in S$ such that $t_{2f_1(\eta_5)^{w_0+3}+3\eta_5+\psi_1(10,10)} \in V(T_{q^\beta_{-\xi_\beta}})$, then $$(t_{2f_1(\eta_5)^{w_0+3}+3\eta_5+\psi_1(10,10)+1}, \allowbreak \dots,t_{2f_1(\eta_5)^{w_0+3}+3\eta_5+\psi_1(20,20)},q^\beta_1,\dots,q^\beta_{\kappa_\beta})$$ 
contains a subsequence that is an $(\alpha_0,\alpha,\beta,\kappa_\beta,\xi_\beta)$-blocker for $(j,t,-1,0,0,t)$, contradicting the lexicographical maximality.
Hence for every $\beta \in S$, 
$$q^\beta_{-\xi_\beta} \in V(T_{t_{2f_1(\eta_5)^{w_0+3}+3\eta_5+\psi_1(10,10)}})-\{t_{2f_1(\eta_5)^{w_0+3}+3\eta_5+\psi_1(10,10)}\}.$$

Note that for each $v \in \bigcup_{\gamma=0}^{w_0}V(M^{(\alpha_0-1)}_\gamma)$, if there exists a monochromatic path in \linebreak  $G[\bigcup_{\gamma=0}^{w_0}V(M^{(\alpha_0-1)}_\gamma)]+E_{j,t}^{(i_t)}$ from $v$ to $X_t$, then $A_{L^{(i_{t},\lvert V(T) \rvert+1,s+2)}}(\{v\}) \cap X_{V(T_{t'})}-(X_{t'} \cup I_j^\circ) = \emptyset$; if $A_{L^{(i_{t},\lvert V(T) \rvert+1,s+2)}}(\{v\}) \cap X_{V(T_{t'})}-(X_{t'} \cup I_j^\circ) \neq \emptyset$ and there exists no monochromatic path in $G[\bigcup_{\gamma=0}^{w_0}V(M^{(\alpha_0-1)}_\gamma)]+E_{j,t}^{(i_t)}$ from $v$ to $X_t$, then since $t_{\eta_5^3} \in V(T_{t'})$ and $(t_1,t_2,\dots,t_\kappa)$ is a blocker for $(j,t,-1,0,0,t)$, we know $v \in V(G)-(X_{V(T_{t_{\eta_5^3+\eta_5+1}})} \cup X_t)$ and $A_{L^{(i_{t},\lvert V(T) \rvert+1,s+2)}}(\{v\}) \cap X_{V(T_{t'})}-(X_{t'} \cup I_j^\circ) \subseteq V(G)-X_{V(T_{t_{\eta_5^3+\eta_5 \cdot f_1(\eta_5)^{w_0}}})}$ by \cref{claim:sizebeltfull}.
Hence \begin{align*}
A_{L^{(i_t,\lvert V(T) \rvert+1,s+2)}}(\bigcup_{\gamma=0}^{w_0}V(M^{(\alpha_0-1)}_\gamma)) \cap X_{V(T_{t'})}-(X_{t'} \cup I_j^\circ) & \subseteq V(G)-X_{V(T_{t_{\eta_5^3+\eta_5 \cdot f_1(\eta_5)^{w_0}}})} \\
& \subseteq V(G)-X_{V(T_{t_{2f_1(\eta_5)^{w_0+3}+3\eta_5+\psi_1(10,10)}})}.
\end{align*}

For every $\beta \in S$, let $\alpha^*_\beta \in \{\alpha,\beta\}$ such that $A_{L^{(i_t,-1,0)}}(M_{\alpha^*_\beta}) \cap (X_{V(T_t)}-X_t) \cap X_{V(T_{q^\beta_1})}-X_{q^\beta_1} = \emptyset$ and $A_{L^{(i_t,-1,0)}}(M_{\alpha^*_\beta}) \cap (X_{V(T_t)}-X_t) \cap X_{V(T_{q^\beta_{-\xi_\beta}})}-(X_{q^\beta_{-\xi_\beta}} \cup I_j^\circ) = \emptyset$.
So for every $\beta \in S$, $(t_{2f_1(\eta_5)^{w_0+3}+3\eta_5+\psi_1(5,5)+1}, \dots, t_{2f_1(\eta_5)^{w_0+3}+3\eta_5+\psi_1(10,10)}, q^\beta_{-\xi_\beta},\dots,q^{\beta}_{\kappa_\beta})$ is an $(\alpha_0,\alpha,\beta,\psi_2(\alpha,\beta),\psi_3(\alpha,\beta)-\psi_3(0,0)+\psi_1(5,5))$-blocker for $(j,t,\lvert V(T) \rvert+1,s+2,t')$.
Note that for every $\beta \in S$, $\beta \geq \alpha'$ and $\psi_3(\alpha,\beta)-\psi_3(0,0)+\psi_1(5,5) \geq \psi_3(\alpha,\alpha')$.
Hence, if $\alpha^*_\beta = \alpha$ for some $\beta \in S$, then $(t_{2f_1(\eta_5)^{w_0+3}+3\eta_5+\psi_1(5,5)+1}, \allowbreak \dots, t_{2f_1(\eta_5)^{w_0+3}+3\eta_5+\psi_1(10,10)}, q^\beta_{-\xi_\beta},\dots,q^{\beta}_{\kappa_\beta})$ contains a subsequence that is an $(\alpha_0,\alpha,\alpha',\psi_2(\alpha,\alpha'), \allowbreak \psi_3(\alpha,\alpha'))$-blocker for $(j,t,\lvert V(T) \rvert+1,s+2,t')$, so we are done.
So we may assume that $\alpha^*_\beta =\beta$ for every $\beta \in S$.
Then $(t_{2f_1(\eta_5)^{w_0+3}+3\eta_5+\psi_1(5,5)+1}, \dots, t_{2f_1(\eta_5)^{w_0+3}+3\eta_5+\psi_1(10,10)}, q^\beta_{-\xi_\beta}, \allowbreak \dots,q^{\beta}_{\kappa_\beta})$ contains a subsequence that is an $(\alpha_0,\alpha,\alpha',\psi_2(\alpha,\alpha'),\psi_3(\alpha,\alpha'))$-blocker for $(j,t,\lvert V(T) \rvert+1,s+2,t')$, so we are done.

Therefore, (i) and (ii) hold for all $\alpha_0 \in [0,\alpha_1-2]$.
In particular, there exists no $e \in E_{j,t}^{(i_t,\alpha_1-2)}-E_{j,t}^{(i_t)}$ and $\alpha'' \in [\alpha_1-1]$ such that $e \cap V(M_{\alpha''}^{(\alpha_1-2)}) \neq \emptyset$.
Hence there exists no $e \in E_{j,t}^{(i_t,\lvert V(G) \rvert)}-E_{j,t}^{(i_t)}$ such that $e \cap \bigcup_{\gamma=1}^{\alpha_1-1}V(M_\gamma) \neq \emptyset$.
So $\bigcup_{M'}A_{L^{(i_t,\lvert V(T) \rvert+1,s+2)}}(V(M')) \cap X_{V(T_t)}-X_t \subseteq \bigcup_{\alpha=1}^{\alpha_1-1}A_{L^{(i_t,-1,0)}}(V(M_\alpha)) \cap X_{V(T_t)}-X_t \subseteq X_{V(T_{t_\kappa})} \cap I_j^\circ$, where the first union is over all $\Se_j^\circ$-related monochromatic $E_{j,t}^{(i_t,\lvert V(G) \rvert)}$-pseudocomponents $M'$ in $G[Y^{(i_t,\lvert V(T) \rvert+1,s+2)}]$ with $\sigma(M')<\sigma(M_{\alpha_1})$.

For each $\gamma \in [w_0]$, let $M''_\gamma$ be the $\Se_j^\circ$-related monochromatic $E_{j,t'}^{(i_{t'})}$-pseudocomponent in $G[Y^{(i_{t'},-1,0)}]$ with the $\gamma$-th smallest $\sigma$-value among all $\Se_j^\circ$-related monochromatic $E_{j,t'}^{(i_{t'})}$-pseudocomponents $M'$ in $G[Y^{(i_{t'},-1,0)}]$ with $V(M') \cap X_{t'} \neq \emptyset$ and $A_{L^{(i_{t'},-1,0)}}(V(M')) \cap X_{V(T_{t'})}-X_{t'} \neq \emptyset$ (if there are less than $\gamma$ such $M'$, then let $V(M''_\gamma)=\emptyset$).

Since (i) and (ii) hold for $\alpha_0 \in [0,\alpha_1-2]$, by \cref{claim:blocker2-0}, for every $\alpha \in [2,\alpha_1-1]$ and $\alpha' \in [\alpha+1,w_0]$, if $M_\alpha''$ and $M_{\alpha'}''$ have the same color, and the $s$-segment in $\Se_j^\circ$ containing $V(M_\alpha'')$ whose level equals the color of $M_\alpha''$ contains $V(M_{\alpha'}'')$, then either:
	\begin{itemize}
		\item[(iii)] there exists an $(\alpha-1,\alpha,\alpha',\psi_2(\alpha,\alpha'), \psi_3(\alpha,\alpha')-\psi_3(0,0))$-blocker for $(j,t',-1,0,0,t')$, or 
		\item[(iv)] there exist $t^* \in V(T)$ with $i_{t^*}<i_{t'}$ and $t' \in V(T_{j,t^*})$, a witness $q' \in \partial T_{j,t^*}$ for $X_{q'} \cap I_j \subseteq W_3^{(i_{t^*},\ell)}$ for some $\ell \in [0,\lvert V(T) \rvert]$, and a monochromatic $E_{j,t^*}^{(i_{t^*})}$-pseudocomponent in $G[Y^{(i_{t^*},\lvert V(T) \rvert+1,s+2)}]$ intersecting $X_{q'}$ and $V(M_\alpha'') \cup V(M_{\alpha'}'')$.
	\end{itemize}

Since $\bigcup_{\gamma=1}^{\alpha_1-1}A_{L^{(i_t,-1,0)}}(V(M_\gamma^{(\alpha_1-1)}) \cap X_{V(T_t)}-X_t \subseteq (X_{V(T_{t'})}-X_{t'}) \cap I_j^\circ$, $M''_\gamma = M_\gamma$ for every $\gamma \in [\alpha_1-1]$.

For every $v_0 \in V(P) \cap A_{L^{(i_t,-1,0)}}(V(M_{\alpha_1}^{(\lvert V(G) \rvert)})) - Y^{(i_t,-1,0)}$, since $c(v_0) = c(P)$ and \linebreak  $\bigcup_{\alpha=1}^{\alpha_1-1}A_{L^{(i_t,-1,0)}}(V(M_\alpha)) \cap X_{V(T_t)}-X_t \subseteq I_j^\circ$ and $t_\kappa \in V(T_{t'})-\{t'\}$, we know $v_0 \not \in Y^{(i_t,\lvert V(T) \rvert+1,s+2)}$, so $v_0 \in I_j^\circ$ and there exists $q_{v_0}^* \in \partial T_{j,t}$ such that $v \in X_{V(T_{q_{v_0}^*})}-X_{q_{v_0}^*}$ and $v \not \in Y^{(i_{q_{v_0}^*},-1,0)}$.
Let $v$ be a vertex in $V(P) \cap A_{L^{(i_t,-1,0)}}(V(M^{(\lvert V(G) \rvert)}_{\alpha_1})) - Y^{(i_t,-1,0)}$ such that $q^*_v \neq t'$ if possible.

Now we prove 
	\begin{itemize}
		\item[(e)] For every $v_0 \in V(P) \cap A_{L^{(i_t,-1,0)}}(V(M^{(\lvert V(G) \rvert)}_{\alpha_1})) - Y^{(i_t,-1,0)}$, $q^*_{v_0} = t'$.
	\end{itemize}

Suppose (e) does not hold.
Then $q^*_v \neq t'$ by the choice of $v$.
By \cref{claim:1stdone}, the monochromatic $E_{j,q^*_v}^{(i_{q^*_v})}$-pseudocomponent in $G[Y^{(i_{q^*_v},-1,0)}]$ containing $M_{\alpha_1}$ does not have the smallest $\sigma$-value among all $\Se_j^\circ$-related monochromatic $E_{j,q^*_v}^{(i_{q^*_v})}$-pseudocomponents $M'$ in $G[Y^{(i_{q^*_v},-1,0)}]$ with $A_{L^{(i_{q^*_v},-1,0)}}(V(M')) \cap X_{V(T_{q^*_v})}-X_{q^*_v} \neq \emptyset$.
By \cref{claim:orderpreserving}, since there exists no $e \in E_{j,t}^{(i_t,\lvert V(G) \rvert)}-E_{j,t}^{(i_t)}$ such that $e \cap \bigcup_{\alpha=1}^{\alpha_1-1}V(M_{\alpha}) \neq \emptyset$, there exists a $c$-monochromatic path $R$ in $G+E_{j,q_v^*}^{(i_t)}$ from $\bigcup_{\alpha=1}^{\alpha_1-1}V(M_{\alpha})$ to $X_{q^*_v}$ internally disjoint from $\bigcup_{\alpha=1}^{\alpha_1-1}V(M_{\alpha})$.
Since $A_{L^{(i_t,-1,0)}}(\bigcup_{\alpha=1}^{\alpha_1-1}V(M_\alpha)) \cap X_{V(T_t)}-X_t \subseteq X_{V(T_{t_\kappa})} \cap I_j^\circ$, there exists a subpath $R'$ from $V(M''_{\alpha_1'}) \cap V(R')$ to $V(M''_{\alpha'_2}) \cap V(R')$ internally disjoint from $\bigcup_{\alpha=1}^{w_0}V(M''_{\alpha})$ such that $V(R') \cap A_{L^{(i_{t'},-1,0)}}(V(M''_{\alpha_1'})) \cap X_{V(T_{t'})}-X_{t'} \neq \emptyset$ for some $\alpha_1' \in [\alpha_1-1]$ and $\alpha'_2 \in [\alpha_1'+1,w_0]$.
By \cref{claim:1stdone}, $\alpha_1' \in [2,\alpha_1-1]$.
Let $\alpha_1''$ be the minimum $\gamma \in [2,\alpha_1-1]$ such that there exists $\alpha_2'' \in [\gamma+1,w_0]$ such that there exists a subpath $R''$ from $V(M''_{\gamma}) \cap V(R'')$ to $V(M''_{\alpha''_2}) \cap V(R'')$ internally disjoint from $\bigcup_{\alpha=1}^{w_0}V(M''_{\alpha})$ such that $V(R'') \cap A_{L^{(i_{t'},-1,0)}}(V(M''_{\gamma})) \cap X_{V(T_{t'})}-X_{t'} \neq \emptyset$.
Note that $\alpha_1'' \leq \alpha_1'<\alpha_1$, so (iii) and (iv) hold for every $\alpha \in [2,\alpha_1'']$ and $\alpha' \in [\alpha+1,w_0]$. 
Since $\sigma(M''_{\alpha_1''}) \leq \sigma(M''_{\alpha_1'}) < \sigma(M_{\alpha_1})$, the minimality of $\gamma$ implies that this claim is applicable when $t$ is replaced by $t'$ and $M_{\alpha_1}$ is replaced by $M''_{\alpha_1''}$, so $(t',M''_{\alpha_1''})$ is not an outer-safe pair by the existence of $R''$.
But $\sigma(M''_{\alpha_1''})<\sigma(M_{\alpha_1})$, so $(t',M''_{\alpha_1''})$ is an outer-safe pair by the assumption of this claim, a contradiction.

Hence (e) holds. 
So $M_{\alpha_1}'' \supseteq M_{\alpha_1}^{(\lvert V(G) \rvert)} \supseteq M_{\alpha_1}$ and $V(P) \cap A_{L^{(i_{t'},-1,0)}}(V(M''_{\alpha_1})) \cap X_{V(T_{t'})}-X_{t'} \neq \emptyset$.
Note that by an argument similar with the one for showing that $q^*_v =t'$, $M''_{\alpha_1}=M^{(\lvert V(G) \rvert)}_{\alpha_1}$.

Since $v \not \in Y^{(i_{t'},-1,0)}$, $V(P) \not \subseteq Y^{(i_{t'},-1,0)}$.

Suppose there exist $\beta_2^{(0)} \in [w_0]$ with $M_{\beta_2}^{(0)} \not \subseteq M''_{\alpha_1}$ and a subpath $P'$ of $P$ from $V(M''_{\alpha_1})$ to $V(M_{\beta_2}^{(0)})$ internally disjoint from $\bigcup_{\gamma=1}^{w_0}V(M_\gamma^{(0)})$ such that $V(P') \cap A_{L^{(i_{t'},-1,0)}}(V(M''_{\alpha_1})) \cap X_{V(T_{t'})}-X_{t'} \neq \emptyset$. 
Since $X_{t'} \cap I_j \subseteq Y^{(i_t,\lvert V(T) \rvert+1,s+2)}$, $V(P') \subseteq X_{V(T_{t'})}$ and $P'$ is internally disjoint from $X_{t'}$.
So there exists $\alpha''_2 \in [w_0]$ such that $M''_{\alpha''_2}$ contains $M^{(0)}_{\beta_2}$. 
Since $P'$ is a subpath of $P$, $\alpha''_2 \in [\alpha_1,w_0]$.
Since $M_{\beta_2^{(0)}} \not \subseteq M''_{\alpha_1}$, $\alpha''_2 \in [\alpha_1+1,w_0]$.
Let $\alpha_1''=\alpha_1$.
For each $\beta \in [2]$, let $\Q_\beta = \{M^{(\alpha_1-1)}_\gamma: \gamma \in [w_0], M^{(\alpha_1-1)}_\gamma \subseteq M_{\alpha_\beta''}''\}$.
Since $\alpha_2'' \neq \alpha_1$, there exists no $e \in E_{j,t}^{(i_t,\alpha_1)}-E_{j,t}^{(i_t,\alpha_1-1)}$ such that $e \cap \bigcup_{Q \in \Q_1}V(Q) \neq \emptyset \neq e \cap \bigcup_{Q \in \Q_2}V(Q)$.
By the construction of $E_{j,t}^{(i_{t},\alpha_1)}$, for every $Q_1 \in \Q_1$ and $Q_2 \in \Q_2$, either there exists $t^* \in V(T)$ with $i_{t^*}<i_{t}$ and $t' \in V(T_{j,t^*})$, a witness $q' \in \partial T_{j,t^*}$ for $X_{q'} \cap I_j \subseteq W_3^{(i_{t^*},\ell)}$ for some $\ell \in [0,\lvert V(T) \rvert]$, and a monochromatic $E_{j,t^*}^{(i_{t^*})}$-pseudocomponent in $G[Y^{(i_{t^*},\lvert V(T) \rvert+1,s+2)}]$ intersecting $X_{q'}$ and $V(M''_{\alpha_1'}) \cup V(M''_{\alpha_2'})$, or there exists a blocker for $Q_1$ and $Q_2$.
If the former holds, then by \cref{claim:isolation,claim:pseudocomptouchingW3} and the existence of $P'$, there exists a monochromatic $E_{j,t^*}^{(i_{t^*})}$-pseudocomponent in $G[Y^{(i_{t^*},\lvert V(T) \rvert+1,s+2)}]$ intersecting $X_{q'}$ and $V(M_{\alpha_1})$, contradicting (c).
So desired blockers for $Q_1 \in \Q_1$ and $Q_2 \in \Q_2$ exist.
But (iii) and (iv) hold for every $\alpha \in [2,\alpha_1-1]$, so \cref{claim:blocker2-0} implies that the conditions of this claim hold, when $t$ is replaced by $t'$ and $M_{\alpha_1}$ replaced by $M''_{\alpha_1}$. 
Since $\sigma(M''_{\alpha_1}) \leq \sigma(M_{\alpha_1})$ and $i_t'>i_t$, the choice of $(t,M_{\alpha_1})$ for this claim and the existence of $P'$ implies that $(t',M''_{\alpha_1})$ is not outer-safe.
But $V(M''_{\alpha_1})=V(M^{(\lvert V(G) \rvert)}_{\alpha_1}) \subseteq Y^{(i_{t},\lvert V(T) \rvert+1,s+2)}$, a contradiction.

This together with (e) and the fact that $M_{\alpha_1}''=M_{\alpha_1}^{(\lvert V(G) \rvert)}$ imply that $M_{\alpha_2} \subseteq M''_{\alpha_1} = M_{\alpha_1}^{(\lvert V(G) \rvert)}$.
Since $M''_\gamma=M_\gamma$ for every $\gamma \in [\alpha_1-1]$, $V(P)-Y^{(i_{t'},-1,0)} \neq \emptyset$, so $V(P) \not \subseteq M_{\alpha_1}''=M_{\alpha_1}^{(\lvert V(G) \rvert)}$.
Recall that there exists $\alpha^* \in \{\alpha_1,\alpha_2\}$ such that $A_{L^{(i_t,-1,0)}}(V(M_{\alpha^*})) \cap (X_{V(T_t)}-X_t) \cap X_{V(T_{t_1})}-X_{t_1}=\emptyset$.
Since $V(P)$ is internally disjoint from $X_t$, by (d) and \cref{claim:sizebeltfull}, there exists a vertex $o$ in $V(P) \cap A_{L^{(i_t,\lvert V(T) \rvert+1,s+2)}}(V(M^{(\lvert V(G) \rvert)}_{\alpha_1})) \cap X_{V(T_{t'})}-(X_{t'} \cup X_{V(T_{t_{2\eta_5^3}})}))$.
So $o \in V(P) \cap A_{L^{(i_{t'},-1,0)}}(V(M''_{\alpha_1})) \cap X_{V(T_{t'})}-(X_{t'} \cup X_{V(T_{t_{2\eta_5^3}})}))$.

Let $z$ be the node of $T$ with $i_z \geq i_{t'}$ such that either $o \in X_{V(T_{j,z})}$ or there exists $z' \in \partial T_{j,z}$ such that $o \in X_{V(T_{z'})}-X_{z'}$ and $t_\kappa \not \in V(T_{z'})$, and subjec to this, $i_z$ is minimum.
Since $o \in A_{L^{(i_{t'},-1,0)}}(V(M''_{\alpha_1}))-X_{V(T_{t_{2\eta_5^3}})}$, an augument similar to the one for showing (d) shows that $t_{3\eta_5^3} \not \in V(T_{j,z})$.
If $o \in X_{V(T_{j,z})}$, then since $t_{3\eta_5^3} \not \in V(T_{j,z})$ and $o \in A_{L^{(i_{t'},-1,0)}}(V(M''_{\alpha_1}))-X_{V(T_{t_{2\eta_5^3}})}$, $o \in Y^{(i_z,0,\alpha_1)}$ with $c(o) \neq c(P)$, a contradiction.
So there exists $z' \in \partial T_{j,z}$ such that $o \in X_{V(T_{z'})}-X_{z'}$ and $t_\kappa \not \in V(T_{z'})$.
Then an argument similar to the one for showing (e) leads to a contradiction.
This proves the claim.
\end{proof}

\subsection{Pairs are nice and outer-safe}

The goal of this subsection is to show \cref{claim:nicesafe} and other claims which state that all ``interesting'' pairs are nice and outer-safe, and monochromatic pseudocomponents more or less correctly predict the future monochromatic components.

\begin{claim} \label{claim:underassumpsafeprep}
Let $j \in [\lvert \V \rvert-1]$.
Let $t \in V(T)$.
Let $M$ be an $\Se_j$-related monochromatic $E_{j,t}^{(i_t)}$-pseudocomponent in $G[Y^{(i_t,-1,0)}]$ intersecting $X_t$. 
Let $t^* \in V(T)$ be the node of $T$ such that $t \in V(T_{j,t^*})$ and some $E_{j,t^*}^{(i_{t^*})}$-pseudocomponent in $G[Y^{(i_{t^*},\lvert V(T) \rvert+1,s+2)}]$ intersects $V(M)$ and $X_q$ for some witness $q \in \partial T_{j,t^*} \cup \{t^*\}$ for $X_q \cap I_j \subseteq W_3^{(i_{t^*},\ell)}$ for some $\ell \in [-1,\lvert V(T) \rvert]$, and subject to this, $i_{t^*}$ is minimum.
Assume there exist $t' \in \partial T_{j,t^*} - V(T_t)$, $u,v \in X_{t'}$, a $c$-monochromatic path $P_{t'}$ in $G[X_{V(T_{t'})}]$ but not in $G[Y^{(i_{t^*},\lvert V(T) \rvert+1,s+2)}]$ from $u$ to $v$ internally disjoint from $X_{t'}$ such that there exists a monochromatic $E_{j,t^*}^{(i_{t^*},\lvert V(G) \rvert)}$-pseudocomponent $M^*$ in $G[Y^{(i_{t^*},\lvert V(T) \rvert+1,s+2)}]$ with $\sigma(M^*)=\sigma(M)$ containing $u$ but not $v$.
Let $M'$ be the monochromatic $E_{j,t'}^{(i_{t'})}$-pseudocomponent in $G[Y^{(i_{t'},-1,0)}]$ containing $M^*$.

Assume that $(z,M_z)$ is a nice pair and an outer-safe pair for every $z \in V(T)$ and $\Se_j^\circ$-related monochromatic $E_{j,z}^{(i_z)}$-pseudocomponent $M_z$ in $G[Y^{(i_z,-1,0)}]$ intersecting $X_z$ with $\sigma(M_z)<\sigma(M)$. 

If $(t',M')$ is an outer-safe pair, then $V(P_{t'}) \subseteq Y^{(i_{t'},-1,0)}$. 
\end{claim}

\begin{proof}
Suppose $V(P_{t'}) \not \subseteq Y^{(i_{t'},-1,0)}$.

Since $t' \in \partial T_{j,t^*} - V(T_t)$, $t \neq t^* \neq t'$.
Since $M'$ contains $M^*$, $\sigma(M') \leq \sigma(M^*)=\sigma(M)$.
Since $P_{t'}$ is a $c$-monochromatic path and $V(P_{t'}) \not \subseteq Y^{(i_{t'},-1,0)}$, by \cref{claim:1stdone}, there exists an $\Se_j^\circ$-related monochromatic $E_{j,t'}^{(i_{t'})}$-pseudocomponent $Q'$ in $G[Y^{(i_{t'},-1,0)}]$ intersecting $X_{t'}$ with $A_{L^{(i_{t'},-1,0)}}(V(Q')) \cap X_{V(T_{t'})}-X_{t'} \neq \emptyset$ and $\sigma(Q')<\sigma(M')$.
Since $\sigma(Q')<\sigma(M') \leq \sigma(M^*)=\sigma(M)$, $(t',Q')$ is an outer-safe pair and a nice pair.
By \cref{claim:pseudocomptouchingW3}, $t'$ is a witness for $X_{t'} \cap I_j \subseteq W_3^{(i_{t^*},\ell)}$ for some $\ell \in [0,\lvert V(T) \rvert]$.
Let $q^* \in V(T)$ be the node of $T$ such that $t' \in V(T_{j,q^*})$ and some $E_{j,q^*}^{(i_{q^*})}$-pseudocomponent in $G[Y^{(i_{q^*},\lvert V(T) \rvert+1,s+2)}]$ intersects $V(Q')$ and $X_q$ for some witness $q \in \partial T_{j,q^*} \cup \{q^*\}$ for $X_q \cap I_j \subseteq W_3^{(i_{q^*},\ell)}$ for some $\ell \in [-1,\lvert V(T) \rvert]$, and subject to this, $i_{q^*}$ is minimum.
Since $t^*$ is a candidate for $q^*$, $i_{q^*} \leq i_{t^*}$.
So by \cref{claim:underassumptouch2}, there exists an $\Se_j^\circ$-related monochromatic $E_{j,t^*}^{(i_{t^*},\lvert V(G) \rvert)}$-pseudocomponent $Q$ in $G[Y^{(i_{t^*},\lvert V(T) \rvert+1,s+2)}]$ intersecting $X_{t'}$ with $A_{L^{(i_{t^*},\lvert V(T) \rvert+1,s+2)}}(V(Q)) \cap X_{V(T_{t'})}-X_{t'} \neq \emptyset$ and $\sigma(Q)=\sigma(Q')<\sigma(M^*)$.

For every $\ell \in [0,\lvert V(G) \rvert-1]$ and $x \in \{u,v\}$, let $Q_x^{(\ell)}$ be the monochromatic $E_{j,t^*}^{(i_{t^*},\ell)}$-pseudocomponent in $G[Y^{(i_{t^*},\lvert V(T) \rvert+1,s+2)}]$ containing $x$.
For every $\ell \in [0,\lvert V(G) \rvert-1]$, let $Q^{(\ell)}$ be the monochromatic $E_{j,t^*}^{(i_{t^*},\ell)}$-pseudocomponent in $G[Y^{(i_{t^*},\lvert V(T) \rvert+1,s+2)}]$ such that $\sigma(Q^{(\ell)})$ is the $(\ell+1)$-th smallest among all monochromatic $E_{j,t^*}^{(i_{t^*},\ell)}$-pseudocomponents in $G[Y^{(i_{t^*},\lvert V(T) \rvert+1,s+2)}]$.

Let $\L = \{\ell \in [0,\lvert V(G) \rvert-1]: Q^{(\ell)}$ is $\Se_j^\circ$-related$, V(Q^{(\ell)}) \cap X_{t'} \neq \emptyset, A_{L^{(i_{t^*},\lvert V(T) \rvert+1,s+2)}}(V(Q^{(\ell)})) \cap X_{V(T_{t'})}-X_{t'} \neq \emptyset,\sigma(Q^{(\ell)})<\min\{Q_u^{(\ell)},Q_v^{(\ell)}\}\}$. 
By the existence of $Q$, $\L \neq \emptyset$. 

For each $\ell \in \L$, let $M^{(\ell)}_1,M^{(\ell)}_2,\dots,M^{(\ell)}_{\beta}$ (for some $\beta \in {\mathbb N}$) be the $\Se_j^\circ$-related monochromatic $E_{j,t^*}^{(i_{t^*},\ell)}$-pseudocomponents $M''$ intersecting $X_{t'}$ with $A_{L^{(i_{t^*},\lvert V(T) \rvert+1,s+2)}}(V(M'')) \cap X_{V(T_{t'})}-X_{t'} \neq \emptyset$ such that $\sigma(M^{(\ell)}_1)<\sigma(M^{(\ell)}_2)<\dots<\sigma(M^{(\ell)}_{\beta})$ and $V(M^{(\ell)}_\gamma) = \emptyset$ for every $\gamma \geq \beta+1$.
Let $m_\L = \max\L$.

Let $M'_1,M'_2,\dots,M'_{\beta}$ (for some $\beta' \in {\mathbb N})$ be the $\Se_j^\circ$-related monochromatic $E_{j,t'}^{(i_{t'})}$-pseudocomponents $M''$ intersecting $X_{t'}$ with $A_{L^{(i_{t'},-1,0)}}(V(M'')) \cap X_{V(T_{t'})}-X_{t'} \neq \emptyset$ such that $\sigma(M'_1)<\sigma(M'_2)<\dots<\sigma(M'_{\beta'})$.
Let $V(M'_\gamma) = \emptyset$ for every $\gamma \geq \beta'+1$.

Let $\alpha_1 \in [w_0]$ be the minimum index such that $V(M'_{\alpha_1}) \cap \{u,v\} \neq \emptyset$.
Since $M^*$ contains $u$ but not $v$, there exists no monochromatic $E_{j,t^*}^{(i_{t^*},\lvert V(G) \rvert)}$-pseudocomponent containing both $u$ and $v$.
So by the construction of $E_{j,t^*}^{(i_{t^*},\lvert V(G) \rvert)}$, $(V(Q_u^{(m_\L)}) \cup V(Q_v^{(m_\L)})) \cap X_{t^*} \neq \emptyset$ and $\bigcup_{M'}A_{L^{(i_{t^*},\lvert V(T) \rvert+1,s+2)}}(V(M')) \cap X_{V(T_{t^*})}-X_{t^*} \subseteq (X_{V(T_{t'})}-X_{t'}) \cap I_j^\circ$, where the union is over all $\Se_j^\circ$-related $E_{j,t^*}^{(i_{t^*},m_\L)}$-pseudocomponents $M'$ in $G[Y^{(i_{t^*},\lvert V(T) \rvert+1,s+2)}]$ with $V(M') \cap X_{t'} \neq \emptyset$, $A_{L^{(i_{t^*},\lvert V(T) \rvert+1,s+2)}}(V(M')) \cap X_{V(T_{t^*})} - X_{t^*} \neq \emptyset$ and $\sigma(M') \leq \sigma(Q^{(m_\L)})$.
So by \cref{claim:orderpreserving}, we know for every $\alpha \in [\alpha_1-1]$, $M'_\alpha = M_\alpha^{(\lvert V(G) \rvert)}$ by induction on $\alpha$.
This together with \cref{claim:orderpreserving} imply that $M'_{\alpha_1} = M^{(\lvert V(G) \rvert)}_{\alpha_1}$.
So $\lvert M'_{\alpha_1} \cap \{u,v\} \rvert = 1$.
Hence there exists $\alpha_2 \in [\alpha_1+1,w_0]$ such that $M'_{\alpha_2} \cap \{u,v\} \neq \emptyset$.
Note that $M^*$ is contained in $M'_{\alpha_1}$ or $M'_{\alpha_2}$.
So $\sigma(M'_{\alpha_1}) \leq \sigma(M^*)$ and $M' \in \{M'_{\alpha_1},M'_{\alpha_2}\}$.

For $\ell \in [2]$, let $\Q_\alpha = \{M_\gamma^{(m_\L)}: \gamma \in [w_0], M_\gamma^{(m_\L)} \subseteq M'_{\alpha_\ell}\}$.
By the existence of $P_{t'}$, the minimality of $i_{t^*}$ and the construction of $E_{j,t^*}^{(i_{t^*},\lvert V(G) \rvert)}$, for every $M_{\gamma_1}^{(m_\L)} \in \Q_1$ and $M_{\gamma_2}^{(m_\L)} \in \Q_2$, we have $(V(M_{\gamma_1}^{(m_\L)}) \cup V(M_{\gamma_2}^{(m_\L)})) \cap X_{t^*} \neq \emptyset$ and there exists an $(m_\L+1, \gamma_1,\gamma_2,\psi_2(\gamma_1,\gamma_2),\psi_{3}(\gamma_1,\gamma_2))$-blocker for $(j,t^*,\lvert V(T) \rvert+1,s+2,m_\L,t')$. 
Hence, by \cref{claim:orderpreserving}, $V(M'_{\alpha_1}) \cap X_{t^*} \neq \emptyset$.
Since $\sigma(M'_{\alpha_1}) \leq \sigma(M^*)= \sigma(M)$, for every $\beta \in [\alpha_1-1]$, $(t',M'_\beta)$ is a nice pair.
Hence by \cref{claim:blocker2-0}, there exists an $(\alpha_1-1,\alpha_1,\alpha_2,\psi_2(\alpha_1,\alpha_2),\psi_{3}(\alpha_1,\alpha_2)-\psi_3(0,0))$-blocker for $(j,t',-1,0,0,t')$.

Similarly, by the construction of $E_{j,t^*}^{(i_{t^*},m_\L)}$ and \cref{claim:blocker2-0}, for every $\alpha \in [2,\alpha_1-1]$ and $\alpha' \in [\alpha+1,w_0]$, if $M'_{\alpha}$ and $M'_{\alpha'}$ have the same color, and the $s$-segment in $\Se_j^\circ$ containing $V(M'_\alpha)$ whose level equals the color of $M'_\alpha$ contains $V(M'_{\alpha'})$, then either there exists an $(\alpha-1,\alpha,\alpha',\psi_2(\alpha,\alpha'),\psi_3(\alpha,\alpha')-\psi_3(0,0))$-blocker for $(j,t',-1,0,0,t')$, or there exist $t''' \in V(T)$ with $i_{t'''}<i_{t^*}$ and $T_{j,t^*} \subseteq T_{j,t'''}$, a witness $q' \in \partial T_{j,t'''}$ for $X_{q'} \cap I_j \subseteq W_3^{(i_{t'''},\ell)}$ for some $\ell \in [0,\lvert V(T) \rvert]$, and a monochromatic $E_{j,t'''}^{(i_{t'''})}$-pseudocomponent in $G[Y^{(i_{t'''},\lvert V(T) \rvert+1,s+2)}]$ intersecting $V(M'_{\alpha}) \cup V(M'_{\alpha'})$ and $X_{q'}$.

Since $V(P_{t'}) \not \subseteq Y^{(i_{t'},-1,0)}$, by \cref{claim:1stdone}, $\alpha_1 \neq 1$.
If $M'_{\alpha_1}=M'$, then $(t',M'_{\alpha_1})=(t',M')$ is an outer-safe pair; if $M'_{\alpha_1} \neq M'$, then $M'=M'_{\alpha_2}$, so $\sigma(M'_{\alpha_1}) < \sigma(M'_{\alpha_2}) = \sigma(M') \leq \sigma(M^*)=\sigma(M)$, and hence $(t',M'_{\alpha_1})$ is an outer-safe pair.
By \cref{claim:blocker2}, there exists no $c$-monochromatic path $O^*$ in $G$ from $V(M'_{\alpha_1}) \cap V(O^*)$ to $V(M_{\alpha_2}') \cap V(O^*)$ internally disjoint from $\bigcup_{\gamma=1}^{w_0}V(M_{\gamma}')$ such that $V(O^*) \cap A_{L^{(i_{t'},-1,0)}}(V(M_{\alpha_1}')) \cap X_{V(T_{t'})}-X_{t'} \neq \emptyset$, contradicting the existence of $P_{t'}$.
This proves the claim.
\end{proof}

\begin{claim} \label{claim:underassumpsafe}
Let $j \in [\lvert \V \rvert-1]$.
Let $t \in V(T)$.
Let $M$ be an $\Se_j$-related monochromatic $E_{j,t}^{(i_t)}$-pseudocomponent in $G[Y^{(i_t,-1,0)}]$ intersecting $X_t$. 
Assume that $(z,M_z)$ is a nice pair and an outer-safe pair for every $z \in V(T)$ and $\Se_j^\circ$-related monochromatic $E_{j,z}^{(i_z)}$-pseudocomponent $M_z$ in $G[Y^{(i_z,-1,0)}]$ intersecting $X_z$ with $\sigma(M_z)<\sigma(M)$. 
Then $(t,M)$ is an outer-safe pair.
\end{claim}

\begin{proof}
Suppose to the contrary that $(t,M)$ is not an outer-safe pair, and subject to this, $i_t$ is as small as possible.

Let $t^* \in V(T)$ be the node of $T$ such that $t \in V(T_{j,t^*})$ and some $E_{j,t^*}^{(i_{t^*})}$-pseudocomponent in $G[Y^{(i_{t^*},\lvert V(T) \rvert+1,s+2)}]$ intersects $V(M)$ and $X_q$ for some witness $q \in \partial T_{j,t^*} \cup \{t^*\}$ for $X_q \cap I_j \subseteq W_3^{(i_{t^*},\ell)}$ for some $\ell \in [-1,\lvert V(T) \rvert]$, and subject to this, $i_{t^*}$ is minimum.
Since $(t,M)$ is not an outer-safe pair, there exist $t' \in \partial T_{j,t^*} - V(T_t)$, $u,v \in X_{t'}$, a monochromatic path $P_{t'}$ in $G[Y^{(i_t,-1,0)} \cap X_{V(T_{t'})}]$ but not in $G[Y^{(i_{t^*},\lvert V(T) \rvert+1,s+2)}]$ from $u$ to $v$ internally disjoint from $X_{t'}$ such that there exists a monochromatic $E_{j,t^*}^{(i_{t^*},\lvert V(G) \rvert)}$-pseudocomponent $M^*$ in $G[Y^{(i_{t^*},\lvert V(T) \rvert+1,s+2)}]$ with $\sigma(M^*)=\sigma(M)$ containing $u$ but not $v$.
In particular, $t \neq t^*$.
By \cref{claim:pseudocomptouchingW3}, $t'$ is a witness for $X_{t'} \cap I_j \subseteq W_3^{(i_{t^*},\ell)}$ for some $\ell \in [0,\lvert V(T) \rvert]$.

Let $M'$ be the monochromatic $E_{j,t'}^{(i_{t'})}$-pseudocomponent in $G[Y^{(i_{t'},-1,0)}]$ containing $M^*$.
By \cref{claim:isolation,claim:pseudocomptouchingW3}, $t^* \in V(T)$ is the node of $T$ such that $t' \in V(T_{j,t^*})$ and some $E_{j,t^*}^{(i_{t^*})}$-pseudocomponent in $G[Y^{(i_{t^*},\lvert V(T) \rvert+1,s+2)}]$ intersects $V(M')$ and $X_q$ for some witness $q \in \partial T_{j,t^*} \cup \{t^*\}$ for $X_q \cap I_j \subseteq W_3^{(i_{t^*},\ell)}$ for some $\ell \in [-1,\lvert V(T) \rvert]$, and subject to this, $i_{t^*}$ is minimum.
Note that $t' \neq t^*$ and $\sigma(M') \leq \sigma(M^*)$.

Suppose $V(P_{t'}) \subseteq Y^{(i_{t'},-1,0)}$.
If $\sigma(M')=\sigma(M^*)$, then since $t' \in \partial T_{j,t^*} \cap V(T_{t'})$, it is a contradiction by \cref{claim:underassumpnounder}.
So $\sigma(M')<\sigma(M^*)=\sigma(M)$.
Hence $(t',M')$ is an outer-safe pair by assumption.
By \cref{claim:underassumpoutersafeunique}, there uniquely exists a monochromatic $E_{j,t^*}^{(i_{t^*},\lvert V(G) \rvert)}$-pseudocomponent $Q^*$ in $G[Y^{(i_{t^*},\lvert V(T) \rvert+1,s+2)}]$ intersecting $V(M')$.
Since $M^* \subseteq M'$, $Q^*=M^*$ by the uniqueness of $Q^*$.
By \cref{claim:underassumpoutersafeunique}, $\sigma(Q^*)=\sigma(M')$.
So $\sigma(M')=\sigma(M^*)$, a contradiction.

Hence $V(P_{t'}) \not \subseteq Y^{(i_{t'},-1,0)}$.
By \cref{claim:underassumpsafeprep}, $(t',M')$ is not an outer-safe pair.
By the assumption of this claim, $\sigma(M') \geq \sigma(M)$.
Since $M' \supseteq M^*$, $\sigma(M') \leq \sigma(M^*)=\sigma(M)$.
So $\sigma(M')=\sigma(M)$.
Since $(t',M')$ is not outer-safe and $\sigma(M')=\sigma(M)$, by the maximality of $i_t$, $i_{t'} \geq i_t$.
But $V(P_{t'}) \subseteq Y^{(i_t,-1,0)}$ and $V(P_{t'}) \not \subseteq Y^{(i_{t'},-1,0)}$.
So $i_t>i_{t'}$, a contradiction.
This proves the claim.
\end{proof}

\begin{claim} \label{claim:underassumpnice}
Let $j \in [\lvert \V \rvert-1]$.
Let $t \in V(T)$.
Let $M$ be an $\Se_j$-related monochromatic $E_{j,t}^{(i_t)}$-pseudocomponent in $G[Y^{(i_t,-1,0)}]$ intersecting $X_t$. 
Assume that $(z,M_z)$ is a nice pair and an outer-safe pair for every $z \in V(T)$ and $\Se_j^\circ$-related monochromatic $E_{j,z}^{(i_z)}$-pseudocomponent $M_z$ in $G[Y^{(i_z,-1,0)}]$ intersecting $X_z$ with $\sigma(M_z)<\sigma(M)$. 
Then $(t,M)$ is a nice pair.
\end{claim}

\begin{proof}
By \cref{claim:underassumpsafe}, $(t,M)$ is an outer-safe pair.
Let $t^*$ be the node of $T$ with $t \in V(T_{j,t^*})$ such that some $E_{j,t^*}^{(i_{t^*})}$-pseudocomponent in $G[Y^{(i_{t^*},\lvert V(T) \rvert+1,s+2)}]$ intersecting $V(M)$ intersects $X_{t'}$ for some witness $t' \in \partial T_{j,t^*} \cup \{t^*\}$ for $X_{t'} \cap I_j \subseteq W_3^{(i_{t^*},\ell)}$ for some $\ell \in [-1,\lvert V(T) \rvert]$, and subject to this, $i_{t^*}$ is minimum.
By \cref{claim:underassumpoutersafeunique}, there uniquely exists a monochromatic $E_{j,t^*}^{(i_{t^*},\lvert V(G) \rvert)}$-pseudocomponent $M^*$ in $G[Y^{(i_{t^*},\lvert V(T) \rvert+1,s+2)}]$ intersecting $V(M)$; $V(M^*) \cap X_t \neq \emptyset$ and $\sigma(M^*)=\sigma(M)$; if $t \neq t^*$ and $A_{L^{(i_t,-1,0)}}(V(M)) \cap X_{V(T_t)}-X_t \neq \emptyset$, then $A_{L^{(i_{t^*},\lvert V(T) \rvert+1,s+2)}}(V(M^*)) \cap X_{V(T_t)}-X_t \neq \emptyset$.

Suppose to the contrary that $(t,M)$ is not a nice pair.
So $t \neq t^*$, and there exist a $c$-monochromatic path $P$ in $G$ intersecting $V(M)$ with $V(P) \subseteq X_{V(T_{t^*})}$ internally disjoint from $X_{V(T_t)}$ and a vertex $u_P \in V(P) \cap X_{V(T_{j,t^*})}$ such that there exists no monochromatic $E_{j,t^*}^{(i_{t^*},\lvert V(G) \rvert)}$-pseudocomponent $M'$ in $G[Y^{(i_{t^*},\lvert V(T) \rvert+1,s+2)}]$ such that $V(M') \cap V(M) \neq \emptyset \neq V(M') \cap X_{V(T_t)}$, $M'$ contains $u_P$, and if $A_{L^{(i_t,-1,0)}}(V(M)) \cap X_{V(T_t)}-X_t \neq \emptyset$, then $A_{L^{(i_{t^*},\lvert V(T) \rvert+1,s+2)}}(V(M')) \cap X_{V(T_t)}-X_t \neq \emptyset$.
So $M^*$ does not contain $u_P$.
By \cref{claim:isolation,claim:pseudocomptouchingW3}, some monochromatic $E_{j,t^*}^{(i_{t^*},\lvert V(G) \rvert)}$-pseudocomponent $Q$ in $G[Y^{(i_{t^*},\lvert V(T) \rvert+1,s+2)}]$ contains $u_P$.
Hence $Q \neq M^*$.

Since $V(P) \cap V(M) \neq \emptyset$ and $V(P) \subseteq X_{V(T_{t^*})}$, $V(P) \cap V(M^*) \neq \emptyset$ by the uniqueness of $M^*$.
Since $P$ is internally disjoint from $X_{V(T_t)}$, by tracing along $P$ from $M^*$, there exists $t' \in \partial T_{j,t^*} - V(T_t)$, $u,v \in X_{t'}$, a $c$-monochromatic path $P_{t'}$ in $G[X_{V(T_{t'})}]$ but not in $G[Y^{(i_{t^*},\lvert V(T) \rvert+1,s+2)}]$ from $u$ to $v$ internally disjoint from $X_{t'}$ such that $M^*$ contains $u$ but not $v$.
Since $M^*$ does not contain $v$, $M$ does not contain $v$ by \cref{claim:isolation,claim:pseudocomptouchingW3} and the uniqueness of $M^*$.
So $V(P_{t'}) \not \subseteq Y^{(i_t,-1,0)}$.
Let $M'$ be the monochromatic $E_{j,t'}^{(i_{t'})}$-pseudocomponent in $G[Y^{(i_{t'},-1,0)}]$ containing $M^*$.
So $\sigma(M') \leq \sigma(M^*)=\sigma(M)$.
Hence $(t',M')$ is an outer-safe pair by \cref{claim:underassumpsafe}.
By \cref{claim:underassumpsafeprep}, $V(P_{t'}) \subseteq Y^{(i_{t'},-1,0)}$.

Since $(t,M)$ is outer-safe, $V(P_{t'}) \not \subseteq Y^{(i_{t},-1,0)}$.
So $i_{t'}>i_t$.
Since $t' \not \in V(T_t)$ and $V(M^*) \cap X_t \neq \emptyset$ and $u \in V(M^*)$, $V(M^*)$ intersects the bag at the common ancestor of $t$ and $t'$.
So $V(P_{t'}) \subseteq Y^{(i_{t'},-1,0)}$ implies $V(P_{t'}) \subseteq Y^{(i_{t},-1,0)}$.
This proves the claim.
\end{proof}

\begin{claim} \label{claim:nicesafe}
Let $j \in [\lvert \V \rvert-1]$.
Let $t \in V(T)$.
Let $M$ be an $\Se_j^\circ$-related monochromatic $E_{j,t}^{(i_t)}$-pseudocomponent in $G[Y^{(i_t,-1,0)}]$ intersecting $X_t$. 
Let $t^*$ be the node of $T$ with $t \in V(T_{j,t^*})$ such that some $E_{j,t^*}^{(i_{t^*})}$-pseudocomponent in $G[Y^{(i_{t^*},\lvert V(T) \rvert+1,s+2)}]$ intersecting $V(M)$ intersects $X_{t'}$ for some witness $t' \in \partial T_{j,t^*} \cup \{t^*\}$ for $X_{t'} \cap I_j \subseteq W_3^{(i_{t^*},\ell)}$ for some $\ell \in [-1,\lvert V(T) \rvert]$, and subject to this, $i_{t^*}$ is minimum.
Then 
	\begin{itemize}
		\item $(t,M)$ is a nice pair and an outer-safe pair, and
		\item if $P$ is a $c$-monochromatic path in $G[X_{V(T_{t^*})}]$ intersecting $V(M) \cap X_t$ internally disjoint from $X_{V(T_t)}$ and $u_P$ is a vertex in $V(P) \cap X_{V(T_{j,t^*})}$, then there exists a monochromatic $E_{j,t^*}^{(i_{t^*},\lvert V(G) \rvert)}$-pseudocomponent in $G[Y^{(i_{t^*},\lvert V(T) \rvert+1,s+2)}]$ containing $V(M) \cap V(P) \cap X_t$ and containing $u_P$.
	\end{itemize}
\end{claim}

\begin{proof}
Suppose that $(t,M)$ is either not a nice pair or not an outer-safe pair, and subject to this, $\sigma(M)$ is as small as possible.
By \cref{claim:underassumpsafe}, $(t,M)$ is an outer-safe pair.
By \cref{claim:underassumpnice}, $(t,M)$ is a nice pair.
Hence the first statement of this claim holds.

Now we prove the second statement of this claim.
By \cref{claim:underassumpoutersafeunique}, there exists uniquely a monochromatic $E_{j,t^*}^{(i_{t^*},\lvert V(G) \rvert)}$-pseudocomponent $M^*$ in $G[Y^{(i_{t^*},\lvert V(T) \rvert+1,s+2)}]$ intersecting $V(M)$.
If $t=t^*$, then since $V(P) \subseteq X_{V(T_{t^*})}$ and $P$ is internally disjoint from $X_{V(T_t)}$, $P$ is a $c$-monochromatic path in $G[X_t]$, so $M^*$ containing $P$ by \cref{claim:isolation,claim:pseudocomptouchingW3}.
Hence we may assume that $t \neq t^*$.
Since $(t,M)$ is nice, the uniqueness of $M^*$ implies that $V(M^*) \cap V(M) \neq \emptyset \neq V(M^*) \cap X_{V(T_t)}$ and $M^*$ contains $u_P$.
By \cref{claim:isolation,claim:pseudocomptouchingW3,claim:underassumpoutersafeunique}, $V(M^*)$ contains $V(M) \cap X_t$, so $V(M^*)$ contains $V(M) \cap V(P) \cap X_t$.
This proves the claim.
\end{proof}

\begin{claim} \label{claim:turelycontain}
Let $j \in [\lvert \V \rvert-1]$.
Let $t \in V(T)$.
Let $M$ be an $\Se_j$-related monochromatic $E_{j,t}^{(i_t)}$-pseudocomponent in $G[Y^{(i_t,-1,0)}]$ intersecting $X_t$.
Let $t^*$ be the node of $T$ with $t \in V(T_{j,t^*})$ such that some $E_{j,t^*}^{(i_{t^*})}$-pseudocomponent in $G[Y^{(i_{t^*},\lvert V(T) \rvert+1,s+2)}]$ intersecting $V(M)$ intersects $X_{t'}$ for some witness $t' \in \partial T_{j,t^*} \cup \{t^*\}$ for $X_{t'} \cap I_j \subseteq W_3^{(i_{t^*},\ell)}$ for some $\ell \in [-1,\lvert V(T) \rvert]$, and subject to this, $i_{t^*}$ is minimum.
Let $P$ be a $c$-monochromatic path in $G$ intersecting $V(M) \cap X_t$ internally disjoint from $X_{V(T_t)}$.
Let $u_P$ be a vertex of $P$.
If either $u_P \in X_{V(T_{j,t^*})}$ or $\sigma(u_P)=\sigma(P)$, then $M$ contains $u_P$.
\end{claim}

\begin{proof}
Suppose this claim does not hold.
Among all counterexamples, we choose $(t,M,P,u_P)$ such that $i_t$ is as small as possible.

Suppose $t=t^*$.
If $V(P) \subseteq X_{V(T_t)}$, then since $P$ is internally disjoint from $X_{V(T_t)}$, $P$ is a $c$-monochromatic path in $G[X_t]$, so $P \subseteq M$, a contradiction.
So $t$ is not the root of $T$.
Let $p$ be the parent of $t$.
So $V(P) \cap X_p \neq \emptyset$.
Hence $p$ is a candidate for $t^*$, so $t \neq t^*$, a contradiction.

So $t \neq t^*$.
We may assume that $u_P$ is an end of $P$.
Let $v_0$ be a vertex in $V(P) \cap V(M)$.
We may assume that $v_0 \neq u_P$, for otherwise we are done.

By \cref{claim:nicesafe,claim:underassumpoutersafeunique}, there uniquely exists a monochromatic $E_{j,t^*}^{(i_{t^*},\lvert V(G) \rvert)}$-pseudocomponent $M^*$ in $G[Y^{(Y^{(i_{t^*},\lvert V(T) \rvert+1,s+2)})}]$ intersecting $V(M)$.
By \cref{claim:isolation,claim:pseudocomptouchingW3}, $M^*$ contains $v_0$.

Let $\C$ be the collection of the maximal subpaths of $P$ contained in $G[X_{V(T_{j,t^*})}]$.
Denote $\C$ by $\{P_\alpha: \alpha \in [\lvert \C \rvert]\}$.
So for each $\alpha \in [\lvert \C \rvert]$, there exist distinct vertices $v_{\alpha-1},u_\alpha \in V(P)$ such that $P_\alpha$ is from $v_{\alpha-1}$ to $u_\alpha$.
By renaming, we may assume that $P$ passes through $v_0,u_1,v_1,u_2,v_2,\dots,u_{\lvert \C \rvert},u_P$ in the order listed.
Note that for every $\alpha \in [\lvert \C \rvert-1]$, there exist $t_\alpha \in \partial T_{j,t^*} \cup \{t^*\}$ and a subpath $Q_\alpha$ of $P$ from $u_{\alpha}$ to $v_\alpha$ internally disjoint from $X_{t_\alpha}$ such that $V(Q_\alpha)-X_{V(T_{j,t^*})} \neq \emptyset$. 
Since $P$ is internally disjoint from $X_{V(T_t)}$, $t_\alpha \not \in V(T_t)$.

We shall prove that for every $\alpha \in [\lvert \C \rvert-1]$, $\{u_\alpha,v_\alpha,u_{\lvert \C \rvert}\} \subseteq V(M^*)$ by induction on $\alpha$.
For each $\alpha \in [\lvert \C \rvert]$, since $Q_\alpha$ is in $G[X_{V(T_{j,t^*})}]$, we know if $M^*$ contains $v_{\alpha-1}$, then by \cref{claim:isolation,claim:pseudocomptouchingW3}, $M^*$ contains $u_\alpha$.
In particular, $u_1 \in V(M^*)$.
So it suffices to show that for every $\alpha \in [\lvert \C \rvert-1]$, if $M^*$ contains $u_\alpha$, then $M^*$ contains $v_\alpha$.

Now we fix $\alpha$ to be an element in $[\lvert \C \rvert-1]$, and assume that $M^*$ contains $u_\alpha$.
We first assume $t_\alpha=t^*$.
Let $Z_u$ and $Z_v$ be the monochromatic $E_{j,t^*}^{(i_{t^*})}$-pseudocomponents in $G[Y^{(i_{t^*},-1,0)}]$ containing $u_\alpha$ and $v_\alpha$, respectively.
Let $v_{Q_\alpha}$ be the vertex of $Q_\alpha$ with $\sigma(v_{Q_\alpha})=\sigma(Q_\alpha)$.
For each $x \in \{u_\alpha,v_\alpha\}$, let $Q_{x}$ be the subpath of $Q_\alpha$ between $x$ and $v_{Q_\alpha}$.
Note that each $Q_x$ is internally disjoint from $X_{V(T_{t^*})}$.
Since $i_{t^*}<i_t$ and we choose $i_t$ to be minimal among all couterexamples of this claim, we know that both $Z_u$ and $Z_v$ contain $v_{Q_\alpha}$, so $Z_u=Z_v$.
Since $M^*$ contains $u_\alpha$, $M^*$ contains $Z_u=Z_v$, so $M^*$ contains $v_\alpha$.

So we may assume $t_\alpha \in \partial T_{j,t^*}-V(T_t)$.
Hence $Q_\alpha$ is a $c$-monochromatic path in $G[X_{V(T_{t^*})}]$ internally disjoint from $X_{V(T_t)}$, and $v_\alpha$ is a vertex of $P$ in $X_{V(T_{j,t^*})}$.
And $Q_\alpha$ contains $u_\alpha \in V(M^*) \subseteq V(M)$.
Since $(t,M)$ is a nice pair, there exists a monochromatic $E_{j,t^*}^{(i_{t^*},\lvert V(G) \rvert)}$-pseudocomponent $M'$ in $G[Y^{(i_{t^*},\lvert V(T) \rvert+1,s+2)}]$ such that $V(M') \cap V(M) \neq \emptyset$ and $M'$ contains $v_\alpha$.
By the uniqueness of $M^*$, $M'=M^*$.
So $M^*$ contains $v_\alpha$.

This shows that for every $\alpha \in [\lvert \C \rvert-1]$, $\{u_\alpha,v_\alpha,u_{\lvert \C \rvert}\} \subseteq V(M^*)$.
Note that if $u_P \in X_{V(T_{j,t^*})}$, then $u_P = u_{\lvert \C \rvert} \in V(M^*) \subseteq V(M)$, a contradiction.
So $u_P \not \in X_{V(T_{j,t^*})}$.
Hence $\sigma(u_P)=\sigma(P)$, and there exists a subpath $P^*$ of $P$ from $u_{\lvert \C \rvert}$ to $u_P$ internally disjoint from $X_{V(T_{t^*})}$.
Let $M_0$ be the monochromatic $E_{j,t^*}^{(i_{t^*})}$-pseudocomponent in $G[Y^{(i_{t^*},-1,0)}]$ containing $u_{\lvert \C \rvert}$.
So $V(P^*) \cap V(M_0) \cap X_{t^*} \neq \emptyset$ and $\sigma(u_P)=\sigma(P^*)$.
Since $i_{t^*}<i_t$ and we choose $i_t$ to be minimal among all couterexamples of this claim, $M_0$ contains $u_P$.
Since $M$ contains $u_{\lvert \C \rvert}$, $M \supseteq M_0$.
So $M$ contains $u_P$.
This proves the claim.
\end{proof}

\begin{claim} \label{claim:truelysame}
Let $j \in [\lvert \V \rvert-1]$.
Let $t \in V(T)$.
Let $M_1$ and $M_2$ be $\Se_j$-related monochromatic $E^{(i_t)}_{j,t}$-pseudocomponents in $G[Y^{(i_t,-1,0)}]$ intersecting $X_t$. 
Let $P$ be a $c$-monochromatic path from $V(M_1)$ to $V(M_2)$ internally disjoint from $X_{V(T_t)}$.
Then $M_1=M_2$.
\end{claim}

\begin{proof}
Let $v_P$ be the vertex of $P$ such that $\sigma(v_P)=\sigma(P)$.
By \cref{claim:turelycontain}, both $M_1$ and $M_2$ contain $v_P$.
So $M_1=M_2$.
\end{proof}

\begin{claim} \label{claim:turelycontain2}
Let $j \in [\lvert \V \rvert-1]$.
Let $t_1,t_2 \in V(T)$ with $t_2 \in V(T_{t_1})$.
For each $\alpha \in [2]$, let $M_\alpha$ be an $\Se_j$-related monochromatic $E_{j,t_\alpha}^{(i_{t_\alpha})}$-pseudocomponent in $G[Y^{(i_{t_\alpha},-1,0)}]$ intersecting $X_{t_\alpha}$.
Let $P$ be a $c$-monochromatic path in $G$ from $V(M_2) \cap X_{t_2}$ to $V(M_1)$ internally disjoint from $X_{V(T_{t_2})}$.
Then $M_2 \supseteq M_1$.
\end{claim}

\begin{proof}
Suppose to the contrary that $M_2 \not \supseteq M_1$.
Among all counterexamples, we choose $(t_1,t_2,M_1, \allowbreak M_2, \allowbreak P)$ such that $i_{t_2}$ is as small as possible.
By \cref{claim:truelysame}, $t_2 \in V(T_{t_1})-\{t_1\}$.

So by taking a subpath, we may assume that $P$ is internally disjoint from $V(M_1)$.
Let $v_1$ be the end of $P$ in $M_1$, and let $v_2$ be the end of $P$ in $V(M_2) \cap X_{t_2}$.
Note that $v_1 \not \in V(M_2)$, for otherwise $M_2 \supseteq M_1$.
By \cref{claim:truelysame}, $v_1 \in X_{V(T_{t_1})}$.

Let $t_2^*$ be the node of $T$ with $t_2 \in V(T_{j,t_2^*})$ such that some monochromatic $E_{j,t_2^*}^{(i_{t_2^*})}$-pseudocomponent in $G[Y^{(i_{t_2^*},\lvert V(T) \rvert+1,s+2)}]$ intersecting $V(M_2)$ intersects $X_{t'}$ for some witness $t' \in \partial T_{j,t_2^*} \cup \{t_2^*\}$ for $X_{t'} \cap I_j \subseteq W_3^{(i_{t_2^*},\ell)}$ for some $\ell \in [-1,\lvert V(T) \rvert]$, and subject to this, $i_{t_2^*}$ is minimum.

Suppose $v_1 \in X_{V(T_{t_2^*})}$. 
Since $v_1 \not \in V(M_2)$, by \cref{claim:turelycontain}, $v_1 \not \in X_{V(T_{j,t_2^*})}$.
So there exists $q \in \partial T_{j,t_2^*}-V(T_{t_2})$ such that $v_1 \in X_{V(T_q)}-X_q$, and there exists a path $Q$ in $M_1$ from $v_1$ to $X_{q}$ internally disjoint from $X_q$.
If $Q$ is a path in $G$, then by applying \cref{claim:turelycontain} to the path $P \cup Q$, we know $M_2$ contains $V(Q) \cap X_q$, so $M_2 \supseteq M_1$, a contradiction.
So $Q$ is not a path in $G$.
Hence there exist $q^* \in V(T)$ with $T_{j,t_2^*} \subset T_{j,q^*}$, $q' \in \partial T_{j,q^*} \cap X_{V(T_q)}-X_q$ and $e \in E(Q) \cap E_{j,q^*}^{(i_{q^*},\lvert V(G) \rvert)} \neq \emptyset$ such that $e \subseteq X_{q'}$ and $q'$ is a witness for $X_{q'} \cap I_j \subseteq W_3^{(i_{q^*},\ell)}$ for some $\ell \in [0,\lvert V(T) \rvert]$.
We may assume that there exists a subpath $Q'$ of $Q$ from $V(Q) \cap X_q$ to an element in $e$ with $E(Q') \subseteq E(G)$.
So $i_{q^*}<i_{t_2^*}$.
Since $v_1 \in X_{V(T_q)}-X_q$, there exists a subpath $P_1$ of $P \cup Q$ from a vertex $u_1$ in $X_{q}$ to an element of $e$ internally disjoint from $X_{V(T_{q})}$ with $E(P_1) \subseteq E(G)$.
So $V(P_1) \cap X_{q'} \neq \emptyset$.
Since $u_1 \in X_{V(T_{j,t_2^*})}$, $u_1 \in V(M_2)$ by \cref{claim:turelycontain}.
By the existence of $P_1$ and \cref{claim:isolation,claim:pseudocomptouchingW3}, $V(M_2) \cap X_{q'} \neq \emptyset$, so $q^*$ is a candidate for $t_2^*$.
But $i_{q^*}<i_{t_2^*}$, a contradiction.

Hence $v_1 \not \in X_{V(T_{t_2^*})}$. 
By the existence of $P$, the parent of $t_2$ is a candidate for $t_2^*$, so $t_2^* \neq t_2$.
Let $P'$ be the maximal subpath of $P$ from $v_2$ to $X_{t_2^*}$.
Let $u_2$ be the end of $P'$ other than $v_2$ if possible.
Since $u_2 \in X_{t_2^*} \subseteq X_{V(T_{j,t_2^*})}$, by \cref{claim:turelycontain}, $M_2$ contains $u_2$.
Let $M_u$ be the monochromatic $E_{j,t_2^*}^{(i_{t_2^*})}$-pseudocomponent in $G[Y^{(i_{t_2^*},-1,0)}]$ containing $u_2$.
So there exists a subpath $P''$ of $P$ from $u_2$ to $v_1$ internally disjoint from $X_{V(T_{t^*_2})}$.
Since $v_1 \in X_{V(T_{t_1})} - X_{V(T_{t_2^*})}$, $t_2^{*} \in V(T_{t_1})$.
Since $i_{t_2^*}<i_{t_2}$, the minimality of $i_{t_2}$ among all counterexamples implies that $M_u \supseteq M_1$.
Since $M_2$ contains $u_2 \in V(M_u)$, $M_2 \supseteq M_u \supseteq M_1$, a contradiction.
This proves the claim.
\end{proof}

\subsection{The size of the second type of monochromatic component}
\label{subsec:SecondType}

This subsection finishes the proof of \cref{no apex 0} by showing that the size of every $c$-monochromatic component intersecting $\bigcup_{j=1}^{\lvert \V \rvert-1}I_j^\circ$ is bounded, complementing \cref{claim:centralcomponentsize}.

A parade $(t_1,t_2,\dots,t_k)$ is \defn{tamed} if for every $\alpha \in [k-1]$, $t_{\alpha+1} \in \partial T_{j,t_\alpha}$.

\begin{claim} \label{claim:changingdesceasingsimple}
Let $j \in [\lvert \V \rvert-1]$.
Let $\kappa = \kappa_1$.
Let $(t_1,t_2,\dots,t_{\kappa})$ be a parade that is a subsequence of a tamed parade. 
For every $\beta \in [\kappa]$, let $p^{(\beta)}=(p^{(\beta)}_1,p^{(\beta)}_2,\dots,p^{(\beta)}_{\lvert V(G) \rvert})$ be the $(i_{t_\beta},-1,0,j)$-pseudosignature. 
Let $\alpha \in [\lvert V(G) \rvert]$ such that $p^{(1)}_\alpha$ is not a zero sequence. 
Let $M_1$ be the monochromatic $E^{(i_{t_1})}_{j,t_1}$-pseudocomponent in $G[Y^{(i_{t_1},-1,0)}]$ defining $p^{(1)}_\alpha$.
For every $\beta \in [2,\kappa]$, let $M_\beta$ be the monochromatic $E^{(i_{t_\beta})}_{j,t_\beta}$-pseudocomponent in $G[Y^{(i_{t_\beta},-1,0)}]$ containing $M_1$.
For every $\alpha' \in [\alpha-1]$ in which $p^{(1)}_{\alpha'}$ is a nonzero sequence, let $Q^{(1)}_{\alpha'}$ be the monochromatic $E^{(i_{t_1})}_{j,t_1}$-pseudocomponent in $G[Y^{(i_{t_1},-1,0)}]$ defining $p^{(1)}_{\alpha'}$, and for every $\beta \in [2,\kappa]$, let $Q^{(\beta)}_{\alpha'}$ be the monochromatic $E^{(i_{t_\beta})}_{j,t_\beta}$-pseudocomponent in $G[Y^{(i_{t_\beta},-1,0)}]$ containing $Q^{(1)}_{\alpha'}$.
Assume that for every $\alpha' \in [\alpha-1]$ in which $p^{(1)}_{\alpha'}$ is a nonzero sequence, $V(Q^{(1)}_{\alpha'})=V(Q^{(\kappa)}_{\alpha'})$.
If $V(M_{\beta+1}) \neq V(M_{\beta})$ for every $\beta \in [\kappa-1]$, then there exists $\alpha^* \in [\alpha-1]$ such that $p^{(\kappa)}_{\alpha^*}$ is lexicographically smaller than $p^{(1)}_{\alpha^*}$, and $p^{(\kappa)}_{\beta}=p^{(1)}_\beta$ for every $\beta \in [\alpha^*-1]$
\end{claim}

\begin{proof}
Suppose to the contrary that this claim does not hold.
We may assume that among all counterexamples, $\alpha$ is minimal. 

Since for every $\alpha' \in [\alpha-1]$ in which $p^{(1)}_{\alpha'}$ is a nonzero sequence, $V(Q^{(1)}_{\alpha'})=V(Q^{(\kappa)}_{\alpha'})$, we know for every $\beta \in [\kappa]$, $V(Q^{(\beta)}_{\alpha'})=V(Q^{(1)}_{\alpha'})$.
So, together with \cref{claim:zerosignature}, for every $\alpha' \in [\alpha-1]$, $\beta \in [\kappa-1]$ and $\gamma \in [\lvert V(G) \rvert]$, the $\gamma$-th entry of $p^{(\beta+1)}_{\alpha'}$ is at most the $\gamma$-th entry of $p^{(\beta)}_{\alpha}$.
Hence if there exist $\beta \in [2,\kappa]$ and $\alpha' \in [\alpha-1]$ such that $p^{(\beta)}_{\alpha'} \neq p^{(1)}_{\alpha'}$, then $p^{(\kappa)}_{\alpha'}$ is lexicographically smaller than $p^{(1)}_{\alpha'}$ and we are done if we further choose $\alpha'$ to be minimum.
	
So for every $\beta \in [\kappa]$ and $\alpha' \in [\alpha-1]$, $p^{(\beta)}_{\alpha'} = p^{(1)}_{\alpha'}$.
Hence for every $\beta \in [\kappa]$ and $\alpha' \in [\alpha-1]$ in which $p^{(1)}_{\alpha'}$ is a nonzero sequence, $Q^{(\beta)}_{\alpha'}$ defines $p^{(\beta)}_{\alpha'}$, and $A_{L^{(i_{t_1},-1,0)}}(V(Q^{(1)}_{\alpha'})) \cap X_{V(T_{t_1})}-X_{t_1} = A_{L^{(i_{t_\beta},-1,0)}}(V(Q^{(\beta)}_{\alpha'})) \cap X_{V(T_{t_\beta})}-X_{t_\beta} \subseteq X_{V(T_{t_\kappa})}-X_{t_\kappa}$.
For every $\alpha' \in [\alpha-1]$, since $V(Q^{(1)}_{\alpha'}) \cap X_{t_1} \neq \emptyset$, we know $A_{L^{(i_{t_1},-1,0)}}(V(Q^{(1)}_{\alpha'})) \cap X_{V(T_{t_1})}-X_{t_1} = A_{L^{(i_{t_\kappa},-1,0)}}(V(Q^{(\kappa)}_{\alpha'})) \cap X_{V(T_{t_\kappa})}-X_{t_\kappa} \subseteq X_{V(T_{t_\kappa})}-(X_{t_\kappa} \cup Z_{t_1}) \subseteq (X_{V(T_{t_\kappa})}-X_{t_\kappa}) \cap I_j^\circ$.
Since $(t_1,\dots,t_\kappa)$ is a subsequence of a tamed parade, for every $\beta \in [\kappa-1]$ and $\alpha' \in [\alpha-1]$, $(A_{L^{(i_{t_\beta},-1,0)}}(V(Q^{(\beta)}_{\alpha'})) \cap X_{V(T_{t_\beta})}-X_{t_\beta}) \cap Z_{t_\beta} = \emptyset$.

Let $\kappa'=\kappa_0+w_0+3$.
Since $A_{L^{(i_{t_{\beta}},-1,0)}}(V(Q^{(\beta)}_{\alpha'})) \cap X_{V(T_{t_{\beta}})}-X_{t_{\beta}} \subseteq (X_{V(T_{t_{\kappa}})}-X_{t_{\kappa}}) \cap I_j^\circ$ for every $\beta \in [\kappa-2]$ and $\alpha' \in [\alpha-1]$, we know that for any distinct $\beta_1,\beta_2 \in [\kappa]$ with $t_{\beta_2} \in V(T_{t_{\beta_1}})-\{t_{\beta_1}\}$, if $X_{t_{\beta_2}} \cap I_j \subseteq X_{t_{\beta_1}} \cap I_j$, then since $M_{\beta_2} \neq M_{\beta_1}$, some element in $E(M_{\beta_2})-E(M_{\beta_1})$ is a fake edge contained in $X_{t_{\beta_1}} \cap I_j$ by the rules of adding fake edges.
So for every $\beta \in [\kappa]$, there are at most ${w_0 \choose 2}$ elements $\beta' \in [\beta+1,\kappa]$ such that $X_{t_{\beta'}} \cap I_j \subseteq X_{t_{\beta}} \cap I_j$.
Hence there exists a subsequence $(t_{j_1},t_{j_2},...)$ of $(t_1,t_2,...,t_\kappa)$ of length at least $\kappa/w_0^2$ such that for any distinct $\beta_1,\beta_2 \in [\lceil \kappa/w_0^2 \rceil]$ with $\beta_1<\beta_2$, $X_{t_{j_{\beta_2}}} \cap I_j \not \subseteq X_{t_{j_{\beta_1}}} \cap I_j$.
Since $\kappa/w_0^2 \geq \kappa_1/w_0^2 \geq N_{\ref{paradefans}}(w_0,\kappa')$ and $(t_1,t_2,\dots,t_{\kappa})$ is a parade with $\lvert X_{t_\beta} \cap I_j \rvert \leq w_0$ for every $\beta \in [\kappa]$, by \cref{paradefans}, there exists a $(t_1',t'_{\kappa'},\ell^*)$-fan $(t_1',t_2',\dots,t'_{\kappa'})$ in $(T,\X|_{G[I_j]})$ of size $\kappa'$ (for some $\ell^* \in [0,w_0]$) that is a subsequence of $(t_1,t_2,\dots,t_\kappa)$ such that $X_{t_\beta'} \cap I_j-X_{t_1'} \neq \emptyset$ for every $\beta \in [\kappa']$.

Since $(t_1',t_2',\dots,t'_{\kappa'})$ is a subsequence of $(t_1,t_2,\dots,t_\kappa)$, for simplicity, by removing some nodes in $(t_1,t_2,\dots,t_\kappa)$, we may assume $(t_1,t_2,\dots,t_{\kappa'})=(t_1',t_2',\dots,t'_{\kappa'})$.

Let $\beta^* \in [\kappa'-\kappa_0-2]$.
Since $V(M_{\beta^*+1}) \neq V(M_{\beta^*})$ and $A_{L^{(i_{t_{\beta^*}},-1,0)}}(V(Q^{(\beta^*)}_{\alpha'})) \cap X_{V(T_{t_{\beta^*}})}-X_{t_{\beta^*}} = A_{L^{(i_{t_{\beta^*+1}},-1,0)}}(V(Q^{(\beta^*+1)}_{\alpha'})) \cap X_{V(T_{t_{\beta^*+1}})}-X_{t_{\beta^*+1}} \subseteq (X_{V(T_{t_{\kappa}})}-X_{t_{\kappa}}) \cap I_j^\circ$ for every $\alpha' \in [\alpha-1]$, there exists $e_{\beta^*} \in E(M_{\beta^*+1}) \cap E_{j,t_{\beta^*+1}}^{(i_{t_{\beta^*+1}})}-E(M_{\beta^*})$ such that $\lvert e \cap V(M_{\beta^*}) \rvert=1$.
So there exists $z \in V(T)$ with $T_{t_{\beta^*+1}} \subseteq T_z \subseteq T_{t_{\beta^*}}$ and $i_{t_{\beta^*}} \leq i_z < i_{t_{\beta^*+1}}$ such that $e_{\beta^*} \in E_{j,z}^{(i_z,\lvert V(G) \rvert)}-E_{j,z}^{(i_z)}$.
Hence there exists $\ell \in [0,\lvert V(G) \rvert-1]$ such that $e_{\beta^*} \in E_{j,z}^{(i_z,\ell+1)}-E_{j,z}^{(i_z,\ell)}$.
Let $t' \in \partial T_{j,z}$ be the witness for $X_{t'} \cap I_j \subseteq W_3^{(i_z,\ell')}$ for some $\ell' \in [0,\lvert V(T) \rvert]$ such that $e_{\beta^*} \subseteq X_{t'}$.
Since $V(M_{\beta^*}) \cap X_{t_1} \neq \emptyset \neq V(M_{\beta^*}) \cap X_{t'}$, $V(M_{\beta^*}) \cap X_z \neq \emptyset$.

Denote $e_{\beta^*}$ by $\{u,v\}$.
Let $M_u,M_v$ be the monochromatic $E_{j,z}^{(i_z,\ell)}$-pseudocomponents in \linebreak $G[Y^{(i_z,\lvert V(T) \rvert+1,s+2)}]$ containing $u,v$, respectively.
Let $M$ be the $\Se_j^\circ$-related monochromatic $E_{j,z}^{(i_z,\ell)}$-pseudocomponent in $G[Y^{(i_z,\lvert V(T) \rvert+1,s+2)}]$ such that $\sigma(M)$ is the $(\ell+1)$-th smallest among all monochromatic $E_{j,z}^{(i_z,\ell)}$-pseudocomponents in $G[Y^{(i_z,\lvert V(T) \rvert+1,s+2)}]$.
So $\sigma(M) < \min\{\sigma(M_u),\sigma(M_v)\}$ and $M$ is $\Se_j^\circ$-related.
Since $M_{\beta^*}$ contains $u$ or $v$, $M_{\beta^*}$ is contained in one of $M_u$ and $M_v$.
So $\sigma(M)<\sigma(M_{\beta^*})$.

Since $\beta^* \leq \kappa'-\kappa_0-2$ and $(t_1,\dots,t_{\kappa'})$ is a subsequence of a tamed parade, $t_{\kappa'-\kappa_0} \not \in V(T_{j,z})$.
For every $\alpha' \in [\alpha-1]$ in which $p^{(1)}_{\alpha'}$ is a nonzero sequence, let $Q^{(z)}_{\alpha'}$ be the monochromatic $E_{j,z}^{(i_z)}$-pseudocomponent in $G[Y^{(i_z,-1,0)}]$ containing $Q^{(1)}_{\alpha'}$.
Since $i_z<i_{t_{\kappa'-\kappa_0-1}}$, for every $\alpha' \in [\alpha-1]$, $A_{L^{(i_{z},\lvert V(T) \rvert+1,s+2)}}(V(Q^{(z)}_{\alpha'})) \cap X_{V(T_{z})}-X_{z} \subseteq X_{V(T_{t_{\kappa'}})}-(X_{t_{\kappa'}} \cup Z_z)$.
By \cref{claim:orderpreserving}, since $Q^{(1)}_{\alpha'}=Q^{(z)}_{\alpha'}=Q^{(\kappa)}_{\alpha'}$ for every $\alpha' \in [\alpha-1]$, we know $t_{\kappa'-\kappa_0} \in V(T_{t'})-\{t'\}$ and $\{Q^{(z)}_{\alpha'}: \alpha' \in [\alpha-1]$, $p^{(1)}_{\alpha'}$ is a nonzero sequence$\}$ contains the set of the monochromatic $E_{j,z}^{(i_z,\ell)}$-pseudocomponents $M'$ in $G[Y^{(i_z,\lvert V(T) \rvert+1,s+2)}]$ with $\sigma(M') \leq \sigma(M)$.
So $E_{j,z}^{(i_z,\ell)}-E_{j,z}^{(i_z)} \subseteq X_{t'}$.

Since $(t_1,t_2,\dots,t_{\kappa'})$ is a $(t_1,t_{\kappa'},\ell^*)$-fan in $(T,\X|_{G[I_j]})$, for every $\kappa'-\kappa_0 \leq \beta_1<\beta_2 \leq \kappa'$, $X_{t_{\beta_1}} \cap I_j-X_{t'}$ and $X_{t_{\beta_2}} \cap I_j-X_{t'}$ are disjoint non-empty sets.
Since $(V(M_u) \cup V(M_v)) \cap X_z \supseteq V(M_{\beta^*}) \cap X_z \neq \emptyset$ and $e_{\beta^*} \in E_{j,z}^{(i_z,\ell+1)}-E_{j,z}^{(i_z,\ell)}$, $(t_{\kappa'-(\psi_2(w_0,w_0)+\psi_3(w_0,w_0))+1}, t_{\kappa'-(\psi_2(w_0,w_0)+\psi_3(w_0,w_0))+2}, ..., t_{\kappa'})$ is not an $(\ell+1,w_0,w_0,\psi_2(w_0,w_0),\psi_3(w_0,w_0))$-blocker, so for every $x \in \{u,v\}$, 
\begin{align*}
 & A_{L^{(i_z,\lvert V(T) \rvert+1,s+2)}}(V(M_x)) \cap (X_{V(T_z)}-X_z) \cap X_{V(T_{t_{\kappa'-\kappa_0+1}})}-X_{t_{\kappa'-\kappa_0+1}} \\
 =\; & A_{L^{(i_z,\lvert V(T) \rvert+1,s+2)}}(V(M_x)) \cap (X_{V(T_z)}-X_z) \cap X_{V(T_{t_{\kappa'-(\psi_2(w_0,w_0)+\psi_3(w_0,w_0))+1}})} - X_{t_{\kappa'-(\psi_2(w_0,w_0)+\psi_3(w_0,w_0))+1}} \\
 \neq\; & \emptyset,
\end{align*}
and hence $V(M_x) \cap X_{t_{\kappa'-\kappa_0}} \cap I_j \neq \emptyset$.
Since $\lvert e_{\beta^*} \cap V(M_{\beta^*}) \rvert=1$, there exists $x \in \{u,v\}$ such that $(V(M_{\beta^*+1})-V(M_{\beta^*})) \cap X_{t_{\kappa'-\kappa_0}} \cap I_j \supseteq V(M_{x}) \cap X_{t_{\kappa'-\kappa_0}} \cap I_j \neq \emptyset$.

Since $\beta^*$ is an arbitrary element in $[\kappa'-\kappa_0-2]$, we know $\lvert X_{t_{\kappa'-\kappa_0}} \cap I_j \rvert \geq \kappa'-\kappa_0-2 \geq w_0+1$, a contradiction.
This proves the claim.
\end{proof}

\begin{claim} \label{claim:changingdecreasingsimple2}
Let $j \in [\lvert \V \rvert-1]$.
Let $\gamma^* \in [w_0]$ and $\kappa \in {\mathbb N}$.
Let $(t_1,t_2,\dots,t_{\kappa})$ be a parade that is a subsequence of a tamed parade. For every $\beta \in [\kappa]$, let $p^{(\beta)}=(p^{(\beta)}_1,p^{(\beta)}_2,\dots,p^{(\beta)}_{\lvert V(G) \rvert})$ be the $(i_{t_\beta},-1,0,j)$-pseudosignature. 
Let $\alpha \in [\lvert V(G) \rvert]$ such that $p^{(1)}_\alpha$ is not a zero sequence, and $\alpha$ is the $\gamma^*$-th smallest index $\ell$ such that $p^{(1)}_\ell$ is a nonzero sequence.
Let $M_1$ be the monochromatic $E^{(i_{t_1})}_{j,t_1}$-pseudocomponent in $G[Y^{(i_{t_1},-1,0)}]$ defining $p^{(1)}_\alpha$.
For every $\beta \in [2,\kappa]$, let $M_\beta$ be the monochromatic $E^{(i_{t_\beta})}_{j,t_\beta}$-pseudocomponent in $G[Y^{(i_{t_\beta},-1,0)}]$ containing $M_1$.
For every $\alpha' \in [\alpha-1]$ in which $p^{(1)}_{\alpha'}$ is a nonzero sequence, let $Q^{(1)}_{\alpha'}$ be the monochromatic $E^{(i_{t_1})}_{j,t_1}$-pseudocomponent in $G[Y^{(i_{t_1},-1,0)}]$ defining $p^{(1)}_{\alpha'}$, and for every $\beta \in [2,\kappa]$, let $Q^{(\beta)}_{\alpha'}$ be the monochromatic $E^{(i_{t_\beta})}_{j,t_\beta}$-pseudocomponent in $G[Y^{(i_{t_\beta},-1,0)}]$ containing $Q^{(1)}_{\alpha'}$.

Assume $(p^{(\beta)}_1,\dots,p^{(\beta)}_{\alpha-1})$ is identical for every $\beta \in [\kappa]$.
If $V(M_{\beta+1}) \neq V(M_{\beta})$ for every $\beta \in [\kappa]$, then $\kappa \leq \phi_1(\gamma^*)$.
\end{claim}

\begin{proof}
We shall prove this claim by induction on $\gamma^*$.
When $\gamma^*=1$, since $\phi_1(1)=\kappa_1$, this claim follows from \cref{claim:changingdesceasingsimple}.
So we may assume $\gamma^* \geq 2$ and this claim holds for smaller $\gamma^*$.

Suppose to the contrary that $\kappa > \phi_1(\gamma^*)$.
Let $Z=\{z_1,z_2,\dots,z_{\gamma^*}\}$ be the set of elements in $[\alpha]$ such that for each $\ell \in [\gamma^*]$, $z_\ell$ is the $\ell$-th smallest index $\ell'$ such that $p^{(1)}_{\ell'}$ is a nonzero sequence.
For every $\alpha' \in [\gamma^*]$, let $S_{\alpha'}=\{i \in [\kappa-1]: V(Q^{(\beta+1)}_{z_{\alpha'}}) \neq V(Q^{(\beta)}_{z_{\alpha'}})\}$.

Suppose that for every $\alpha' \in [\gamma^*-1]$, $\lvert S_{\alpha'} \rvert \leq \phi_1(\alpha')$.
Then there exists a set $R \subseteq [\kappa-1]$ of consecutive positive integers with $\lvert R \rvert \geq \frac{\kappa-1}{1+\sum_{\alpha'=1}^{\gamma^*-1}\lvert S_{\alpha'} \rvert} \geq \frac{\kappa-1}{1+\sum_{\alpha'=1}^{\gamma^*-1}\phi_1(\alpha')} \geq \kappa_1$ such that $R \cap \bigcup_{\alpha'=1}^{\gamma^*-1}S_{\alpha'} = \emptyset$.
But since $\lvert R \rvert \geq \kappa_1$ and $V(M_{\beta+1}) \neq V(M_\beta)$ for every $\beta \in R$, \cref{claim:changingdesceasingsimple} implies that there exists $\alpha^* \in [\alpha-1]$ such that $p^{(a)}_{\alpha^*} \neq p^{(b)}_{\alpha^*}$ for some distinct $a,b \in R \subseteq [\kappa]$, a contradiction.

So there exists the smallest element $\xi$ in $[\gamma^*-1]$ such that $\lvert S_\xi \rvert \geq \phi_1(\xi)+1$.
Since for every $\alpha' \in [\xi-1]$, $\lvert S_{\alpha'} \rvert \leq \phi_1(\alpha')$, there exists a set $R^* \subseteq S_\xi \subseteq [\kappa-1]$ with $\lvert R^* \rvert \geq \lvert S_\xi \rvert/(1+\sum_{\alpha'=1}^{\xi-1}\phi_1(\alpha')) \geq \kappa_1$ such that $R^* \cap \bigcup_{\alpha'=1}^{\xi-1}S_{\alpha'} = \emptyset$ and there exist no $\beta_1 < \beta_2 <\beta_3$ with $\beta_1,\beta_3 \in R^*$ and $\beta_2 \in \bigcup_{\alpha'=1}^{\xi-1}S_{\alpha'}$.
But since $\lvert R^* \rvert \geq \kappa_1$ and $V(Q^{(\beta+1)}_{z_{\xi}}) \neq V(Q^{(\beta)}_{z_{\xi}})$ for every $\beta \in R^*$, \cref{claim:changingdesceasingsimple} implies that there exists $\alpha^* \in [\alpha-1]$ such that $p^{(a)}_{\alpha^*} \neq p^{(b)}_{\alpha^*}$ for some distinct $a,b \in R^*$, a contradiction.
This proves the claim.
\end{proof}

\begin{claim} \label{claim:changingdecreasing}
Let $j \in [\lvert \V \rvert-1]$.
Let $\gamma^* \in [w_0]$ and $\kappa \in {\mathbb N}$.
Let $(t_1,t_2,\dots,t_{\kappa})$ be a parade that is a subsequence of a tamed parade.
For every $\beta \in [\kappa]$, let $p^{(\beta)}=(p^{(\beta)}_1,p^{(\beta)}_2,\dots,p^{(\beta)}_{\lvert V(G) \rvert})$ be the $(i_{t_\beta},-1,0,j)$-pseudosignature. 
Let $\alpha \in [\lvert V(G) \rvert]$ such that $p^{(1)}_\alpha$ is not a zero sequence, and $\alpha$ is the $\gamma^*$-th smallest index $\ell$ such that $p^{(1)}_\ell$ is a nonzero sequence.
For every $\alpha' \in [\alpha]$ in which $p^{(1)}_{\alpha'}$ is a nonzero sequence, let $Q^{(1)}_{\alpha'}$ be the monochromatic $E^{(i_{t_1})}_{j,t_1}$-pseudocomponent in $G[Y^{(i_{t_1},-1,0)}]$ defining $p^{(1)}_{\alpha'}$, and for every $\beta \in [2,\kappa]$, let $Q^{(\beta)}_{\alpha'}$ be the monochromatic $E^{(i_{t_\beta})}_{j,t_\beta}$-pseudocomponent in $G[Y^{(i_{t_\beta},-1,0)}]$ containing $Q^{(1)}_{\alpha'}$.
If for every $\beta \in [\kappa]$, there exists $\alpha_\beta \in [\alpha]$ such that $V(Q^{(\beta+1)}_{\alpha_\beta}) \neq V(Q^{(\beta)}_{\alpha_\beta})$, then $\kappa \leq \phi_3(\gamma^*)$.
\end{claim}

\begin{proof}
Suppose to the contrary that $\kappa \geq \phi_3(\gamma^*)+1$.
Let $Z=\{z_1,z_2,\dots,z_{\gamma^*}\}$ be the subset of $[\alpha]$ such that for each $\alpha' \in [\gamma^*]$, $z_{\alpha'}$ is the $\alpha'$-th smallest index $\ell$ such that $p^{(1)}_{\ell}$ is a nonzero sequence.
For every $\alpha' \in [\gamma^*]$, let $S_{\alpha'}=\{i \in [2,\kappa]: p^{(i)}_{z_{\alpha'}} \neq p^{(i-1)}_{z_{\alpha'}}\}$.

Suppose there exists $\xi \in [\gamma^*]$ such that $\lvert S_{\xi} \rvert \geq \phi_2(\xi)+1$.
We choose $\xi$ to be as small as possible.
So there exists a subset $R$ of $S_\xi$ with $\lvert R \rvert \geq \frac{\lvert S_\xi \rvert}{1+\sum_{\alpha'=1}^{\xi-1}\lvert S_{\alpha'} \rvert} \geq \frac{\phi_2(\xi)+1}{1+\sum_{\alpha'=1}^{\xi-1}\phi_2(\alpha')} \geq (\phi_1(\xi)+1) \cdot w_0f(\eta_5)+1$ such that $R \cap \bigcup_{\alpha'=1}^{\xi-1}S_{\alpha'} = \emptyset$ and there exist no distinct elements $a,b \in R$ such that $[a,b] \cap \bigcup_{\alpha'=1}^{\xi-1}S_{\alpha'} \neq \emptyset$.
Let $S'_\xi = \{\beta \in R: V(Q^{(\beta)}_{z_\xi}) \neq V(Q^{(\beta-1)}_{z_\xi})\}$.
Since $R \cap \bigcup_{\alpha'=1}^{\xi-1}S_{\alpha'} = \emptyset$ and there exist no distinct elements $a,b \in R$ such that $[a,b] \cap \bigcup_{\alpha'=1}^{\xi-1}S_{\alpha'} \neq \emptyset$, $\lvert S'_\xi \rvert \leq \phi_1(\xi)$ by \cref{claim:changingdecreasingsimple2}.
So there exists a subset $R'$ of $R$ with $\lvert R' \rvert \geq \frac{\lvert R \rvert}{\lvert S'_\xi \rvert+1} \geq w_0f(\eta_5)+1$ such that $R' \cap S'_\xi = \emptyset$ and for any distinct elements $a<b$ of $R'$, $[a,b] \cap S'_\xi = \emptyset$.
Since for every $\beta \in R'$, $\beta \in R-S'_\xi$, so $p^{(\beta)}_{z_{\xi}} \neq p^{(\beta-1)}_{z_{\xi}}$ and $V(Q^{(\beta)}_{z_\xi}) = V(Q^{(\beta-1)}_{z_\xi})$, and hence $p^{(\beta)}_{z_{\xi}} < p^{(\beta-1)}_{z_{\xi}}$.
So $\lvert R' \rvert \leq w_0f(\eta_5)$ by \cref{claim:pseudosignaturenumbernonzero}, a contradiction.

Therefore, for every $\alpha' \in [\gamma^*]$, $\lvert S_{\alpha'} \rvert \leq \phi_2(\alpha')$.
Hence there exists a subset $R''$ of $[2,\kappa]$ with 
$$\lvert R'' \rvert \geq \frac{\kappa-1}{1+\sum_{\alpha'=1}^{\gamma^*}\lvert S_{\alpha'} \rvert} \geq \frac{\kappa-1}{1+\sum_{\alpha'=1}^{\gamma^*}\phi_2(\alpha')} \geq \frac{\phi_3(\gamma^*)}{1+\sum_{\alpha'=1}^{\gamma^*}\phi_2(\alpha')}$$ 
such that $R''$ consists of consecutive integers and $R'' \cap \bigcup_{\alpha'=1}^{\gamma^*}S_{\alpha'} = \emptyset$.
So there exists $R^* \subseteq R''$ with $\lvert R^* \rvert \geq \lvert R'' \rvert/\gamma^* \geq \phi_1(\gamma^*)+1$ such that there exists $x^* \in Z$ such that for every $x \in R^*$, $V(Q^{(x+1)}_{x^*}) \neq V(Q^{(x)}_{x^*})$.
But it contradicts \cref{claim:changingdecreasingsimple2}.
This proves the claim.
\end{proof}

\begin{claim} \label{claim:nonzerotruncation}
Let $j \in [\lvert \V \rvert-1]$.
Let $M$ be a $c$-monochromatic component with $S_M  \in \Se_j^\circ$.
Let $t \in K^*(M)$ such that $K^*(M) \cap V(T_t)-\{t\} \neq \emptyset$.
Then:
	\begin{itemize}
		\item There exists a monochromatic $E_{j,t}^{(i_t)}$-pseudocomponent $M'$ in $G[Y^{(i_t,-1,0)}]$ such that $V(M') \cap X_{t} \neq \emptyset \neq V(M') \cap V(M)$ and $A_{L^{(i_t,-1,0)}}(V(M')) \cap X_{V(T_{t})}-X_t \neq \emptyset$.
		\item There exists a monochromatic $E_{j,t}^{(i_t)}$-pseudocomponent $M'$ in $G[Y^{(i_t,-1,0)}]$ such that $V(M') \cap X_{t} \neq \emptyset$ and $M'$ contains $v_M$, where $v_M$ is the vertex of $M$ with $\sigma(v_M)=\sigma(M)$.
		\item For every monochromatic $E_{j,t}^{(i_t)}$-pseudocomponent $M'$ in $G[Y^{(i_t,-1,0)}]$ with $V(M') \cap X_{t} \neq \emptyset \neq V(M') \cap V(M)$, $A_{L^{(i_t,-1,0)}}(V(M')) \cap X_{V(T_{t})}-X_t \neq \emptyset$.
	\end{itemize}
\end{claim}

\begin{proof}
Since $K^*(M) \cap V(T_t)-\{t\} \neq \emptyset$, there exists $t' \in K^*(M) \cap X_{V(T_t)}-\{t\}$.
So there exists $v \in V(M) \cap Y^{(i_{t'},\lvert V(T) \rvert+1,s+2)} \cap X_{V(T_{t'})} - Y^{(i_t,\lvert V(T) \rvert+1,s+2)}$.
Hence $v \not \in Y^{(i_t,-1,0)}$.
Let $P$ be a path in $M$ from $v$ to $v_M$, where $v_M$ is the vertex of $M$ with $\sigma(v_M)=\sigma(M)$.
Since $t \in K^*(M)$, $t \in V(T_{r_M})$.
So $V(P) \cap X_t \neq \emptyset$ and there exists a subpath $P'$ of $P$ from $v$ to $X_t$ internally disjoint from $X_t$.
Since $v \not \in Y^{(i_t,-1,0)}$, there exists a monochromatic $E_{j,t}^{(i_t)}$-pseudocomponent $M'$ in $G[Y^{(i_t,-1,0)}]$ such that $V(M') \cap X_{t} \neq \emptyset \neq V(M') \cap V(P)$ and $A_{L^{(i_t,-1,0)}}(V(M')) \cap X_{V(T_{t})}-X_t \neq \emptyset$.
Since $P \subseteq M$, $V(M') \cap V(M) \neq \emptyset$.
This proves the first statement of this claim.
In addition, since $V(P) \cap X_t \neq \emptyset$, there exists a subpath $P''$ from $v_M$ to a vertex $u''$ in $X_t$ internally disjoint from $X_{V(T_t)}$.
Let $M''$ be the monochromatic $E_{j,t}^{(i_t)}$-pseudocomponent in $G[Y^{(i_t,-1,0)}]$ containing $u''$.
So $P''$ is a $c$-monochromatic path in $G$ intersecting $V(M'') \cap X_t$ internally disjoint from $X_{V(T_t)}$.
Since $\sigma(v_M)=\sigma(M)$ and $P \subseteq M$, $\sigma(v_M)=\sigma(P'')$.
By \cref{claim:turelycontain}, $M''$ contains $v_M$.
This proves the second statement of this claim.

Now we prove the third statement of this claim.
Let $Z_1 = \bigcup_{M'''}V(M''')$, where the union is over all monochromatic $E_{j,t}^{(i_t)}$-pseudocomponents in $G[Y^{(i_t,-1,0)}]$ such that $V(M''') \cap X_{t} \neq \emptyset \neq V(M''') \cap V(M)$ and $A_{L^{(i_t,-1,0)}}(V(M''')) \cap X_{V(T_{t})}-X_t \neq \emptyset$.
By the first statment of this claim, $Z_1 \neq \emptyset$.
Let $Z_2 = \bigcup_{M'''}V(M''')$, where the union is over all monochromatic $E_{j,t}^{(i_t)}$-pseudocomponents in $G[Y^{(i_t,-1,0)}]$ such that $V(M''') \cap X_{t} \neq \emptyset \neq V(M''') \cap V(M)$ and $A_{L^{(i_t,-1,0)}}(V(M''')) \cap X_{V(T_{t})}-X_t = \emptyset$.
Suppose $Z_2 \neq \emptyset$.
There exists a path $P'$ in $M$ from $Z_2$ to $Z_1$ internally disjoint from $Z_1 \cup Z_2$.
Note that the neighbor $u$ of the vertex in $V(P') \cap Z_2$ in $P'$ belongs to $A_{L^{(i_t,-1,0)}}(Z_2)$.
So $u \not \in X_{V(T_t)}-X_t$.
Hence there exists a subpath of $P'$ from $Z_2$ to $Z_1$ internally disjoint from $Z_1 \cup Z_2 \cup X_{V(T_t)}$.
Let $Q_1$ be the monochromatic $E_{j,t}^{(i_t)}$-pseudocomponents in $G[Y^{(i_t,-1,0)}]$ such that $V(Q_1) \subseteq Z_1$ and $Q_1$ contains the vertex in $Z_1 \cap V(P')$. 
Since $P'$ is internally disjoint from $X_{V(T_t)}$, by \cref{claim:truelysame}, $Q_1$ equals the monochromatic $E_{j,t}^{(i_t)}$-pseudocomponents in $G[Y^{(i_t,-1,0)}]$ containing $V(P) \cap Z_2$.
So $V(Q_1) \subseteq Z_1 \cap Z_2$, a contradiction.
Hence $Z_2=\emptyset$.
This proves the third statement of this claim.
\end{proof}

For any $i \in [0,\lvert V(T) \rvert-1]$, $j \in [\lvert \V \rvert-1]$ and $c$-monochromatic component $M$ with $S_M \in \Se_j^\circ$, the \mathdef{$(i,j,M)$}{truncation} is the sequence $(b_1,\dots,b_{\lvert V(G) \rvert})$ defined as follows.
\begin{itemize}
	\item Let $(p_1,p_2,\dots,p_{\lvert V(G) \rvert})$ be the $(i,-1,0,j)$-pseudosignature.
	\item Let $\alpha^*$ be the largest index $\gamma$ such that $p_\gamma$ is a nonzero sequence and the monochromatic $E_{j,t}^{(i_t)}$-pseudocomponent (where $t$ is the node of $T$ with $i_t=i$) in $G[Y^{(i,-1,0)}]$ defining $p_\gamma$ intersects $M$. (If no such index $\gamma$ exists, then define $\alpha^*=0$.)
	\item For every $\alpha \in [\alpha^*]$, define $b_\alpha :=p_\alpha$.
	\item For every $\alpha \in [\lvert V(G) \rvert]-[\alpha^*]$, define $b_\alpha :=0$.
\end{itemize}

\begin{claim} \label{claim:truncationchange}
Let $j \in [\lvert \V \rvert-1]$.
Let $M$ be a $c$-monochromatic component with $S_M  \in \Se_j^\circ$.
Let $(t_1,t_2,\dots,t_{\eta_5+3})$ be a parade such that $t_\alpha \in K^*(M)$ for every $\alpha \in [\eta_5+3]$.
Then there exists $\alpha^* \in [2,\eta_5+2]$ such that either 
	\begin{itemize}
		\item the $(i_{t_{\alpha^*}},j,M)$-truncation is different from the $(i_{t_1},j,M)$-truncation, or
		\item for each $\alpha \in \{1,\alpha^*\}$, there exists a monochromatic $E_{j,t_\alpha}^{(i_{t_\alpha})}$-pseudocomponent $M_\alpha$ in $G[Y^{(i_{t_\alpha},-1,0)}]$ with $V(M_\alpha) \cap X_{t_\alpha} \neq \emptyset \neq V(M_\alpha) \cap V(M)$ and $A_{L^{(i_{t_\alpha},-1,0)}}(V(M_\alpha)) \cap X_{V(T_{t_\alpha})}-X_{t_\alpha} \neq \emptyset$ such that $M_1 \subseteq M_{\alpha^*}$ and $V(M_1) \neq V(M_{\alpha^*})$.
	\end{itemize}
\end{claim}

\begin{proof}
For every $\alpha \in [\eta_5+2]$, let $\C_\alpha=\{V(Q): Q$ is a monochromatic $E_{j,t_\alpha}^{(i_{t_\alpha})}$-pseudocomponent $M'$ in $G[Y^{(i_{t_\alpha},-1,0)}]$ with $V(M') \cap X_{t_\alpha} \neq \emptyset \neq V(M') \cap V(M)$ and $A_{L^{(i_{t_\alpha},-1,0)}}(V(M')) \cap X_{V(T_{t_\alpha})}-X_{t_\alpha} \neq \emptyset\}$.

Suppose to the contrary that this claim does not hold.
So for every $\alpha \in [\eta_5+2]$, the $(i_{t_\alpha},j,M)$-truncation equals the $(i_{t_1},j,M)$-truncation.
Hence, for every $\alpha \in [\eta_5+2]$ and every $Q_\alpha \in \C_\alpha$, there exists $Q_1' \in \C_1$ such that $\sigma(Q_\alpha)=\sigma(Q_1')$, so $Q_\alpha \supseteq Q'_1$.
Therefore for every $\alpha \in [\eta_5+2]$, $\C_\alpha = \C_1$, for otherwise second outcome of this claim holds.

For every $\alpha \in [2,\eta_5+2]$, since $t_\alpha \in K^*(M)$, there exists $v_\alpha \in V(M) \cap X_{V(T_{t_\alpha})} \cap Y^{(i_{t_\alpha},\lvert V(T) \rvert+1,s+2)}-Y^{(i_{t_{\alpha-1}},\lvert V(T) \rvert+1,s+2)}$.
So $\{v_\alpha: \alpha \in [2,\eta_5+2]\}$ is a set of $\eta_5+1$ different vertices.
By \cref{claim:sizebeltfull}, there exists $\alpha^* \in [2,\eta_5+2]$ such that $v_{\alpha^*} \not \in X_{V(T_{t_{\eta_5+2}})}$.

Let $v_M$ be the vertex of $M$ such that $\sigma(v_M)=\sigma(M)$.
Let $P$ be a path in $M$ from $v_{\alpha^*}$ to $v_M$.
Since $t_{\alpha^*-1} \in K^*(M)$, $V(P) \cap X_{t_{\alpha^*-1}} \neq \emptyset$.
Let $P'$ be the subpath of $P$ from $v_{\alpha^*}$ to $X_{t_{\alpha^*-1}}$ internally disjoint from $X_{t_{\alpha^*-1}}$.

Let $Q$ be the monochromatic $E_{j,t_{\alpha^*-1}}^{(i_{t_{\alpha^*-1}})}$-pseudocomponent in $G[Y^{(i_{t_{\alpha^*-1}},-1,0)}]$ containing the vertex in $V(P') \cap X_{t_{\alpha^*-1}}$.
Since $v_{\alpha^*} \not \in Y^{(i_{t_{\alpha^*-1}},-1,0)}$, $A_{L^{(i_{t_{\alpha^*-1}},-1,0)}}(V(Q)) \cap V(P')-X_{t_{\alpha^*-1}} \neq \emptyset$.
Since $P'$ is internally disjoint from $X_{t_{\alpha^*-1}}$, $V(Q) \in \C_{\alpha^*-1}$.
Since $\C_{\eta_5+2}=\C_1=\C_{\alpha^*-1}$, $V(Q) \in \C_{\eta_5+2}$.
Since both the $(i_{t_{\eta_5+2}},j,M)$-truncation and $(i_{t_{\alpha^*-1}},j,M)$-truncation equal the $(i_{t_1},j,M)$-truncation, $A_{L^{(i_{t_{\alpha^*-1}},-1,0)}}(V(Q)) \cap X_{V(T_{t_{\alpha^*-1}})}-X_{t_{\alpha^*-1}} \subseteq X_{V(T_{t_{\eta_5+2}})}-X_{t_{\eta_5+2}}$.
Hence $V(P') \cap X_{V(T_{t_{\eta_5+2}})} - X_{t_{\eta_5+2}} \neq \emptyset$.
Since $v_{\alpha^*} \not \in X_{V(T_{t_{\eta_5+2}})}$, there exists a subpath $P''$ of $P'$ from $v_{\alpha^*}$ to $X_{t_{\eta_5+2}}$ internally disjoint from $X_{t_{\eta_5+2}}$.

Let $Q'$ be the monochromatic $E_{j,t_{\eta_5+2}}^{(i_{t_{\eta_5+2}})}$-pseudocomponent in $G[Y^{(i_{t_{\eta_5+2}},-1,0)}]$ containing $X_{t_{\eta_5+2}} \cap V(P'')$.
Since $t_{\eta_5+3} \in K^*(M) \cap X_{V(T_{t_{\eta_5+2}})}-\{t_{\eta_5+2}\}$, by \cref{claim:nonzerotruncation}, $A_{L^{(i_{t_{\eta_5+2}},-1,0)}}(V(Q')) \cap X_{V(T_{t_{\eta_5+2}})}-X_{t_{\eta_5+2}} \neq \emptyset$.
So $V(Q') \in \C_{\eta_5+2}=\C_{\alpha^*-1}$.
Since $v_{t_{\alpha^*}} \not \in Y^{(i_{t_{\alpha^*-1}},-1,0)}$ and $V(Q') \in \C_{\alpha^*-1}$, $v_{t_{\alpha^*}} \not \in V(Q')$.
So $A_{L^{(i_{t_{\alpha^*-1}},-1,0)}}(V(Q')) \cap V(P'') \neq \emptyset$.
Let $u \in A_{L^{(i_{t_{\alpha^*-1}},-1,0)}}(V(Q')) \cap V(P'')$.
Since $u \not \in Y^{(i_{t_{\alpha^*-1}},-1,0)}$, $u \not \in X_{t_{\alpha^*-1}}$.
Since both the $(i_{t_{\eta_5+2}},j,M)$-truncation and $(i_{t_{\alpha^*-1}},-1,0)$-truncation equal the $(i_{t_1},j,M)$-truncation, $A_{L^{(i_{t_{\alpha^*-1}},-1,0)}}(V(Q')) \cap X_{V(T_{t_{\alpha^*-1}})}-X_{t_{\alpha^*-1}} \subseteq X_{V(T_{t_{\eta_5+2}})}-X_{t_{\eta_5+2}}$.
Since $V(P'') \subseteq X_{V(T_{t_{\alpha^*-1}})}$, $u \in A_{L^{(i_{t_{\alpha^*-1}},-1,0)}}(V(Q')) \cap X_{V(T_{t_{\alpha^*-1}})}-X_{t_{\alpha^*-1}} \subseteq X_{V(T_{t_{\eta_5+2}})}-X_{t_{\eta_5+2}}$.
But $P''$ is internally disjoint from $X_{V(T_{t_{\eta_5+2}})}$, a contradiction.
This proves the claim.
\end{proof}

\begin{claim} \label{claim:truncationremaining}
Let $j \in [\lvert \V \rvert-1]$.
Let $M$ be a $c$-monochromatic component with $S_M \in \Se_j^\circ$.
Let $t_1,t_2 \in K^*(M)$ such that $t_2 \in V(T_{t_1})-\{t_1\}$ and $K^*(M) \cap X_{V(T_{t_2})}-X_{t_2} \neq \emptyset$.
Let $\alpha^*$ be the largest index $\ell$ such that the $\ell$-th entry of the $(i_{t_1},j,M)$-truncation is not a zero sequence.
For $\alpha \in [2]$, let $(q^{(\alpha)}_1,\dots,q^{(\alpha)}_{\lvert V(G) \rvert})$ be the $(i_{t_\alpha},j,M)$-truncation.
Let $\alpha \in [0,\alpha^*-1]$.
If $q^{(1)}_\beta=q^{(2)}_\beta$ for every $\beta \in [\alpha]$, then either
	\begin{itemize}
		\item there exists $\beta^* \in [\alpha+1,\alpha^*]$ such that $q^{(2)}_{\beta^*}$ is not a zero sequence, or
		\item there exists $\beta^* \in [\alpha]$ such that $q^{(1)}_{\beta^*}$ and $q^{(2)}_{\beta^*}$ are nonzero sequences and the vertex-set of the monochromatic $E_{j,t_2}^{(i_{t_2})}$-pseudocomponent in $G[Y^{(i_{t_2},-1,0)}]$ defining $q^{(2)}_{\beta^*}$ is different from the vertex-set of the monochromatic $E_{j,t_1}^{(i_{t_1})}$-pseudocomponent in $G[Y^{(i_{t_1},-1,0)}]$ defining $q^{(1)}_{\beta^*}$.
	\end{itemize}
\end{claim}

\begin{proof}
Suppose this claim does not hold.
So for every $\beta \in [\alpha+1,\alpha^*]$, $q^{(2)}_\beta$ is a zero sequence.
And for every $\beta^* \in [\alpha]$, if $q^{(1)}_{\beta^*}$ and $q^{(2)}_{\beta^*}$ are nonzero sequences, then the vertex-set of the monochromatic $E_{j,t_2}^{(i_{t_2})}$-pseudocomponent in $G[Y^{(i_{t_2},-1,0)}]$ defining $q^{(2)}_{\beta^*}$ equals the vertex-set of the monochromatic $E_{j,t_1}^{(i_{t_1})}$-pseudocomponent in $G[Y^{(i_{t_1},-1,0)}]$ defining $q^{(1)}_{\beta^*}$.

Let $v_M$ be the vertex of $M$ such that $\sigma(v_M)=\sigma(M)$.
For every $k \in [2]$, by \cref{claim:nonzerotruncation}, there exists a monochromatic $E_{j,t_k}^{(i_{t_k})}$-pseudocomponent $M_k$ in $G[Y^{(i_{t_k},-1,0)}]$ with $V(M_k) \cap X_{t_k} \neq \emptyset$ containing $v_M$.
By \cref{claim:nonzerotruncation}, for each $k \in [2]$, $A_{L^{(i_{t_k},-1,0)}}(V(M_k)) \cap X_{V(T_{t_k})}-X_{t_k} \neq \emptyset$, so there exists $\gamma^*_k \in [\sigma(v_M)]$ such that $q^{(k)}_{\gamma^*_k}$ is not a zero sequence and $M_k$ defines $q^{(k)}_{\gamma^*_k}$.

By the definition of $\alpha^*$, $\gamma^*_1 \in [\alpha^*]$.
Since $v_M \in V(M_1) \cap V(M_2)$, $M_1 \subseteq M_2$.
So $\gamma^*_2 \leq \gamma^*_1$.
Since $q^{(2)}_{\gamma_2^*}$ is not a zero sequence, $\gamma^*_2 \not \in [\alpha+1,\alpha^*]$.
Since $\gamma^*_2 \leq \gamma^*_1 \leq \alpha^*$, $\gamma^*_2 \in [\alpha]$.
In particular, $\alpha \geq 1$.
Since $v_M \in V(M_2)$ and $V(M_2) \cap X_{t_2} \neq \emptyset$, $V(M_2) \cap X_{t_1} \neq \emptyset$.
So by \cref{claim:orderpreserving}, $q^{(1)}_{\gamma_2^*}$ is not a zero sequence.

For every $\gamma \in [\lvert V(G) \rvert]$, if $q^{(1)}_\gamma$ is not a zero sequence, then let $Q_\gamma$ be the monochromatic $E_{j,t_1}^{(i_{t_1})}$-pseudocomponent in $G[Y^{(i_{t_1},-1,0)}]$ defining $q^{(1)}_\gamma$; otherwise, let $V(Q_\gamma)=\emptyset$.

Let $Z_1 = \bigcup_{\gamma=1}^{\alpha}V(Q_\gamma)$. 
Since $q^{(1)}_{\gamma_2^*}$ is not a zero sequence and $\gamma_2^* \leq \alpha$, $Z_1 \neq \emptyset$.
Since $\gamma_2^* \in [\alpha]$, $V(Q_{\gamma_2^*})=V(M_2)$.
Since $M_2$ contains $v_M$, $Q_{\gamma_2^*}$ contains $v_M$.
Hence $Z_1 \cap V(M) \neq \emptyset$.

Let $Z_2 = \bigcup_{\gamma = \alpha+1}^{\alpha^*}(V(Q_\gamma))$.
Since $(q^{(1)}_1,\dots,q^{(1)}_{\lvert V(G) \rvert})$ is the $(i_{t_1},j,M)$-truncation, $V(Q_{\alpha^*}) \cap V(M) \neq \emptyset$.
So $Z_2 \cap V(M) \neq \emptyset$.

Since $Z_1 \cap V(M) \neq \emptyset \neq Z_2 \cap V(M)$, there exists a path $P$ in $M$ from $Z_1$ to $Z_2$ internally disjoint from $Z_1 \cup Z_2$.
By \cref{claim:truelysame,claim:nonzerotruncation}, $P$ is contained in $G[X_{V(T_{t_1})}]$ and is internally disjoint from $X_{t_1}$.
Let $\beta^*$ be the index in $[\alpha]$ such that $Q_{\beta^*}$ contains $V(P) \cap Z_1$.
Since $\beta^* \leq \alpha$, $q^{(2)}_{\beta^*}=q^{(1)}_{\beta^*}$ is not a zero sequence.
So there exists a monochromatic $E_{j,t_2}^{(i_{t_2})}$-pseudocomponent $Q^*$ in $G[Y^{(i_{t_2},-1,0)}]$ defining $q^{(2)}_{\beta^*}$.
Since $\beta^* \in [\alpha]$, $V(Q^*) = V(Q_{\beta^*})$.
Since $q_{\beta^*}^{(1)}=q_{\beta^*}^{(2)}$, $A_{L^{(i_{t_1},-1,0)}}(V(Q_{\beta^*})) \cap X_{V(T_{t_1})}-X_{t_1} = A_{L^{(i_{t_2},-1,0)}}(V(Q^*)) \cap X_{V(T_{t_2})}-X_{t_2} \subseteq X_{V(T_{t_2})}-X_{t_2}$.
Since $P$ is contained in $G[X_{V(T_{t_1})}]$ and is internally disjoint from $X_{t_1}$, $V(P) \cap A_{L^{(i_{t_1},-1,0)}}(V(Q_{\beta^*})) \cap X_{V(T_{t_1})}-X_{t_1} \neq \emptyset$.
So $V(P) \cap X_{V(T_{t_2})}-X_{t_2} \neq \emptyset$.

Let $v$ be the vertex in $Z_2 \cap V(P)$.
Let $\beta_v$ be the index in $[\lvert V(G) \rvert]$ such that $Q_{\beta_v}$ contains $v \in V(P) \cap Z_2 \subseteq V(M) \cap \bigcup_{\gamma = \alpha+1}^{\alpha^*}(V(Q_\gamma))$.
So $\beta_v \in [\alpha+1,\alpha^*]$.

Suppose there exists $\beta'' \in [\lvert V(G) \rvert]$ such that $q^{(2)}_{\beta''}$ is not a zero sequence and the monochromatic $E_{j,t_2}^{(i_{t_2})}$-pseudocomponent $Q'$ in $G[Y^{(i_{t_2},-1,0)}]$ defining $q^{(2)}_{\beta''}$ contains $Q_{\beta_v}$.
Since $Q'$ contains $Q_{\beta_v}$, $\beta'' \in [\beta_v]$.
Since $\beta'' \leq \beta_v \leq \alpha^*$ and $q^{(2)}_{\beta''}$ is a nonzero sequence, $\beta'' \in [\alpha]$, so $V(Q')=V(Q_{\beta''})$.
Hence $V(Q_{\beta''})$ contains $V(Q_{\beta_v})$.
So $Z_1 \cap Z_2 \supseteq V(Q_{\beta''}) \cap V(Q_{\beta_v}) \neq \emptyset$, a contradiction.

Hence there exists no $\beta'' \in [\lvert V(G) \rvert]$ such that $q^{(2)}_{\beta''}$ is not a zero sequence and the monochromatic $E_{j,t_2}^{(i_{t_2})}$-pseudocomponent in $G[Y^{(i_{t_2},-1,0)}]$ defining $q^{(2)}_{\beta''}$ contains $Q_{\beta_v}$.

Suppose $v \in X_{V(T_{t_2})}$.
Since $V(Q_{\beta_v}) \cap X_{t_1} \neq \emptyset$ and $Q_{\beta_v}$ contains $v$, $V(Q_{\beta_v}) \cap X_{t_2} \neq \emptyset$.
So by \cref{claim:nonzerotruncation}, there exists $\beta'' \in [\beta_v]$ such that $q^{(2)}_{\beta''}$ is not a zero sequence and the monochromatic $E_{j,t_2}^{(i_{t_2})}$-pseudocomponent in $G[Y^{(i_{t_2},-1,0)}]$ defining $q^{(2)}_{\beta''}$ contains $Q_{\beta_v}$, a contradiction.

Hence $v \in V(G)-X_{V(T_{t_2})}$.
Since $V(P) \cap X_{V(T_{t_2})}-X_{t_2} \neq \emptyset$, there exists $v' \in V(P) \cap X_{t_2}$ such that the subpath $P'$ of $P$ between $v$ and $v'$ is internally disjoint from $X_{V(T_{t_2})}$.
So there exists a monochromatic $E_{j,t_2}^{(i_{t_2})}$-pseudocomponent $R_2$ in $G[Y^{(i_{t_2},-1,0)}]$ containing $v'$.
By \cref{claim:nonzerotruncation}, there exists $\gamma_{R_2} \in [\lvert V(G) \rvert]$ such that $q^{(2)}_{\gamma_{R_2}}$ is not a zero sequence, and $R_2$ defines $q^{(2)}_{\gamma_{R_2}}$.
But by \cref{claim:turelycontain2}, $R_2 \supseteq Q_{\beta_v}$, a contradiction.
This proves the claim.
\end{proof}

We say a sequence $(a_1,a_2,\dots,a_m)$ is a \defn{substring} of another sequence $(b_1,b_2,\dots,b_n)$ (for some $m,n \in {\mathbb N}$) if there exists $\alpha \in [0,n-m]$ such that $a_i=b_{\alpha+i}$ for every $i \in [m]$.

\begin{claim} \label{claim:truncationmaintain}
Let $j \in [\lvert \V \rvert-1]$.
Let $M$ be a $c$-monochromatic component  with $S_M \in \Se_j^\circ$.
Let $\kappa \in {\mathbb N}$.
Let $(t_1,t_2,\dots,t_\kappa)$ be a parade in $T$ such that $t_\alpha \in K^*(M)$ for every $\alpha \in [\kappa]$.
For each $\alpha \in [\kappa]$, let $(a^{(\alpha)}_1,\dots,a^{(\alpha)}_{\lvert V(G) \rvert})$ be the $(i_{t_\alpha},j,M)$-truncation.
Let $\alpha^*$ be an index such that:
	\begin{itemize}
		\item $a^{(1)}_{\alpha^*}$ is a nonzero sequence, and
		\item for every $\alpha \in [\alpha^*-1]$, $a^{(1)}_\alpha = a^{(2)}_\alpha = \dots = a^{(\kappa)}_\alpha$.
	\end{itemize}
Let $n_0$ be the number of indices $\gamma \in [\alpha^*]$ such that $a^{(1)}_\gamma$ is a nonzero sequence.
Then $\kappa \leq h(w_0-n_0)$.
\end{claim}

\begin{proof}
Note that $w_0-n_0 \geq 0$ by \cref{claim:pseudosignaturenumbernonzero}.
Suppose to the contrary that this claim does not hold.
That is, $\kappa \geq h(w_0-n_0)+1$.
We further assume that $w_0-n_0$ is as small as possible among all counterexamples.

For every $t \in \{t_1,t_2,\dots,t_\kappa\}$ and $\alpha \in [\lvert V(G) \rvert]$, let $a^{(t)}_\alpha=a^{(\beta)}_\alpha$, where $\beta$ is the integer such that $t=t_\beta$. 
For every $\alpha \in [\lvert V(G) \rvert]$ and $t \in \{t_1,t_2,\dots,t_\kappa\}$, if $a^{(t)}_\alpha$ is not a zero sequence, then let $M_\alpha^{(t)}$ be the monochromatic $E_{j,t}^{(i_{t})}$-pseudocomponent in $G[Y^{(i_{t},-1,0)}]$ defining $a^{(t)}_\alpha$; otherwise, let $V(M^{(t)}_\alpha)=\emptyset$.
Since each $t_\alpha$ is in $K^*(M)$, every subsequence of $(t_1,t_2,\dots,t_{\kappa})$ is a subsequence of a tamed parade.
For every $t \in \{t_1,t_2,\dots,t_\kappa\}$, let $\alpha^*_t$ be the largest index $\gamma$ such that $a^{(t)}_\gamma$ is a nonzero sequence.

Let $\tau_1 = (\kappa-1)/(\eta_5+3)$.
By \cref{claim:truncationchange}, there exists a subsequence $(t_1',t_2',\dots, t'_{\tau_1})$ of $(t_1,t_2,\dots,t_{\kappa-1})$ such that for every $\alpha \in [\tau_1-1]$, 
	\begin{itemize}
		\item[(i)] either the $(i_{t'_\alpha},j,M)$-truncation is different from the $(i_{t'_{\alpha+1}},j,M)$-truncation, or there exists $\gamma_\alpha \in [\lvert V(G) \rvert]$ such that $V(M^{(t'_{\alpha+1})}_{\gamma_\alpha}) \neq V(M^{(t'_\alpha)}_{\gamma_\alpha})$.
	\end{itemize}
Since $(t_1',t_2',\dots, t'_{\tau_1})$ satisfies (i), by \cref{claim:changingdecreasing} and the assumption of this claim, there exists $\beta \in [\phi_3(w_0)+1]$ such that $a^*_{t'_\beta} \geq \alpha^*$.
For each $\gamma \in [\tau_1-\phi_3(w_0)]$, we redefine $t'_\gamma$ to be $t'_{\beta+\gamma-1}$.
So $a^*_{t_1'} \geq \alpha^*$.

Let $\tau_2 = (\tau_1-\phi_3(w_0))/2^{w_0}$.
By \cref{claim:zerosignature}, there exists a substring $(t_1'',t_2'',\dots,t''_{\tau_2})$ of $(t_1',t_2',\dots,t_{\tau_1-\phi_3(w_0)}')$ such that for each $\alpha \in [\alpha^*_{t_1'}]$, 
	\begin{itemize}
		\item[(ii)] either $a^{(t)}_\alpha$ is a zero sequence for every $t \in \{t_1'',t_2'',\dots,t''_{\tau_2}\}$, or $a^{(t)}_\alpha$ is a nonzero sequence for every $t \in \{t_1'',t_2'',\dots,t''_{\tau_2}\}$.
	\end{itemize}
Since $(t_1'',t_2'',\dots,t''_{\tau_2})$ is a substring of $(t_1',t_2',\dots,t'_{\tau_1})$, $(t_1'',t_2'',\dots,t''_{\tau_2})$ satisfies (i).

Let $\tau_3 = \tau_2/(\phi_3(w_0)+1)$.
By \cref{claim:changingdecreasing}, there exists a substring $(t_1''',t_2''',\dots,t_{\tau_3}''')$ of $(t_1'',t_2'',\dots,t''_{\tau_2})$ such that 
	\begin{itemize}
		\item[(iii)] for every $\alpha \in [\alpha^*_{t_1'}]$ and $t,t' \in \{t_1''',t_2''',\dots,t'''_{\tau_3}\}$, $V(M^{(t)}_\alpha)=V(M^{(t')}_\alpha)$.
	\end{itemize}
Since $(t_1''',t_2''',\dots,t'''_{\tau_3})$ is a substring of $(t_1'',t_2'',\dots,t''_{\tau_2})$, $(t_1''',t_2''',\dots,t'''_{\tau_3})$ satisfies (i)-(iii).

Let $\tau_4 = \frac{\tau_3}{f(\eta_5)w_0^2+1}$.
Since $(t_1''',t_2''',\dots,t'''_{\tau_3})$ satisfies (iii), for every $\alpha \in [\alpha^*_{t_1'}]$ and $1 \leq \beta < \beta' \leq \tau_3$, the sum of the entries of $a^{(t'''_{\beta'})}_\alpha$ is at most the sum of the entries of $a^{(t'''_{\beta})}_\alpha$.
By \cref{claim:pseudosignaturenumbernonzero}, the sum of the entries of $a^{(t'''_{\beta})}_\alpha$ is at most $f(\eta_5)w_0$.
So there exists a substring $(q_0,q_1,q_2,\dots,q_{\tau_4-1})$ of $(t'''_1,t'''_2,\dots,t'''_{\tau_3})$ such that 
	\begin{itemize}
		\item[(iv)] for every $\alpha \in [\alpha^*_{t_1'}]$ and $t,t' \in \{q_0,q_1,\dots,q_{\tau_4-1}\}$, $V(M^{(t)}_\alpha)=V(M^{(t')}_\alpha)$ and $a^{(t)}_\alpha=a^{(t')}_\alpha$.
	\end{itemize}
Since $(q_0,q_1,\dots,q_{\tau_4-1})$ is a substring of $(t_1''',t_2''',\dots,t'''_{\tau_3})$, $(q_0,q_1,\dots,q_{\tau_4-1})$ satisfies (i)-(iv).

Since $(q_0,q_1,\dots,q_{\tau_4-1})$ satisfies (i) and (iv), there exists $\beta^* \in \{0,1\}$ such that $\alpha^*_{q_{\beta^*}} \geq \alpha^*_{t_1'}+1$.
For each $\gamma \in [\tau_4-1]$, redefine $q_\gamma$ to be $q_{\beta^*+\gamma-1}$. 
Then $(q_1,q_2,\dots,q_{\tau_4-1})$ satisfies (i)-(iv) and $\alpha^*_{q_1} \geq \alpha^*_{t_1'}+1$.
For each $\beta \in [2,\tau_4-1]$, by taking $(t_1,t_2,\alpha^*,\alpha)=(q_1,q_{\beta},\alpha^*_{q_1},\alpha^*_{t_1'})$ in \cref{claim:truncationremaining}, since $(q_1,\dots,q_{\tau_4-1})$ satisfies (iv), there exists $\gamma^*_\beta \in [\alpha^*_{t_1'}+1,\alpha^*_{q_1}]$ such that $a^{(q_{\beta})}_{\gamma^*_\beta}$ is a nonzero sequence. 
For each $\beta \in [2,\tau_4-1]$, we choose $\gamma^*_\beta \neq \alpha^*_{q_1}$ if possible.

Let $\tau_5 = (\tau_4-1)/2$.
By \cref{claim:zerosignature}, if $\gamma^*_{\tau_5} \neq \alpha^*_{q_1}$, then $a^{(q_\beta)}_{\gamma^*_{\tau_5}}$ is a nonzero sequence for every $\beta \in [\tau_5]$.
Since we choose $\gamma^*_\beta \neq \alpha^*_{q_1}$ for each $\beta \in [2,\tau_4-1]$ if possible, we know if $\gamma^*_{\tau_5} = \alpha^*_{q_1}$, then $\gamma^*_\beta=\alpha^*_{q_1}$ for every $\beta \in [\tau_5,\tau_4-1]$ by \cref{claim:zerosignature}.
Hence there exists a substring $(q_1',q_2',\dots,q'_{\tau_5})$ of $(q_1,q_2,\dots,q_{\tau_4-1})$ such that 
\begin{itemize}
	\item[(v)] there exists $\gamma^* \in [\alpha^*_{t_1'}+1,\alpha^*_{q_1}]$ such that $a^{(t)}_{\gamma^*}$ is a nonzero sequence for every $t \in \{q_1',q_2',\dots,q'_{\tau_5}\}$.
\end{itemize}
Since $(q_1',q_2',\dots,q'_{\tau_5})$ is a substring of $(q_1,q_2,\dots,q_{\tau_4-1})$, $(q_1',q_2',\dots,q'_{\tau_5})$ satisfies (i)-(v).
Note that $\alpha^*_{q_1'} \geq \gamma^*$.

Let $\tau_6 = \frac{\tau_5}{2^{w_0} \cdot (\phi_3(w_0)+1) \cdot (f(\eta_5)w^2_0+1)}$.
By \cref{claim:changingdecreasing,claim:pseudosignaturenumbernonzero,claim:zerosignature}, there exists a substring $(q_1'',q_2'',\dots,q''_{\tau_6})$ of $(q_1',q_2',\dots,q'_{\tau_5})$ satisfying (i)-(v) and
\begin{itemize}
	\item[(vi)] for every $\alpha \in [\alpha^*_{q_1'}]$ and $t,t' \in \{q''_1,q''_2,\dots,q''_{\tau_6}\}$, $V(M^{(t)}_\alpha)=V(M^{(t')}_\alpha)$ and $a^{(t)}_\alpha=a^{(t')}_\alpha$.
\end{itemize}
Since $(q_1'',q_2'',\dots,q''_{\tau_6})$ satisifies (i) and (vi), there exists $\beta_1^* \in \{1,2\}$ such that $a^*_{q''_{\beta_1^*}} \geq \alpha^*_{q_1'} +1 \geq \gamma^*+1$. 
For every $\gamma \in [\tau_6-1]$, redefine $q''_\gamma = q''_{\gamma-1+\beta_1^*}$.
So there exists the smallest integer $\xi^* \in [\gamma^*+1,\lvert V(G) \rvert]$ such that $a^{(q_1'')}_{\xi^*}$ is a nonzero sequence.
By (vi) and \cref{claim:zerosignature}, for every $\alpha \in [\xi^*-1]$, $a^{(q_k'')}_\alpha$ is identical for every $k \in [\tau_6-1]$.

Let $Z_0=\{\gamma \in [\alpha^*]: a^{(1)}_\gamma$ is a nonzero sequence$\}$.
So $\lvert Z_0 \rvert = n_0$.
Let $Z'_0=\{\gamma \in [\xi^*]: a^{(q_1'')}_\gamma$ is a nonzero sequence$\}$.
Let $n_0'=\lvert Z_0' \rvert$. 
Recall that for every $\alpha \in [\alpha^*-1]$, $a_\alpha^{(1)}=a_\alpha^{(2)}=\dots=a_\alpha^{(\kappa)}$.
Since $\xi^* \geq \gamma^*+1 \geq \alpha^*+1$, $Z_0-\{\alpha^*\} \subseteq Z_0'$.
By (v), $\{\gamma^*,\xi^*\} \subseteq Z_0'-Z_0$.
So $n_0' \geq n_0+1$.
In particular, $w_0-n_0 \geq 1$.
By the minimality of $w_0-n_0$, $\tau_6-1 \leq h(w_0-n_0') \leq h(w_0-n_0-1)$.
But $\tau_6-1 \geq h(w_0-n_0-1)+1$ since $\kappa \geq h(w_0-n_0)+1$, a contradiction.
This proves the claim.
\end{proof}

\begin{claim} \label{claim:KMpath}
Let $j \in [\lvert \V \rvert-1]$.
Let $M$ be a $c$-monochromatic component  with $S_M \in \Se_j^\circ$.
Let $\tau \in {\mathbb N}$.
Let $t_1,t_2,\dots,t_\tau$ be elements in $K^*(M)$ such that $t_{i+1} \in V(T_{t_i})-\{t_i\}$ for every $i \in [\tau-1]$.
Then $\tau \leq \eta_6$.
\end{claim}

\begin{proof}
Suppose $\tau \geq \eta_6+1$.
For each $\alpha \in [\tau]$, let $(a^{(\alpha)}_1,a^{(\alpha)}_2,\dots,a^{(\alpha)}_{\lvert V(G) \rvert})$ be the $(i_{t_\alpha},j,M)$-truncation.
Since $\eta_6+1 \geq 2$, by \cref{claim:nonzerotruncation}, there exists $\gamma^*$ such that $a^{(1)}_{\gamma^*}$ is a nonzero sequence.
We may assume that $\gamma^*$ is as small as possible.
So $a^{(1)}_\gamma$ is a zero sequence for every $\gamma \in [\gamma^*-1]$.
By \cref{claim:zerosignature}, for every $\gamma \in [\gamma^*-1]$ and $\alpha \in [\tau]$, $a^{(\alpha)}_\gamma$ is a zero sequence.
By \cref{claim:truncationmaintain}, $\tau \leq h(w_0-1)$.
So $\eta_6 \leq h(w_0-1)-1$, a contradiction.
\end{proof}

For $j \in [\lvert \V \rvert-1]$ and a $c$-monochromatic component $M$  with $S_M \in \Se_j^\circ$, we define $T_M$ to be the rooted tree obtained from $T_{r_M}$ by contracting each maximal subtree of $T_{r_M}$ rooted at a node in $K^*(M)$ containing exactly one node in $K^*(M)$ into a node.
Note that $\lvert V(T_M) \rvert = \lvert K^*(M) \rvert$.

\begin{claim} \label{claim:TMheight}
Let $j \in [\lvert \V \rvert-1]$.
Let $M$ be a $c$-monochromatic component  with $S_M \in \Se_j^\circ$.
If $P$ is a directed path in $T_M$, then $\lvert V(P) \rvert \leq \eta_6$.
\end{claim}

\begin{proof}
By the definition of $T_M$, there exist nodes $t_1,t_2,\dots,t_{\lvert V(P) \rvert}$ in $K^*(M)$ such that for each $\alpha \in [\lvert V(P) \rvert-1]$, $t_{\alpha+1} \in V(T_{t_\alpha})-\{t_\alpha\}$.
By \cref{claim:KMpath}, $\lvert V(P) \rvert \leq \eta_6$.
\end{proof}

\begin{claim} \label{claim:TMdegree}
Let $j \in [\lvert \V \rvert-1]$.
Let $M$ be a $c$-monochromatic component  with $S_M \in \Se_j^\circ$.
Then the maximum degree of $T_M$ is at most $f(\eta_5)$.
\end{claim}

\begin{proof}
For every node $t \in K^*(M)-\{r_M\}$, let $q_t$ be the node in $K(M)$ such that $t \in \partial T_{j,q_t}$.
Let $x$ be a node of $T_M$.
It suffices to show that the degree of $x$ in $T_M$ is at most $f(\eta_5)$.

Let $R_x$ be the subtree of $T_{r_M}$ in $T$ such that $x$ is obtained by contracting $R_x$.
We say a node $q$ in $R_x$ is \defn{important} if $q=q_t$ for some $t \in K^*(M)-\{r_M\}$.
Note that the degree of $x$ equals the number of nodes $t$ in $K^*(M)-\{r_M\}$ such that there exists an important node $q$ in $R_x$ with $q=q_t$.
Let $r_x$ be the root of $R_x$.
Since $r_x$ is the only element in $V(R_x) \cap K^*(M)$, for every important node $q$, $V(M) \cap Y^{(i_q,\lvert V(T) \rvert+1,s+2)} \cap X_{V(T_q)} = V(M) \cap Y^{(i_{r_x},\lvert V(T) \rvert+1,s+2)} \cap X_{V(T_q)}$.

Suppose there exists an important node $q^*$ and node $t^* \in K^*(M)-\{r_M\}$ with $q^*=q_{t^*}$ such that $A_{L^{(i_{q^*},\lvert V(T) \rvert+1,s+2)}}(Y^{(i_{q^*},\lvert V(T) \rvert+1,s+2)} \cap V(M)) \cap X_{V(T_{t^*})}-X_{t^*} = \emptyset$.
Since $t^* \in K^*(M)-\{r_M\}$, there exists $v \in V(M) \cap (Y^{(i_{t^*},\lvert V(T) \rvert+1,s+2)}-Y^{(i_{q^*},\lvert V(T) \rvert+1,s+2)}) \cap X_{V(T_{t^*})}$.
By \cref{claim:isolation}, since $V(M) \cap X_{q^*} \neq \emptyset$ and $X_{q^*} \cap I_j \subseteq W_3^{(i_{q^*},-1)}$ and $t^* \in \partial T_{j,q^*}$, we have $v \not \in X_{t^*}$.
Since $M$ is connected, there exists a path $P$ in $M$ from $v$ to $V(M) \cap Y^{(i_{q^*},\lvert V(T) \rvert+1,s+2)}$ internally disjoint from $Y^{(i_{q^*},\lvert V(T) \rvert+1,s+2)}$.
Let $u$ be the neighbor of the end of $P$ in $V(M) \cap Y^{(i_{q^*},\lvert V(T) \rvert+1,s+2)}$.
So $u \in A_{L^{(i_{q^*},\lvert V(T) \rvert+1,s+2)}}(V(M) \cap Y^{(i_{q^*},\lvert V(T) \rvert+1,s+2)})$.
Since $A_{L^{(i_{q^*},\lvert V(T) \rvert+1,s+2)}}(Y^{(i_{q^*},\lvert V(T) \rvert+1,s+2)} \cap V(M)) \cap X_{V(T_{t^*})}-X_{t^*} = \emptyset$, $u \not \in X_{V(T_{t^*})}-X_{t^*}$.
Since $u \in V(M)$ and $t^* \in \partial T_{j,q^*}$, $u \not \in X_{t^*}$ by \cref{claim:isolation}.
Since $v \in X_{V(T_{t^*})}-X_{t^*}$, some internal node of $P$ other than $u$ belongs to $X_{t^*}$.
But $V(P) \cap X_{t^*} \subseteq V(M) \cap X_{t^*} \subseteq V(M) \cap X_{V(T_{j,q^*})} = V(M) \cap Y^{(i_{q^*},\lvert V(T) \rvert+1,s+2)} \cap X_{V(T_{j,q^*})}$ by \cref{claim:isolation}.
Hence $P$ is not internally disjoint from $Y^{(i_{q^*},\lvert V(T) \rvert+1,s+2)}$, a contradiction.

Therefore, for every important node $q$ and node $t \in K^*(M)-\{r_M\}$ with $q=q_t$, we have $A_{L^{(i_q,\lvert V(T) \rvert+1,s+2)}}(Y^{(i_q,\lvert V(T) \rvert+1,s+2)} \cap V(M)) \cap X_{V(T_t)}-X_t \neq \emptyset$.
Hence for every $t \in K^*(M)-\{r_M\}$ in which there exists an important node $q$ in $R_x$ with $q=q_t$, 
\begin{align*}
& A_{L^{(i_{r_x},\lvert V(T) \rvert+1,s+2)}}(Y^{(i_{r_x},\lvert V(T) \rvert+1,s+2)} \cap V(M)) \cap X_{V(T_t)}-X_t \\
=\;  & A_{L^{(i_{r_x},\lvert V(T) \rvert+1,s+2)}}(Y^{(i_{r_x},\lvert V(T) \rvert+1,s+2)} \cap V(M) \cap X_{V(T_q)}) \cap X_{V(T_t)}-X_t \\
=\;  & A_{L^{(i_{r_x},\lvert V(T) \rvert+1,s+2)}}(Y^{(i_q,\lvert V(T) \rvert+1,s+2)} \cap V(M) \cap X_{V(T_q)}) \cap X_{V(T_t)}-X_t \\
\supseteq\; & A_{L^{(i_q,\lvert V(T) \rvert+1,s+2)}}(Y^{(i_q,\lvert V(T) \rvert+1,s+2)} \cap V(M) \cap X_{V(T_t)}) \cap X_{V(T_t)}-X_t \\
=\; & A_{L^{(i_q,\lvert V(T) \rvert+1,s+2)}}(Y^{(i_q,\lvert V(T) \rvert+1,s+2)} \cap V(M)) \cap X_{V(T_t)}-X_t \\
\neq\; & \emptyset.
\end{align*}
Note that 
\begin{align*}
	& \lvert A_{L^{(i_{r_x},\lvert V(T) \rvert+1,s+2)}}(Y^{(i_{r_x},\lvert V(T) \rvert+1,s+2)} \cap V(M)) \cap X_{V(T_{r_x})} - X_{r_x} \rvert \\
	\leq\; & \lvert N^{\geq s}(Y^{(i_{r_x},\lvert V(T) \rvert+1,s+2)} \cap V(M) \cap X_{V(T_{r_x})}) \rvert \\
	\leq\; & f(\lvert Y^{(i_{r_x},\lvert V(T) \rvert+1,s+2)} \cap V(M) \cap X_{V(T_{r_x})} \rvert) \\
\leq\; & f(\eta_5)
\end{align*}
by \cref{claim:sizebeltfull}.
So there are at most $f(\eta_5)$ nodes $t \in K^*(M)-\{r_M\}$ in which there exists an important node $q$ in $R_x$ with $q=q_t$.
Therefore, the degree of $x$ is at most $f(\eta_5)$.
\end{proof}

\begin{claim} \label{claim:componentsizemiddlefinal}
If $M$ is a $c$-monochromatic component  with $V(M) \cap \bigcup_{j=1}^{\lvert \V \rvert-1}I_j^\circ \neq \emptyset$, then $\lvert V(M) \rvert \leq \eta_7$.
\end{claim}

\begin{proof}
Since $V(M) \cap \bigcup_{j=1}^{\lvert \V \rvert-1}I_j^\circ \neq \emptyset$, there exists $j \in [\lvert \V \rvert-1]$ such that $S_M \in \Se_j^\circ$.
By \cref{claim:TMheight,claim:TMdegree}, $T_M$ is a rooted tree with maximum degree at most $f(\eta_5)$ and the longest directed path in $T_M$ from the root contains at most $\eta_6$ nodes.
So $\lvert K^*(M) \rvert = \lvert V(T_M) \rvert \leq \sum_{k=0}^{\eta_6-1}(f(\eta_5))^k \leq (f(\eta_5))^{\eta_6}$.
By \cref{claim:bddbyK*}, $\lvert V(M) \rvert \leq \lvert K^*(M) \rvert \cdot \eta_5 \leq \eta_5 \cdot (f(\eta_5))^{\eta_6}=\eta_7$.
\end{proof}

We are now ready to complete the proof of the main lemma.

\begin{proof}[Proof of \cref{no apex 0}.]
Let $M$ be a $c$-monochromatic component.
If $V(M) \cap I_j^\circ=\emptyset$ for every $j \in [\lvert \V \rvert-1]$, then $\lvert V(M) \rvert \leq \eta_4 \leq \eta^*$ by \cref{claim:centralcomponentsize}.
If $V(M) \cap I_j^\circ \neq \emptyset$ for some $j \in [\lvert \V \rvert-1]$, then $\lvert V(M) \rvert \leq \eta_7 \leq \eta*$ by \cref{claim:componentsizemiddlefinal}.
Therefore, $\lvert V(M) \rvert \leq \eta^*$.
\end{proof}

\section{Allowing Apex Vertices}
\label{AllowingApexVertices}

The following lemma implies \cref{ltwmain} by taking $k=1$ and $Y_1=\emptyset$.

\begin{lemma} 
\label{bounded layered tw}
For all $s,t,w,k,\xi\in\mathbb{N}$, there exists $\eta^*\in\mathbb{N}$ such that for every graph $G$ containing no $K_{s,t}$-subgraph, if $Z$ is a subset of $V(G)$ with $\lvert Z \rvert \leq \xi$, $\V=(V_1,V_2,\dots,V_{\lvert \V \rvert})$ is a $Z$-layering of $G$, $(T,\X)$ is a tree-decomposition of $G-Z$ with $\V$-width at most $w$, $Y_1$ is a subset of $V(G)$ with $\lvert Y_1 \rvert \leq k$, $L$ is an $(s,\V)$-compatible list-assignment of $G$ such that $(Y_1,L)$ is an $(s,\V)$-standard pair, then there exists an $L$-coloring $c$ of $G$ such that every $c$-monochromatic component contains at most $\eta^*$ vertices.
\end{lemma}

\begin{proof}
Let $s,t,w,k,\xi\in\mathbb{N}$.
Let $f$ be the function $f_{s,t}$ in \cref{BoundedGrowth}. 
Let $h_0$ be the identity function. 
For $i\in\mathbb{N}$, let $h_i$ be the function defined by $h_i(x):=x+f(h_{i-1}(x))$ for every $x \in \mathbb{N}_0$.
Let $\eta_1 := h_{s+1}(k+\xi)$.
Let $\eta_2$ be the number $\eta^*$ in \cref{no apex} taking $s=s$, $t=t$, $w=w+\xi$ and $\eta=h_{s+2}(k+\xi)+(s+1)\xi$.
Define $\eta^* := \max\{\eta_1,\eta_2\}$.

Let $G$ be a graph with no $K_{s,t}$-subgraph, $Z$ a subset of $V(G)$ with $|Z| \leq \xi$, $\V=(V_1,V_2,\dots,V_{\lvert \V \rvert})$ a $Z$-layering of $G$, $(T,\X)$ a tree-decomposition of $G-Z$ with $\V$-width at most $w$, $Y_1$ a subset of $V(G)$ with $\lvert Y_1 \rvert \leq k$, $L$ an $(s,\V)$-compatible list-assignment of $G$ such that $(Y_1,L)$ is an $(s,\V)$-standard pair.
Say $\X=(X_p: p \in V(T))$.

We may assume that $\lvert \V \rvert \geq 2$, since we may add empty layers into $\V$.
Let $(Y^{(0)},L^{(0)})$ be a $(Z \cup Y_1,0)$-progress of $(Y_1,L)$.
For each $i \in [s+2]$, define $(Y^{(i)},L^{(i)})$ to be an $(N_G^{\geq s}(Y^{(i-1)}),i)$-progress of $(Y^{(i-1)},L^{(i-1)})$.
Note that every $L^{(s+2)}$-coloring of $G$ is an $L$-coloring of $G$.
It is clear that $\lvert Y_1^{(i)} \rvert \leq h_i(\lvert Z \cup Y_1 \rvert) \leq h_i(\xi+k)$ for every $i \in [0,s+2]$ by \cref{BoundedGrowth} and induction on $i$.

\begin{claim}
\label{bounded layered tw claim1}
For every $L^{(s+2)}$-coloring of $G$, every monochromatic component intersecting $Z \cup Y_1$ is contained in $G[Y_1^{(s+1)}]$. 
\end{claim}

\begin{proof}
We shall prove that for each $i \in [s+2]$ and for every $L^{(i)}$-coloring of $G$, every monochromatic component colored $i$ and intersecting $Z \cup Y_1$ is contained in $G[Y^{(i-1)}]$.

Suppose to the contrary that there exist $i \in [s+2]$, an  $L^{(i)}$-coloring of $G$, and a monochromatic component $M$ colored $i$ and intersecting $Z \cup Y_1$ such that $V(M) \not \subseteq Y^{(i-1)}$. 
Since $M$ intersects $Z \cup Y_1 \subseteq Y^{(i-1)}$, we have $V(M) \cap Y^{(i-1)} \neq \emptyset$.
So there exist $y \in V(M) \cap Y^{(i-1)}$ and $v \in N_G(y) \cap V(M) -Y^{(i-1)}$.
In particular, $L^{(i-1)}(y)=\{i\} \subseteq L^{(i)}(v) \subseteq L^{(i-1)}(v)$.
Since $(Y^{(i-1)},L^{(i-1)})$ is an $(s,\V)$-standard pair, (L2) implies that $v \not \in N_G^{<s}(Y^{(i-1)})$.
Since $v \in N_G(y) \subseteq N_G(Y^{(i-1)})$, we have $v \in N_G^{\geq s}(Y^{(i-1)})$.
But since $(Y^{(i)},L^{(i)})$ is an $(N_G^{\geq s}(Y^{(i-1)}),i)$-progress, $i \not \in L^{(i)}(v)$, a contradiction.

This proves that for each $i \in [s+2]$ and for every $L^{(i)}$-coloring of $G$, every monochromatic component colored $i$ and intersecting $Z \cup Y_1$ is contained in $G[Y^{(i-1)}]$. Since every $L^{(s+2)}$-coloring of $G$ is an $L^{(i)}$-coloring of $G$ for every $i \in [s+2]$, every monochromatic component with respect to $L^{(s+2)}$ intersecting $Z \cup Y_1$ is contained in $G[\bigcup_{i=1}^{s+2}Y^{(i-1)}] \subseteq G[Y^{(s+1)}]$.
\end{proof}

Let $G'$ be the graph obtained from $G-Z$ by adding a copy $z_i$ of $z$ into $V_i$ for every $z \in Z$ and $i \in [\lvert \V \rvert]$ and adding an edge $z_iv$ for every edge $zv \in E(G)$ with $z \in Z$ and $v \in V_i$. 
For $i \in [\lvert \V \rvert]$, define $V_i' :=V_i \cup \{z_i: z \in Z\}$. 
Let $\V' :=(V_1',V_2',\dots,V_{\lvert \V \rvert}')$.
Then $\V'$ is a layering of $G'$.
Let $(T,\X')$ be the tree-decomposition of $G'$, where $\X':=(X_p': p \in V(T))$ and $X'_q := X_q \cup \{z_i: z \in Z, i \in [\lvert \V \rvert]\}$ for every $q \in V(T)$. 
Then $(T,\X')$ is a tree-decomposition of $G'$ with $\V'$-width at most $w+\xi$.

Let $Y_1':=Y^{(s+2)} \cup \{z_i: z \in Z, i \in [\lvert \V \rvert]\}$, let $L'(v):=L^{(s+2)}(v)$ for every $v \in V(G)-Z$, and let $L'(z_i):=L^{(s+2)}(z)$ for every $z \in Z$ and $i \in [\lvert \V \rvert]$.
Note that $(Y_1',L')$ satisfies (L1)--(L3).
But $(Y_1',L')$ is possibly not $(s,\V')$-compatible since some $L'(z_i)$ might contain a color that is not allowed for $V_i$.
For each $i \in [\lvert \V' \rvert]$ and $z \in Z$ with $L'(z_i)=\{i'\}$ for some $i' \in [s+2]$ with $i \equiv i' \pmod{s+2}$,  delete $z_i$ from $G'$ and add edges $z_{i-1}v$ for each edge $z_iv \in E(G')$ if $i>1$ (or delete $z_i$ from $G'$ and add edges $z_{i+1}v$ for each $z_iv \in E(G')$ if $i=1$). Note that if $z_i$ is deleted from $G'$, then $z_{i-1}$ and $z_{i+1}$ are not deleted from $G'$ (since $L'(z_{i-1})=L'(z_{i})=L'(z_{i+1})$). 
Let $G''$ be the resulting graph. 
For each $i \in [\lvert \V \rvert]$, let  $V_i'':=V_i' \cap V(G'')$.
Let $\V'':=(V_1'',V_2'',\dots,V_{\lvert \V \rvert}'')$.
So $\V''$ is a layering of $G''$.
Let $Y_1'':=Y_1' \cap V(G'')$ and $L'':=L'|_{V(G'')}$.
Then $L''$ is an $(s,\V'')$-compatible list-assignment of $G''$, and $(Y_1'',L'')$ is an  $(s,\V'')$-standard pair of $G''$. 
Let $X''_q:=X'_q \cap V(G'')$ for every $q \in V(T)$, and let $\X'':=(X''_q: q \in V(T))$.
Then $(T,\X'')$ is a tree-decomposition of $G''$ with $\V''$-width at most $w+\xi$.

\begin{claim}
\label{bounded layered tw claim2}
$G''$ does not contain $K_{s,t}$ as a subgraph.
\end{claim}

\begin{proof}
Suppose to the contrary that some subgraph $H''$ of $G''$ is isomorphic to $K_{s,t}$.
Note that for every $z \in Z$, $N_{G'}(z_i) \cap N_{G'}(z_j)=\emptyset$ for all distinct $i,j \in [\lvert \V'' \rvert]$.
So for every $z \in Z$, no vertex of $G''$ is adjacent to two vertices in $\{z_i \in V(G''): i \in [\lvert \V'' \rvert]\}$.
Hence for every $z \in Z$, each part of the bipartition of $H''$ contains at most one vertex in $\{z_i: i \in [\lvert \V'' \rvert]\}$.
Furthermore, if $z_iz'_j \in E(H'')$ for some $z,z' \in Z$ and $i,j \in [\lvert \V'' \rvert]$, then $z \neq z'$.
Therefore, for every $z \in Z$, $H''$ contains at most one vertex in $\{z_i: i \in [\lvert \V'' \rvert]\}$. 
Note that for every $z \in Z$ and $i \in [\lvert \V'' \rvert]$ with $z_i \in V(G'')$, we have $N_{G''}(z_i) \subseteq N_G(z)$.

Define $H$ to be the subgraph of $G$ obtained from $H''$ by replacing each vertex $z_i \in V(H'')$ (for some $z \in Z$ and $i \in [\lvert \V'' \rvert]$) by $z$.
Then $H$ is isomorphic to $K_{s,t}$, which is a contradiction.
So $G''$ does not contain $K_{s,t}$ as a subgraph.
\end{proof}

Note that for every $s$-segment $S$ of $\V''$, 
$$\lvert Y''_1 \cap S \rvert \leq \lvert Y^{(s+2)}_1 \cap S \rvert + (s+1)\lvert Z \rvert \leq h_{s+2}(k+\xi)+(s+1)\xi.$$
By \cref{no apex}, there exists an $L''$-coloring $c''$ of $G''$ such that every $c''$-monochromatic component contains at most $\eta_2$ vertices.
Let $c$ be the function defined by $c(v):=c''(v)$ for every $v \in V(G)-Z$ and $c(z):=c''(z_i)$ for every $z \in Z$ and some $i \in [\lvert \V'' \rvert]$ with $z_i \in V(G'')$.
Note that for each $z \in Z$, $L''(z_i)$ is a constant 1-element subset of $L^{(s+2)} (z)$ for all $i$, so $c$ is well-defined.
Therefore $c$ is an $L^{(s+2)}$-coloring of $G$ and hence an $L$-coloring of $G$.

Let $M$ be a $c$-monochromatic component.
Since $c$ is an $L^{(s+2)}$-coloring of $G$, if $V(M) \cap (Z \cup Y_1) \neq \emptyset$, then $M$ is contained in $G[Y_1^{(s+1)}]$ by \cref{bounded layered tw claim1}, which contains at most $h_{s+1}(k+\xi) =\eta_1 \leq \eta^*$ vertices.
If $V(M) \cap (Z \cup Y_1) = \emptyset$, then $M$ is a $c''$-monochromatic component contained in $G''$, so $M$ contains at most $\eta_2 \leq \eta^*$ vertices.
This proves that $c$ has clustering at most $\eta^*$. 
\end{proof}

\section{Open Problems}
\label{sec:openproblems}

The 4-Color Theorem \citep{RSST97} is best possible, even in the setting of clustered coloring. That is, for all $c$ there are planar graphs for which every 3-coloring has a monochromatic component of size greater than $c$; see \citep{WoodSurvey}. These examples have unbounded maximum degree. This is necessary since  \citet{EJ14} proved that every planar graph with bounded maximum degree is 3-colorable with bounded clustering. The  examples mentioned in \cref{Introduction}, in fact, contain large $K_{2,t}$ subgraphs. The following question naturally arises: does every planar graph with no $K_{2,t}$ subgraph have a 3-coloring with clustering $f(t)$, for some function $f$?  More generally, is every graph with layered treewidth $k$ and with no $K_{2,t}$ subgraph 3-colorable with clustering $f(k,t)$, for some function $f$? For $s\geq2$, is every graph with layered treewidth $k$ and with no $K_{s,t}$ subgraph $(s+1)$-colorable with clustering $f(k,s,t)$, for some function $f$? In our companion paper \citep{LW2} we prove an affirmative answer to the weakening of this question with ``layered treewidth'' replaced by ``treewidth'' for $s\geq 1$.

\bigskip

\noindent{\bf Note:} When a version of this paper was under review, \citet{d} answered the first question mentioned above in positive by proving that for any surface $\Sigma$ and positive integer $t$, every $P_t''$-subgraph free graph that can be drawn in $\Sigma$ is 3-colorable with bounded clustering, where $P_t''$ is the graph obtained from a $t$-vertex path by adding two vertices adjacent to all vertices in this path.
In fact, his result applies to list-coloring.
We refer interested readers to \cite{d}.

\bigskip

\noindent{\bf Acknowledgement:} The authors thank the anonymous reviewer for their careful reading and suggestions.

\printindex
\end{document}